\documentclass{amsart}

\usepackage{amsmath}
\usepackage{amsfonts}
\usepackage{amssymb}
\usepackage{amsthm}
\usepackage{comment}
\usepackage{epsfig}
\usepackage{psfrag}
\usepackage{mathrsfs}
\usepackage{amscd}
\usepackage[all]{xy}
\usepackage{rotating}
\usepackage{lscape}
\usepackage{amsbsy}
\usepackage{verbatim}
\usepackage{moreverb}
\usepackage{color}
\usepackage{bbm}
\usepackage{eucal}
\usepackage{soul}
\usepackage[normalem]{ulem}
\usepackage{xcolor}
\usepackage{soul}
\newcommand{\mathcolorbox}[2]{\colorbox{#1}{$\displaystyle #2$}}

\usepackage{tikz}
\usetikzlibrary{patterns,shapes.geometric,arrows,decorations.markings,arrows.meta}
\usepackage{tikz-3dplot}

\usepackage{caption}
\usepackage{subcaption}

\colorlet{lightgray}{black!15}

\tikzset{->-/.style={decoration={
  markings,
  mark=at position .5 with {\arrow{>}}},postaction={decorate}}}
\tikzset{midarrow/.style={decoration={
    markings,
    mark=at position {#1} with {\arrow{>}}},postaction={decorate}}}

\newtheorem{theorem}{Theorem}[section]
\newtheorem{prop}[theorem]{Proposition}
\newtheorem{lemma}[theorem]{Lemma}
\newtheorem{cor}[theorem]{Corollary}

\theoremstyle{definition}
\newtheorem{definition}[theorem]{Definition}

\newtheorem{observation}[theorem]{Observation}
\newtheorem{construction}[theorem]{Construction}
\newtheorem{terminology}[theorem]{Terminology}
\newtheorem{remark}[theorem]{Remark}
\newtheorem{example}[theorem]{Example}
\newtheorem{q}[theorem]{Question}
\newtheorem{notation}[theorem]{Notation}
\newtheorem{criterion}[theorem]{Criterion}

\theoremstyle{remark}

\definecolor{orange}{rgb}{.95,0.5,0}
\definecolor{light-gray}{gray}{0.75}
\definecolor{brown}{cmyk}{0, 0.8, 1, 0.6}
\definecolor{plum}{rgb}{.5,0,1}

\DeclareMathOperator{\Link}{\sf Link}
\DeclareMathOperator{\Fin}{\sf Fin}

\DeclareMathOperator{\Vect}{\cV{\sf ect}}

\DeclareMathOperator{\pr}{\mathsf{pr}}
\DeclareMathOperator{\ev}{\mathsf{ev}}

\DeclareMathOperator{\uno}{\mathbbm{1}}

\DeclareMathOperator{\Alg}{\sf Alg}
\DeclareMathOperator{\CAlg}{\sf CAlg}

\DeclareMathOperator{\unzip}{\sf Unzip}

\DeclareMathOperator{\TwAr}{\sf TwAr}
\DeclareMathOperator{\cSpan}{\sf cSpan}

\DeclareMathOperator{\Psh}{\sf PShv}

\DeclareMathOperator{\Aut}{\sf Aut}
\DeclareMathOperator{\colim}{{\sf colim}}

\DeclareMathOperator{\limit}{{\sf lim}}

\DeclareMathOperator{\Fun}{{\sf Fun}}

\DeclareMathOperator{\Map}{{\sf Map}}

\DeclareMathOperator{\exit}{\sf Exit}
\DeclareMathOperator{\Exit}{\bcE{\sf xit}}

\DeclareMathOperator{\cylr}{{\sf Cylr}}
\DeclareMathOperator{\cylo}{{\sf Cylo}}

\DeclareMathOperator{\Cat}{{\sf Cat}}

\DeclareMathOperator{\Ar}{{\sf Ar}}

\DeclareMathOperator{\psh}{\sf PShv}

\DeclareMathOperator{\Diff}{{\sf Diff}}

\DeclareMathOperator{\op}{\mathsf{op}}

\DeclareMathOperator{\bun}{\sf Bun}

\DeclareMathOperator{\cMfld}{{\sf c}\cM\mathsf{fld}}
\DeclareMathOperator{\cBun}{{\sf c}\cB\mathsf{un}}
\DeclareMathOperator{\Bun}{\cB\mathsf{un}}

\DeclareMathOperator{\sd}{\mathsf{sd}}

\DeclareMathOperator{\cls}{\mathsf{cls}}
\DeclareMathOperator{\act}{\mathsf{act}}
\DeclareMathOperator{\rf}{\mathsf{ref}}

\DeclareMathOperator{\opn}{\mathsf{open}}
\DeclareMathOperator{\emb}{\mathsf{emb}}

\DeclareMathOperator{\cbl}{\mathsf{cbl}}

\DeclareMathOperator{\pcbl}{\mathsf{p.cbl}}

\DeclareMathOperator{\Mfd}{{\cM}\mathsf{fd}}
\DeclareMathOperator{\cMfd}{{\sf c}{\cM}\mathsf{fd}}
\DeclareMathOperator{\Mfld}{{\cM}\mathsf{fd}}

\DeclareMathOperator{\Emb}{\mathsf{Emb}}

\DeclareMathOperator{\strat}{\mathsf{Strat}}
\DeclareMathOperator{\Strat}{\cS\mathsf{trat}}
\DeclareMathOperator{\kan}{\mathsf{Kan}}

\DeclareMathOperator{\poset}{\mathsf{Poset}}

\DeclareMathOperator{\spaces}{\cS\mathsf{paces}}
\DeclareMathOperator{\Spaces}{\cS\mathsf{paces}}

\DeclareMathOperator{\Disk}{\cD{\mathsf{isk}}}

\DeclareMathOperator{\cDisk}{{\sf c}\cD{\mathsf{isk}}}
\DeclareMathOperator{\sing}{\mathsf{Sing}}

\DeclareMathOperator{\vfr}{\sf vfr}
\DeclareMathOperator{\fr}{\sf fr}
\DeclareMathOperator{\sfr}{\sf sfr}

\DeclareMathOperator{\Bord}{\cB{\sf ord}}

\DeclareMathOperator{\Corr}{{\sf Corr}}

\DeclareMathOperator{\BO}{{\mathsf BO}}

\def\ot{\otimes}

\DeclareMathOperator{\oo}{\infty}

\DeclareMathOperator{\tr}{\triangleright}
\DeclareMathOperator{\tl}{\triangleleft}

\newcommand{\lag}{\langle}
\newcommand{\rag}{\rangle}

\newcommand{\w}{\widetilde}
\newcommand{\un}{\underline}
\newcommand{\ov}{\overline}

\newcommand{\ra}{\rightarrow}
\newcommand{\la}{\leftarrow}
\newcommand{\xra}{\xrightarrow}
\newcommand{\xla}{\xleftarrow}

\def\cA{\mathcal A}\def\cB{\mathcal B}\def\cC{\mathcal C}\def\cD{\mathcal D}
\def\cE{\mathcal E}\def\cF{\mathcal F}
\def\cJ{\mathcal J}\def\cK{\mathcal K}\def\cL{\mathcal L}
\def\cM{\mathcal M}\def\cO{\mathcal O}\def\cP{\mathcal P}
\def\cR{\mathcal R}\def\cS{\mathcal S}\def\cT{\mathcal T}
\def\cU{\mathcal U}\def\cV{\mathcal V}\def\cX{\mathcal X}
\def\cZ{\mathcal Z}

\def\DD{\mathbb D}

\def\PP{\mathbb P}
\def\RR{\mathbb R}

\def\ZZ{\mathbb Z}

\def\sB{\mathsf B}\def\sC{\mathsf C}\def\sD{\mathsf D}

\def\sO{\mathsf O}
\def\sS{\mathsf S}\def\sT{\mathsf T}
\def\sV{\mathsf V}

\def\bdelta{\mathbf\Delta}

\def\fB{\frak B}\def\fC{\frak C}\def\fD{\frak D}
\def\fF{\frak F}

\def\bcD{\boldsymbol{\mathcal D}}\def\bcE{\boldsymbol{\mathcal E}}

\DeclareMathOperator{\Stri}{\boldsymbol{\cS}{\sf tri}}
\DeclareMathOperator{\btheta}{\boldsymbol{\Theta}}

\begin{document}

\title{Factorization homology I: higher categories}
\author{David Ayala}
\author{John Francis}
\author{Nick Rozenblyum}

\address{Department of Mathematics\\Montana State University\\Bozeman, MT 59717}
\email{david.ayala@montana.edu}
\address{Department of Mathematics\\Northwestern University\\Evanston, IL 60208}
\email{jnkf@northwestern.edu}
\address{Department of Mathematics\\University of Chicago\\Chicago, IL 60637}
\email{nick@math.uchicago.edu}
\thanks{DA was supported by the National Science Foundation under award 1507704. JF was supported by the National Science Foundation under awards 1207758 and 1508040. Parts of this paper were written while JF was a visitor at the Mathematical Sciences Research Institute and at Universit\'e Pierre et Marie Curie.}

\begin{abstract}
We construct a pairing, which we call factorization homology, between framed manifolds and higher categories. The essential geometric notion is that of a vari-framing of a stratified manifold, which is a framing on each stratum together with a coherent system of compatibilities of framings along links between strata.
Our main result constructs labeling systems on disk-stratified vari-framed $n$-manifolds from $(\infty,n)$-categories. These $(\oo,n)$-categories, in contrast with the literature to date, are not required to have adjoints. 
This allows the following conceptual definition: the factorization homology 
\[
\int_M\mathcal{C}
\]
of a framed $n$-manifold $M$ with coefficients in an $(\infty,n)$-category $\mathcal{C}$ is the classifying space of $\cC$-labeled disk-stratifications over $M$. 
\end{abstract}



\maketitle

\tableofcontents

\section*{Introduction}

In this work, we construct the factorization homology of $n$-manifolds with coefficients in $(\oo,n)$-categories. We posit  this forms the fundamental relation between manifold topology and higher category theory, answering a question which we now motivate and describe.

\smallskip

In 1988, Atiyah \cite{atiyah} proposed a mathematical framework for topological quantum field theory modeled on Segal's earlier axioms for conformal field theory \cite{segalconformal}. An explosion in physically motivated topology over the previous five years informed his proposal. These advances were carried out by new studies of gauge theory; this includes both Atiyah \& Bott's analysis of the Morse theory of the Yang--Mills functional to compute the cohomology of algebraic bundles on Riemann surfaces in \cite{atiyahbott}, as well as Donaldson's revolution in smooth 4-manifold topology based on the self-dual Yang--Mills equations in \cite{donaldson}. These advances led to Atiyah's open challenge, given at the Hermann Weyl Symposium, to marry other low-dimensional topology invariants, such as the Casson invariant and the Jones polynomial, with mathematical physics. Witten answered this challenge by introducing Chern--Simons theory \cite{witten}, a gauge theory in which the standard Yang--Mills action is replaced by the Chern--Simons 3-form of the connection. At a physical level of rigor, Witten showed that the Jones polynomial is the expectation associated to loop observables in Chern--Simons theory. 

\smallskip

Atiyah's proposed axioms were most influenced by Chern--Simons theory and Witten's notion of topologically invariant quantum field theories. In Chern--Simons theory, a 3-manifold $M$ is assigned an element in a vector space $Z(\partial M)$ associated to its boundary. 
This association satisfies a local-to-global expression with respect to surgery on manifolds. 
Atiyah added axioms to encode this surgery-locality in terms of Thom's cobordism theory: in this now ubiquitous definition, a topological quantum field theory is a functor from a category whose object are $(n-1)$-manifolds and whose morphisms are $n$-dimensional cobordisms.

\smallskip

By the early 1990s, it had become clear that if codimension-1 boundary conditions form a vector space, then higher codimension defects should correspond to higher categorical objects. Earliest publications of this include works of Lawrence \cite{lawrence}, Freed \cite{freed}, Ooguri \cite{ooguri1} and \cite{ooguri2}, and Crane--Yetter \cite{craneyetter}, but the insight is often attributed collectively to many mathematicians, including Baez, Dolan, Kapranov, Kazhdan, Reshetikhin, Turaev, Voevodsky, and others. Relevant works include \cite{baezdolan}, \cite{kv}, \cite{rt}, \cite{walker}, and \cite{kapustin}; see in particular, Freed's work on quantum groups \cite{freed2} and the Baez--Dolan cobordism hypothesis, which specified many features that should be true of this connection between manifolds and higher category theory in terms of an extensive surgery-locality based on Morse theory.

\smallskip

While it appeared clear that higher categories bore a close connection to field theory, a basic question remained unanswered: what is it that connects them? For instance, field theories are defined by integration -- is there integration on the categorical side? Or does the category theory only serve as an elaborate system of bookkeeping?

\smallskip

During this same period, Beilinson \& Drinfeld introduced a beautiful theory of chiral and factorization algebras, an algebro-geometric approach to chiral conformal field theory; their work was finally published a decade later in \cite{bd}. Therein, they devised a fantastic procedure---chiral homology---in which one integrates a chiral algebra coherently over all configuration spaces of a curve to produce a conformal field theory. The conformal blocks of the field theory occur as the zeroth chiral homology group. They defined algebro-geometric forms of standard vertex algebras, and calculated their chiral homologies in several cases of especial interest, including lattice algebras and central extensions of enveloping algebras of Lie algebras.

\smallskip

This theory of factorization algebras, and of coherently integrating over all configuration spaces at once, inspired and connected with a number of works in differing areas. These include: quantum groups in \cite{bfs}; manifold topology and mapping spaces in \cite{HA}, \cite{salvatore}, \cite{segallocal}, and \cite{fact}; $\ell$-adic cohomology and bundles on curves in \cite{dennisjacob}. In mathematical physics, Costello \cite{kevin} developed a rigorous system of renormalization for perturbative quantum field theories based on the Batalin--Vilkovisky formalism \cite{bv}. Analyzed in great depth by Costello \& Gwilliam \cite{kevinowen}, the quantum observables in these renormalized theories obtain the structure of a factorization algebra in a topological sense. Assuming the theory is perturbative, then the global observables are computed by an analogous process of factorization homology: one integrates over all embedded disks, rather than configuration spaces. This theory accommodates a wealth of examples, from perturbative Chern--Simons theory to twisted supersymmetric gauge theories \cite{kevin2}.

\smallskip

Consequently, for conformal field theory as well as for perturbative quantum field theory, our basic motivating question has an answer: there is integration on the categorical side, and it is chiral/factorization homology. The field theory itself is implemented by integration over manifolds from an algebraic input, which is a chiral/factorization/$\cE_n$-algebra. 
However, this forms only a partial solution to our basic question, because Chern--Simons theory and the other field theories involved are not perturbative.
Their perturbative sectors do not account for the entire theories.
Said differently, the global observables in these theories are not computed as the factorization homology of the point-local observables, which are organized as an $\cE_n$-algebra:
there is always a natural assembly map from the factorization homology of the local observables to the global observables,
\[
\int_M {\sf Obs}(\RR^n)
\longrightarrow 
{\sf Obs}(M)~,
\]
but this map need not be an equivalence. Failure for this map to be an equivalence is to be expected if, for example, the space of maps in a quantized sigma model is larger than the formal neighborhood of the subspace of classical solutions.

\smallskip

From the point of view of the cobordism hypothesis of Baez--Dolan, further developed by Lurie and Hopkins--Lurie in \cite{cobordism} after Costello \cite{cycat}, certain higher categories are given by the Morita theory of $\cE_n$-algebras. These account for those TQFTs whose value on a point is Morita equivalent to an $\cE_n$-algebra, i.e., to an $(\oo,n)$-category with a single object and a single $k$-morphism for $k<n$. (The collection of $n$-morphisms then forms an $\cE_n$-algebra, just as the collection of 1-morphisms in a category with a single object forms an algebra.) For this special class of $(\oo,n)$-categories, the output field theory, as expected by the cobordism hypothesis, can be implemented by taking factorization homology of $\cE_n$-algebras. Consequently, we can now give a more precise rephrasing of our basic question.

\begin{q}
What higher codimensional enhancement of chiral/factorization homology implements topological quantum field theory?
\end{q}

That is, we wish to solve the theoretical problem of comparing category theory and field theory, after Baez--Dolan, within the philosophy of Beilinson--Drinfeld. In the narrative we pursue in this introduction, this theory should fill the last entry in the following table.

\medskip

\begin{center}
    \begin{tabular}{|p{4cm}|  p{4cm} |p{6cm}| }
    \hline
        {\bf Physics} & {\bf Algebra} & {\bf Integration} \\ \hline
    \begin{center}CFT \end{center} & \begin{center}chiral algebra\end{center}  & \begin{center}chiral homology (\cite{bd})\end{center} \\ \hline
    \begin{center}perturbative TQFT \end{center} & \begin{center}$\cE_n$-algebra/stack\end{center}  & \begin{center}factorization homology (\cite{HA}, \cite{fact}) \end{center}   \\ \hline
    \begin{center}perturbative QFT \end{center}  & \begin{center}factorization algebra\end{center}  & \begin{center}factorization homology (\cite{kevinowen})\end{center}  \\ \hline
       \begin{center} TQFT  \end{center} & \begin{center}$(\oo,n)$-category\end{center}  &  \\ \hline

    \end{tabular}
\end{center}

\medskip

Our proposed solution, which we again call factorization homology, has a simple summary: rather than integrating over configuration spaces -- i.e., over the moduli space of finite subsets -- integrate over a moduli space of disk-stratifications. The conclusion of this paper is that this heuristic definition can be made well-defined.

\smallskip

Before describing what technical features this problem absorbs and how they are overcome, we first make an observation and comment. In the diagram above, we have listed $(\oo,n)$-categories instead of $(\oo,n)$-categories with duals or adjoints. As far as we are aware, the TQFT literature to date has uniformly emphasized the necessity of adjoints in the category theory; these adjoints mirror categorically the Morse theory and surgery-locality of Atiyah's axioms and the cobordism hypothesis after \cite{baezdolan} and \cite{cobordism}. However, examples such as Donaldson theory have not fit into these axioms. There are genuine topological obstructions to defining the requisite Floer theory on the full bordism category; see \cite{unitary}. In particular, the monopole Floer homology of Kronheimer--Mrowka \cite{kronheimermrowka} is defined only on a bordism category whose morphisms are connected bordisms. We are hopeful that these important Floer theories may still fit in the factorization paradigm after Beilinson \& Drinfeld, exactly because we can fill in the missing square in the above diagram without requiring adjoints in the coefficient $(\infty,n)$-categories.

\smallskip

For instance, Donaldson theory (\cite{donaldson2}, \cite{kronheimer}, \cite{km1}, \cite{km2}, \cite{km3}) and Seiberg--Witten theory (\cite{taubes}) are closely related to embedding spaces such as $\Emb(\Sigma, M)$, of surfaces $\Sigma$ in a 4-manifold $M$. However, as discussed in \cite{fact}, the factorization homology of $\cE_4$-algebras is only sensitive to the space $T_\infty\Emb(\Sigma, M)$, the limit of the Goodwillie--Weiss embedding calculus tower \cite{weiss}. Little is known about the convergence of the embedding tower in this case: it is just outside the range of convergence established by Goodwillie--Klein \cite{goodwillieklein}. For instance, it is not established whether this canonical map, at the level of connected components,
\[
\pi_0 \Emb(\Sigma, M)\longrightarrow\pi_0 T_\infty\Emb(\Sigma, M)
\]
is injective or surjective. We imagine the general role of our present techniques in differential topology as advancing beyond the range of convergence of the Goodwillie--Weiss tower.

\smallskip

We now describe our solution. First, we recall the corresponding simpler case in codimension-0, factorization homology with coefficients in an $\cE_n$-algebra. If $A$ is an $\cE_n$-algebra and $M$ is an framed $n$-manifold, one heuristically constructs factorization homology as
\begin{center} $\displaystyle\int_M A  \ \approx  \ \Bigl|A$-labeled $n$-disks in  $M\Bigr|$~,\end{center}
the classifying space of a category, an object of which is a collection of disjointly embedded $n$-disks in $M$ each of which is labeled by a point of $A$. There are several important classes of morphisms.
\begin{enumerate}
\item {\bf compositions}: two disks are embedded in a third disk, and the labels multiply in $A$.
\item {\bf units}: a disk is added to a configuration, labeled by the unit of $A$.
\item {\bf coherence}: disks are moved through an isotopy of embedding.

\end{enumerate}

\medskip

We wish to make a corresponding construction where the $\cE_n$-algebra $A$ is replaced by an $(\oo,n)$-category $\cC$. The factorization homology of a framed $n$-manifold $M$ with coefficients in $\cC$ should be
\begin{center} $\displaystyle\int_M\cC  \ \approx  \ \Bigl|\cC$-labeled disk-stratifications  of  $M\Bigr|$~,\end{center}
the classifying space of a category, an object of which consists of a coherent system of:
\begin{itemize}
\item a stratification of $M$, each closed component of which is a $k$-disk;
\item a $k$-morphism of $\cC$ for each $k$-dimensional component of the stratification of $M$.
\end{itemize}
There are several important classes of morphisms.
\begin{enumerate}
\item\label{one} {\bf refinements/compositions}: a stratum is refined away, forgotten, and the labels are composed.
\item\label{two} {\bf creations/units}: a new stratum is created, labeled by identity morphisms.
\item {\bf coherence}: a stratification is moved via diffeomorphism to another stratification.
\end{enumerate}
This template for making factorization homology is, however, afflicted by the absence of any known model for $(\oo,n)$-categories which can define such a system of labels. Most models for $(\oo,n)$-categories are constructed in terms of presheaves on a combinatorially defined category, such as $\btheta_n$ or the $n$-fold product $\bdelta^n$, and none of these are manifestly suitable for decorating a disk-stratification.

\smallskip

We encountered a similar, easier, obstruction in the guiding simpler case above. One cannot obviously define factorization homology with coefficients in an algebra for the little $n$-cubes operad. One requires an intermediate notion, namely the operad of framed embeddings (which is infinite dimensional but homotopy finite dimensional) and a comparison result, that the two operads are weakly homotopy equivalent. This allows one to Kan extend the $\cE_n$-algebra along the inclusion of rectilinear embeddings into framed embedding without altering the homotopy type; once one has expressed the $\cE_n$-algebra in terms of framed embeddings, the definition is manifestly well-defined.

\smallskip

We solve this issue in our setting in three steps. In the first step, we construct an $\oo$-category of labeling systems for stratifications on framed $n$-manifolds. In the second step, we show that the $\infty$-category of
$(\oo,n)$-categories maps into that of labeling systems. 
In the third step, we define factorization homology with coefficients in the specified labeling systems. 
We elaborate on these steps below.

\smallskip

{\bf First step}: In our antecedent work on striation sheaves \cite{striat}, we constructed an $\oo$-category $\cBun$ whose objects are compact conically smooth stratified spaces and whose morphisms include refinements and stratum-creating maps, exactly as in points (\ref{one}) and (\ref{two}) above. Now, starting from $\cBun$, we restrict to the $\oo$-subcategory $\cDisk\subset \cBun$ of objects which are disk-stratified, as above. We then introduce the notion of a variform framing -- for short, vari-framing -- on a stratified space. A vari-framing consists of a framing on each stratum together with compatibilities between these framings in links of strata. From this, we define $\cDisk^{\vfr}_n$ as the collection of compact disk-stratified manifolds of dimension less or equal to $n$ and equipped with a vari-framing. Lastly, the $\oo$-category of labeling systems is \[\Fun(\cDisk^{\vfr}_n,\spaces)~,\] space-valued functors on vari-framed compact disk-stratified $n$-manifolds.

\smallskip

{\bf Second step}: We use Rezk's presentation \cite{rezk-n} of the $\oo$-category of $(\oo,n)$-categories $\Cat_{(\oo,n)}$ as a full $\oo$-subcategory of $\psh(\btheta_n)$, presheaves on Joyal's category $\btheta_n$ of \cite{joyaltheta}. We construct a functor out of Joyal's category
\begin{equation}\label{e1}
\lag - \rag \colon \btheta_n^{\op} \longrightarrow \cDisk^{\vfr}_n
\end{equation}
which we call {\it cellular realization} (Definition~\ref{def.cellular}).

\smallskip

This cellular realization functor~(\ref{e1}) carries an object in $\btheta_n$ to its associated pasting diagram, which is viewed as a stratified space.  
In a sense, the $\infty$-category $\cDisk^{\vfr}_n$ is crafted just so that this cellular realization functor exists and is fully faithful.
The $\infty$-category $\cMfd^{\vfr}_n$ of \emph{vari-framed stratified $n$-manifolds}, which houses $\cDisk^{\vfr}_n$ as those that are disk-stratified, is likewise designed around the following essential issues involved in connecting higher categories to manifolds.
\begin{itemize}
\item
$\cMfd^{\vfr}_n$ contains moduli spaces of framed smooth $n$-manifolds.
More generally, an object of $\cMfd^{\vfr}_n$ is a stratified space with dimension at most $n$, equipped with a vari-framing.

\item
Morphisms in $\btheta_n$ do not in any sense determine stratified maps between their pasting diagrams.  
Consequently, for there to exist a functor $\btheta_n^{\op} \to \cMfd^{\vfr}_n$ calls for the consideration of an exotic notion of morphisms between stratified spaces in the target, possibly with prescribed tangential structure.

\end{itemize}

\smallskip

{\bf Third step}: Lastly, we left Kan extend from $\cDisk^{\vfr}_n$ to $\cMfd_n^{\vfr}$. That is, factorization homology is the composite
\[
\int: \Cat_{(\oo,n)}\longrightarrow \Fun(\cDisk^{\vfr}_n,\spaces) \longrightarrow \Fun(\cMfd_n^{\vfr},\spaces)
\]
where the first functor is that of the second step, and the second functor is left Kan extension along the inclusion $\cDisk^{\vfr}_n \subset \cMfd_n^{\vfr}$. Equivalently, the factorization homology
\[\int_M \cC\]
is the classifying space of the Grothendieck construction of the composite functor
\[\cDisk^{\vfr}_{n/M}\longrightarrow \cDisk^{\vfr}_n\overset{\cC}\longrightarrow\spaces\]
where the functor $\cC$ is the right Kan extension of $\cC:\btheta_n^{\op}\ra \spaces$ along the cellular realization functor $\btheta_n^{\op}\ra \cDisk^{\vfr}_n$.

\smallskip

We now state the main result of the present work.
This result can be interpreted as a construction of space-valued invariants of smooth framed $n$-manifolds for each $(\infty,n)$-category; this invariant is manifestly functorial and continuous in each of these inputs.

\begin{theorem}\label{main.theorem}
\footnote{
Per the erratum of~\S\ref{sec.erratum}, this functor, which is defined in Definition~\ref{def.fact.homology} in~\S\ref{sec.fact.def}, need not be fully faithful for $n \geq 3$ as was asserted in the previous version of this work.  
}
There is functor from the $\infty$-category of $(\oo,n)$-categories into space-valued functors of vari-framed $n$-manifolds:
\[
\int: \Cat_{(\oo,n)} \longrightarrow \Fun\bigl(\cMfd_n^{\vfr},\spaces\bigr)
~.
\] 
For $n<3$, this functor is fully faithful. 
\end{theorem}

In future works, we apply this higher codimension form of factorization homology to construct topological quantum field theories.

\subsection*{Future works}
This work is the third paper in a larger program, currently in progress. We now outline a part of this program, in order of logical dependency. This part consists of a number of papers, the last of which proves the cobordism hypothesis, after Baez--Dolan \cite{baezdolan}, Costello \cite{cycat}, Hopkins--Lurie (unpublished), and Lurie \cite{cobordism}.
\begin{enumerate}

\item[\bf \cite{aft1}:] {\bf Local structures on stratified spaces}, by the first two authors with Hiro Lee Tanaka, establishes a theory of stratified spaces based on the notion of conical smoothness.
This theory is tailored for the present program, and intended neither to supplant or even address outstanding theories of stratified spaces.
This theory of conically smooth stratified spaces and their moduli is closed under the basic operations of taking products, open cones of compact objects, restricting to open subspaces, and forming open covers, and it enjoys a notion of derivative which, in particular, gives the following: 
\begin{itemize}
\item[~] For the open cone $\sC(L)$ on a compact stratified space $L$, taking the derivative at the cone-point implements a homotopy equivalence between \emph{spaces} of conically smooth automorphisms
\[
\Aut\bigl(\sC(L)\bigr)~\simeq~\Aut(L)~.
\] 
\end{itemize}
This work also introduces the notion of a constructible bundle, along with other classes of maps between stratified spaces.

\item[\bf \cite{striat}:]  {\bf A stratified homotopy hypothesis} proves stratified spaces are parametrizing objects for $\infty$-categories.
Specifically, we construct a functor $\exit\colon \strat \to \Cat_\infty$ and show that the resulting restricted Yoneda functor $\Cat_\infty \to \Psh(\strat)$ is fully faithful.
The image is characterized by specific geometric descent conditions.
We call these presheaves \emph{striation sheaves}. 
We develop this theory so as to construct particular examples of $\infty$-categories by hand from stratified geometry: $\Bun$, $\Exit$, and variations thereof. 
As striation sheaves, $\Bun$ classifies constructible bundles, $\Bun \colon  K\mapsto \{X\xra{\sf cbl} K\}$,
while $\Exit$ classifies constructible bundles with a section.

\item[\bf Present:]
In the present work, we construct a fiberwise tangent classifier $\sT^{\sf fib} \colon \Exit \to \Vect^{\sf inj}$ to an $\infty$-category of vector spaces and injections thereamong.  
We use this to define $\infty$-categories $\cMfd_n^{\vfr}$ of \emph{vari-framed compact $n$-manifolds}, and $\cMfd_n^{\sfr}$ of \emph{solid $n$-framed compact $n$-manifolds}. As a striation sheaf, $\cMfd_n^{\vfr}$ classifies proper constructible bundles equipped with a trivialization of their fiberwise tangent classifier, and $\cMfd_n^{\sfr}$ classifies proper constructible bundles equipped with an injection of their fiberwise tangent classifier into a trivial  $n$-dimensional vector bundle.
We then construct a functor $\fC\colon (\cMfd_n^{\vfr})^{\op} \to \Cat_{(\infty,n)}$ between $\infty$-categories, and use this to define \emph{factorization homology}. This takes the form of a functor between $\infty$-categories
\[
\int \colon \Cat_{(\infty,n)} \longrightarrow \Fun(\cMfd_n^{\vfr}, \Spaces)
\]
Subsequent papers establish analogous results for $(\oo,n)$-categories with adjoints, as they relate to solidly $n$-framed compact manifolds.

\item[\bf \cite{bord}:]
{\bf The cobordism hypothesis}, by the first two authors, proves the cobordism hypothesis.  
Namely, for $\cX$ a symmetric monoidal $(\infty,n)$-category with adjoints and with duals, the space of fully extended (framed) topological quantum field theories is equivalent to the underlying $\infty$-groupoid of $\cX$: 
\[
\Map^{\ot}(\Bord_n^{\fr}, \cX) ~\simeq~ \cX^\sim~.
\]  
The cobordism hypothesis is a limiting consequence of the tangle hypothesis, one form of which states that, for $\ast \xra{\uno} \cC$ a pointed $(\infty,n+k)$-category with adjoints, there is a canonical identification 
$\Map^{\ast/}(\cT{\sf ang}_{n\subset n+k}^{\fr}, \cC) \simeq k{\sf End}_\cC(\uno)$ between the space of pointed functors and the space of $k$-endomorphisms of the point in $\cC$.  
The tangle hypothesis is proved in two steps.
The first step establishes versions of the factorization homology functors above in which the higher categories are replaced by \emph{pointed} higher categories, and the manifolds are replaced by possibly non-compact manifolds.  
The second step shows that the pointed $(\infty,n+k)$-category $\cT{\sf ang}^{\fr}_{n\subset n+k}$, as a copresheaf on $\Mfd_{n+k}^{\sfr}$, is represented by the object $\RR^k$. The bordism hypothesis follows from the tangle hypothesis, represented by the equivalence $\Bord_n^{\fr} \simeq \varinjlim \Omega^k \RR^k$ as copresheaves on $\Mfd_n^{\sfr}$.

\end{enumerate}

\subsection*{Linear overview}
We conclude the introduction by a linear overview of this work, followed by a comparison with spiritually similar works.

\smallskip

{\bf Section 1} recalls the requisite definitions and results on stratified spaces from the antecedent works \cite{striat} and \cite{aft1}. In the joint work \cite{aft1} with Hiro Lee Tanaka, a theory of smoothly stratified spaces founded on the key technical notion of conical smoothness was developed. This technical feature allowed for well-behaved homotopy types of mapping spaces and such bedrock results as an inverse function theorem, an isotopy extension theorem, and the unzipping construction, which is a functorial resolution of singularities. One could take the collective results of \cite{aft1} as meaning that there is a well-behaved theory of smooth moduli of stratified spaces, in particular with smooth parameter spaces and in which all fibers are generic.

\smallskip

In \cite{striat}, we developed this theory further, showing that it extends to a well-behaved theory of {\it singular} moduli of stratified spaces. 
An $\oo$-category $\Bun$ encodes this theory of singular moduli of stratified spaces, with specialization maps from special to generic fibers. A morphism in $\Bun$ can be regarded as a constructible bundle over the standardly stratified interval $\{0\}\subset [0,1]$. 
A close relative, the absolute exit-path $\oo$-category $\Exit$, similarly encodes the theory of pointed singular moduli. The construction of $\Bun$ and $\Exit$ formed the main result of that work. Their existence is non-formal, because 1-morphisms do not obviously compose: one wants to compose by gluing intervals end-to-end, but the resulting total space no longer maps constructibly to the interval. One must resolve singularities and retract floating strata to fix the total space.

\smallskip

Consequently, the existence of $\Bun$ and $\Exit$ as $\oo$-categories requires one to verify horn-filling conditions by hand using a sort of d\'evissage of conically smooth stratified structures. In order to perform this by-hand construction, we broke the problem into two conceptual steps. First, we introduced striation sheaves, sheaves on stratified spaces which satisfy additional descent conditions, and we proved that $\Bun$ and $\Exit$ are striation sheaves. This required showing this theory of singular moduli satisfied descent for blow-ups and for gluing along consecutive strata, among other conditions.
Second, we proved that striation sheaves are equivalent to $\oo$-categories. To do so, we showed that there is a fully faithful functor
\[
\Strat\overset{\exit}\longrightarrow \Cat_{\oo}
\]
given by a complete Segal space form of the exit-path $\oo$-category of Lurie \cite{HA} and MacPherson. Our construction of $\exit$ is by restricting the Yoneda embedding along a functor
\[{\sf st}:\bdelta\hookrightarrow \Strat\]
defined by sends the object $[p]$ to the topological $p$-simplex $\Delta^p$ with the standard stratification -- the conical stratification given by regarding $\Delta^p$ as the $p$-fold cone on a point. By analysis of the homotopy type of conically smooth stratified maps between cones, we obtained that the functor ${\sf st}$ is fully faithful, and thus that there is an embedding of $\psh(\bdelta)$, hence $\oo$-categories, into presheaves on stratified spaces. The result then followed by applying a d\'evissage of stratified spaces, showing that the values of a striation sheaf are determined by two values (on a point and on a 1-simplex) after a combination of resolving singularities and induction on depth of singularity type.

\smallskip

{\bf Section 2} begins the new material of the present work. Intuitively, presheaves on the $\oo$-category $\cBun$ should present a theory of ``$(\oo,\oo)$-categories with pseudoisotopic duals'' in the same way that presheaves on the simplex category $\bdelta$ presents usual $\oo$-categories. However, in this work we do not want to express ``$(\oo,\oo)$-categories with pseudoisotopic duals'' in terms of manifolds; our goal is merely to express $(\oo,n)$-categories in terms of manifolds. Consequently, we modify $\cBun$ in two ways for this purpose.
\begin{enumerate}
\item We restrict to objects of $\cBun$ of dimension less than or equal to $n$.
This has the effect of eliminating noninvertible $k$-morphisms for $k>n$.
\item We introduce a stratified tangential structure -- a vari-framing.
This has the effect of eliminating the duals and pseudoisotopies.
\end{enumerate}
In order to define the vari-framing, in \S\ref{sec.tgt-char} we enhance the usual definition of the tangent bundle of a stratified space (see \S2.1 of \cite{pflaum}) so as to be defined in singular families parametrized by stratified spaces. A vari-framing of a single stratified space is a trivialization of its constructible tangent bundle.
This is a more subtle notion than just a framing on the underlying space.
A vari-framing of a family is then a trivialization of the fiberwise constructible tangent bundle.
We then define $\cMfd_n^{\vfr}$ as the sheaf on stratified spaces which classifies such proper families of vari-framed $n$-manifolds. To verify that $\cMfd_n^{\vfr}$ satisfies the striation sheaf axioms, and so defines an $\oo$-category, we reduce to showing that the functor $\Exit\ra \Bun$ satisfies the exponentiability property of Giraud \cite{giraud} and Conduch\'e \cite{conduche} -- this is shown in the appendix.

\smallskip

{\bf Section 3} equates the topology of the constructions of \S 2 with combinatorics.
First, we cut down the topology by restricting to $\cDisk^{\vfr}_n$, a full $\oo$-subcategory of $\cMfd_n^{\vfr}$ whose objects are disk-stratified. This notion is defined in \S3.3.
The principal construction of this section is a cellular realization functor
\[\xymatrix{\btheta_n^{\op} \ar[rr]^{\lag-\rag}&& \cDisk^{\vfr}_n}\]
from Joyal's category $\btheta_n$, the $n$-fold wreath product of $\bdelta$, the usual simplex category. We construct our cellular realization in two steps.
\begin{itemize}
\item For $n=1$, we directly construct the cellular realization $\btheta_1^{\op} = \bdelta^{\op} \longrightarrow \cDisk_1^{\vfr}$.
\item We construct a functor $\bigl(\cDisk_1^{\vfr}\bigr)^{\wr n}\longrightarrow \cDisk^{\vfr}_n$.
\end{itemize}
This cellular realization then establishes the following pairing of concepts:
\begin{center}
    \begin{tabular}{|p{5cm}|  p{6.8cm}  |}
    \hline
        {\bf Category theory} & {\bf Stratified manifolds}  \\ \hline
    an $(\oo,n)$-category $\cC$ & a vari-framed disk-stratified $n$-manifold $M$ \\ \hline
    a $k$-morphism  & a connected $k$-dimensional stratum    \\ \hline
    composition of $k$-morphisms & merging $k$-dimensional strata by refinement  \\ \hline
        identity $k$-morphisms & creating a $k$-dimensional stratum  \\ \hline
        source \& target maps & eliminating strata by closed morphisms \\ \hline

    \end{tabular}
\end{center}

\smallskip

{\bf Section 4} is formal and short. As a result of \S3, we can extend an $(\oo,n)$-category $\cC$ to a functor
\[
\xymatrix{
\btheta_n^{\op}\ar[d]\ar[rr]^{\cC}&&\spaces\\
\cDisk^{\vfr}_n\ar@{-->}[rru]_{\cC}\\}\] by right Kan extension along $\btheta_n^{\op}\hookrightarrow \cDisk^{\vfr}_n$, which we give the same notation. 
To define factorization homology $\int_M \cC$ for a general vari-framed $n$-manifold $M$, we left Kan extend:
\[
\xymatrix{
\cDisk^{\vfr}_n\ar[d]\ar[rr]^{\cC}
&&
\spaces
\\
\cMfd_n^{\vfr}\ar@{-->}[rru]_{\displaystyle\int\cC}
&&
.
}
\]
This completes the construction of the fully faithful functor $\Cat_{(\oo,n)}\longrightarrow \Fun(\Mfd_n^{\vfr},\spaces)$.

\smallskip

{\bf Section 5} is an appendix concerning some checkable criteria in higher category theory.
This appendix has four parts, each of which is entirely formal and contains essentially no new ideas.  
The first part establishes the notion of a \emph{monomorphism} among $\infty$-categories, which we find to be a convenient way to articulate comparisons among $\infty$-categories.
The second part records some easy facts about $\infty$-categories of cospans.
The third part addresses when base change among $\infty$-categories is a left adjoint.
This is an essential aspect of how we endow fiberwise structures on constructible families of stratified spaces.  
The fourth part simply explains how the univalence condition among Segal $\btheta_n$-spaces is implied by lacking non-trivial higher idempotents.

\subsection*{Comparison with other works}
Our notion of factorization homology is a direct generalization, from the $\cE_n$-algebra case, of the labeled configuration spaces of Salvatore \cite{salvatore} and Segal \cite{segallocal}. These are both special cases of factorization homology, or topological chiral homology, after Lurie \cite{HA}.
An approach to manifold invariants via labeling fine stratifications of a manifold by a higher category was first well-demonstrated in 3-dimensions by Turaev--Viro~\cite{turaev-viro} as state-sums.

\smallskip

The essential notion underlying this work, of defining a homology theory by integrating over disk-stratifications, was earlier conceived by Morrison--Walker \cite{blob}.
Likewise, this is the essential notion for their blob homology. Our factorization homology is thus a spiritual cousin of their blob homology.
However, there are significant differences in concept, execution, and result.
We now survey these differences, with the dual purpose of discussing some features of our present framework that embody certain key developmental choices in our setup.  

\smallskip

First, there is a difference in end result: we prove that an $(\oo,n)$-category defines an input for factorization homology.  To accomplish this takes the combined work of the present paper, of \cite{striat}, and of \cite{aft1}, to fuse combinatorics and differential topology. It makes use of: the introduction of the notion of conical smoothness of stratifications and a host of results about the differential topology thereof; the striation sheaf model of $\oo$-categories; the striation sheaf property of $\Bun$, showing existence of composition of morphisms via resolution of singularities; the homotopy equivalence between conically smooth diffeomorphisms of $\DD^n$ and its space of vari-framings. 
Morrison \& Walker have not yet shown that their blob homology can take as input an $(\oo,n)$-category, with or without duals/adjoints. 
Instead, they conceive their own notion of an $(\oo,n)$-category, with duals, and suggest that examples of interest will naturally fall within their framework.

\smallskip

A technical difference is that our definitions of homology are, in detail, quite different and not easily comparable.
We define factorization homology as a colimit over an $\oo$-category $\cDisk_{n/M}^{\vfr}$ whose morphisms are compositions of four basic types: (1) refining strata, (2) creating strata, (3) isotoping strata, and (4) eliminating strata. Morrison \& Walker define blob homology as a colimit over a poset $\fD(M)$ of stratifications of $M$, the morphisms in which are of type (1), namely refining strata. Also, factorization homology is naturally defined on the $\oo$-category of $(\oo,n)$-categories -- in particular, it is a homotopy invariant of an $(\oo,n)$-category.
In contrast, their work does not consider naturality in functors or homotopy-invariance in the higher category variable.

\smallskip

Variants of these four types of morphisms all occur in \cite{blob}, where they are called anti-refinements, pinched products, isotopies, and the boundary natural transformations.
However, how to compose these maps, such as how to compose a refinement and pinched product, is not part of their schema. 
We address this problem of composition by verifying that the simplicial space determined by $\Bun$ satisfies the Segal condition; so composition is defined only up to coherent homotopy.  
Thereafter, the entities $\Disk_n^\cB$ and $\Mfd_n^\cB$ are $\infty$-categories, thereby organizing the compositions of these types of morphisms. 
We do not think that homotopy coherent compositions of these classes of morphisms in these $\oo$-categories satisfy a universal property with respect to the individual classes, i.e., can be viewed as a condition.
In particular, there is no four-term factorization system for these morphisms, and such a factorization system is still far weaker than an axiomatization. 
Consequently, we do not see how our model in terms of $\Bun$ and $\Mfd_n^\cB$ could be characterized by the axioms given at the end of \S6.1 of \cite{blob}.

\smallskip

Regarding (3), another difficulty of comparison is the absence of common point-set refinements of two disk-stratifications. If one allows refining strata to be isotoped, then common refinements do exist. 
This leads to organizing morphisms in $\Bun$ as \emph{spaces}, the paths in which account for isotopies.  
A lack of topology on mapping spaces would also obstruct any comparison with combinatorial models of $(\oo,n)$-categories. In particular, we are able to define a fully faithful functor $\btheta_n^{\op} \ra \cDisk^{\vfr}_n$ exactly because the righthand side is topologized; without a topology on mapping spaces (e.g., allowing isotopies of stratifications as invertible morphisms) a discrete version of the righthand side---as is used in the blob setting---need not receive any functor from $\btheta_n^{\op}$ or any other collection of combinatorial generators for $(\oo,n)$-categories. One can work with a discrete category of refinements and take the Dwyer--Kan localization with respect to stratified isotopies, but this leads to the difficult problem of identifying this Dwyer--Kan localization. Further, it is unclear how one would match this approach to refinements with the other morphisms, such as creations morphisms: even if the $\oo$-subcategory $\Disk_n^{\rf}$ of refinements among disk-stratified manifolds might be realized as a Dwyer--Kan localization of an ordinary category of refinements, it unclear how to even conjecturally extend this picture to the entire $\oo$-category $\Disk_n$, the construction of which is inherently $\oo$-categorical and lacking strict compositions -- see the verification of the Segal condition for $\Bun$ from \cite{striat}.

\smallskip

A last difference stems from our introduction of the vari-framing. The rigid geometric structure of the vari-framing allows for factorization homology to take coefficients in $(\oo,n)$-categories, rather than $(\oo,n)$-categories with adjoints. If one used a more naive notion of a framing, such as a framing on the ambient manifold, then this would require the input $(\oo,n)$-category to have adjoints. The existence of adjoints is an extremely restrictive condition on an $(\oo,n)$-category, and becomes more restrictive as $n$ increases, so this allows factorization homology to be defined with far more possible inputs.

\subsection*{Acknowledgements}
JF and NR thank Alexander Beilinson for his inspiring mathematics and for his kind encouragement.
DA thanks Richard Hepworth, discussions with whom informed our treatment of $\btheta_n$.
DA extends warm appreciation to Ana Brown for her persistent support.
All of the authors thank the anonymous referee for positive and careful comments, which improved the content and presentation of this article a great deal.

\section{Recollections of striation sheaves}
In this section, we recall definitions and results from our antecedent works. The reference for \S1.1 is \cite{aft1}, and the reference for the subsequent sections is \cite{striat}. This section is only an overview, so see those works for precise definitions and details; all the assertions below are substantiated in those works.

\subsection{Stratified spaces}

The work \cite{aft1} presents a theory of smoothly stratified spaces founded on a key technical feature of conical smoothness. 
We summarize this theory below, first indicating its operational advantages.

A \emph{conically smooth} structure on a stratified topological space is analogous to, and generalizes, a \emph{smooth} structure on a topological manifold.
In a \emph{conically smooth} stratified space, each stratum, and each link between strata, has the structure of a smooth manifold (possibly with corners).
This regularity on links imposed by conical smoothness removes the possibility of maps between stratified spaces that witness sheering along strata.  
For instance, consider the stratification $\RR^{\{1\}}\subset \RR^n$ of Euclidean space by its $1$-axis; this stratified space admits a standard conically smooth structure.
Now, it is possible to construct a stratified self-homeomorphism of this stratified space that restricts to the $1$-axis and its complement as a smooth map, yet which does not determine a homeomorphism between the link $S^{n-2}\times \RR^{\{1\}}$ of the 1-axis.\footnote{
For the case $n=2$, take the piecewise-linear self-homeomorphism of $\RR^2$ whose value on $(x,y)\in \RR^2$ is $(x,y)$ if $x\leq 0$ and is $(x,x+y)$ if $x>0$.}
Such a self-homeomorphism does not respect this standard conically smooth structure.
The essential regularity imposed by conical smoothness is this.
Let $L$ and $M$ be compact smooth manifolds.
Consider their \emph{open cones}, $\sC(L)$ and $\sC(M)$ (which are introduced within the coming page) -- these are stratified by the cone-point.  
A stratified embedding $f\colon\sC(L)\hookrightarrow \sC(M)$ that is conically smooth determines a smooth embedding $\sD_\ast f\colon L \hookrightarrow M$ between links of cone-points.
This map $\sD_\ast f$ is defined via a limit quotient, just like derivatives of smooth maps.
Defined in such a way, the association $f\mapsto \sD_\ast f$ can be enhanced as a deformation retraction from the space $\bigr\{\sC(L) \hookrightarrow \sC(M)\bigr\}$ of conically smooth embeddings onto the space $\bigl\{L\hookrightarrow M\bigr\}$ of smooth embeddings.

We give now give a rough definition of conical smoothness. 
Following the development in~\cite{aft1}, we do this in two steps.
We first define a \emph{stratified topological space}; we then define a \emph{conically smooth stratified space}, (or simply a \emph{stratified space}, for short).

A \emph{stratified topological space} is a topological space $X$ together with a continuous map $X\to P$ to a poset endowed with the topology for which $U\subset P$ is closed if it is downward closed: $q\in U$ and $p\leq q$ implies $p\in U$.  
For $X\to P$ a stratified topological space, and for $p\in P$, its \emph{$p$-stratum} is the subspace $X_p$ which is the preimage of $p$.  
The \emph{depth} of a stratified topological space $X\to P$ is the maximal cardinality of a linearly ordered sub-poset of $P$.
We typically denote a stratified topological space $X\to P$ simply as $X$ alone, when the poset and the map to it are understood.  
A \emph{stratified map} from a stratified topological space $X\to P$ to a stratified topological space $Y\to Q$ is a continuous map $f\colon X\to Y$ and a functor $\ov{f} \colon P \to Q$ for which the diagram
\[
\xymatrix{
X  \ar[rr]^-f  \ar[d]  
&&
Y \ar[d]
\\
P  \ar[rr]^-{\ov{f}}
&&
Q
}
\]
commutes.

Here are a few notions among stratified topological spaces.
\begin{enumerate}

\item
For $X\to P$ a compact stratified topological space, its \emph{open cone} is the stratified topological space
\[
\sC(X) := \ast \underset{\{0\}\times X} \coprod  [0,1) \times X
~\subset~ \ov{\sC}(X) := \ast \underset{\{0\}\times X} \coprod  [0,1] \times X
\qquad \longrightarrow \qquad
P^{\tl} := \ast \underset{\{0\} \times P }\amalg \{0<1\} \times P
\]
in which the map is defined upon declaring the $\ast$ to be the $\ast$-stratum; 
likewise, its \emph{closed cone} is
\[
\ov{\sC}(X) := \ast \underset{\{0\}\times X} \coprod  [0,1] \times X
\qquad \longrightarrow \qquad
P^{\tl} ~.
\]

\item
For $X\to P$ and $Y\to Q$ stratified topological spaces, their \emph{product} is the stratified topological space
\[
X\times Y 
\longrightarrow
P\times Q~.
\]

\item 
For $X\to P$ a stratified topological space, and for $U\subset X$ an open subset, its \emph{inherited stratification} is the canonical continuous map $U \to P_U$ where $P_U$ is the subposet of $P$ defined by the image of $U$.
A stratified map $g$ from $U \to R$ to $X\to P$ is a \emph{stratified open embedding} if the continuous map $U\to X$ is an open embedding and $U_r \ra X_{gr}$ is an open embedding for each $r\in R$.

\item
An \emph{open cover} of a stratified topological space $X\to P$ is a collection of stratified open embeddings to $X\to P$ whose images cover $X$.

\end{enumerate}
The collection of \emph{$C^0$-stratified space} is the smallest collection of stratified topological spaces that contains the stratified topological space $\emptyset \to \emptyset$ and that is closed under each of the above formations.  

\begin{remark}
If $X\to P$ is a $C^0$-stratified space and $P$ has depth zero, then $X$ is a topological manifold.
Conversely, each topological manifold $X$ determines a $C^0$-stratified space whose underlying topological space is $X$ and whose poset is the set of connected components of $X$.  

\end{remark}

\begin{remark}
Each $C^0$-stratified space admits an open cover by stratified topological spaces of the form $\RR^i\times {\RR_{\geq 0}}^{i'}\times \sC(Z)$ for $i\geq 0$ and $L$ a compact $C^0$-stratified topological space.
In particular, each $C^0$-stratified space admits an open cover by $C^0$-stratified topological spaces that are finite dimensional and that have finite depth.  

\end{remark}

\begin{remark}
Each $C^0$-stratified space $X = (X\to P)$ has the following properties.
\begin{itemize}
\item
The poset $P$ is necessarily countable.
Also, $P$ is \emph{Artinian}, meaning that each functor $\ZZ_{\leq 0}\to P$ from the poset of non-positive integers and inequality among them, factors through a finite poset over $P$.  

\item
The topological space $X$ is locally compact and Hausdorff.

\end{itemize}

\end{remark}

We now turn to a rough definition of a \emph{$C^\infty$-stratified spaces}, or \emph{conically smooth stratified spaces}, or simply \emph{stratified spaces}.
The definition of a stratified space is by simultaneous induction on depth and topological dimension, so it will appear to be circular.
See~\cite{aft1} for a detailed treatment.

In the case of depth 0, a conically smooth structure is the structure of a smooth manifold with corners.
A \emph{basic} is a stratified space of the form $\RR^i\times {\RR_{\geq 0}}^{i'}\times \sC(Z) = \bigl(\RR^i\times {\RR_{\geq 0}}^{i'} \times \sC(Z) \to \sC(Z) \to P^{\tl}\bigr)$ with $Z=(Z\to P)$ a compact stratified space.
Each basic has a \emph{cone-locus} as well as an \emph{origin}:
\[
\RR^i\times {\RR_{\geq 0}}^{i'}  \subset \RR^i\times {\RR_{\geq 0}}^{i'}\times\sC(Z)
\qquad \text{ and }\qquad 
0 \in \RR^i\times {\RR_{\geq 0}}^{i'} \subset \RR^i\times {\RR_{\geq 0}}^{i'}\times \sC(Z)~.
\]
Scaling in the Euclidean and the cone coordinates defines an action of the monoid $\RR_{>0}$ of positive real numbers on such a basic: $t\cdot (p,s,z)\mapsto (tp,ts,z)$.  
A map between basics $f\colon \RR^i \times {\RR_{\geq 0}}^{i'}\times \sC(Y) \to \RR^j\times {\RR_{\geq 0}}^{j'}\times \sC(Z)$ is \emph{conically smooth} if 
\begin{itemize}
\item it is the restriction of a stratified continuous map $\w{f}\colon \RR^{i+i'} \times \sC(Y) \dashrightarrow \RR^{j+j'}\times \sC(Z)$ whose domain of definition is open;
\item if its restriction $f_{|}\colon \RR^i\times {\RR_{\geq 0}}^{i'}\times (0,1)\times Z\to \RR^j\times {\RR_{\geq 0}}^{j'}\times \sC(Z)$ is conically smooth;
\item if the following condition holds, which breaks up as two cases:
\begin{enumerate}
\item Suppose $f$ factors through the complement:
\[
f\colon \RR^i\times {\RR_{\geq 0}}^{i'}\times \sC(Y) ~\longrightarrow~ 
\bigl(   \RR^j\times {\RR_{\geq 0}}^{j'}\times\sC(Z)  \bigr)
\smallsetminus 
( \RR^j   \times {\RR_{\geq 0}}^{j'}  )  
= \RR^j\times {\RR_{\geq 0}}^j\times (0,1)\times Z~.
\]
In this case, the condition is that this factored map is conically smooth.

\item
Suppose $f$ carries $\RR^i\times {\RR_{\geq 0}}^{i'}$ into $\RR^j\times {\RR_{\geq 0}}^{j'}$.
Denote the map $\w{f}$ in coordinates $\w{f} = (\w{f}_{\sf Euc} , \w{f}_{\sC(Z)})$.  
For each $(p,s,y)\in \RR^i\times {\RR_{\geq 0}}^{i'}\times \sC(Y)$, and each $v\in \RR^{i+i'}$, the condition is that the limit 
\[
\underset{t\to0}{\sf lim}~\Bigl(~ \frac{\w{f}_{\sf Euc}(p+tv,ts,y) - \w{f}_{\sf Euc}(p)}{t} ~,~ \frac{\w{f}_{\sC(Z)}(p+tv,ts,y)}{t} ~ \Bigr)~{}~ \in ~{}~\RR^j\times {\RR_{\geq 0}}^{j'}\times \sC(Z)
\]
exists and is again \emph{conically smooth} in the arguments $(p,s,y)$ and $v$.  

\end{enumerate}
\end{itemize}
A \emph{conically smooth} atlas on a $C^0$-stratified space $X$ is a collection of stratified open embeddings from basics into $X$ whose images cover $X$ and whose transition maps are conically smooth.
A \emph{conically smooth stratified space}, or \emph{stratified space} for short, is a stratified topological space equipped with a maximal conically smooth atlas.
We typically omit the maximal atlas of a stratified space when referring to it.  

\begin{remark}
The above heuristic definition of a stratified space might appear to be circular.  
With more care, this seemingly circular definition can be crafted into an inductive definition.  
The induction is simultaneously by dimension and by depth.  
This induction is founded on the fact that, by definition, each stratified space is openly covered by basics, and each basic $\RR^i\times {\RR_{\geq 0}}^{i'}\times \sC(Z)$ has finite depth and dimension and $Z$ has strictly less depth and dimension than does the basic.  
See~\cite{aft1} for a detailed treatment.  

\end{remark}

\begin{notation}
For $\RR^i\times {\RR_{\geq 0}}^{i-k}\times \sC(L)$ a basic singularity type, we typically compress the first two factors as $\RR^{i}\times \sC(L)$ in how we denote it.  

\end{notation}

For $X$ and $Y$ stratified spaces, a stratified map $X\to Y$ is \emph{conically smooth} if it restricts to and over each chart of $X$ and $Y$ as a conically smooth atlas between basics, as outlined above.  
The identity map on any stratified space is conically smooth, and conically smooth maps are closed under composition.
The resulting category of stratified spaces and conically smooth maps between them is
\[
\strat~.
\]
\begin{remark}
For $X\to P$ a stratified space such that $P$ has depth zero, then $X$ is a smooth manifold.  
Conversely, each smooth manifold $X$ determines a stratified space whose underlying topological space is that of $X$ and whose poset is the set of connected components of $X$.  
In this way, we regard each smooth manifold as a stratified space. 
This is assignment of a stratified space to each smooth manifold describes a fully faithful functor
\[
{\sf Man}~\hookrightarrow~\strat~. 
\]

\end{remark}

Here are a number of notable classes of morphisms in $\strat$.
\begin{definition}\label{def.classes-of-maps}
Let $f\colon X\ra Y$ be a conically smooth map of stratified spaces.
\begin{itemize}
\item {\bf Embedding (${\sf emb}$):}
$f$ is an \emph{open embedding} if it is an isomorphism onto its image as well as open map of underlying topological spaces.
\item {\bf Refinement (${\sf ref}$):}
$f$ is a \emph{refinement} if it is a homeomorphism of underlying topological spaces, and, for each stratum $X_p\subset X$, the restriction $f_|\colon X_p \to Y$ is an isomorphism onto its image.
\item {\bf Open (${\sf open}$):}
$f$ is \emph{open} if it is an open embedding of underlying topological spaces and a refinement onto its image.  

\item {\bf Fiber bundle:}
$f$ is a \emph{fiber bundle} if there is a collection of pullback diagrams in $\strat$,
\[
\xymatrix{
F_\alpha\times O_\alpha \ar[r]  \ar[d]
&
X  \ar[d]
\\
O_\alpha \ar[r]^-{\psi_\alpha}
&
Y,
}
\]
for which $\bigl\{ \psi_\alpha\colon O_\alpha \to Y\bigr\}$ is a collection of open embeddings covering $Y$.

\item {\bf Constructible bundle ($\cbl$):}
 $f$ is a \emph{weakly constructible bundle} if, for each stratum $Y_q\subset Y$, the restriction $f_{|}\colon f^{-1}Y_q \to Y_q$ is a fiber bundle. The definition of a constructible bundle is inductive based on depth: in the base case of smooth manifolds, $f\colon X\ra Y$ is a constructible bundle if it is a fiber bundle; in the inductive step of the definition, $f\colon X\ra Y$ is a constructible bundle if it is a weakly constructible bundle and, additionally, if for each stratum $Y_q\subset Y$ the natural map
 \[
\Link_{f^{-1}Y_q}(X) \longrightarrow f^{-1}Y_q \underset{Y_q}\times \Link_{Y_q}(Y)
\]
is a constructible bundle.

\item {\bf Proper constructible ($\pcbl$):}
$f$ belongs to the class ($\pcbl$) if it is a constructible bundle and it is \emph{proper}, i.e., if $f^{-1}C\subset X$ is compact for each compact subspace $C\subset Y$. $f$ belongs to either of the classes $(\pcbl, {\sf surj})$ or $({{\pcbl}, {\sf inj}})$ if it is proper constructible as well as either surjective or injective, respectively. 

\end{itemize}

\end{definition}

There are two natural cosimplicial stratified spaces.  
\begin{definition}\label{def.simplices}
The \emph{extended} cosimplicial stratified space is the functor 
\[
\Delta^\bullet_e\colon \bdelta\longrightarrow \strat~,\qquad [q]\mapsto \Delta^q_e:= \bigl\{ \{0,\dots,p\}\xra{t} \RR\mid \sum_i t_i = 1 \bigr\}
\]
with the values on morphisms standard.
The \emph{standard} cosimplicial stratified space is the functor
\[
\Delta^\bullet \colon \bdelta\longrightarrow \strat~,\qquad [p]\mapsto \Bigl(\Delta^p\ni t \mapsto {\sf Max}\{i\mid t_i\neq 0\}\in [p]\Bigr)
\]
with the values on morphisms standard.  
\end{definition}

\begin{remark}
Note the isomorphism of stratified spaces $\Delta^q_e \cong \RR^q$ with the smooth Euclidean space, as well as the isomorphism of stratified spaces $\Delta^p\cong \ov{\sC}(\Delta^{p-1})$, where $\ov{\sC}$ is the \emph{closed} cone which is stratified by the left-cone on the stratifying poset for $\Delta^{p-1}$.  

\end{remark}

The extended cosimplicial stratified space accommodates a natural enrichment of $\strat$ over Kan complexes.  
\begin{definition}[$\Strat$]\label{def.Strat}
The $\infty$-category $\Strat$ is that associated to the $\kan$-enriched category for which an object is a stratified space and the Kan complex of morphisms from $X$ to $Y$ is the simplicial set
\[
\Strat(X,Y)~:=~\strat_{/\Delta^\bullet_e}(X\times\Delta^\bullet_e,Y\times \Delta^\bullet_e)~;
\]
composition is given by composition in $\strat$ over $\Delta^\bullet_e$.  
There are a number of notable subsidiary $\infty$-categories
\[
\Strat^{\sf ref}~,~\Strat^{\sf emb}~\ra~\Strat^{\opn}~\to~\Strat~\la~\Strat^{\pcbl}~\la~\Strat^{\pcbl, {\sf surj}}~,~\Strat^{{\pcbl}, {\sf inj}}
\]
which are each associated to Kan enriched categories that have the same objects as $\strat$ yet with Hom-Kan complexes consisting of those simplices of $\strat_{/\Delta^\bullet_e}(-,-)$ which are fiberwise over $\Delta^\bullet_e$ of the indicated class.  
\end{definition}
\noindent
(The regularity along strata ensured by conical smoothness can be used to verify that these Hom-simplicial sets are indeed Kan complexes.)

Manifest is a functor $\strat \to \Strat$ between $\infty$-categories.
This functor has the property that, for each stratified space $X$, the morphism $X\times \RR\xra\pr X$ in $\strat$ is carried to an equivalence in $\Strat$.  
\begin{theorem}[\cite{striat}]\label{Strat-localize}
The functor $\strat \to \Strat$ witnesses an equivalence between $\infty$-categories:
\[
\strat[\RR^1\times -^{-1}]\xra{~\simeq~} \Strat
\]
from the localization on the collection of morphisms which are projections off of $\RR$.

\end{theorem}

\subsection{Exit-paths}

The enrichment $\Strat$ of stratified spaces allows for a very natural presentation of the exit-path $\oo$-category of a stratified space, after Lurie \cite{HA} and MacPherson. 
Following \cite{striat}, we use the standard simplices to define the exit-path $\oo$-category $\exit(X)$ as a complete Segal space.
As a simplicial object, $\exit(X)$ is the stratified version of the singular simplicial object $\sing_\bullet(X)$.  

\begin{definition}\label{def.exit-path}
The \emph{exit-path $\infty$-category functor} is the restricted Yoneda functor
\[
\exit\colon \strat \longrightarrow \Psh(\bdelta)~,\qquad X\mapsto \Bigl([p]\mapsto \Strat(\Delta^p,X) = |\strat(\Delta^p\times \Delta^\bullet_e,X)| \Bigr)~.
\]
\end{definition}

The following is one of the main results of \cite{striat}.

\begin{theorem}[\cite{striat}]\label{thm.exit-facts}
The functor $\exit$ admits a unique factorization in the following diagram among $\infty$-categories,
\begin{equation}\label{exit-to-cat}
\xymatrix{
\strat \ar[rr]^-{\exit}  \ar[d]^-{\sf loc}_-{\rm Thm~\ref{Strat-localize}}
&&
\Psh(\bdelta)  
\\
\Strat  \ar@{-->}[rr]^-{\exit}
&&
\Cat_\infty  ,    \ar@{_{(}->}[u]
}
\end{equation}
as a functor from the localization in Theorem~\ref{Strat-localize} to the $\infty$-category of $\infty$-categories, here incarnated as complete Segal spaces.  
Also, the following diagrams in $\Strat$ are colimit diagrams, and this factorizing functor~(\ref{exit-to-cat}) carries each of these diagrams to colimit diagrams among $\infty$-categories:
\begin{itemize}

\item open hypercovering diagrams $\cU^{\tr} \to \strat \to \Strat$;

\item blow-up diagrams
\[
\xymatrix{
{\sf Link}_{X_0}(X)  \ar[r]  \ar[d]
&
\unzip_{X_0} (X)  \ar[d]  
\\
X_0  \ar[r]
&
X;
}
\]

\item
iterated cone diagrams
\[
\xymatrix{
\ov{\sC}(\emptyset) \ar[rr]^-{\sC(\emptyset \hookrightarrow L)}  \ar[d]
&&
\ov{\sC}(L) \ar[d]
\\
\ov{\sC}^2(\emptyset) \ar[rr]^-{\sC^{2}(\emptyset\hookrightarrow L)}
&&
\ov{\sC}^2(L)~;
}
\]

\item the univalence diagram 
\[
\xymatrix{
\Delta^{\{0<2\}}\amalg \Delta^{\{1<3\}}  \ar[r] \ar[d]
&
\Delta^{\{0<1<2<3\}} \ar[d]
\\
\Delta^{\{0=2\}}\amalg \Delta^{\{1=3\}}  \ar[r]
&
\ast.
}
\]
\end{itemize}

\end{theorem}

The main geometric ingredient supporting the proof of Theorem~\ref{thm.exit-facts} is the next result.
For $\cX$ an $\infty$-category, the \emph{$\infty$-category of arrows in $\cX$}
\[
\Ar(\cX)~:=~\Fun([1],\cX)
\]
is that of functors from $[1]$.
\begin{lemma}[\cite{striat}]\label{exit-explicit}
Let $X = (X\to P)$ be a stratified space.
There is a natural identification of the maximal $\infty$-subgroupoid of $\exit(X)$
\[
\exit(X)^\sim~\simeq~ \underset{p\in P}\coprod X_p
\]
as the coproduct of the underlying spaces of the strata of $X$.
For each pair of strata $p,q\in P$ there is an identification of the space of morphisms in $\exit(X)$ with source in $X_p$ and target in $X_q$,
\[
\Ar(\exit(X))_{|X_p\times X_q}~\simeq~ \Link_{X_p}(X_{|P^{p/}})_q~,
\]
as the underlying space of the $q$-stratum of the link of $X_p$ in the open stratified subspace $X_{|P^{p/}}\subset X$.  

\end{lemma}

\subsection{Striation sheaves}

\begin{definition}
The $\infty$-category $\Stri$ of \emph{striation sheaves} is the full $\infty$-subcategory of $\Psh(\strat)$ consisting of those presheaves $\cF$ that are $\RR^1$-invariant and that carry the opposites of each of the distinguished colimit diagrams of Theorem~\ref{thm.exit-facts} to limit diagrams among spaces.  
\end{definition}

\begin{remark}
Using the identification of $\Strat$ as a localization of $\strat$ from Theorem~\ref{Strat-localize}, the $\infty$-category $\Stri$ could be equivalently defined as consisting of those presheaves $\cF\in \psh(\Strat)$ that carry the opposites of each of the distinguished colimit diagrams of Theorem~\ref{thm.exit-facts} to limit diagrams among spaces.
\end{remark}

The standard cosimplicial stratified space $\bdelta \xra{\Delta^\bullet} \strat$ yields the restriction functor $\Stri \to \Psh(\bdelta)$.  
\begin{theorem}[\cite{striat}]\label{strict} 
Restriction along the cosimplicial stratified space $\Delta^\bullet$ determines an equivalence of $\infty$-categories
\[
\Stri~\simeq~\Cat_\infty
\]
between striation sheaves and $\infty$-category of $\infty$-categories, incarnated here as complete Segal spaces.
This equivalence sends an $\oo$-category $\cC$ to the presheaf on $\strat$ that takes the values
\[
K \mapsto\Map_{\Cat_\infty}(\exit(K), \cC)~.
\]
\end{theorem}

We use striation sheaves to make interesting $\oo$-categories by hand from smooth stratified geometry. The principal such object is the $\oo$-category $\Bun$, the construction of which we now indicate.

\begin{definition}[$\Bun$ and $\Exit$]\label{def.Bun}
$\Bun$ is the presheaf on stratified spaces that classifies constructible bundles:
\[
\Bun\colon K \mapsto |\{X\underset{\pi,~ \cbl}\longrightarrow K\times \Delta^\bullet_e\}|~,
\]
the moduli space of constructible bundles over $K$.  
$\cBun$ is the subpresheaf on stratified spaces that classifies \emph{proper} constructible bundles:
\[
\cBun\colon K \mapsto |\{X\underset{\pi,~ \sf p.cbl}\longrightarrow K\times \Delta^\bullet_e\}|~,
\]
the moduli space of proper constructible bundles.  
\\
$\Exit$ is the presheaf on stratified spaces that classifies constructible bundles equipped with a section:
\[
\Exit\colon K\mapsto |\bigl\{X \underset{\pi,~\cbl}{\overset{\sigma}\leftrightarrows} K\times \Delta^\bullet_e \mid \sigma \pi = 1 \bigr\}|~.
\]

\end{definition}

The cumulative result of the work of \cite{striat}, and of all the regularity around substrata ensured by conical smoothness, is the following.

\begin{theorem}[\cite{striat}]
The presheaves $\Bun$ and $\cBun$ and $\Exit$ are striation sheaves and, consequently, form $\oo$-categories via the equivalence $\Stri \underset{\rm Thm~\ref{strict}}\simeq \Cat_\infty$.
Furthermore, the natural functor
\[
\cBun \longrightarrow \Bun
\]
is a monomorphism of $\infty$-categories (in the sense of \S\ref{sec.monos}).  

\end{theorem}

\begin{remark}
The $\infty$-subcategory $\cBun \subset \Bun$ consists only of \emph{compact} stratified spaces, yet is \emph{not} full.
For instance, there is a unique morphism from $\emptyset$ to $S^1$ in $\Bun$; it is represented by the constructible bundle $S^1 \times ( \Delta^1\smallsetminus\{0\} ) \ra \Delta^1$.
Because this constructible bundle is not proper, it does not represent a morphism in $\cBun$.  
Furthermore, there is no morphism in $\cBun$ from $\emptyset$ to $S^1$.  
\end{remark}

Forgetting sections defines a functor $\Exit \to \Bun$. 
The next result in particular identifies the fibers of this functor.  
\begin{prop}[\cite{striat}]\label{prop.Exit}
For each functor $\exit(K)\xra{(X\xra{\pi} K)} \Bun$ classifying the indicated constructible bundle, the diagram among $\infty$-categories
\[
\xymatrix{
\exit(X)  \ar[rr]^-{(X\underset{K}\times X \underset\pr{\overset{\sf diag}\leftrightarrows} X)} \ar[d]_-{\exit(\pi)}
&&
\Exit \ar[d]
\\
\exit(K)\ar[rr]^-{(X\xra{\pi}K)}
&&
\Bun
}
\]
is a pullback diagram.

\end{prop}

The reader might notice that the above description of $\Bun$, as well as each variant, is derived from more primitive data.
We explain this, for we make use of such a maneuver in~\S\ref{sec.iterated-bdls} of this article.  
\begin{definition}[\S6 of~\cite{striat}]
The ordinary category $\bun$ is that for which an object is a constructible bundle $X\to K$ and a morphism from $(X\to K)$ to $(X'\to K')$ is a pullback diagram among stratified spaces
\[
\xymatrix{
X  \ar[r] \ar[d]
&
X'  \ar[d]
\\
K \ar[r]
&
K'.
}
\]
Composition is given by concatenating such diagrams horizontally and composing horizontal arrows.  

\end{definition}

In~\S6 of~\cite{striat} we prove that the projection
\[
\bun \longrightarrow \strat~,\qquad (X\to K)\mapsto K
\]
is a right fibration, which we can straighten to a functor
\[
\strat^{\op} \longrightarrow {\sf Gpd}
\]
to the ordinary category of groupoids.
The presheaf $\Bun$ can be obtained as the functor
\[
\Bun\colon \strat^{\op} \longrightarrow \Fun(\bdelta^{\op},{\sf Gpd}) \xra{~|-|~} \Spaces~,\qquad X\mapsto |\bun(X\times \Delta^\bullet_e)|
\]
where here the arrow labeled as $|-|$ is a standard nerve functor.  
More generally, for $\fF\in \Fun(\strat^{\op},{\sf Gpd})$ a groupoid-valued functor, the expression $\cF\colon X\mapsto |\fF(X\times \Delta^\bullet_e)|$ defines a space-valued presheaf.  
This association $\fF\mapsto \cF$ is the \emph{topologizing diagram} functor, and is developed in~\S2 of~\cite{striat}.

\subsection{Classes of morphisms}\label{sec.classes}
Here we name several classes of morphisms in $\Bun$.  
To identify these morphisms, it is convenient to use that morphisms in the $\oo$-category $\Bun$ can be constructed as mapping cylinders of stratified maps in two different ways, as cylinders of open maps or as reversed cylinders of proper constructible maps. See \S6.6 of~\cite{striat} for the following.
\begin{theorem}\label{thm.class-maps}
There are functors 
\[
\Strat^{\opn}  \overset{\cylo}\longrightarrow  \Bun  \overset{\cylr}\longleftarrow (\Strat^{\pcbl})^{\op}
\]
each of which is a monomorphism.  
\end{theorem}

This allows for the following definition.

\begin{definition}\label{def.classes}
\begin{itemize}
\item[~]
\item The $\oo$-subcategory $\Bun^{\cls}\subset \Bun$ of {\it closed} morphisms is the image of $(\Strat^{{\pcbl}, {\sf inj}})^{\op}$. 
\item The $\oo$-subcategory $\Bun^{\sf cr}\subset \Bun$ of {\it creation} morphisms is the image of $(\Strat^{\pcbl,{\sf surj}})^{\op}$.

\item The $\infty$-subcategory $\Bun^{\sf ref}\subset \Bun$ of \emph{refinement} morphisms is the image of $\Strat^{\sf ref}$.

\item The $\infty$-subcategory $\Bun^{\sf emb}\subset \Bun$ of \emph{open embedding} morphisms is the image of $\Strat^{\emb}$.  

\item The $\infty$-subcategory $\Bun^{\sf act}\subset \Bun$ of \emph{active} morphisms is the smallest containing the creation and refinement and embedding morphisms.  

\end{itemize}

\end{definition}

\begin{remark}
So, an object in $\cBun$ is given by a compact stratified space.
A morphism in $\cBun$ is given by a compact stratified space equipped with a constructible bundle to the standardly stratified $\Delta^1$.  
Figure~\ref{fig.99} depicts a refinement morphism in $\cBun$; Figure~\ref{fig.100} depicts a creation morphism in $\cBun$.  

\end{remark}

\begin{figure}
\begin{tikzpicture}

\draw[fill=lightgray] (0,0) to (4,0) to (4,2) to (0,2) to (0,0);

\draw[line width=2] (-0.03,0) to (4.03,0);
\draw (0,0) to (0,2);
\draw[line width=2] (-0.03,2) to (4.03,2);
\draw (4,0) to (4,2);

\draw[fill] (0,1) circle [radius=0.07];

\draw[->] (2,-0.5) to (2,-1);

\draw[fill] (0,-1.5) circle [radius=0.07];
\draw[fill] (0,0) circle [radius=0.07];
\draw[fill] (0,2) circle [radius=0.07];
\draw (0, -1.5) to (4,-1.5);
\end{tikzpicture}
\caption{A refinement morphism in $\cBun$ between two stratifications of the closed interval.}  
\label{fig.99}
\end{figure}
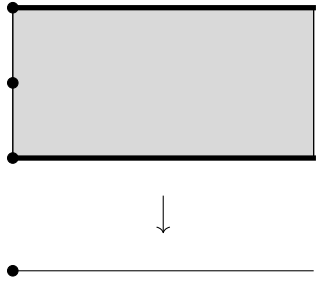

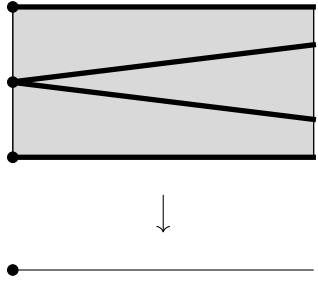
\begin{figure}
\begin{tikzpicture}

\draw[fill=lightgray] (0,0) to (4,0) to (4,2) to (0,2) to (0,0);

\draw[line width=2] (-0.03,0) to (4.03,0);
\draw (0,0) to (0,2);
\draw[line width=2] (-0.03,2) to (4.03,2);
\draw (4,0) to (4,2);

\draw[fill] (0,1) circle [radius=0.07];

\draw[line width=2] (-0.03,1) to (4.03,0.5);
\draw[line width=2] (-0.03,1) to (4.03,1.5);

\draw[->] (2,-0.5) to (2,-1);

\draw[fill] (0,-1.5) circle [radius=0.07];
\draw[fill] (0,0) circle [radius=0.07];
\draw[fill] (0,2) circle [radius=0.07];
\draw (0, -1.5) to (4,-1.5);
\end{tikzpicture}
\caption{A creation morphism in $\cBun$ between two stratifications of the closed interval.}
\label{fig.100}
\end{figure}

\begin{remark}
The intersection of the $\infty$-subcategories $\cBun\cap \Bun^{\sf emb} \simeq \cBun^\sim$ is the maximal $\infty$-subgroupoid of $\Bun$ consisting of the compact stratified spaces.

\end{remark}

We note these classes of morphisms have an equivalent definition in terms properties of {\it links}. That is, consider a morphism $X_0\ra X_1$ in $\Bun$ represented by a constructible bundle $X\ra \Delta^1$. The morphism is closed (respectively, a creation) if the natural map $\Link_{X_0}(X)\to X_0$ is an embedding (respectively, is surjective) and the open conically smooth map
\[\Link_{X_0}(X)\times[0,1)\cong \unzip_{X_0}(X)_{|\Link_{\{0\}}(\Delta^1)\times[0,1)}\] can be chosen to be an isomorphism.
This is the geometry behind the following result.

\begin{lemma}[\cite{striat}]\label{closed-active}
The pair of $\infty$-subcategories $(\Bun^{\sf cls},\Bun^{\sf act})$ forms a factorization system on $\Bun$.  

\end{lemma}

Recall the unstraightening construction from \S3.2 of~\cite{HA}, which constructs a monomorphism
\[
\Fun(\cD^{\op},\Cat_{\oo})\longrightarrow \Cat_{\oo/\cD}~.
\]
The essential image consists of {\it Cartesian fibrations} over $\cD$.

\begin{lemma}\label{exit-fbn}
Each solid diagram among $\infty$-categories
\[
\xymatrix{
\{0\}  \ar[r]  \ar[d]
&
\Exit  \ar[d]
\\
[1]  \ar[r]^-o  \ar@{-->}[ur]
&
\Bun
}
\]
in which $o$ classifies either a refinement morphism or an embedding morphism of $\Bun$ admits a filler that classifies a coCartesian morphism over $\Bun$.  
Each solid diagram among $\infty$-categories
\[
\xymatrix{
\{1\}  \ar[r]  \ar[d]
&
\Exit  \ar[d]
\\
[1]  \ar[r]^-c  \ar@{-->}[ur]
&
\Bun
}
\]
in which $c$ classifies either a closed morphism or a creation morphism of $\Bun$ admits a filler that classifies a Cartesian morphism over $\Bun$.  
\end{lemma}
\begin{proof}
The first statement follows because 
$\Exit_{|\Bun^{\sf ref,emb}} \to \Bun^{\sf ref,emb}\simeq \Strat^{\sf open}$ is the unstraightening of the functor 
\[
\exit\colon \Strat^{\sf open} \longrightarrow \Cat_\infty~.
\]
The second statement follows because 
$\Exit_{|\Bun^{\sf cls,cr}} \to \Bun^{\sf cls,cr}\simeq (\Strat^{\pcbl})^{\op}$ is the unstraightening of the functor 
\[
\exit\colon ((\Strat^{\pcbl})^{\op})^{\op} \longrightarrow \Cat_\infty~.
\]
\end{proof}

We mirror Definition~\ref{def.classes} with classes of morphisms in $\Exit$.
We do this in a way that reflects the natural handedness of the various restrictions $\Exit_{|\Bun^{\psi}} \to \Bun^{\psi}$.  
\begin{definition}\label{def.Exit-classes}
\begin{itemize}
\item[~]
\item The $\oo$-subcategory $\Exit^{\cls}\subset \Exit$ of {\it closed} morphisms consists of those morphisms in $\Exit$ that are Cartesian over $\Bun$ and whose image in $\Bun$ is a closed morphism.

\item The $\oo$-subcategory $\Exit^{\sf cr}\subset \Exit$ of {\it creation} morphisms consists of those morphisms in $\Exit$ that are Cartesian over $\Bun$ and whose image in $\Bun$ is a creation morphism.

\item The $\oo$-subcategory $\Exit^{\sf ref}\subset \Exit$ of {\it refinement} morphisms consists of those morphisms in $\Exit$ that are coCartesian over $\Bun$ and whose image in $\Bun$ is a refinement morphism.

\item The $\oo$-subcategory $\Exit^{\sf emb}\subset \Exit$ of {\it embedding} morphisms consists of those morphisms in $\Exit$ that are coCartesian over $\Bun$ and whose image in $\Bun$ is an embedding morphism.  

\end{itemize}

\end{definition}

\section{Tangential structures}

The $\infty$-category of constructible bundles is vast.  
We are primarily interested in the $\infty$-subcategory of $\Bun$ classifying constructible bundles whose fibers are stratified $n$-manifolds, as well as variations which account for infinitesimal structures thereon.  
We recognize these entities as $\Bun^\cB$ for appropriately chosen tangential structures $\cB$.  
Forgetting structure defines a functor $\Bun^\cB \to \Bun$; this functor will have partial fibration properties, mirroring the definition of $\infty$-operads as functors $\cO \to \Fin_\ast$ as developed in~\cite{HA} (Definition~2.1.1.10).  
In this section, we give a general framework for such structures; manipulations among them will be key for the main results of this article.  
We begin by generalizing the standard tangent bundle of smooth manifolds.

\subsection{Constructible tangent bundle}\label{sec.tgt-shf}
Here we define, for each stratified space $X$, its \emph{constructible} tangent bundle: 
$\sT X \to X$.
For a treatment in the context of Whitney stratified spaces, see \S2.1 of~\cite{pflaum}.
Our definition is premised on the principle that a constructible vector bundle $E\to X$ is a constructible system of vector spaces on $X$.  
Embracing this perspective, Theorem~A.9.3 of~\cite{HA} then characterizes such a constructible system on $X$ as a functor from $\exit(X)$ to an $\infty$-category of vector spaces.  
This abstracts a parallel transport system. 
The constructible tangent bundle of a stratified space $X$ will thusly take such a form: it is a functor $\sT_X\colon \exit(X) \to \Vect$ to an $\infty$-category of vector spaces.
We construct this functor $\sT_X$, for each stratified space $X$, so as to accommodate a number of functorialities in the argument $X$.  
Such functorialities are succinctly implemented by constructing the \emph{fiberwise} constructible tangent bundle of a constructible bundle $X\xra{\pi} K$; this takes the form of a functor $\sT^{\sf fib}_\pi\colon \exit(X) \to \Vect$.  
Stipulating such functorialities completely characterizes this (fiberwise) constructible tangent bundle of each stratified space.   
In order to succinctly manage homotopy coherence issues that are abundant, our treatment of (fiberwise) constructible tangent bundles is through such a characterization.  
The main outcome of this section is Definition~\ref{def.tangent}, the definition of the \emph{
fiberwise
constructible tangent bundle}
\[
\sT^{\sf fib}\colon \Exit \longrightarrow \Vect^{\sf inj}~.
\]

\subsubsection{\bf Synopsis}
In essence, we define the constructible tangent bundle of a stratified space $X$ by naming a sheaf of vector spaces on $X$, the stalk-dimensions of which are finite and vary lower semi-continuously.
In terms of parallel transport systems, such data is classified by a functor
\[
\sT_X \colon \exit(X) \to \Vect^{\sf inj}
\] 
from the exit-path $\infty$-category of $X$ to an $\infty$-category of vector spaces and injections among them.
In the case that $X=M$ is an ordinary smooth manifold, the functor $\exit(M)\to\Vect^{\sf inj}$ is the classifying map of its tangent bundle $M\to \underset{i\geq 0}\amalg \BO(i)$.

We thus seek an assignment of each stratified space $X$ to a functor $\exit(X) \to \Vect^{\sf inj}$.
We do this in such as way as to reveal an assortment of functorialities of this assignment in the argument $X$.
To learn what functorialities we might expect, we can consider the case of ordinary smooth manifolds.
Namely, a smooth map $f\colon M\to N$ between ordinary smooth manifolds determines a map of vector bundles $\sD f\colon \sT M \to f^\ast \sT N$ on $M$.  
This map is an isomorphism whenever $f$ is an open embedding, but otherwise it typically will not be an isomorphism.
Therefore we expect functoriality for the constructible tangent bundle with respect to open embeddings among stratified spaces.
Thinking of vector fields as infinitesimal automorphisms of a manifold, we also expect that the constructible tangent bundle of a stratified space restricts along strata: for $g\colon X_0\hookrightarrow X$ the inclusion of a stratum, we expect an isomorphism of vector bundles $\sD g\colon \sT X_0 \cong g^\ast \sT X$ on $X_0$.  
This latter expected functoriality does not appear in the case of ordinary smooth manifolds.  
(Note the distinction between this indication of \emph{tangent} and that supplied by algebraic geometry whose sections cannot be integrated when $X$ is not smooth, even in characteristic zero.)

In our Definition~\ref{def.tangent}, the 
fiberwise
constructible tangent bundle will be a functor between $\infty$-categories
\[
\sT^{\sf fib}\colon \Exit \longrightarrow \Vect^{\sf inj}~.  
\]
For each stratified space $X$, there is a natural functor $\exit(X) \to \Exit$, and this functor $\sT^{\sf fib}$ will then specialize to the expected functor $\sT_X\colon \exit(X) \to \Vect^{\sf inj}$.  
Furthermore, the functor $\sT^{\sf fib}$ specializes along the monomorphisms 
$(\Strat^{\ast/})^{\emb} \hookrightarrow \Exit \hookleftarrow \bigl((\Strat^{\ast/})^{\sf p.cls.inj}\bigr)^{\op}$
as the above expected functorialities.
In this way, our definition of the 
fiberwise
constructible tangent bundle as the functor $\sT^{\sf fib}$ accommodates our intuition.

It is awkward to define this (fiberwise) constructible tangent bundle directly. 
Simply requiring it to possess the expected functorialities and to restrict to the familiar notion for smooth manifolds completely characterizes it.  
As such, we articulate a sense in which structures on stratified spaces that satisfy suitable descent are completely characterized by their values on Euclidean spaces.

\begin{remark}\label{parallel-transport}
Our Definition~\ref{def.tangent} of the
fiberwise
constructible tangent bundle as a functor between $\infty$-categories $\Exit\xra{\sT^{\sf fib}} \Vect^{\sf inj}$ manifestly carries isotopies among stratified spaces, such as manifolds, to equivalences of vector bundle data.
This phrasing greatly consolidates tangential data, but it also loses a considerable amount of infinitesimal geometric information.  
For instance, for $M$ a smooth manifold, as a $\Vect^{\sf inj}$-valued sheaf, $\sT_M$ is simply a kind of local system of vector spaces on $M$, which one can regard as a continuous parallel transport system of vector spaces on $M$.  
Of course, such parallel transport on the tangent bundle of $M$ is only defined upon a choice of connection, which has rich infinitesimal geometry.  Because the space of connections is contractible (for it is convex), the phrasing in terms of the $\infty$-category $\Vect^{\sf inj}$ collapses this choice.

\end{remark}

\subsubsection{\bf Vector spaces}

Here we define a constructible sheaf $\Vect$ on stratified spaces that classifies stratum-wise vector bundles, in which the dimensions of the fibers may vary across strata.
We first define conically smooth vector bundles, in which the dimensions of the fibers are locally constant across strata.
(All vector spaces are understood to be finite dimensional real vector spaces.)

\begin{definition}
A \emph{conically smooth vector bundle} $V\to K$ is a conically smooth map of stratified spaces $V\to K$ together with 
\begin{itemize}
\item a conically smooth section $K\xra{0} V$;

\item  a conically smooth map $V\underset{K}\times V \xra{~+~}V$ over $K$;

\item a conically smooth map $\RR\times V\to V$ over $K$.

\end{itemize}
These data satisfy the following local triviality condition.
\begin{itemize}

\item There is an open cover $\cU$ of $K$ for which, for each $U\in \cU$, there is a vector space $V_U=(0\in V_U,+,\cdot)$ and an isomorphism from the restricted data over $U$:
\[
(U\xra{0_{|U}} V_{|U}~,~+_{|U}~,~\cdot_{|U})~\cong~ (U\xra{{\sf id}_U\times \{0\}}U\times V_U~,~{\sf id}_U\times +~,~{\sf id}_U\times \cdot)~.  
\]

\end{itemize}
A \emph{map of vector bundles} from $(V\to K)$ to $(W\to L)$ is a commutative square of conically smooth maps
\[
\xymatrix{
V  \ar[r]^-{F}  \ar[d]
&
W  \ar[d]
\\
K \ar[r]^-{f}
&
L
}
\]
for which, for each $x\in K$, the map of fibers $F_{|x}\colon V_{|x} \to W_{|f(x)}$ is a linear map between vector spaces.  
Such a map of vector bundles has \emph{locally constant rank} if ${\sf Ker}(F) \subset V$ is a sub-vector bundle over $K$. 

\end{definition}
\noindent
Note that conically smooth vector bundles, and maps among them, form a category in which composition is given by concatenating such squares horizontally and composing horizontal arrows.

\begin{remark}
We point out that each conically smooth vector bundle $V\to K$ forgets to a vector bundle on the underlying topological space of $K$.

\end{remark}

In the next definition, we make use of the standard nerve functor $\Fun(\bdelta^{\op},{\sf Gpd}) \xra{|-|} \Spaces$ from simplicial groupoids to the $\infty$-category of spaces.  
\begin{definition}[$\Vect$]\label{def.Vect}
The simplicial space $\Vect\colon \bdelta^{\op}\to \Spaces$ is given by
\[
[p]~\mapsto~\Bigl|\Bigl\{V_0  \xra{~f_1~} V_1\xra{~f_2~} \dots\xra{~f_p~} V_p\text{ over }\Delta^\bullet_e\Bigr\}\Bigr|~,
\]
the space of composable sequences of morphisms of (finite rank) vector bundles over the cosimplicial smooth manifold $\Delta^\bullet_e$ 
for which, for each $0\leq i \leq j\leq p$, the composition $f_j\circ \dots \circ f_i$ has locally constant rank.  
The simplicial subspaces
\[
\Vect^\sim~\subset~\Vect^{\sf inj}~,~\Vect^{\sf surj}~\subset~\Vect
\]
consist, for each $p\geq 0$, of those $p$-simplices $V_0\xra{f_1}\dots \xra{f_p} V_p$ for which each $f_i$ is a fiberwise isomorphism, a fiberwise injection, and a fiberwise surjection, respectively.   

\end{definition}

\begin{observation}
The simplicial spaces $\Vect^{\sf inj}$ and $\Vect^{\sf surj}$ are complete Segal spaces, and therefore present $\infty$-categories.  
The simplicial space $\Vect^\sim$ is a constant simplicial space, and therefore presents an $\infty$-groupoid. 
This $\infty$-groupoid $\Vect^\sim$ is identified as the maximal $\infty$-subgroupoid of both $\Vect^{\sf inj}$ and $\Vect^{\sf surj}$.  
The simplicial space $\Vect$ does \emph{not} satisfy the Segal condition; therefore we do not regard it as an $\infty$-category.  

\end{observation}

\begin{observation}
There is a natural identification of $\infty$-groupoids
\[
\underset{n\geq 0} \coprod \BO(n)~\simeq ~\Vect^\sim~.
\]
For $i$ and $j$ dimensions, there is an identification of the space of $1$-simplices from $\RR^i$ to $\RR^j$
\[
\Vect(\RR^i,\RR^j)~\simeq~\underset{0\leq d \leq i,j} \coprod    \bigl(\sO(i)/\sO(i-d)\bigr)\underset{\sO(d)}\times \bigl(\sO(j)/\sO(j-d)\bigr)~.
\]
In particular, there are natural identifications of the spaces of morphisms 
\[
\Vect^{\sf inj}(\RR^i,\RR^j)~\simeq~ \sO(j)/\sO(j-i)
\qquad  \text{ and }\qquad
\Vect^{\sf surj}(\RR^i,\RR^j)~\simeq~ \sO(i)/\sO(i-j)~.  
\]
Also, the zero vector space is initial in the $\infty$-category $\Vect^{\sf inj}$, final in the $\infty$-category $\Vect^{\sf surj}$.
\end{observation}

\begin{remark}
We give a second description of the $\infty$-category $\Vect^{\sf inj}$.
Direct sum of vector spaces gives the $\infty$-groupoid
\[
\underset{n\geq 0}\coprod \BO(n)
\]
a monoidal structure.  
This monoidal $\infty$-groupoid is specified by the $\infty$-category
\[
\fB\Bigl( \underset{n\geq 0}\coprod \BO(n) \Bigr)~,
\]
which is its deloop -- as a univalent Segal space, it evaluates as
\[
\fB\Bigl( \underset{n\geq 0}\coprod \BO(n) \Bigr)
\colon \bdelta^{\op} \ni [p]
~\mapsto ~
\Bigl( \underset{n\geq 0}\coprod \BO(n) \Bigr)^{\times p}
\in \Spaces~.
\]
In particular, there is a unique object $\ast \in \fB\Bigl( \underset{n\geq 0}\coprod \BO(n) \Bigr)$.
There is a canonical identification between $\infty$-categories
\[
\Vect^{\sf inj}
{~\simeq~}
\fB\Bigl( \underset{n\geq 0}\coprod \BO(n) \Bigr)^{\ast/}
\]
involving the $\infty$-undercategory.  

\end{remark}

\begin{remark}\label{vect-dual}
Taking fiberwise linear duals implements an equivalence $(-)^\vee\colon \Vect^{\op}\simeq \Vect$.  
This equivalence restricts to $\Vect^\sim$ as the identity, and it restricts as an equivalence $(\Vect^{\sf inj})^{\op}\simeq \Vect^{\sf surj}$.
\end{remark}

\begin{remark}
We find it convenient to sometimes work with the simplicial space $\Vect$ even though it does not present an $\infty$-category. 
For instance, as in~\S1 of~\cite{HTT}, the notion of a limit diagram, and of a zero-object, in a simplicial space can be defined as usual once the notion of a final object in a simplicial space is defined, which we do as follows:
\begin{itemize}
\item[~]
Let $\cV$ be a simplicial space.
Let $[0]\xra{V}\cV$ be an object. 
The simplicial space $\cV_{/V}$ is that for which $\Map([p],\cV_{/V}) :\simeq \Map\bigl([p+1],\{p+1\}),(\cV,V)\bigr)$ for each $p\geq 0$.  
The object $V\in \cV$ is \emph{final} if the projection $\cV_{/V} \to \cV$ is an equivalence of simplicial spaces.
\end{itemize}

\end{remark}

\begin{observation}
The zero vector space is a zero-object of the simplicial space $\Vect$.
Furthermore, each map $[1]\to \Vect$, which classifies a linear map $V\xra{f} W$ between vector spaces, canonically extends to a limit diagram $[1]\times[1]\to \Vect$:
\[
\xymatrix{
{\sf Ker}(f) \ar[r]^-i  \ar[d]
&
V  \ar[d]^-f
\\
0  \ar[r]
&
W
}
\]
in which $i$ is an injection.
Using Remark~\ref{vect-dual}, if $f$ is surjective then the above square diagram in $\Vect$ is also a pushout, and there is a natural equivalence ${\sf Ker}(i)^{\vee} \simeq W$. 

\end{observation}

Fiberwise direct sum of vector spaces defines a symmetric monoidal structure on $\Vect$,
\[
\Vect~\in~\CAlg\bigl(\Psh(\bdelta)^\times\bigr)~,
\]
as we now explain.
First note that the simplicial space $\Vect$ is defined as the nerve of a simplicial object in simplicial groupoids.  
Fiberwise direct sum of vector bundles lifts this simplicial object in simplicial groupoids to a simplicial object in symmetric monoidal simplicial groupoids:
\[
\bigoplus \colon \bigl((V\to \Delta^\bullet_e),(W\to \Delta^\bullet_e)\bigr)~\mapsto~ \bigl(V\underset{\Delta^\bullet_e}\times W\to \Delta^\bullet_e\bigr)~.
\]
The existence of the proposed symmetric monoidal structure follows because geometric realization commutes with products.

\begin{observation}
The symmetric monoidal structure on $\Vect$ restricts as symmetric monoidal structures on each of the $\infty$-categories $\Vect^{\sf inj}$, $\Vect^{\sf surj}$, and $\Vect^\sim$.

\end{observation}

\subsubsection{\bf The constructible tangent bundle}

In~\S2 of~\cite{striat} we considered the \emph{topologizing diagram functor}
\begin{equation}\label{def.td}
{\sf Trans} \xra{~ \fF \mapsto |\fF(-\times\Delta^\bullet_e)|~} \Stri~\underset{\text{\cite{striat}}}\simeq~\Cat_\infty
\end{equation}
from transversality sheaves, which are certain right fibrations $\fF\to \strat$ between ordinary categories, to striation sheaves.

\begin{observation}\label{td-prods}
Because geometric realization commutes with finite products, the topologizing diagram functor~(\ref{def.td}) preserves finite products.

\end{observation}

\begin{cor}\label{td-sym}
The topologizing diagram functor~(\ref{def.td}) carries commutative algebras in ${\sf Trans}$ to symmetric monoidal $\infty$-categories.  
In particular, the following $\infty$-categories have natural symmetric monoidal structures:
\begin{eqnarray}
\nonumber
\Bun
&
\text{\rm with }
&
\bigl((X\to K),(Y\to K)\bigr)\mapsto \bigl(X\underset{K}\times Y\to K\bigr)~;
\\
\nonumber
\cBun
&
\text{\rm with }
&
\bigl((X\to K),(Y\to K)\bigr)\mapsto \bigl(X\underset{K}\times Y\to K\bigr)~;
\\
\nonumber
\Exit
&
\text{\rm with }
&
\bigl((X\overset{\sigma}{\leftrightarrows} K),(Y\overset{\tau}{\leftrightarrows} K)\bigr)\mapsto \bigl(X\underset{K}\times Y\overset{\sigma\times \tau}{\leftrightarrows} K\bigr)~.
\end{eqnarray}

\end{cor}

The poset $\ZZ_{\geq 0}$ of non-negative integers carries a natural symmetric monoidal structure given by addition.  
Notice the functor
\[
{\sf dim}\colon \Exit \longrightarrow \ZZ_{\geq 0}~,\qquad \Bigl(\exit(K)\xra{(X\overset{\sigma}{\underset{\pi}{\leftrightarrows}} K)}\Exit\Bigr)~\mapsto~ \Bigl(\exit(K)\xra{k \mapsto {\sf dim}_{\sigma(k)}((\pi^{-1}k)_p)}\ZZ_{\geq 0}\Bigr)
\]
given by assigning to each such indicated $K$-point of $\Exit$ the functor $\exit(K) \to \ZZ_{\geq 0}$ whose value on $k\in K$ is the local dimension of the stratum of the fiber in which $\sigma(k)$ lies.
Explicitly, in the case that $K=\ast$ is a point, this functor assigns to each pointed stratified space $(x\in X)$ the local dimension ${\sf dim}_x(X_p)$ of the stratum $X_p\subset X$ in which $x$ lies.  
This functor is naturally a symmetric monoidal functor.

Notice the functor
\begin{equation}\label{dim.functor}
{\sf dim}\colon \Vect^{\sf inj} \longrightarrow \ZZ_{\geq 0}
\end{equation}
that carries each vector space to its dimension; this functor is naturally a symmetric monoidal functor.  
Notice the functor between $\infty$-categories
\begin{equation}\label{vect-to-exit}
\Vect^\sim\longrightarrow \Exit~,\qquad |(V\to \Delta^p\times \Delta^\bullet_e)|\mapsto |(V\overset{\sf zero}{\leftrightarrows} \Delta^p\times \Delta^\bullet_e)|~,
\end{equation}
which is expressed here as a map of simplicial spaces,
Because the underlying stratified space of the direct sum of two vector spaces agrees with the product of their underlying stratified spaces, this last functor is naturally a symmetric monoidal functor.  
Because vector space dimension equals topological dimension, this symmetric monoidal functor naturally lies over the symmetric monoidal $\infty$-category $\ZZ_{\geq 0}$.

We are now prepared to present our definition of the (fiberwise) constructible tangent bundle.  
The definition relies on a characterization, which is Proposition~\ref{unique-action} below, which we prove in~\S\ref{sec.char}.  
\begin{definition}\label{def.tangent}
The \emph{
fiberwise
constructible tangent bundle} is the unique symmetric monoidal functor under $\Vect^\sim$,
\[
\sT^{\sf fib} \colon \Exit \longrightarrow \Vect^{\sf inj}~,
\]
that carries both closed morphisms and embedding morphisms to equivalences.  

\end{definition}

\begin{prop}\label{unique-action}
The fiberwise constructible tangent bundle exists and is uniquely characterized by the conditions of Definition~\ref{def.tangent}.  

\end{prop}

\subsubsection{\bf Elucidating $\sT^{\sf fib}$}\label{tangent-explicit}
We postpone the construction of the 
fiberwise
constructible tangent bundle $\sT^{\sf fib}$ to~\S\ref{sec.char}.
Here we provide a simple description of $\sT^{\sf fib}$ value-wise, without taking care to manage its coherent functoriality.
Let $X= (X\to P)$ be a stratified space.
Lemma~3.3.5 of~\cite{striat} gives that the underlying $\infty$-groupoid of $\exit(X)$ is the coproduct of spaces $\underset{p\in P} \coprod X_p$; each cofactor $X_p$ is the underlying space of a smooth manifold.  
Lemma~3.3.5 of~\cite{striat} also gives that the space of morphisms in $\exit(X)$ from the $X_p$ component to the $X_q$ component is the space $L_{pq}$, which is the $q$-stratum of the link of $X_p\subset X$; it is the underlying space of a smooth manifold.  
As established in~\S7.3 of~\cite{aft1}, this smooth manifold is equipped as a smooth proper fiber bundle $X_p\xla{\pi_{pq}} L_{pq}$ as well as a smooth open embedding $L_{pq}\times(0,1) \xra{\gamma_{pq}} X_q$. 
This $\gamma_{pq}$ is determined up to a choice of a collaring of the link of $X_p$ in $X$ within the blow-up, or unzip, of $X$ along $X_p$, which is a contractible choice. 
The functor $\sT_X\colon \exit(X) \to \Vect^{\sf inj}$ can be described on objects and morphisms as follows.
\begin{itemize}
\item {\bf Objects:} $\sT_X$ restricts to the component $X_p$ of the underlying $\infty$-groupoid as the local system 
\[
\sT X_p \colon X_p \longrightarrow \Vect^{\sf inj}
\]
classifying the tangent bundle of the smooth manifold $X_p$.

\item {\bf Morphisms:} $\sT_X$ restricts to the component $L_{pq}$ of the mapping space as the diagram of local systems
\[
\xymatrix{
L_{pq} \ar[d]_-{(\pi_{pq},{\gamma_{pq}}_|)}     \ar[rr]
&&
\Ar(\Vect^{\sf inj})  \ar[d]^-{(\ev_s,\ev_t)}
\\
X_p\times X_q  \ar[rr]^-{\sT X_p\times \sT X_q}
&&
\Vect^{\sf inj}\times \Vect^{\sf inj}
}
\]
in which the top horizontal functor
$L_{pq}\to \Ar(\Vect^{\sf inj})$
classifies the composite morphism between vector bundles over $L_{pq}$,
\[
\pi_{pq}^\ast \sT X_p     
\overset{~\sD\pi_{pq}^\vee~}\hookrightarrow
\sT L_{pq}   
\overset{~\rm zero~}\hookrightarrow
\sT L_{pq}\oplus \epsilon^1    
\overset{~\sD\gamma_{pq}~} \hookrightarrow
\gamma_{pq}^\ast \sT X_q     ~,
\]
with $\sD\pi_{pq}^\vee$ the dual of the derivative of $\pi_{pq}$, and with the middle map the inclusion as zero in the second coordinate.

\end{itemize}
\begin{remark}
In the above we made use of an identification $V\simeq V^\vee$ for each object in $\Vect$.
Such an identification is tantamount to a choice of an inner product on $V$, the space of which is contractible.  
It is because of these contractible choices that we opted for a hands-off construction of the constructible tangent bundle, which occupies~\S\ref{sec.tgt-char}.
Indeed, we found it impractical to directly manage these contractible choices in order to define a functor between $\infty$-categories $\Exit \xra{\sT^{\sf fib}}\Vect^{\sf inj}$.  

\end{remark}

\subsection{Framings: one stratified space at a time}\label{sec.framings}
Here, we define \emph{tangential structures} and give some examples thereof.
We do this through the heuristic description of the constructible tangent bundle given in the previous section~\S\ref{tangent-explicit}, still postponing the rigorous construction thereof to the coming section~\S\ref{sec.tgt-char}.
Our choice to postpone this construction is a consequence of our choice to define the constructible tangent bundle as it is characterized by its functoriality in $\Exit$.
Indeed, this functoriality requires the \emph{fiberwise} constructible tangent bundle of a constructible bundle---we postpone this degree of generalization so as not to complicate the discussion of a framings on a single stratified space.

\begin{definition}\label{def.tangential-structure}
A \emph{tangential structure} is an $\infty$-category $\cB$ equipped with a functor $\cB \to \Vect^{\sf inj}$.
The \emph{$\infty$-category of tangential structures} is the $\infty$-category $\Cat_{\infty/\Vect^{\sf inj}}$.
\end{definition}

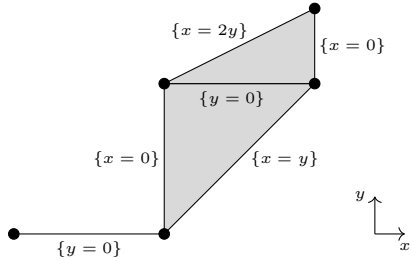
\begin{figure}[ht]
\centering

\begin{tikzpicture}
\tikzstyle{every node}=[font=\tiny]
	
\draw (0,0) to (2,0);
\draw[fill=lightgray] (2,0) to (2,2) to (4,2) to (2,0);
\draw[fill=lightgray] (2,2) to (4,2) to (4,3) to (2,2);

\draw[fill] (0,0) circle [radius=0.07];
\draw[fill] (2,0) circle [radius=0.07];
\draw[fill] (2,2) circle [radius=0.07];
\draw[fill] (4,2) circle [radius=0.07];
\draw[fill] (4,3) circle [radius=0.07];

\node at (1,-0.2) {$\{y=0\}$};
\node at (3.6,1) {$\{x=y\}$};
\node at (2.6,2.7) {$\{x=2y\}$};
\node at (1.5,1) {$\{x=0\}$};
\node at (4.5,2.5) {$\{x=0\}$};
\node at (2.9,1.8) {$\{y=0\}$};

\draw[->] (4.8,0) to (5.2,0) node[below]{$x$};
\draw[->] (4.8,0) to (4.8,0.5) node[left]{$y$};

\end{tikzpicture}

\caption{A framed stratified $2$-manifold with its constructible tangent bundle; the 1-dimensional strata are labelled by the embedding of the tangent bundle.}

\end{figure}

\begin{notation}
We often denote a tangential structure $\cB \to \Vect^{\sf inj}$ simply as its domain $\cB$, with the functor to $\Vect^{\sf inj}$ understood.
\end{notation}

\begin{definition}\label{def.B.framings}
Let $\cB\to \Vect^{\sf inj}$ be a tangential structure.  
Let $X$ be a stratified space.
A \emph{$\cB$-framing on $X$} is a lift $\varphi$ as in the commutative diagram among $\infty$-categories:
\[
\xymatrix{
&&
\cB  \ar[d]
\\
\sT_X\colon \exit(X)  \ar[r]  \ar@{-->}[urr]^-\varphi
&
\Exit  \ar[r]^-{\sT^{\sf fib}}
&
\Vect^{\sf inj}  .
}
\]
A \emph{$\cB$-framed stratified space} is a pair $(X,\varphi)$ consisting of a stratified space $X$ together with a $\cB$-framing $\varphi$ on $X$.

\end{definition}

\begin{remark}
If the base change $\cB_{|X} \to \exit(X)$ is the unstraightening of a functor $\exit(X) \xra{\cB_{|X}} \Spaces$, such a lift is the datum of a point in the limit $\varphi \in {\sf lim}\bigl(\exit(X) \xra{\cB_{|X}} \spaces\bigr)$.
Through Theorem~A.9.3 of~\cite{HA}, we interpret such a $\varphi$ as a global section of a constructible sheaf on the stratified space $X$.  
In this way, we think of the notation of a tangential structure as an expansion of that of a constructible sheaf on the site of stratified spaces and open embeddings among them.  

\end{remark}

\begin{observation}\label{taus-limits}
The $\infty$-category $\Cat_{\infty/\Vect^{\sf inj}}$ of tangential structures is presentable.
In particular, it admits products, which are fiber products of $\infty$-categories over $\Vect^{\sf inj}$.  
\end{observation}

\begin{construction}[$\cB = \un{\cS}$]\label{def.tau-A}
Taking products with $\Vect^{\sf inj}$ defines a functor
\[
\Cat_\infty \ni \cS~\mapsto~ \un{\cS}:= \bigl(\cS\times \Vect^{\sf inj} \xra\pr \Vect^{\sf inj} \bigr)
\in \Cat_{\infty/\Vect^{\sf inj}} 
\]
from the $\infty$-category $\Cat_\infty$ of $\infty$-categories to the $\infty$-category of tangential structures. 

\end{construction}

\begin{example}
Let $\cS$ be an $\infty$-category.
An $\un{\cS}$-framing on a stratified space $X$ is simply a functor $\exit(X) \to \cS$.  
\end{example}

\begin{example}\label{ex.n}
Let $n\geq 0$.
Consider the full $\infty$-subcategory $\Vect^{\sf inj}_{\leq n}\subset \Vect^{\sf inj}$ consisting of those vector spaces whose dimension is at most $n$.  
The fully faithful functor $\Vect^{\sf inj}_{\leq n} \hookrightarrow \Vect^{\sf inj}$ is a tangential structure.
A $\Vect^{\sf inj}_{\leq n}$-framing on a stratified space $X$ is unique if it exists, and it exists if and only if the dimension of $X$ is at most $n$.
\end{example}

\begin{notation}\label{ex.99}
Let $n\geq 0$.
Let $\cB\to \Vect^{\sf inj}$ be a tangential structure.
We denote the tangential structure
\[
\cB_{\leq n}:= \cB_{|\Vect^{\sf inj}_{\leq n}}\longrightarrow \Vect^{\sf inj}
\]
which is the product of $\cB\to \Vect^{\sf inj}$ and $\Vect^{\sf inj}_{\leq n}\to \Vect^{\sf inj}$ in the $\infty$-category of tangential structures.  

\end{notation}

We now introduce the essential geometric notion new to this article: \emph{vari-framings}.  
Recall that a framing on a smooth $n$-manifold $M$ is an identification $\epsilon^n_M\simeq \sT_M$ of vector bundles over $M$ between the trivial rank $n$-bundle and its tangent bundle.  
A vari-framing is a direct imitation of this classical notion.  
\begin{definition}\label{def.epsilon}
The \emph{vari-framing} tangential structure is the functor
\[
{\sf vfr}:=\Bigl(\ZZ_{\geq 0} \xra{~\epsilon^\bullet~} \Vect^{\sf inj}~,\qquad (i\leq j)\mapsto (\RR^i\xra{\sf inc}\RR^j)~\Bigr)~,
\]
taking values at Euclidean vector spaces and standard inclusions among them as the first coordinates. 

\end{definition}

\begin{remark}
Let $X$ be a space, which we identify as an $\infty$-groupoid.  
A functor $X\to \Vect^{\sf inj}$ classifies a vector bundle over $X$.
A lift of this functor along $\ZZ_{\geq 0} \xra{\epsilon^\bullet} \Vect^\sim$ is a trivialization of this vector bundle.  

\end{remark}

\begin{notation}
For each dimension $n$, the \emph{$n$-vari-framing} tangential structure is the composite functor
\[
{\sf vfr}_{\leq n} := {\sf vfr}_{n}~:=~\bigl( [n]\xra{~i\mapsto i~}\ZZ_{\geq 0} \longrightarrow \Vect^{\sf inj}  \bigr) ~.
\]

\end{notation}

\begin{remark}\label{remark.1}
Let $X$ be a stratified space.
A \emph{vari-framing on $X$} is a lift
\[
\xymatrix{
&&&
\ZZ_{\geq 0}  \ar[d]^-{\sf \epsilon^\bullet}
\\
\exit(X)  \ar[rrr]_-{\sT_X}  \ar@{-->}[urrr]^-{\varphi}
&&&
\Vect^{\sf inj} .
}
\]
A \emph{$n$-vari-framing (on $X$)} is exactly the same as a vari-framing on $X$ provided the dimension of $X$ is at most $n$. If the dimension of $X$ is greater than $n$, then there does not exist an $n$-vari-framing on $X$.
\end{remark}

\begin{notation}\label{def.vfr.X}
For $X$ a stratified space, the \emph{space of vari-framings (on $X$)} is 
\[
{\sf vfr}(X)~:=~ \Map_{/\Vect^{\sf inj}}\bigl(\exit(X) , \ZZ_{\geq 0}\bigr)~,
\]
the space of such lifts as in Remark~\ref{remark.1}.

\end{notation}

\begin{remark}\label{epsilon.section}
The functor $\epsilon^\bullet \colon \ZZ_{\geq 0} \to \Vect^{\sf inj}$ of Definition~\ref{def.epsilon} is a section of the functor ${\sf dim}\colon \Vect^{\sf inj} \to \ZZ_{\geq 0}$ from~(\ref{dim.functor}).

\end{remark}

\begin{observation}\label{99}
Let $X$ be a stratified space.
In light of Remark~\ref{epsilon.section}, for $\varphi$ a vari-framing on $X$, the functor $\varphi\colon \exit(X) \to \ZZ_{\geq 0}$ necessarily agrees with the dimension functor ${\sf dim}\colon \exit(X) \to \ZZ_{\geq 0}$.  
Consequently, a  vari-framing on $X$ is precisely an equivalence between functors
\[
\epsilon^{\sf dim}_X~\underset{\varphi}\simeq~\sT_X~\colon \exit(X) \longrightarrow \Vect^{\sf inj}~.
\]
\end{observation}

\begin{remark}
The term `vari-framing' is short for `variform framing' which reflects that the framing varies over strata.  
See Remark~\ref{vari-framing-explicit}.  

\end{remark}

\begin{example}\label{pre-hemi}
For each dimension $n$ we give an example of a vari-framed stratified $n$-manifold: $\DD^n$ (see Figures~\ref{fig.11} and~\ref{fig.22}).
We will reexamine this in~\S\ref{sec.disks} (see Definition~\ref{hemi-disks}).
The underlying space of $\DD^n$ is the unit disk in $\RR^n$.  
Its stratifying poset $P^{\tr}$ is the right-cone on the poset whose underlying set is $[n-1]\times\{\pm\}$ with partial order $(i,\sigma)\leq (i',\sigma')$ meaning $i<i'$ if $i\neq i'$ and otherwise $\sigma=\sigma'$.  
The stratification $\DD^n \to P^{\tr}$ assigns to a vector $x=(x_1,\dots,x_n)$ the cone-point if $\lVert x\rVert <1$ and otherwise $(i,{\sf sign}(x_{i}))$ where $i:={\sf Max}\{j\mid x_j\neq 0\}$.  
The conically smooth structure on this stratified topological space is inherited from the smooth structure of the closed $n$-disk.  
\\
The dimension constructible bundle $\epsilon^{\sf dim}_{\DD^n}\colon \exit(\DD^n)\to \Vect^{\sf inj}$ is, by definition, the composition
\[
\epsilon^{\sf dim}_{\DD^n}\colon \exit(\DD^n)\longrightarrow P^{\tr} \longrightarrow [n-1]^{\tr} = [n] \xra{i\mapsto \RR^{i}} \Vect^{\sf inj}
\]
given as follows:
the first arrow is the functor to the stratifying poset; the second arrow is the right-cone on the projection $[n-1]\times\{\pm\}\xra{\pr}[n-1]$; and the third arrow is as indicated which carries each relation to the standard inclusions among Euclidean spaces.
\\
The constructible tangent bundle $\sT_{\DD^n} \colon \exit(\DD^n) \to \Vect^{\sf inj}$ is given on objects and morphisms as follows.  
Let $x\in \exit(\DD^n)$ be an object.
The value $\sT_{\DD^n}(x):=\RR^n$ if $\lVert x\rVert <1$; if $\lVert x\rVert =1$ with $i:={\sf Max}\{j\mid x_j\neq 0\}$, this value $\sT_{\DD^n}(x):=x^{\perp_{i}}$ is the orthogonal complement of the vector $x\in \RR^{i}$.
We now describe the value of the functor $\sT_{\DD^n}$ on morphisms in $\exit(\DD^n)$.  
Because the conically smooth structure on $\DD^n$ is inherited from the smooth structure on $\RR^n$, each exit-path $[0,1]\to \DD^n$ is, in particular, a smooth map $[0,1]\to \RR^n$.  
The value of $\sT_{\DD^n}$ on each morphism in $\exit(\DD^n)$, which is an exit-path in $\DD^n$, is implemented by parallel transport with respect to the Levi-Civita connection of the standard Riemannian metric on $\RR^n$, which is flat.  
\\
We now define a vari-framing $\epsilon^{\sf dim}_{\DD^n}\simeq \sT_{\DD^n}$ on $\DD^n$.  
The canonical functor $\exit(\DD^n) \to P^{\tr}$ is an equivalence between $\infty$-categories.
Select a section $P^{\tr} \to \exit(\DD^n)$ of this projection whose value on $(i,\sigma)$ is $\sigma \cdot e_i$, where $e_i$ is the $i^{\rm th}$ standard basis vector in $\RR^n$.  
By direct inspection, the restriction $(\sT_{\DD^n})_{|P^{\tr}}$ of the constructible tangent bundle of $\DD^n$ along this section is identical to the composition $P^{\tr} \to [n] \xra{i\mapsto \RR^i} \Vect^{\sf inj}$ appearing above.  
Therefore, the further restriction $((\sT_{\DD^n})_{|P^{\tr}})_{|\exit(\DD^n)} \simeq \epsilon^{\sf dim}_{\DD^n}$ is identical to the dimension constructible bundle on $\DD^n$.  
The sought vari-framing $\epsilon^{\sf dim}_{\DD^n}\simeq \sT_{\DD^n}$ is then implemented through the equivalence between the composite functor $\exit(\DD^n) \to P^{\tr} \to \exit(\DD^n)$ and the identity functor on $\exit(\DD^n)$.

\end{example}

%
%

\begin{figure}[ht]
\begin{tikzpicture}

\draw[fill] (0,0) circle [radius=0.07];
\draw[fill] (3,0) circle [radius=0.07];
\draw[midarrow=0.5] (0,0) to (3,0);

\node at (0,0.4) {$(0,-)$};
\node at (3,0.4) {$(0,+)$};
\node at (1.5,0.4) {$1$};

\end{tikzpicture}
\caption{A vari-framed hemispherically stratified $1$-disk with strata labelled by the poset.}
\label{fig.11}
\end{figure}
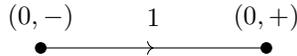

%
%
%
%
%
%
%
%

\begin{figure}[ht]
\centering

\begin{tikzpicture}
\tikzstyle{every node}=[font=\tiny]

\draw[fill=lightgray] (0,1) [out=75, in=180] to node [opacity=0] (TOP) {} (1.5,2) [out=0, in=105] to (3,1) [out=-105, in=0] to node [opacity=0] (BOTTOM) {} (1.5,0) [out=-180, in=-75] to (0,1);

\draw[fill] (0,1) circle [radius=0.07];
\draw[fill] (3,1) circle [radius=0.07];

\draw[->] (0,1) [out=75, in=180] to (1.5,2);
\draw (1.5,2) [out=0, in=105] to (3,1);
\draw[->] (0,1) [out=-75,in=-180] to (1.5,0);
\draw (1.5,0) [out=0, in=-105] to (3,1);

\node at (-0.3,1.25) {$(0,-)$};
\node at (3.3,1.25) {$(0,+)$};
\node at (1.5,2.3) {$(1,+)$};
\node at (1.5,-0.3) {$(1,-)$};
\node at (1.5,1) {$2$};

\draw[->] (3.8,0) to (4.2,0) node[below]{1};
\draw[->] (3.8,0) to (3.8,0.5) node[left]{2};

\end{tikzpicture}
\caption{A vari-framed hemispherically stratified $2$-disk with stata labelled by the poset.}
\label{fig.22}
\end{figure}
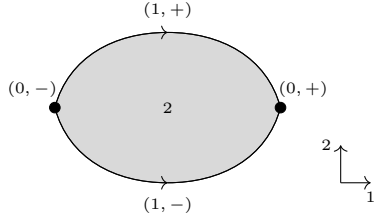

\subsubsection{\bf Solid framings}
Some tangential structures have additional functorialities.
We abstract these as \emph{solid} tangential structures.
Generally, these are less sensitive to stratifications than other tangential structures. 
This is articulated as Proposition~\ref{solid.refs}, to come.  

\begin{definition}\label{def.solid}
A \emph{solid tangential structure} is a right fibration $\cB\to \Vect^{\sf inj}$.

\end{definition}

Any tangential structure determines a solid tangential structure via right fibration-replacement.  
We describe in the case of $\infty$-groupoids. 
For $B\to \Vect^{\sim}$ a functor between $\infty$-groupoids, we denote the pullback,
\[
\xymatrix{
\Ar(\Vect^{\sf inj})_{|B}  \ar[rr]  \ar[d]
&&
\Ar(\Vect^{\sf inj})  \ar[d]^-{{\sf ev}_t}
\\
B    \ar[r]
&
\Vect^{\sim}  \ar[r]
&
\Vect^{\sf inj}  .
}
\]
Evaluation at the source defines a functor 
\[
{\sf ev}_s\colon \Ar(\Vect^{\sf inj})_{|B} 
\longrightarrow \Vect^{\sf inj}~.
\]
\begin{definition}[solid $B$-framings]\label{d2}
Let $n\geq 0$ and $B \to \BO(n)$ be a map between spaces.  
The \emph{solid $B$-framing} tangential structure is the functor
\[
{\sf s} B~:=~\Bigl(\Ar(\Vect^{\sf inj})_{|B} 
\xra{~\ev_s~}
\Vect^{\sf inj}\Bigr)~.
\]
The \emph{$n$-manifold} tangential structure is the solid $B$-framing tangential structure for the case that $B \xra{\simeq} \BO(n)$.
The \emph{solid $n$-framing} tangential structure is the solid $B$-framing tangential structure for the case that $B\simeq \ast \xra{\{\RR^n\}} \BO(n)$ -- this is the tangential structure ${\sf s}\ast = (\Vect^{\sf inj}_{/\RR^n} \to \Vect^{\sf inj})$.

\end{definition}

\subsubsection{\bf Explicating framings}
We explicate these notions of framings. 
We begin with the weakest of the structures above: stratified $n$-manifold structures.

\begin{remark}
Let $X$ be a stratified space.
A \emph{solid $n$-framing} of $X$ is a lift as in the diagram among $\infty$-categories,
\[
\xymatrix{
&&
\Vect^{\sf inj}_{/\RR^n} \ar[d]_-{\ev_s}  
\\
\exit(X) \ar[rr]^-{\sT_X}  \ar@{-->}[urr]^-{\varphi}
&&
\Vect^{\sf inj}
.
}
\]
The \emph{space of solid $n$-framings of $X$} is 
\[
{\sf sfr}_n(X)~:=~\Map_{/\Vect^{\sf inj}}\bigl(\exit(X) , \Vect^{\sf inj}_{/\RR^n}\bigr)~,
\]
the space of functors over $\Vect^{\sf inj}$.

\end{remark}

\begin{remark}
Let $X$ be a stratified space.
An \emph{$n$-manifold structure on $X$} is a functor 
\[
\xymatrix{
&&
\Ar(\Vect^{\sf inj})_{|\BO(n)}  \ar[d]_-{\ev_s}  
\\
\exit(X) \ar[rr]^-{\sT_X}  \ar@{-->}[urr]^-{\varphi}
&&
\Vect^{\sf inj}
.
}
\]

\end{remark}

\begin{remark}
Explicitly, a stratified $n$-manifold is a stratified space $X$ together with a rank $n$ vector bundle $\eta$ on $X$, as well as an injection of constructible vector bundles $\sT_X \hookrightarrow \eta$.  
By way of a zero-section, such data determines an embedding of $X$ into the total space of the cokernel vector bundle: $X\hookrightarrow \eta/\sT_X$.  
The domain of this embedding is a topological manifold of dimension $n$ (with regularity that we will not articulate).  
In this way, we regard $\sT_X\hookrightarrow \eta$ as an infinitesimal thickening of $X$ as a smooth $n$-manifold; in particular, the topological dimension of $X$ is bounded above by $n$.
If $X$ is a smooth $n$-manifold, then the injection $\sT_X\hookrightarrow \eta$ is an isomorphism, and so $\eta$ is no more data than that of $X$ alone.

\end{remark}

\begin{remark}\label{rem.1}
A $B$-structure on a stratified $n$-manifold $(X,\sT_X\hookrightarrow \eta)$ is a lift of the classifying map $X\xra{\eta} \BO(n)$ to $B$.
In particular, a solid $n$-framing on such a stratified $n$-manifold is an isomorphism of vector bundles $\eta \cong \epsilon_X^n$ with the trivial rank $n$ vector bundle over $X$.  
In particular, the space of solid $n$-framings on such a stratified $n$-manifold is a torsor for $\Map\bigl(X,\sO(n)\bigr)$, provided a solid $n$-framing exists; this depends only on the underlying space of $X$. 

\end{remark}

\begin{remark}\label{vari-framing-explicit}
A vari-framing on a stratified $n$-manifold $(X,\sT_X\hookrightarrow \eta)$ is a \emph{trivialization} of $(\sT_X\hookrightarrow \eta)$, by which we mean an equivalence of functors $\exit(X) \to \Ar(\Vect^{\sf inj})$
\[
\bigl(\sT_{X}\hookrightarrow \eta\bigr)~\underset{\varphi}\simeq~ \bigl(\epsilon^{{\sf dim}}_X\hookrightarrow \epsilon^n_X\bigr)~.
\]
Here, $\sT_X\colon \exit(X) \to \Vect^{\sf inj}$ is the functor defined in~\S\ref{tangent-explicit} while
$\epsilon^{\sf dim}_X \colon \exit(X) \to \Vect^{\sf inj}$ is the functor whose restriction to the $i$-dimensional stratum $X_i$ is the constant functor at $\RR^i$, and whose value on the space of morphisms from $X_i$ to $X_j$ is the inclusion $\RR^i\hookrightarrow \RR^j$ as the first coordinates.  
In particular, a vari-framing on a stratified space $X$ determines, for each dimension $i$, an equivalence of vector bundles on  $X_i$
\[
\sT_{X_i} ~\underset{\varphi_i}\simeq~ \epsilon^i_{X_i}~.
\]
Also, a vari-framing on $X$ determines, for each pair of dimensions $i\leq j$ with link system $X_i \xla{\pi_{ij}} L_{ij}\xra{\gamma_{ij}}X_j$, an identification of the projection
\[
\epsilon^i_{L_{ij}}
\underset{\varphi_i}\simeq
\pi_{ij}^\ast \sT_{X_i}
\xla{~D\pi_{ij}~}  
\sT_{L_{ij}}  
\xla{~\pr~}
\sT_{L_{ij}}\oplus \epsilon^1_{L_{ij}} 
\simeq 
(\sT_{X_j})_{|L_{ij}} 
\underset{{\varphi_j}_{|L_{ij}}} \simeq
\epsilon^j_{L_{ij}}
\]
with the standard projection off of the final coordinates. 
There is a similar description for finite sequences of dimensions $i_0\leq \dots \leq i_p$.  

\end{remark}

\begin{remark}
From the description of link data in Remark~\ref{vari-framing-explicit}, a vari-framing on a stratified space is not solely a framing of each stratum, but also coherent compatibility of stratum-by-stratum framings along links between strata.  
\end{remark}

\begin{remark}\label{r1}
Let $X$ be a stratified space with dimension at most $n$.
Let $\ov{\varphi}\colon \sT_X \hookrightarrow \epsilon^n$ be a solid $n$-framing on $X$.  
A \emph{compatible vari-framing on $(X,\ov{\varphi})$} is a factorization 
\[
\xymatrix{
&&
\epsilon^{\sf dim}_X  \ar[d]^-{\rm inclusion}
\\
\sT_X  \ar[rr]^-{\ov{\varphi}}  \ar@{-->}[urr]^-{\varphi}
&&
\epsilon^n_X    
}
\]
in the $\infty$-category $\Fun\bigl( \exit(X) , \Vect^{\sf inj}\bigr)$.
Informally, such a compatible vari-framing is a system, for each stratum $X_p\subset X$, of factorizations between vector bundles over $X_p$:
\[
\xymatrix{
\sT_{X_p}  \ar[dr]_-{\ov{\varphi}}  \ar@{--}[rr]^-{\varphi}  
&&
\epsilon^{{\sf dim}(X_p)}_{X_p}  \ar[dl]_-{\rm inclusion}
\\
&
\epsilon^n_{X_p}
&
,
}
\]
which is coherently compatible across links between strata.  
In particular, such a compatible vari-framing determines a splitting of the solid $n$-framing alone each stratum, and in particular a framing of each stratum.  
\end{remark}

\begin{remark}
A solid $n$-framing on a stratified space $X$ does not generally determine a vari-framing on $X$.  
Figures~\ref{fig.object} and~\ref{fig.object'} illustrate that there can be a multitude of vari-framings that are compatible with a given solid $2$-framing. Here, the solid $2$-framing is understood to be that inherited from the ambient plane in which these pictures are drawn.  
Examining the flow of the 1st coordinate vector field for each of these two vari-framings reveals that these two compatible vari-framings are indeed distinct.
Figure~\ref{fig.33} depicts an object that admits a solid $2$-framing but not a vari-framing, as we now explain.
After Remark~\ref{r1}, a vari-framing on this stratified space would determine a splitting of the ambient solid 2-framing along the circle-stratum.  
No such splitting exists because there is no non-vanishing normal vector field to the circle-stratum that extends as a non-vanishing vector field on the disk it encloses.

\end{remark}

\begin{figure}[ht]
\centering

\begin{tikzpicture}
\tikzstyle{every node}=[font=\tiny]

\draw[fill] (2,0.5) circle [radius=0.07];
\draw[fill] (1,1) circle [radius=0.07];
\draw[->-] (1,1) to (2,0.5);
\draw[fill] (1,3) circle [radius=0.07];
\draw[->-] (1,1) to (2,2);
\draw[->-] (1,3) to (2,2);

\draw[fill=lightgray] (2,2) [out=75, in=180] to node [opacity=0] (TOP) {} (3.5,3) [out=0, in=105] to (5,2) [out=-105, in=0] to node [opacity=0] (BOTTOM) {} (3.5,1) [out=-180, in=-75] to (2,2);

\draw[fill] (2,2) circle [radius=0.07];
\draw[fill] (5,2) circle [radius=0.07];

\draw[midarrow=0.25] (2,2) [out=75, in=180] to (3.5,3);
\draw[midarrow=0.25] (3.5,3) [out=0, in=105] to (5,2);
\draw[midarrow=0.75] (2,2) [out=-75, in=-180] to (3.5,1);
\draw[midarrow=0.75] (3.5,1) [out=0, in=-105] to (5,2);

\draw[fill] (TOP) circle [radius=0.07];
\draw[fill] (BOTTOM) circle [radius=0.07];
\draw[->-] (TOP.center) to (BOTTOM.center);

\draw[fill] (6,2) circle [radius=0.07];
\draw[->-] (5,2) to (6,2);
\draw[fill] (7,2) circle [radius=0.07];

\draw[->-] (6,2) [out=75, in=105] to (7,2);
\draw[->-] (6,2) [out=-75, in=-105] to (7,2);

\draw[->] (5.3,0.8) to (5.8,0.8) node[below]{1};
\draw[->] (5.3,0.8) to (5.3,1.3) node[left]{2};

\end{tikzpicture}
\caption{A 2-vari-framed stratified space.}
\label{fig.object}
\end{figure}
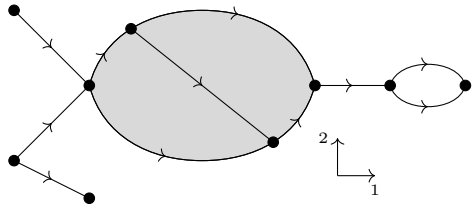

\begin{figure}[ht]
\centering

\begin{tikzpicture}
\tikzstyle{every node}=[font=\tiny]

\draw[fill] (0,1) circle [radius=0.07];
\draw[fill] (1,1) circle [radius=0.07];
\draw[->-] (0,1) to (1,1);
\draw[fill] (1,3) circle [radius=0.07];
\draw[->-] (1,1) to (2,2);
\draw[->-] (1,3) to (2,2);

\draw[fill=lightgray] (2,2) [out=75, in=180] to  (3.5,3) [out=0, in=105] to node [opacity=0] (TOP) {} (5,2) [out=-105, in=0] to  (3.5,1) [out=-180, in=-75] to node [opacity=0] (BOTTOM) {} (2,2);

\draw[fill] (2,2) circle [radius=0.07];
\draw[fill] (5,2) circle [radius=0.07];

\draw[midarrow=0.75] (2,2) [out=75, in=180] to (3.5,3);
\draw[midarrow=0.75] (3.5,3) [out=0, in=105] to (5,2);
\draw[midarrow=0.25] (2,2) [out=-75, in=-180] to (3.5,1);
\draw[midarrow=0.25] (3.5,1) [out=0, in=-105] to (5,2);

\draw[fill] (TOP) circle [radius=0.07];
\draw[fill] (BOTTOM) circle [radius=0.07];
\draw[->-] (BOTTOM.center) to (TOP.center);

\draw[fill] (6,2) circle [radius=0.07];
\draw[fill] (7,2) circle [radius=0.07];
\draw[->-] (5,2) to (6,2);

\draw[->-] (6,2) [out=75, in=105] to (7,2);
\draw[->-] (6,2) [out=-75, in=-105] to (7,2);

\draw[->] (5.3,0.8) to (5.8,0.8) node[below]{1};
\draw[->] (5.3,0.8) to (5.3,1.3) node[left]{2};

\end{tikzpicture}
\caption{A stratified 2-vari-framed space whose underlying solid $2$-framed stratified space is the same as the underlying solid $2$-framed space of Figure~\ref{fig.object}.}
\label{fig.object'}
\end{figure}
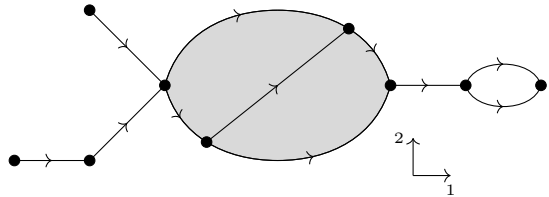

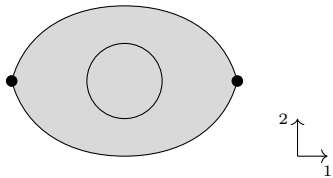
\begin{figure}[ht]
\begin{tikzpicture}

\draw[fill=lightgray] (0,1) [out=75, in=180] to node [opacity=0] (TOP) {} (1.5,2) [out=0, in=105] to (3,1) [out=-105, in=0] to node [opacity=0] (BOTTOM) {} (1.5,0) [out=-180, in=-75] to (0,1);

\draw[fill] (0,1) circle [radius=0.07];
\draw[fill] (3,1) circle [radius=0.07];

\draw (1.5, 1) circle [radius=0.5];

%

\tikzstyle{every node}=[font=\tiny]

\draw[->] (3.8,0) to (4.2,0) node[below]{1};
\draw[->] (3.8,0) to (3.8,0.5) node[left]{2};

\end{tikzpicture}
\caption{A stratified space that admits a solid $2$-framing but not a vari-framing.}
\label{fig.33}
\end{figure}

For $M$ a smooth $n$-manifold that admits a framing, Remark~\ref{rem.1} states that its space of solid $n$-framings 
\[
\sfr_n(M)~\simeq ~\Map(M,\sO(n))
\]
is a torsor for the space of maps to the group of orthogonal transformations of $\RR^n$.
We will give an analogous description of the space of vari-framings on a stratified space.
We first introduce the stratified version of the orthogonal group $\sO(n)$.  
\begin{construction}[$\un{\sO}$]\label{def.O}
We define the $\infty$-category $\un{\sO}$ over the poset $\ZZ_{\geq 0}$ of non-negative integers, as the following complete Segal space.
Space of functors $[p]\to \un{\sO}$ over a functor $[p]\xra{i_\bullet} \ZZ_{\geq 0}$ classifying $i_0\leq \dots \leq i_p$ is 
\[
\sO(\RR^{i_0}\subset \dots \subset \RR^{i_p})~\subset~\sO(\RR^{i_p})~,
\]
the underlying space of the smooth submanifold consisting of those orthogonal linear maps that preserve each standardly embedded $\RR^{i_u}\subset \RR^{i_p}$.  
A morphism $[p]\xra{\rho}[q]\xra{i_\bullet} \ZZ_{\geq 0}$ in $\bdelta_{/\ZZ_{\geq 0}}$ determines the smooth homomorphism $\sO(\RR^{i_0} \subset \cdots \subset \RR^{i_q}) \to \sO(\RR^{i_{\rho(0)}} \subset \cdots \subset \RR^{i_{\rho(p)}})$ given by restricting transformations from $\RR^{i_q}$ to $\RR^{i_{\rho(p)}}\subset \RR^{i_q}$ and relaxing which subspaces are preserved.
It follows that $\un{\sO}$ is indeed a simplicial space over $\ZZ_{\geq 0}$.  
For each $[p]\xra{i_\bullet}\ZZ_{\geq 0}$, taking orthogonal complements defines a diffeomorphism
\[
\sO(\RR^{i_0}\subset \cdots \subset \RR^{i_p})~\cong~ \sO(\RR^{i_0})\times \underset{0<k\leq p}\prod \sO(\RR^{i_{k}-i_{k-1}})~.
\]
This implies the simplicial space $\un{\sO}$ satisfies the Segal condition.  
By inspection, the map of simplicial spaces $\un{\sO}\to \ZZ_{\geq 0}$ has the property that only degenerate simplices are carried to degenerate simplices.
Because the only equivalences of $\ZZ_{\geq 0}$ are identities, the Segal space $\un{\sO}$ satisfies the completeness condition.
Therefore, $\un{\sO}$ presents an $\infty$-category over $\ZZ_{\geq 0}$.  
Finally, because it is the case value-wise as a simplicial space, this $\infty$-category is naturally a group object among $\infty$-categories over $\ZZ_{\geq 0}$.

\end{construction}

\begin{observation}\label{framings-explicit}
For each stratified space $X$ that admits a vari-framing, such a choice of vari-framing on $X$ determines an identification of the space of vari-framings on $X$
\[
\vfr(X)~\simeq~\Map_{/\ZZ_{\geq 0}}\bigl(\exit(X),\un{\sO}\bigr)
\]
with the space of functors over the poset $\ZZ_{\geq 0}$ from $\exit(X)\xra{\sf dim}\ZZ_{\geq 0}$ to the $\infty$-category $\un{\sO}$ of Construction~\ref{def.O}.  

\end{observation}

In the next result we denote the restriction $\un{\sO}_{|\leq n}:=\{0<\dots<n\}\underset{\ZZ_{\geq 0}}\times  \un{\sO}$, which is again a group object among $\infty$-categories.  
\begin{observation}\label{un-O-to-O}
For each dimension $n$, consider the pullback $\infty$-category
\[
\xymatrix{
\Ar(\un{\sO}_{|\leq n})_{|n} \ar[rr] \ar[dd]
&&
\Ar(\un{\sO}_{|\leq n}) \ar[d]^-{\ev_{\sf t}}
\\
&&
\un{\sO}_{|\leq n} \ar[d]
\\
\{n\} \ar[rr]
&&
[n]
;
}
\]
it consists of those arrows whose target lies over $n\in [n]$.  
There is a canonical inclusion $\sO(n) \to \Ar(\un{\sO}_{|\leq n})_{|n}$.
Evaluation at the target is a left adjoint to this inclusion, and in particular witnesses an $\infty$-groupoid completion:
\[
\ev_{\sf t}\colon \sB\Ar(\un{\sO})_{|n}
\xra{~\simeq~}
\sO(n)
~.
\]

\end{observation}

\begin{remark}\label{refs.vfr.vs.sfr}
Let $\ov{X}$ be a smooth $n$-manifold with corners, which we regard as a stratified space in a standard manner: two points in $\ov{X}$ belong to the same stratum if they belong to the same connected component of a face of $\ov{X}$ with respect to its corner structure (see Example~3.5.7 of~\cite{aft1}).
In~\S4 of~\cite{aft1} it is verified that the canonical functor $\exit(\ov{X}) \to X$ witnesses $\infty$-groupoid completion: the resulting map $\sB\exit(X) \xra{\simeq} X$ is an equivalence between spaces.
Suppose $\ov{X}$ is connected, and has exactly one open $n$-dimensional stratum.
The identity morphism $\ov{X}\to X$ witnesses a refinement morphism, in which the codomain---denoted without an overscore---carries the trivial stratification. 
Through Proposition~\ref{solid.refs}, to come, the space of solid $n$-framings on $\sfr_n(\ov{X})$ is the space of framings of the interior $X$ as a smooth $n$-manifold. 
Provided there exists a vari-framing on $\ov{X}$, Observation~\ref{un-O-to-O} gives that
each vari-framing on $\ov{X}$ determines a solid $n$-framing on $X$:
\[
\vfr(\ov{X}) \simeq \Map_{/\ZZ_{\geq 0}}\bigl(\exit(\ov{X}),\un{\sO}\bigr) 
\longrightarrow 
\Map\bigl(X,\sO(n)\bigr)\simeq \sfr_n(X)~.
\]
\end{remark}

We now demonstrate the vast difference between these vari-framings and solid framings.

\begin{remark}
Consider the hemispherical disk $\DD^3$ of Example~\ref{pre-hemi}.
The space $\vfr(\DD^3)$ of vari-framings of $\DD^3$ is canonically the limit of the diagram of spaces:
\[
\xymatrix{
\sO(1)\times\sO(1)\times \sO(1) \ar[d]_-{{\sf id} \times \oplus} \ar[rr]^-{\oplus\times {\sf id}}
&&
\sO(2)\times \sO(1) \ar[d]_-{\oplus} 
&&
\sO(1)\times \sO(1)\times \sO(1) \ar[ll]_-{\oplus\times {\sf id}} \ar[d]_-{{\sf id} \times \oplus}
\\
\sO(1)\times\sO(2) \ar[rr]^-{\oplus} 
&&
\sO(3)
&&
\sO(1)\times \sO(2) \ar[ll]_-{\oplus}
\\
\sO(1)\times\sO(1)\times \sO(1) \ar[u]^-{{\sf id} \times \oplus} \ar[rr]^-{\oplus\times {\sf id}}
&&
\sO(2)\times \sO(1) \ar[u]^-{\oplus} 
&&
\sO(1)\times \sO(1)\times \sO(1) \ar[ll]_-{\oplus\times {\sf id}} \ar[u]^-{{\sf id} \times \oplus}
.
}
\]
In particular, $\Omega^3 \sO(3)$ maps to $\vfr(\DD^3)$ as a monomorphism (i.e., as an inclusion of a union of path components).  
Because $\Omega^3 \sO(3)\simeq \Omega^3 S^3$ has infinitely many path components, we conclude that $\vfr(\DD^3)$ has infinitely many path components.  
On the other hand, using that the underlying space of $\DD^3$ is contractible, the space $\sfr_3(\DD^3)\simeq \Map\bigl(\DD^3,\sO(3)\bigr) \simeq \sO(3)$ has two components.  
\end{remark}

\begin{remark}
Restriction along $\partial \DD^3 \hookrightarrow \DD^3$ defines a map 
$
\vfr(\DD^3) \to \vfr(\partial \DD^3)
$
between spaces of vari-framings.  
Example~\ref{pre-hemi} gives that $\vfr(\DD^3)$ is not empty.
We conclude that $\vfr(\partial \DD^3)$ is not empty.
On the other hand, via Remark~\ref{refs.vfr.vs.sfr}, the refinement $\partial \DD^3 \to S^2$ determines an equivalence $\emptyset = \sfr_2(S^2) \xra{\simeq} \sfr_2(\partial \DD^3)$.  
This demonstrates that, while a smooth $n$-manifold might not admit a solid $n$-framing, it might have a conically smooth refinement that admits a vari-framing.  

\end{remark}

For each stratified space $X$, the assignment 
\[
X\underset{\rm open}\supset U~ \mapsto~ \vfr(U)~\in \Spaces
\]
defines a constructible sheaf of spaces on $X$.
In contrast with ordinary differential topology, where `constructible' and `locally constant' are the same notion, the existence of a vari-framing on a stratified space has strong global implications.  
We demonstrate this as the next remark.

\begin{remark}
Let $\ov{M}$ be a smooth $n$-manifold with boundary, whose boundary is connected.
Regard $\ov{M}$ as a stratified space whose underlying topological space is that of $\ov{M}$, whose stratification $\ov{M}\to \{n-1<n\}$ is such that the $(n-1)$-stratum is precisely the boundary, and whose conically smooth structure is inherited from the smooth structure of $\ov{M}$.  
A choice of a collar-neighborhood of $\partial \ov{M} \subset \ov{M}$ determines a smooth open embedding $\partial \ov{M} \times \RR \hookrightarrow M$ into the interior.  
The space ${\sf vfr}(\ov{M})$ of vari-framings on $\ov{M}$ is the pullback in the diagram,
\[
\xymatrix{
{\sf vfr}(\ov{M})   \ar[rr]  \ar[d]
&&
{\sf fr}(M)  \ar[d]
\\
{\sf fr}(\partial \ov{M})\times {\sf fr}(\RR)    \ar[rr]^-{\times}
&&
{\sf fr}(\partial \ov{M} \times \RR)  ,
}
\]
in which ${\sf fr}(N)$ is the space of framings of a smooth manifold $N$.  
In particular, the space of vari-framings of $\ov{M}$ extending a given framing of its interior is either empty or a torsor for $\Map(\partial \ov{M} , \Omega \RR\PP^{n-1})$.  
Similarly, the space of vari-framings of $\ov{M}$ extending a given framing of its boundary is either empty or is equivalent to the space of based maps $\Map_\ast\bigl( \ov{M}{/\partial \ov{M}} , \sO(n)\bigr)$.
\end{remark}

\begin{remark}
Vari-framings are but one example of a local structure on stratified spaces.
One could invent others: stipulate a structure stratum-wise, each of which could be quite different in nature, in addition to specified interactions among them across links between strata.
For a given such stratified structure $\cB$, the existence of a $\cB$-structure on $X$ will then restrict the global topology of $X$.
In particular, the general obstruction to the existence of a $\cB$-structure need not be measured by the cohomology of the underlying space of $X$, unless $\cB$ is locally constant. 

\end{remark}

\subsection{Characterization of the constructible tangent bundle}\label{sec.tgt-char}
Here we prove Proposition~\ref{unique-action}, which characterizes the (fiberwise) constructible tangent bundle.

\subsubsection{\bf Outline}
We outline the logic of the proof of Proposition~\ref{unique-action}.   
We seek a symmetric monoidal functor $\Exit \to \Vect^{\sf inj}$ under $\Vect^\sim$ that is unique among all such that carry both closed and embedding morphisms to equivalences.
Such a functor is equivalent to the data of a morphism $\Map([\bullet],\Exit) \to \Map([\bullet],\Vect^{\sf inj})$ between simplicial $\cE_\infty$-spaces under $\Map([\bullet],\Vect^\sim)\simeq \Vect^\sim$ that carries embedding and closed simplices 
to degenerate simplices.
Here, $[\bullet]$ is a variable object of $\bdelta$, and $\Map([\bullet],\Exit)$ defines the simplicial space with values $[p]\mapsto \Map([p],\Exit)$.
We can accommodate this localization using cospans.
Namely, the morphism we seek is equivalent to a morphism of simplicial symmetric monoidal functors under $\Vect^\sim$, 
\[
\w{\sT^{\sf fib}}\colon \cSpan(\Exit^{[\bullet]})^{\sf cls\text{-}emb} \to \Map([\bullet],\Vect^{\sf inj})~.
\]
This is because the target of this arrow lies in simplicial symmetric monoidal $\infty$-groupoids.
Consequently, such an arrow factors through the value-wise classifying space with respect to the closed and embedding morphisms, thereby implementing the desired localization.  
We construct $\w{\sT^{\sf fib}}$, and verify it is unique, by induction on $\bullet$.  
The inductive step exploits the following parallel facts.
\begin{enumerate}
\item As a $\Vect^\sim$-module in spaces, $\Map([p],\Vect^{\sf inj})$ is free on $\Map([p-1],\Vect^{\sf inj})$.
This is tantamount to our vector spaces being finite dimensional, and the zero vector space being initial in $\Vect^{\sf inj}$.

\item As a $\Vect^\sim$-module in $\infty$-categories, $\cSpan(\Exit^{[p]})^{\sf cls\text{-}emb}$ receives a final functor from the free $\Vect^\sim$-module on a certain $\infty$-subcategory of $\cSpan(\Exit^{[p]})^{\sf cls\text{-}emb}$.
By virtue of $\w{\sT^{\sf fib}}$ being a morphism of $\Vect^\sim$-modules over $\ZZ_{\geq 0}$, this $\infty$-subcategory must map to the generating space $\Map([p-1],\Vect^{\sf inj})$ of the above fact (1).
Therefore, the restriction of $\w{\sT^{\sf fib}}$ to this $\infty$-subcategory must factor through a standard face map to $\cSpan(\Exit^{[p-1]})^{\sf cls\text{-}emb}$.

\end{enumerate}
The base case of this induction reveals the entirety of the construction of the fiberwise constructible tangent bundle.  
Namely, for $[p]=[0]$ the map $\w{\sT^{\sf fib}}$ is a symmetric monoidal functor under $\Vect^\sim$
\[
 \cSpan(\Exit)^{\sf cls\text{-}emb} \longrightarrow \Vect^\sim~.
\]  
Because the target of this arrow is a symmetric monoidal $\infty$-groupoid, the existence and uniqueness of this arrow is equivalent to the functor $\Vect^\sim \to  \cSpan(\Exit)^{\sf cls\text{-}emb}$ being a final functor.  
Because $\Vect^\sim$ is an $\infty$-groupoid, this finality is the simple observation that the cospan of pointed stratified spaces
\[
(x\in X) \xra{~\sf cls~} (x\in \ov{X}_p)\xla{~\emb~} (0\in \sT_x X_p)
\]
represents a final object in cospans from $(x\in X)$ to vector spaces. Here, the first arrow is opposite of the inclusion $\ov{X}_p\subset X$ of the closure of the stratum $X_p\subset X$ in which $x$ lies; the second arrow is an exponential map from the tangent space at $x$ of the smooth manifold $X_p$.

\subsubsection{\bf Cospans in $\Exit$}

Recall from Definition~\ref{def.Exit-classes} the $\infty$-subcategories
\[
\Exit^{\sf cls}~\subset ~\Exit ~ \supset ~ \Exit^{\sf emb}~,
\]
each of which contains the equivalences.  
For each $\infty$-category $\cK$, consider the $\infty$-category $\Exit^\cK:=\Fun(\cK,\Exit)$ of functors.  
This $\infty$-category is equipped with a pair of $\infty$-subcategories
\[
\Fun^{\sf cls}(\cK,\Exit)~\subset~\Exit^\cK~\supset~\Fun^{\sf emb}(\cK,\Exit)
\]
consisting of the same objects, which are functors $\cK \to \Exit$, and those natural transformations through closed/embedding morphisms.  
Using that constructible bundles pull back (see~\S6 of~\cite{striat}), and that closed morphisms in $\Bun$ are opposites of proper constructible embeddings, this pair of $\infty$-subcategories satisfies Criterion~\ref{cspan-rep} of~\S\ref{sec.cspan}.
Through the conclusions of~\S\ref{sec.cspan}, cospans in $\Exit^\cK$ by closed and embedding morphisms organize as an $\infty$-category 
\[
\cSpan(\Exit^\cK)^{\sf cls\text{-}emb}~.
\]
Because finite pullbacks preserve colimits in $\strat$, through Observation~\ref{cspan-ot-rep} we see that this $\infty$-category inherits a symmetric monoidal structure from that of $\Exit$.  
Furthermore, there is a sequence of composable symmetric monoidal functors 
\begin{equation}\label{seq}
\Vect^\sim \xra{~\sf diag~}\Map(\cK,\Vect^\sim)\longrightarrow \Map(\cK,\Exit)   \underset{\rm Obs~\ref{from-gpd}}\longrightarrow   \cSpan(\Exit^\cK)^{\sf cls\text{-}emb}~.
\end{equation}
The symmetric monoidal $\infty$-category $\cSpan(\Exit^\cK)^{\sf cls\text{-}emb}$ under $\Vect^\sim$ is contravariantly functorial in the argument $\cK$, by construction.  
In particular, we obtain a simplicial symmetric monoidal $\infty$-category $\cSpan(\Exit^{[\bullet]})^{\sf cls\text{-}emb}$ under $\Vect^\sim$.

For $\cK$ an $\infty$-category, consider the full symmetric monoidal $\infty$-subcategory
\[
\cSpan\bigl(\Exit^{\cK^{\tl}}\bigr)_0^{\sf cls\text{-}emb} ~\subset~\cSpan\bigl(\Exit^{\cK^{\tl}}\bigr)^{\sf cls\text{-}emb} 
\]
consisting of those functors $\cK^{\tl} \to \Exit$ whose value on the cone-point $\ast \to \Exit$ classifies a pointed stratified space $(x\in X)$ that is isomorphic to an open cone on some compact stratified space $(\ast\in \sC(L))$ equipped with its cone-point.

The $\oo$-category $\cSpan(\Exit)^{\sf cls\text{-}emb}$ is designed for the following technical result.
This result articulates a sense in which the symmetric monoidal $\infty$-groupoid $\Vect^\sim$, which classifies vector bundles, approximates $\Exit$.
The basic geometric idea is that each pointed stratified space $(x\in X)$ admits a canonical cospan in $\Exit$ to $\sT_x X_p$, the tangent space at $x$ of the stratum $X_p\subset X$ in which $x$ lies.  
\begin{lemma}\label{0-final}
For each compact stratified space $Z$, the restriction of the symmetric monoidal structure
\[
\Vect^\sim  \times \cSpan\bigl(\Exit^{\exit(\ov{\sC}(Z))}\bigr)_0^{\sf cls\text{-}emb} 
\xra{~{}~\times~{}~}
\cSpan\bigl(\Exit^{\exit(\ov{\sC}(Z))}\bigr)^{\sf cls\text{-}emb}
\]
defines a final functor.
In the case $Z=\emptyset$, the functor
\[
\Vect^\sim
\longrightarrow
\cSpan(\Exit)^{\sf cls\text{-}emb}~.
\]
is a final functor.  

\end{lemma}

\begin{proof}
Using Quillen's Theorem~A, we show that for each object $\bigl(X\overset{\sigma}{\leftrightarrows} \ov{\sC}(Z)\bigr)$ of $\cSpan\bigl(\Exit^{\exit(\ov{\sC}(Z))}\bigr)_0^{\sf cls\text{-}emb} $,  the classifying space of the under $\infty$-category
\[
\sB\Bigl(\Vect^\sim  \times \cSpan\bigl(\Exit^{\exit(\ov{\sC}(Z))}\bigr)_0^{\sf cls\text{-}emb} \Bigr)^{(X\overset{\sigma}{\leftrightarrows} \ov{\sC}(Z))/}~\simeq~\ast
\]
is terminal. 
We do this by demonstrating a terminal object in this under $\infty$-category.  
Manifestly, the $\infty$-category $\cSpan\bigl(\Exit^{\exit(\ov{\sC}(Z))}\bigr)^{\sf cls\text{-}emb}$ has a factorization system whose left factor is \emph{closed} morphisms and whose right factor is \emph{embedding} morphisms.  
Therefore, it is sufficient to use the following logic.
\begin{enumerate}
\item 
Demonstrate a terminal object 
\[
\bigl(X \overset{\sigma}{\leftrightarrows} \ov{\sC}(Z)\bigr)
\xra{~\sf cls~}
\bigl(\ov{X} \overset{\ov{\sigma}}{\leftrightarrows} \ov{\sC}(Z)\bigr)
\]
of the under $\infty$-category $\Fun^{\sf cls}\Bigl(\exit(\ov{\sC}(Z)),\Exit\bigr)^{(X\overset{\sigma}{\leftrightarrows} \ov{\sC}(Z))/}$.

\item 
Demonstrate an initial object
\[
\bigl(V\times R \overset{0\times \tau}{\leftrightarrows} \ov{\sC}(Z)\bigr)
\xra{~\sf emb~}
\bigl(\ov{X} \overset{\ov{\sigma}}{\leftrightarrows} \ov{\sC}(Z)\bigr)
\]
of the over $\infty$-category $\Bigl(\Vect^\sim \times \Fun^{\sf emb}_0\bigl(\exit(\ov{\sC}(Z)),\Exit\bigr)\Bigr)_{/(\ov{X}\overset{\ov{\sigma}}{\leftrightarrows} \ov{\sC}(Z))}$.
Here the subscript $0$ indicates the full $\infty$-subcategory of those functors that carry the cone-point to a pointed stratified space $(\sigma(\ast)\in R_{|\ast})$ for which $\sigma(\ast)$ is the unique $0$-dimensional stratum.

\end{enumerate}

The first terminal object $\bigl(\ov{X} \overset{\ov{\sigma}}{\leftrightarrows} \ov{\sC}(Z)\bigr)$ can be described as follows.  
As a properly embedded constructible subspace, $\ov{X}\subset X$ is the intersection of all such that contain the image of the section $\sigma$.  
The projection $\ov{X} \to \ov{\sC}(Z)$ is a constructible bundle; this constructible bundle is equipped with a section over the section $\sigma$, by construction.  
Also by construction, there is a canonical closed morphism $\bigl(X \overset{\sigma}{\leftrightarrows} \ov{\sC}(Z)\bigr) \to \bigl(\ov{X} \overset{\ov{\sigma}}{\leftrightarrows} \ov{\sC}(Z)\bigr)$ in the functor $\infty$-category $\Fun\bigl(\exit(\ov{\sC}(Z)),\Exit\bigr)$.  
The construction of $\ov{X}$ makes this closed morphism manifestly terminal among all such. 

Consequently, we can assume that the canonical closed morphism $X\to \ov{X}$ is an equivalence.
We now face the problem of showing there is an initial embedding morphism $\bigl(V\times R \overset{\tau}{\leftrightarrows} \ov{\sC}(Z)\bigr)
\xra{~\sf emb~}
\bigl(X\overset{\ov{\sigma}}{\leftrightarrows} \ov{\sC}(Z)\bigr)$
between functors $\exit(\ov{\sC}(Z))\to \Exit$.  
Using Lemma~\ref{emb-init}, it is sufficient to argue the existence of an initial object of the over $\infty$-category 
${(\Strat^{\ast/})^{\emb}}_{/(x\in X)}$.
From the very definition of a stratified space in the sense of~\S3 of~\cite{aft1}, there is a basic neighborhood $\bigl((0,\ast)\in \RR^i\times \sC(L)\bigr)\hookrightarrow (x\in X)$, and such basic neighborhoods form a basis for the topology about $x\in X$.  
This implies the full $\infty$-subcategory consisting of basic neighborhoods is initial.  
In~\S4 of~\cite{aft1} it is shown that this $\infty$-subcategory is in fact a contractible $\infty$-groupoid.  
This verifies the desired initiality.

We now verify the second clause of the lemma.
We continue with the same notation above.  
In the case $Z=\emptyset$, the object $\bigl(X\overset{\sigma}{\leftrightarrows} \ov{\sC}(Z)\bigr)$ is simply a pointed stratified space $(x\in X)$.  
As such, $\ov{X}$ is simply the closure $\ov{X}_p\subset X$ of the stratum $x\in X_p\subset X$ in which $x$ lies.
In particular, $x$ lies in the top-dimensional stratum of the stratified space $\ov{X}_p$.  
Therefore, the object $R= \ast$ is a point.  
Thus, the under $\infty$-category $(\Vect^\sim)^{(x\in X)/}$ has a terminal object;
in particular its classifying space is terminal.
Through Quillen's Theorem~A, this proves that the functor $\Vect^\sim \to \cSpan(\Exit)^{\sf cls\text{-}emb}$ is final. 
\end{proof}

In the proof of Lemma~\ref{0-final} above, we made use of the following technical result.  
We denote by 
\[
(\Strat^{\ast /})^{\emb}
\]
the $\infty$-category of pointed stratified spaces and pointed open embeddings among them.
\begin{lemma}\label{emb-init}
Let $Z$ be a compact stratified space.
Let $\exit(\ov{\sC}(Z))\to  \Exit$ be a functor classifying a constructible vector bundle $X\overset{\sigma}{\leftrightarrows} \ov{\sC}(Z)$ equipped with a section.
Evaluation at the cone-point
\[
\Bigl(\Vect^\sim \times \Fun^{\sf emb}_0\bigl(\exit(\ov{\sC}(Z)),\Exit\bigr)\Bigr)_{/(X\overset{\sigma}{\leftrightarrows} \ov{\sC}(Z))}
\longrightarrow
{(\Strat^{\ast/})^{\emb}}_{/(\sigma(\ast)\in X_{|\ast})}
\]
is a right adjoint.  

\end{lemma}

\begin{proof}
In~\S6 of~\cite{striat} we construct a colimit diagram among stratified spaces over $\ov{\sC}(Z)$:
\[
\xymatrix{
&&
\Link_{X_{|\ast|}}(X)\times Z\times (\Delta^1\smallsetminus \Delta^{\{0\}}) \ar[rr]^-{\gamma}  \ar[d]
&&
X_{|Z}\times (\Delta^1\smallsetminus \Delta^{\{0\}})    \ar[dd]
\\
\Link_{X_{|\ast|}}(X)  \ar[d]_-{\pi}  \ar[rr]^-{\{0\}}
&&
\Link_{X_{|\ast|}}(X)\times Z\times \Delta^1  \ar[drr]
&&
\\
X_{|\ast}  \ar[rrrr]
&&
&&
X
}
\]
in which the map $\pi$ is proper and constructible and the map $\gamma$ is open.  
Such a diagram is determined upon a choice of collaring of the above link over the standard collaring $Z\times \Delta^1 \to \ov{\sC}(Z)$.  
Because such collarings form a basis for the topology about $X_{|\ast}\subset X$, the compactness of $Z$ grants the existence of $0<\epsilon\leq 1$ for which the image of the restricted section lies in the image of $\gamma$:
\[
{\sf Image}\bigl(\sigma_{|Z\times (0,\epsilon)}\bigr)~\subset ~{\sf Image}\bigl(\gamma_{|\Link_{X_{|\ast|}}(X)\times Z\times (0,\epsilon)}\bigr)~.
\]
(In the above expression we have made use of a standard identification $\Delta^1\cong [0,1]$ that carries $\Delta^{\{0\}}$ to $\{0\}$.)
Therefore, such collarings can be chosen to guarantee that $\epsilon$ can be taken to be $1$:
\[
{\sf Image}(\sigma_{|Z\times (\Delta^1\smallsetminus \Delta^{\{0\}})})~\subset~{\sf Image}(\gamma)~.
\]
We thusly obtain a functor
\[
\gamma_\ast\pi^\ast\colon {(\Strat^{\ast/})^{\emb}}_{/(\sigma(\ast)\in X_{|\ast})} \longrightarrow \Fun^{\sf emb}\bigl(\exit(\ov{\sC}(Z)),\Exit\bigr)_{/\bigl(X\overset{\sigma}{\leftrightarrows} \ov{\sC}(Z)\bigr)}
\]
given by pullback along $\pi$ and pushforward along $\gamma$; in terms of transversality sheaves this is the assignment
\[
\scriptstyle
X_{|\ast}\underset{\rm open}\supset U~{}~\mapsto~ {}~
\Bigl(
U\underset{\pi^{-1}(U)} \coprod \pi^{-1}(U)\times Z\times \Delta^1 \underset{\pi^{-1}(U)\times Z\times (\Delta^1\smallsetminus \Delta^{\{0\}})}\coprod \gamma\bigl(\pi^{-1}(U)\times Z\times (\Delta^1\smallsetminus \Delta^{\{0\}}\bigr) \Bigr)~\underset{\rm open}\subset~X~.
\]
We now argue that this functor $\gamma_\ast \pi^\ast$ is a left adjoint to evaluation at the cone-point.

By inspection, there is a canonical identification of the composite functor
\[
{(\Strat^{\ast/})^{\emb}}_{/(\sigma(\ast)\in X_{|\ast})} 
\xra{~\gamma_\ast \pi^\ast~}
\Fun^{\sf emb}\bigl(\exit(\ov{\sC}(Z)),\Exit\bigr)_{/\bigl(X\overset{\sigma}{\leftrightarrows} \ov{\sC}(Z)\bigr)}
\xra{~\ev_\ast~}
{(\Strat^{\ast/})^{\emb}}_{/(\sigma(\ast)\in X_{|\ast})}
\]
with the identity functor.  
It remains to construct a counit transformation $\gamma_\ast \pi^\ast \circ \ev_\ast \to {\sf id}$. 
Consider an object $\bigl(W\overset{\sigma}{\leftrightarrows} \ov{\sC}(Z)\bigr)\xra{~\sf emb~}\bigl(X\overset{\sigma}{\leftrightarrows} \ov{\sC}(Z)\bigr)$ 
of 
$\Fun^{\sf emb}\bigl(\exit(\ov{\sC}(Z)),\Exit\bigr)_{/\bigl(X\overset{\sigma}{\leftrightarrows} \ov{\sC}(Z)\bigr)}$.
Evaluation at the cone-point determines the map between morphism spaces
\[
\Map_{/(X\overset{\sigma}{\leftrightarrows} \ov{\sC}(Z))}\Bigl(\gamma_\ast \pi^\ast(\sigma(\ast)\in U), (W\overset{\sigma}{\leftrightarrows} \ov{\sC}(Z))\Bigr)
\longrightarrow
\Map_{/(\sigma(\ast)\in X_{|\ast})}\Bigl((\sigma(\ast)\in U), (x\in W_{|\ast})\Bigr)~.
\]
The fiber $F_f$ of this map over a pointed open embedding $f\colon U\hookrightarrow W_{|\ast}$ is the space of extensions of $f$ to open embeddings
\[
\gamma_\ast \pi^\ast(U) \longrightarrow W
\]
over and under $\ov{\sC}(Z)$.
We must show that this fiber space $F_f$ is contractible. First, note that it is clearly nonempty.

For $K$ a compact stratified space, consider a $K$-point $K\to F_f$. This consists of a conically smooth open embedding
\[
\w{f}\colon \gamma_\ast \pi^\ast(U)\times K~ \hookrightarrow~ W\times K
\]
over and under $\ov{\sC}(Z)\times K$, together with an identification of the restriction $\w{f}_{|U\times K} \simeq f\times {\sf id}_{K}$.  
First, recall that collar-neighborhoods of $X_{|\ast}\subset X$ over collar-neighborhoods of $\ast\in \ov{\sC}(Z)$ form a basis for the topology about $X_{|\ast}$. From this, together with compactness of $K$ and of $Z$, there exists a conically smooth map $\epsilon \colon X_{|\ast} \to (0,1]$. This map has the property that, for each compact subspace $C\subset X_{|\ast}$, the subspace $\gamma\bigl(\pi^{-1}(C)\times [0,\epsilon(x)]\bigr)\subset X$ lies in the image of $\w{f}$.
By the construction of $\gamma_\ast \pi^\ast$ in terms of collarings, we conclude that the map $\w{f}$ extends to an open embedding 
\[
\ov{\w{f}}\colon \gamma_\ast \pi^\ast(U)\times \sC(K)~ \hookrightarrow~ W\times \sC(K)
\]
over and under $\ov{\sC}(Z)\times \sC(K)$, together with an identification of the restriction $\ov{\w{f}}_{|U\times \sC(K)} \simeq f\times {\sf id}_{\sC(K)}$.
Therefore our $K$-point $K\ra F_f$ is null-homotopic. This verifies the contractibility of the space $F_f$.
\end{proof}

\subsubsection{\bf Proof of Proposition~\ref{unique-action}}\label{sec.char}
There is a sequence of fully faithful functors from the $\infty$-category of symmetric monoidal $\infty$-categories
\[
\Small
\CAlg(\Cat_\infty^{\times})\hookrightarrow
\CAlg\bigl(\Psh(\bdelta)^{\times}\bigr)\simeq 
\Fun\bigl(\bdelta^{\op},\CAlg(\Spaces^{\times})\bigr)\hookrightarrow
\Fun\bigl(\bdelta^{\op},\CAlg(\Cat_\infty^{\times})  \bigr)
\]
to the $\infty$-category of simplicial symmetric monoidal $\infty$-categories, as we now explain.  
The first functor is induced by the presentation $\Cat_\infty\subset \Psh(\bdelta)$ as complete Segal spaces, which is fully faithful and preserves finite products.  
The middle equivalence is an adjunction of variables, using that the Cartesian symmetric monoidal structure of presheaf $\infty$-categories is given value-wise.  
The last functor is induced from the fully faithful inclusion $\Spaces \hookrightarrow \Cat_\infty$ as $\infty$-groupoids, which preserves finite products.

By the fully faithfulness of this functor, we can establish the existence of $\sT^{\sf fib}$ by constructing a morphism between simplicial symmetric monoidal $\infty$-categories under $\Vect^\sim\simeq \Map([\bullet],\Vect^\sim)$,
\[
\sT^{\sf fib}\colon \Map\bigl([\bullet],\Exit\bigr) \longrightarrow \Map\bigl([\bullet],\Vect^{\sf inj}\bigr)~,
\]
and argue that it is the unique such whose values on the closed and embedding simplices factor through degenerate simplices.
The sequence~(\ref{seq}) gives a morphism of simplicial symmetric monoidal $\infty$-categories $\Map([\bullet],\Exit) \to \cSpan(\Exit^{[\bullet]})^{\sf cls\text{-}emb}$. 
Of course, for each $p\geq 0$ the symmetric monoidal $\infty$-category $\Map\bigl([p],\Vect^{\sf inj}\bigr)$ is actually a symmetric monoidal $\infty$-groupoid.  
Thus, in light of this assertion that $\sT^{\sf fib}$ carries the given collection of morphisms to equivalences, it is equivalent to argue the existence and uniqueness of a morphism of simplicial symmetric monoidal $\infty$-categories under $\Vect^\sim$,
\begin{equation}\label{w-theta}
\w{\sT^{\sf fib}}\colon \cSpan(\Exit^{[\bullet]})^{\sf cls\text{-}emb} \longrightarrow \Map([\bullet],\Vect^{\sf inj})~.
\end{equation}

For $p\geq 0$, consider the full subcategory $\bdelta_{\leq p}\subset \bdelta$ consisting of those $[q]$ for which $q\leq p$.  
The canonical functor $\underset{p\geq 0} \colim \bdelta_{\leq p} \to \bdelta$ is an equivalence.
It follows that the canonical functor $\Fun\bigl(\bdelta^{\op},\CAlg(\Cat_\infty^\times)\bigr)\to\underset{p\geq 0}{\sf lim} \Fun\bigl(\bdelta_{\leq p},\CAlg(\Cat_\infty^\times)\bigr)$ is again an equivalence.
Therefore, to argue the existence and uniqueness of the morphism of simplicial symmetric monoidal $\infty$-categories~(\ref{w-theta}) under $\Vect^\sim$ it is enough to argue the existence and uniqueness of a morphism of truncated simplicial symmetric monoidal $\infty$-categories under $\Vect^\sim$
\begin{equation}\label{w-theta-p}
\w{\sT^{\sf fib}}_{|\bdelta_{\leq p}}\colon \cSpan(\Exit^{[\bullet]})^{\sf cls\text{-}emb}_{|\bdelta_{\leq p}} \longrightarrow \Map([\bullet],\Vect^{\sf inj})_{|\bdelta_{\leq p}}
\end{equation}
for each $p\geq 0$, coherently compatibly.  
We do this by induction on $p\geq 0$.

The base case of $p=0$ is the assertion that there is a unique symmetric monoidal retraction $\w{\sT^{\sf fib}}_{\{[0\}} \colon \cSpan(\Exit)^{\sf cls\text{-}emb} \to \Vect^\sim$.  
Because $\Vect^\sim$ is in particular an $\infty$-groupoid, this is equivalent to the assertion that the functor $\Vect^\sim\to \cSpan(\Exit)^{\sf cls\text{-}emb}$ induces an equivalence on classifying spaces:
\[
 \Vect^\sim\xra{~\simeq~} \sB \bigl(\cSpan(\Exit)^{\sf cls\text{-}emb}\bigr)~.
\]
This is implied by the second clause of Lemma~\ref{0-final}, which states that the functor $\Vect^\sim \to  \cSpan(\Exit)^{\sf cls\text{-}emb}$ is final.

We proceed by induction and assume that the morphism~(\ref{w-theta-p}) has been defined, and been verified as being unique, for $0\leq p<q$, coherently compatibly with the simplicial morphisms in $\bdelta_{\leq p}$.  
We must argue that there is a unique symmetric monoidal functor under $\Vect^\sim$
\begin{equation}\label{w-theta-q}
\w{\sT^{\sf fib}}_{\{[q]\}} \colon \cSpan(\Exit^{[q]})^{\sf cls\text{-}emb} \longrightarrow \Map([q],\Vect^{\sf inj})
\end{equation}
for which, for each simplicial morphism $[p]\xra{\rho}[q]$ with $p<q$, the restriction $\rho^\ast\w{\sT^{\sf fib}}_{\{[q]\}}$ is coherently identified with $\w{\sT^{\sf fib}}_{\{[p]\}}$.  
For $\cK$ an $\infty$-category, consider the full symmetric monoidal $\infty$-subcategory
\[
\Map_0(\cK^{\tl},\Vect^{\sf inj})~\subset~\Map(\cK^{\tl},\Vect^{\sf inj})
\]
consisting of those functors whose value on the cone-point $\ast$ is the zero vector space.  
The requirement of an identification of the restriction $(\w{\sT^{\sf fib}}_{\{[q]\}})_{|\{[0]\}}$ with $\w{\sT^{\sf fib}}_{\{[0]\}}$ implies that $\w{\sT^{\sf fib}}_{\{[q]\}}$ restricts as a morphism between these symmetric monoidal $\infty$-subcategories under $\Vect^\sim$
\[
\cSpan(\Exit^{[q]})_0^{\sf cls\text{-}emb}\longrightarrow\Map_0([q],\Vect^{\sf inj})~.  
\]
The requirement that $\w{\sT^{\sf fib}}_{\{[q]\}}$ be symmetric monoidal under $\Vect^\sim$ in particular requires a commutative diagram among $\infty$-categories
\[
\xymatrix{
\Vect^\sim \times \cSpan(\Exit^{[q]})_0^{\sf cls\text{-}emb}  \ar[rr]^-{\times}_-{(\rm f)}  \ar[d]_-{{\sf id}\times \w{\sT^{\sf fib}}_{\{[q]\}}}  
&&
\cSpan(\Exit^{[q]})^{\sf cls\text{-}emb}  \ar[d]^-{\w{\sT^{\sf fib}}_{\{[q]\}}}
\\
\Vect^\sim \times  \Map_0([q],\Vect^{\sf inj})  \ar[rr]^-{\oplus}
&&
\Map([q],\Vect^{\sf inj}).
}
\]
Lemma~\ref{0-final}, applied to $Z=\Delta^{q-1}$, states that the functor labeled as (f) is final.
Therefore, since $\Map([q],\Vect^{\sf inj})$ is in particular an $\infty$-groupoid, the existence and uniqueness of $\w{\sT^{\sf fib}}_{\{[q]\}}$ is equivalent to the existence and uniqueness of its restriction $\cSpan(\Exit^{[q]})_0^{\sf cls\text{-}emb} \to \Map_0([q],\Vect^{\sf inj})$.

We have required that restriction along the standard simplicial morphism $[q-1]=\{1<\dots<q\}\xra{\rho} [q]$ determines a commutative diagram among $\infty$-categories
\[
\xymatrix{
\cSpan(\Exit^{[q]})_0^{\sf cls\text{-}emb}  \ar[rr]^-{\rho^\ast}  \ar[d]_-{\w{\sT^{\sf fib}}_{\{[q]\}}}  
&&
\cSpan(\Exit^{[q-1]})^{\sf cls\text{-}emb}  \ar[d]^-{\w{\sT^{\sf fib}}_{\{[q-1]\}}}
\\
\Map_0([q],\Vect^{\sf inj})  \ar[rr]^-{\rho^\ast}_-{(\rm z)}
&&
 \Map([q-1],\Vect^{\sf inj})~.
}
\]
Because the zero vector space is initial in the $\infty$-category $\Vect^{\sf inj}$, the bottom horizontal arrow~($\rm z$) is an equivalence of $\infty$-groupoids.  
Thus, the existence and uniqueness of the functor $\w{\sT^{\sf fib}}_{\{[q]\}}$, which is the left vertical arrow in the above diagram, is implied by the existence and uniqueness of $\w{\sT^{\sf fib}}_{\{[q-1]\}}$, which is the right vertical arrow in the above diagram.  
 Our induction hypothesis on $q$ ensures the existence and uniqueness of $\w{\sT^{\sf fib}}_{\{[q-1]\}}$, coherently compatibly with the simplicial morphisms in $\bdelta_{<q}$.  
Through this logic we conclude the existence and uniqueness of $\w{\sT^{\sf fib}}_{\{[q]\}}$ as in~(\ref{w-theta-q}).  
By construction, this symmetric monoidal functor is coherently compatible with the simplicial morphisms in $\bdelta_{\leq q}$, thereby completing the inductive step.

\subsection{Vertical framings}
In section~\S\ref{sec.framings} we introduced a variety of tangential structures on one stratified space at a time.  
In this section we introduce 
the notion of \emph{fiberwise} framings on a constructible bundle $X\to K$.
Like the case that $K=\ast$, these are functors from the exit-path $\infty$-category of $X$.

\begin{notation}\label{def.theta-sub}
Let $\eta \colon \Exit \to \Vect^{\sf inj}$ be a functor.  
For each constructible bundle $X\xra{\pi} K$, the functor
\[
\eta^{\sf fib}_\pi  \colon \exit(X) \longrightarrow \Vect^{\sf inj}
\]
is the restriction of $\eta$ along the functor $\exit(X)\to \Exit$ classifying $(X\underset{K}\times X \underset\pr{\overset{\sf diag}\leftrightarrows} X)$. (See Proposition~\ref{prop.Exit}.)
The superscript ${\sf fib}$ indicates \emph{fiberwise}.
For each stratified space $X$, the functor
\[
\eta_X\colon \exit(X) \longrightarrow \Vect^{\sf inj}
\]
is the instance of the functor $\eta^{\sf fib}_\pi$ applied to the constructible bundle $X\xra{\pi} \ast$.  

\end{notation}

\begin{definition}\label{def.fib.B.framings}
Let $\cB\to \Vect^{\sf inj}$ be a tangential structure.  
Let $X\xra{\pi} K$ be a constructible bundle between stratified spaces.
A \emph{fiberwise $\cB$-framing on $\pi$} is a lift $\varphi$ as in the commutative diagram among $\infty$-categories involving the fiberwise constructible tangent bundle of $\pi$:
\[
\xymatrix{
&&
\cB  \ar[d]
\\
\sT^{\sf fib}_\pi \colon \exit(X)  \ar[r]  \ar@{-->}[urr]^-\varphi
&
\Exit  \ar[r]^-{\sT^{\sf fib}}
&
\Vect^{\sf inj}  .
}
\]
A \emph{fiberwise $\cB$-framed constructible bundle} is a pair $(X\xra{\pi} K,\varphi)$ consisting of a constructible bundle together with a fiberwise $\cB$-framing on it.

\end{definition}

\begin{observation}\label{B.to.Vect}
Base change along the fiberwise constructible tangent bundle $\sT^{\sf fib}\colon \Exit \to \Vect^{\sf inj}$ defines a functor
\[
(\sT^{\sf fib})^\ast \colon \Cat_{\infty/\Vect^{\sf inj}}  
\longrightarrow
\Cat_{\infty/\Exit}
~,\qquad
(\cB \to \Vect^{\sf inj})\mapsto   \cB_{|\Vect^{\sf inj}}~.
\]
\end{observation}

\begin{notation}\label{B.to.Vect.convention}
In what follows, we implement Observation~\ref{B.to.Vect} implicitly to associate to each tangential structure an $\infty$-category over $\Vect^{\sf inj}$.  

\end{notation}

We temporarily denote the forgetful functor $\Exit \xra{\pi} \Bun$ given by forgetting section data.  
Lemma~\ref{exit-exp} of the appendix verifies that the base change functor 
\[
\pi^\ast \colon {\Cat_\infty}_{/\Bun} ~\rightleftarrows~ {\Cat_\infty}_{/\Exit}\colon \pi_\ast~,
\]
has a right adjoint, as depicted.  
\begin{definition}[$\Bun^{\cB}$]\label{def.bun-tau}
For $\cB \to \Vect^{\sf inj}$ a tangential structure, the $\infty$-category of \emph{$\cB$-framed stratified spaces} is the $\infty$-category
\[
\Bun^{\cB}~:=~ \pi_\ast (\cB_{|\Vect^{\sf inj}})
\]
over $\Bun$.
For each class $\psi$ of morphisms of $\Bun$, the $\infty$-subcategory
\[
\Bun^{\cB,\psi}~:=~\Bun^{\psi}\underset{\Bun}\times \Bun^{\cB}~\subset~ \Bun^{\cB}
\]
which is the pullback over $\Bun^{\psi}$.  

\end{definition}

\begin{prop}\label{partial.fibration}
For each tangential structure $\cB$, the canonical projection $\Bun^{\cB}\to \Bun$ is closed-coCartesian and embedding-Cartesian. 

\end{prop}

\begin{proof}
This is a consequence of the formal result, Proposition~\ref{t4}.  
Namely, first apply that result to where the sequence of functors $\cA \to \cE \to \cB$ is $\cB_{|\Exit}\to \Exit \to \Bun$.
That the functor $\Exit\to \Bun$ is closed-Cartesian is established in~\cite{striat}.
By construction of the fiberwise constructible tangent bundle $\sT^{\sf fib}\colon \Exit \to \Vect^{\sf inj}$, it carries morphisms in $\Exit$ that are closed-Cartesian over $\Bun$ to equivalences.  
In particular, $\cB_{|\Exit}\to \Exit$ is coCartesian over such morphisms in $\Exit$.  
It follows from Lemma~\ref{fib.lemma} that $\Bun^{\cB}\to \Bun$ is closed-coCartesian.

The same logic, though using the dual version of Proposition~\ref{t4}, applies to verify that $\Bun^{\cB}\to \Bun$ is closed-Cartesian.  
\end{proof}

The next result articulates the sense in which solid framings are less sensitive to stratifications than are general tangential structures.
\begin{prop}\label{solid.refs}
If $\cB\to \Vect^{\sf inj}$ is a solid tangential structure, then the canonical functor $\Bun^{\cB} \to \Bun$ is a refinement-Cartesian fibration, and the Cartesian monodromy is an equivalence.  

\end{prop}

\begin{proof}
The same logic as in the proof of the second part of Proposition~\ref{partial.fibration} applies to verify that $\Bun^{\cB}\to \Bun$ is refinement-Cartesian, provided $\cB\to \Vect^{\sf inj}$ is a Cartesian fibration.  
\end{proof}

Using Proposition~\ref{partial.fibration}, we have the following.

\begin{observation}\label{tau-closed-covers}
The fiberwise constructible tangent bundle $\sT^{\sf fib}\colon \Exit \to \Vect^{\sf inj}$ carries closed morphisms to equivalences. Consequently, for each tangential structure $\cB$, the projection $\Bun^{\cB}\to \Bun$ preserves and detects limit diagrams that factor through $\Bun^{\cB,\cls}$.  

\end{observation}

\begin{lemma}\label{fact-sys}
For each tangential structure $\cB$, the pair of $\infty$-subcategories $(\Bun^{\cB,\cls},\Bun^{\cB,\act})$ is a factorization system on $\Bun^{\cB}$.

\end{lemma}

\begin{proof}
In~\S6 of~\cite{striat} we prove the case in which $\cB \to \Vect^{\sf inj}$ is an equivalence.
The result follows because $\Bun^{\cB} \to \Bun$ is closed-coCartesian.  
\end{proof}

Recall from Example~\ref{ex.n} the tangential structure $\Vect^{\sf inj}_{\leq n} \to \Vect^{\sf inj}$; from Notation~\ref{ex.99} the notation $\cB_{\leq n}$.  
\begin{notation}[$\Bun_{\leq n}$]\label{def.dim-n}
We use the simplified notation: 
\[
\Bun_{\leq n}~:=~ \Bun^{\Vect^{\sf inj}_{\leq n}}~.  
\]
For each tangential structure $\cB\to \Vect^{\sf inj}$, we use the simplified notation:
\[
\Bun^{\cB}_{\leq n}:=\Bun^{\cB_{\leq n}}~.
\]  
\end{notation}

\begin{remark}
Explicitly, a functor $\exit(K)\xra{(X\to K)}\Bun$ factors through $\Bun_{\leq n}$ if the fibers of $X\to K$ are bounded above in dimension by $n$.
Indeed, the space of lifts of $\sT_{\pi}\colon \exit(X)\to \Exit\xra{\sT^{\sf fib}} \Vect^{\sf inj}$ through $\Vect^{\sf inj}_{\leq n}$ is either empty or contractible, depending on the dimension of the fibers of $X\to K$.  

\end{remark}

\begin{remark}\label{r7}
Inspecting Notation~\ref{def.dim-n}, which imports Notation~\ref{ex.99}, reveals that, for each tangential structure $\cB$, the canonical functor
\[
\Bun^\cB_{\leq n}
\longrightarrow
\Bun^\cB
\]
is fully faithful, and that its image consists of those $\cB$-structured stratified spaces whose dimension is bounded above by $n$.  

\end{remark}

Recall from Construction~\ref{def.tau-A}, for each $\infty$-category $\cS$, the tangential structure $\un{\cS} = \cS\times \Vect^{\sf inj}$.  

\begin{observation}\label{Bun-tau-A}
Let $\cS$ be an $\infty$-category.
For each functor $\cK \to \Bun$, the $\infty$-category of sections
\[
\Fun_{/\Bun}\bigl(\cK,\Bun^{\un{\cS}}\bigr)~\simeq~\Fun\bigl(\Exit_{|\cK},\cS\bigr)
\]
is canonically identical to that of functors from $\Exit_{|\cK}$ to $\cS$.
In particular, an object of the $\infty$-category $\Bun^{\un{\cS}}$ is a stratified space $X$ together with a functor $\exit(X)\xra{\varphi} \cS$.  

\end{observation}

Finally, we consider vari-framed stratified spaces, and variations thereon, as they organize as an $\infty$-category.
These $\infty$-categories are designed so that they classify constructible families of vari-framed stratified spaces.
This is the primary $\infty$-category of this article.  
\begin{definition}\label{def.vari-fr}
Let $n\geq 0$.
The $\infty$-category of \emph{vari-framed (stratified) $n$-manifolds} is
\[
\Mfld^{\vfr}_n~:=~\Bun^{\vfr_n}~.
\]
For $B\to \BO(n)$ a map between spaces, the $\infty$-category of \emph{solid $B$-framed (stratified) $n$-manifolds} is
\[
\Mfld_n^B~:=~\Bun^{{\sf s}B}~.
\]
The $\infty$-category of \emph{(stratified) $n$-manifolds}, and the $\infty$-category of \emph{solid $n$-framed (stratified) $n$-manifolds}, are the special cases
\[
\Mfld_n~:=~\Mfld^{\BO(n)}_n\qquad \text{ and }\qquad \Mfld_n^{\sfr}~:=~\Mfld^{\ast}_n~.
\]

\end{definition}

\begin{remark}
Following up on Remark~\ref{r7}, the canonical functor
\[
\Mfd_n^{\vfr}
\longrightarrow
\Bun^{\vfr}
\]
is fully faithful, and it factors through $\Bun^{\vfr}_{\leq n}$ (of Notation~\ref{def.dim-n}) as an equivalence between $\infty$-categories over $\Bun$:
\[
\Mfd^{\vfr}_n~\simeq~\Bun^{\vfr}_{\leq n}~.
\]
\end{remark}

\begin{observation}
For each dimension $n$, there are projections
\[
\Bun^{\vfr}~\hookleftarrow~
\Mfld^{\vfr}_n~\rightarrow~\Mfld^{\sfr}_n~.  
\]

\end{observation}

\begin{remark}
In ordinary differential topology, a family of framed $n$-manifolds can be taken as a smooth fiber bundle $E\xra{\pi} B$ together with a trivialization of the fiberwise tangent bundle: 
$\sT_\pi:={\sf Ker}(\sT_E\xra{\sD\pi}\pi^\ast \sT_B) \underset{\varphi}\simeq \epsilon^n_E$.  
We imitate this definition simply by replacing the fiber bundle by a constructible bundle between stratified spaces.
\end{remark}

\begin{remark}
An object of $\Mfd_n^{\vfr}$ is a stratified space $X$ of dimension bounded above by $n$ together with an equivalence $\epsilon^{\sf dim}_X \underset{\varphi}\simeq \sT_X$.  
In particular, the underlying topological space of $X$ need not be a topological manifold, let alone of dimension $n$.  
Our choice for this terminology for the $\infty$-category $\Mfd_n^{\vfr}$ reflects the examples of objects therein that drive our interest: those $(X,\varphi)$ for which $X$ is a refinement of a smooth $n$-manifold.  

\end{remark}

\begin{example}\label{stand-cr}
For each $0\leq i \leq n$, the hemispherical $i$-disk $\DD^i$ of Example~\ref{pre-hemi} is a vari-framed $i$-manifold.
Through Definition~\ref{d2}, we regard $\DD^i$ as a vari-framed $n$-manifold.
The boundary $\partial \DD^i$ too is a vari-framed $n$-manifold.
The inclusion of stratified spaces $c\colon \partial \DD^i \hookrightarrow \DD^i$ is an injective constructible bundle.
The reversed mapping cylinder (see \S6.6 of~\cite{striat}) is a constructible bundle ${\sf Cylr}(c) \to \Delta^1$, thereby defining a morphism in $\cBun$ which is closed.
The canonical functor $\exit({\sf Cylr}(c)) \to \Exit \xra{\sT^{\sf fib}} \Vect^{\sf inj}$ is equipped with a lift to $\vfr$, thereby defining a closed morphism in $\Mfd_n^{\vfr}$:
\[
c\colon \DD^i\longrightarrow \partial \DD^i~;
\]
see Figure~\ref{fig.44}.
Similarly, for each $n\geq i\geq j\geq 0$, the standard projection $a\colon \DD^i \to \DD^j$ between stratified spaces is a surjective constructible bundle.
The canonical functor $\exit({\sf Cylr}(a)) \to \Exit \xra{\sT^{\sf fib}} \Vect^{\sf inj}$ is equipped with a lift to $\vfr$.  
This data defines a creation morphism in $\Mfd_n^{\vfr}$:
\[
a\colon \DD^j\longrightarrow \DD^i~;
\]
see Figure~\ref{fig.55}.

\end{example}

\begin{figure}[ht]
\begin{tikzpicture}

\draw[line width=2] (-0.02,0) to (4,0);
\draw[midarrow=0.5] (0,0) to (0,2);
\draw[line width=2] (-0.02,2) to (4,2);

\draw[->] (2,-0.5) to (2,-1);

\draw[fill] (0,-1.5) circle [radius=0.07];
\draw (0, -1.5) to (4,-1.5);

\draw[fill] (0,0) circle [radius=0.07];
\draw[fill] (0,2) circle [radius=0.07];

\end{tikzpicture}
\caption{A closed morphism in $\cMfd^{\vfr}_1$ from $\DD^1$ to $\partial \DD^0$: by definition of $\cBun$, and thereafter of $\cMfd^{\vfr}_1$, a morphism is, in particular, a proper constructible bundle over the standardly stratified 1-simplex $\Delta^1 = \bigl( \{0\} \subset [0,1]\bigr)$.  
The source, and the target, of such a morphism in $\cBun$ is the fiber over $\{0\}\subset [0,1]$, and over $\{1\}\subset[0,1]$, respectively.  
}
\label{fig.44}
\end{figure}
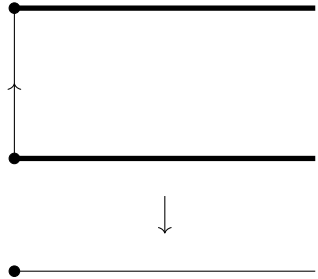

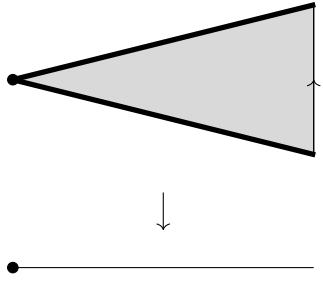
\begin{figure}[ht]
\begin{tikzpicture}

\draw[fill=lightgray] (0,1) to (4,0) to (4,2) to (0,1);

\draw[line width=2] (0,1) to (4.02,0);
\draw[line width=2] (0,1) to (4.02,2);

\draw[fill] (0,1) circle [radius=0.07];

\draw[midarrow=0.5] (4,0) to (4,2);

\draw[->] (2,-0.5) to (2,-1);

\draw[fill] (0,-1.5) circle [radius=0.07];
\draw (0, -1.5) to (4,-1.5);

\end{tikzpicture}
\caption{A creation morphism in $\cMfd^{\vfr}_1$ from $\DD^0$ to $\DD^1$: see the caption for Figure~\ref{fig.44} for how this picture indeed depicts a morphism in $\cMfd^{\vfr}_1$.
}
\label{fig.55}
\end{figure}

\begin{remark}
The $\infty$-category $\Mfd_n^{\vfr}$ captures several functorialities.
We summarize these as the following monomorphisms (see~\S\ref{sec.monos} for a discussion of monomorphisms -- it is an $\infty$-categorical notion of `faithful representation').
\begin{itemize}
\item For each vari-framed stratified space $X$ of dimension bounded above by $n$, there is a monomorphism 
\[
{\sf BAut}^{\vfr}(X)~ \hookrightarrow~ \Mfd_n^{\vfr}~.
\]  
In particular, for $M$ a smooth framed $n$-manifold, there is a monomorphism ${\sf BDiff}^{\fr}(M) \hookrightarrow \Mfd_n^{\vfr}$.
Therefore, a functor from $\Mfd_n^{\vfr}$ determines continuous representations of $\Diff^{\fr}(M)$ for each framed $n$-manifold $M$.  

\item For each vari-framed $n$-manifold $M$, there is a monomorphism from the moduli space of open embeddings 
\[
\Bigl\{ U \underset{\emb}\hookrightarrow M\Bigr\}  ~\hookrightarrow~\Mfd_{n/M}^{\vfr}~.
\]
In this way, descent with respect to open covers can be accounted for.  

\item For each vari-framed $n$-manifold $M$, there is a monomorphism from the moduli space of constructible proper embeddings
\[
\Bigl\{ X \underset{\sf p.cbl.emb}\hookrightarrow M\Bigr\}  ~\hookrightarrow~(\Mfd_{n}^{\vfr})^{M/}~.
\]
In this way, descent with respect to cutting along strata can be accounted for.

\item For each smooth framed $k$-manifold $B$, with $k\leq n$, is a monomorphism from the moduli space of smooth framed proper codimension-$(n-k)$ fiber bundles
\[
\Bigl\{E^n\underset{\sf p.sm.bdl}{\xra{~\pi~}} B~,~ {\sf Ker}(D\pi)\underset{\varphi}\simeq \epsilon_E^{n-k}\Bigr\}~\hookrightarrow~(\Mfd_n^{\vfr})^{B/}~.
\]
Therefore, a functor from $\Mfd_n^{\vfr}$ determines transfer-type maps for each such framed fiber bundle.  

\item For each smooth framed $n$-manifold $M=(M,\varphi)$, there is a monomorphism from the moduli space of vari-framed refinements
\[
\Bigl\{\w{M} \underset{\sf ref} \longrightarrow M~,~\sT_{\w{M}} \underset{\w{\varphi}/\varphi}\simeq \epsilon^{\sf dim}_{\w{M}}  \Bigr\} ~\hookrightarrow~\Mfd_{n/M}^{\vfr}~.
\]
Therefore, a functor from $\Mfd_n^{\vfr}$ determines composition-type maps for each vari-framed refinement.

\end{itemize}

\end{remark}

After Observation~\ref{Bun-tau-A}, we reconsider the constructible tangent bundle of Definition~\ref{def.tangent} as well as the fiberwise dimension constructible bundle of Definition~\ref{def.epsilon}.  
\begin{observation}\label{ob.1}
The functors $\sT^{\sf fib}\colon \Exit \to \Vect^{\sf inj}$ and $\epsilon^{\sf fib.dim}\colon \Exit\to \Vect^{\sf inj}$ both define sections
\[
\sT^{\sf fib}\colon \Bun 
\longrightarrow
\Bun^{\un{\Vect^{\sf inj}}}
\qquad \text{ and }\qquad
\epsilon^{\sf fib.dim}\colon \Bun 
\longrightarrow
\Bun^{\un{\Vect^{\sf inj}}}
\]
of the canonical projection $\Bun^{\un{\Vect^{\sf inj}}} \to \Bun$.

\end{observation}

Unpacking definitions, we recognize, for each tangential structure $\cB$, the $\infty$-category $\Bun^\cB$ as a pullback.
\begin{observation}\label{ob.2}
For each tangential structure, $\cB\to \Vect^{\sf inj}$, there is a canonical pullback diagram among $\infty$-categories
\[
\xymatrix{
\Bun^\cB  \ar[rr]  \ar[d]
&&
\Bun^{\un{\cB}}  \ar[d]
\\
\Bun  \ar[rr]^-{\sT^{\sf fib}}
&&
\Bun^{\un{\Vect^{\sf inj}}}  .
}
\]

\end{observation}

Observation~\ref{ob.2} immediately gives the following.
\begin{observation}\label{vfr-as-pb}
There are canonical pullback diagrams among $\infty$-categories
\[
\xymatrix{
\Mfld^{\vfr}_n  \ar[rr]  \ar[d]
&&
\Bun^{\vfr}  \ar[rr]  \ar[d]
&&
\Bun^{\un{\Vect^{\sf inj}}}  \ar[d]
\\
\Bun_{\leq n} \ar[rr]
&&
\Bun   \ar[rr]^-{(\epsilon^{\sf fib.dim},\sT^{\sf fib})}
&&
\Bun^{\Vect^{\sf inj}\times \Vect^{\sf inj}}.
}
\]

\end{observation}

\section{Disks}
A stratified space naturally accommodates two types of gluing procedures: unions of open subsets, thereby making use of the underlying topology; splicing along strata, thereby making use of the stratification.  
Correspondingly, there are two notions of descent for invariants of stratified spaces.
These are interrelated.
For instance, if a stratified space is sufficiently finely stratified, then these two notions of descent refine each other, in a locally constant sense: regular neighborhoods of strata determines an open cover which, up to isotopy, refines any other open cover.
We articulate this intuition of `sufficiently finely stratified' as a \emph{disk}-stratification.
To define the notion of a disk-stratification we consider suspensions of compact stratified spaces, and suspensions of structures thereon.  
To do so efficiently, we introduce such suspension by way of \emph{wreath} product.

\subsection{Iterated constructible bundles}\label{sec.iterated-bdls}
Here we exploit the universal nature of $\Bun$ as a classifying object for constructible bundles.  
To this end, we utilize the point-set entity $\bun$ from which $\Bun$ is derived (see~\S6.3 of~\cite{striat} for a review).  
The main output of this section is Definition~\ref{def.circ}. This establishes the functor $\Bun^{\un{\Bun}} \xra{\circ}\Bun$ of which we will make ongoing use.

Consider the simplicial category
\[
\bun^\bullet \colon \bdelta^{\op} \longrightarrow \Cat
\]
whose category of $p$-simplices is the subcategory of $\Fun([p]^{\op},\strat)$ whose objects are those $X_\bullet \colon [p]^{\op}\to \strat$ for which, for each $0\leq i\leq j \leq p$, $X_j\to X_i$ is a constructible bundle; the morphisms $X_\bullet \to Y_\bullet$ are those natural transformations for which, for each $0\leq i \leq j \leq p$, the square
\[
\xymatrix{
X_j \ar[r]  \ar[d] 
&
Y_j \ar[d]   
\\
X_i \ar[r] 
&
Y_i
}
\]
is a pullback.
The simplicial structure functors of $\Fun([\bullet]^{\op}, \strat)$ restrict to $\bun^\bullet$, using the result from~\S6 of~\cite{striat} that constructible bundles compose.

Among the simplicial structure functors, restriction along each $\{0<\dots<i\} \hookrightarrow \{0<\dots<p\}$ induces a right fibration
\begin{equation}\label{bun-to-strat}
\bun^{p} \longrightarrow \bun^{i}~.
\end{equation}

\begin{observation}\label{bun-n-trans}
The projection~(\ref{bun-to-strat}) for the case $i=0$ is a transversality sheaf.  
This follows by induction on $p$, after the base case $p=1$ which is proved in~\S6 of~\cite{striat}.  

\end{observation}

\begin{notation}[$\Bun^p$]\label{def.bun-n-stri}
After Observation~\ref{bun-n-trans}, the main result of~\cite{striat} associates to each such functor $\bun^p \to \bun^0=\strat$ an $\infty$-category 
\[
\Bun^p~.
\]
Heuristically, the $\infty$-category $\Bun^p$ classifies $p$-fold sequences of constructible bundles.  

\end{notation}

\begin{observation}\label{bun-to-bun-other}
Consider the subcategory $\bdelta_-\subset \bdelta$ consisting of the same objects and morphisms that preserve minima.
The $\infty$-categories $\Bun^p$ assemble as a functor
\[
\Bun^\bullet\colon \bdelta_-^{\op} \to \Cat_{\infty}~.
\]

\end{observation}

\begin{observation}\label{bun-bun}
Consider a functor $\exit(K) \xra{(Y\to K)} \Bun$ classifying the indicated constructible bundle.
From the Construction~\ref{def.tau-A} of $\Bun^{\un{\Bun}}$, there is a canonical monomorphism spaces of functors
\[
\Map_{/\Bun}\bigl(\exit(K),\Bun^{\un{\Bun}}\bigr)~\subset~\Map\bigl(\exit(Y),\Bun\bigr)
\]
which is contravariantly functorial in the variable $K$.  
As such, we have the following description of the space of functors
\[
\Map\bigl(\exit(K), \Bun^{\un{\Bun}}\bigr)~\simeq~|\{X \xra{\cbl} Y\xra{\cbl} K\times \Delta^\bullet_e\}|~.
\]
Better, for each $p>0$, there is a canonical equivalence of $\infty$-categories
\[
\Bun^p ~\simeq~ \Bun^{\un{\Bun^{p-1}}}~.
\]

\end{observation}

Observations~\ref{bun-to-bun-other} and~\ref{bun-bun} combine, which we highlight as the following.
\begin{definition}\label{def.circ}
For each $n>0$, the functor between $\infty$-categories
\[
\circ \colon \Bun^n \longrightarrow \Bun
\]
is the assignment of $K$-points 
\[
\bigl(\exit(K)\xra{(X_n\xra{\pi_n} \cdots \xra{\pi_1} K)}\Bun^{\un{\Bun}}\bigr)\mapsto \bigl(\exit(K)\xra{(X_n\xra{\pi_1\circ \dots \circ \pi_n} K)}\Bun\bigr)~.
\]

\end{definition}
\noindent
Because each distinguished class $\psi$ of morphisms in $\Bun$ is closed under composition, the restriction of this functor factors:
\[
\circ^\psi \colon \Bun^{\un{\Bun^{\psi}},\psi}\longrightarrow  \Bun^{\psi}~.
\]

\begin{remark}
We use the notation $\circ$ in Definition~\ref{def.circ} to evoke \emph{composition}.
\end{remark}

Finally, we note that the results in this section are equally valid upon replacing the role of constructible bundles by \emph{proper} constructible bundles.
We highlight this as the following.
\begin{observation}
There is an $\infty$-subcategory $\cBun^n$ of $\Bun^n$ that classifies $n$-fold sequences of \emph{proper} constructible bundles.
Furthermore, there is a factorization
\[
\cBun^n \xra{~\circ~} \cBun
\]
of the restriction of $\Bun^n \xra{\circ} \Bun$.  
\end{observation}

\subsection{Iterated framings}
Here we show that the functor $\Bun^{\un{\Bun}} \xra{\circ} \Bun$ of Definition~\ref{def.circ} respects various notions of framings.
This is articulated as Corollary~\ref{vfr-iterate} which establishes a functor $\Bun^{\vfr}\underset{\Bun}\times \Bun^{\un{\Bun}^{\vfr}} \xra{\circ} \Bun^{\vfr}$.
In brief, this is a functorial construction of a fiberwise vari-framing on a constructible bundle $X\to K$ for each factorization through fiberwise vari-framed constructible bundles $X\to Y\to K$.

We first show that the parametrizing $\infty$-category $\Exit$ respects the functor $\Bun^{\un{\Bun}}\xra{\circ}\Bun$ of Definition~\ref{def.circ}.  
\begin{lemma}\label{exit-over-bun-bun}
There is a filler in the diagram among $\infty$-categories
\[
\xymatrix{
\Exit  \ar[d]  
&&
\Exit_{|\Bun^{\un{\Bun}}}  \ar[rr]  \ar@{-->}[ll]_-\pr  \ar[d]
&&
\Exit  \ar[d]
\\
\Bun  
&&
\Bun^{\un{\Bun}}  \ar[ll]_-\pr  \ar[rr]^-{\circ} 
&&
\Bun.
}
\]
\end{lemma}
\begin{proof}
As striation sheaves, this filler is the map of presheaf on $\strat$ for which, for each stratified space $K$, the map on spaces of $K$-points is the assignment
\[
\Small
\xymatrix{
X\times \Delta^\bullet_e \ar[rr]_-{q}
&&
Y\times \Delta^\bullet_e  \ar[rr]_-{p}
&&
K\times \Delta^\bullet_e  \ar@/_1.5pc/[llll]^-{\sigma}
&
\mapsto 
&
X\times \Delta^\bullet_e \ar[rr]_-{pq}
&&
K\times \Delta^\bullet_e  \ar@/_1.5pc/[ll]^-{\sigma}~.
}
\] 
\end{proof}

The next result articulates how, for $\cA$ and $\cA'$ $\infty$-categories, $\un{\cA}$-structures combine with $\un{\cA'}$-structures over the functor $\Bun^{\un{\Bun}}\xra{\circ}\Bun$ of Definition~\ref{def.circ}.
\begin{cor}\label{exp-mon}
Each functor $\cA \times \cA' \to \cA''$ among $\infty$-categories canonically determines a filler in the diagram among $\infty$-categories
\[
\xymatrix{
\Bun^{\un{\cA}}\underset{\Bun}  \times \Bun^{\un{\Bun}^{\un{\cA'}}}  \ar@{-->}[rr]    \ar[d]
&&
\Bun^{\un{\cA''}}  \ar[d]
\\
\Bun^{\un{\Bun}}  \ar[rr]^-{\circ}  
&&
\Bun.
}
\]

\end{cor}

\begin{proof}
Through Observation~\ref{Bun-tau-A}, which explicates what $\Bun^{\un{\cS}}$ classifies for each $\infty$-category $\cS$, the problem is to construct a canonical functor 
\[
\Exit_{|\Bun^{\un{\cA}}}\underset{\Bun}  \times \Bun^{\un{\Bun^{\un{\cA'}}}} \longrightarrow \cA''
\]
satisfying the locality of Observation~\ref{Bun-tau-A}.
This locality will be manifest from the construction of the functor.  
Lemma~\ref{exit-over-bun-bun}, just above, offers the functor
\[
\Exit_{|\Bun^{\un{\cA}}}\underset{\Bun}  \times \Bun^{\un{\Bun^{\un{\cA'}}}}  \longrightarrow \Exit_{|\Bun^{\un{\cA}}} \times \Exit_{|\Bun^{\un{\cA'}}}~.
\]
The counit of the adjunction defining $\cA'\mapsto \Bun^{\un{\cA'}}$ gives the functor
\[
 \Exit_{|\Bun^{\un{\cA}}} \times \Exit_{|\Bun^{\un{\cA'}}} \longrightarrow \cA\times \cA'~.  
\]
The result follows by composing with $\cA\times \cA'\to \cA''$. 
\end{proof}

Corollary~\ref{exp-mon} applied to the direct sum functor $\Vect^{\sf inj}\times \Vect^{\sf inj} \xra{\oplus} \Vect^{\sf inj}$ gives the next result.  
\begin{cor}\label{vect-sum}
Direct sum of vector spaces $\Vect^{\sf inj}\times \Vect^{\sf inj}\xra{\oplus}\Vect^{\sf inj}$ determines a functor 
\[
\Bun^{\un{\Vect^{\sf inj}}}\underset{\Bun}  \times \Bun^{\un{\Bun}^{\un{\Vect^{\sf inj}}}}  
\xra{~{}~\circ~{}~}
\Bun^{\un{\Vect^{\sf inj}}}
\]
over $\Bun^{\un{\Bun}}\xra{\circ}\Bun$.

\end{cor}

\begin{cor}\label{B.sum}
Let $\cB$, $\cB'$, and $\cB''$ be tangential structures.
Each functor $\cB\times \cB'\to \cB''$ over the direct sum functor $\Vect^{\sf inj}\times \Vect^{\sf inj}\xra{\oplus} \Vect^{\sf inj}$ determines a commutative square among $\infty$-categories:
\[
\xymatrix{
\Bun^{\un{\cB}}\underset{\Bun}\times \Bun^{\un{\Bun^{\un{\cB'}}}}  \ar[rr]   \ar[d]
&&
\Bun^{\un{\cB''}}   \ar[d]
\\
\Bun^{\un{\Vect^{\sf inj}}}\underset{\Bun}\times \Bun^{\un{\Bun^{\un{\Vect^{\sf inj}}}}}  \ar[rr]^-{\circ}
&&
\Bun^{\un{\Vect^{\sf inj}}}  .
}
\]

\end{cor}

The next result articulates how the functor of Corollary~\ref{vect-sum} just above respects the fiberwise constructible tangent bundle as well as the fiberwise dimension constructible bundle. 
Recall from Observation~\ref{ob.1} the rephrasing of the fiberwise constructible tangent bundle, as well as the fiberwise dimension constructible bundle, as sections of the canonical projection $\Bun^{\un{\Vect^{\sf inj}}}\to \Bun$.
\begin{lemma}\label{splittings}
The diagram among $\infty$-categories
\[
\xymatrix{
\Bun^{\un{\Vect^{\sf inj}}}\underset{\Bun}\times \Bun^{\un{\Bun^{\un{\Vect^{\sf inj}}}}}  \ar[d]_-{\circ}
&&&
\Bun^{\un{\Bun}}  \ar[rrr]^-{(\sT^{\sf fib},\Bun^{\sT^{\sf fib}})}  \ar[lll]_-{(\epsilon^{\sf fib.dim},\Bun^{\epsilon^{\sf fib.dim}})}  \ar[d]_-{\circ}
&&&
\Bun^{\un{\Vect^{\sf inj}}}\underset{\Bun}\times \Bun^{\un{\Bun^{\un{\Vect^{\sf inj}}}}}   \ar[d]^-{\circ}
\\
\Bun^{\un{\Vect^{\sf inj}}}  
&&&
\Bun  \ar[rrr]^-{\sT^{\sf fib}}  \ar[lll]_-{\epsilon^{\sf fib.dim}}
&&&
\Bun^{\un{\Vect^{\sf inj}}}
}
\]
commutes.

\end{lemma}

\begin{proof}
In this proof we will use the following notation.
\begin{itemize}
\item[~]
For $X\xra{p}K$ a constructible bundle, we use the notations
\[
\sT^{\sf fib}_p\colon \exit(X) \longrightarrow \Vect^{\sf inj}\qquad\text{ and }\qquad \epsilon_p^{{\sf dim}_{(-)}(p^{-1}(p(-))}\colon \exit(X) \longrightarrow \Vect^{\sf inj}
\]
for the restrictions of $\sT^{\sf fib}$ and $\epsilon^{\sf fib.dim}$ along $\exit(X)\simeq \Exit_{|\exit(K)} \hookrightarrow \Exit$.  
\end{itemize}

Fix a $K$-point $X\xra{q}Y \xra{p}K$ of $\bun^2$.  
Inspecting their definitions, there are short exact sequences of $\Vect$-valued functors from $\exit(X)$
\[
0\longrightarrow \sT^{\sf fib}_q  \longrightarrow \sT^{\sf fib}_{pq}  \xra{~Dq~} q^\ast \sT^{\sf fib}_q \longrightarrow 0
\]
and
\[
0\longrightarrow \epsilon_X^{{\sf dim}_{(-)}(q^{-1}q(-))}  \longrightarrow \epsilon_X^{{\sf dim}_{(-)}({pq}^{-1}(pq(-))} \longrightarrow q^\ast \epsilon_Y^{{\sf dim}_{(-)}(p^{-1}(p(-))} \longrightarrow 0~.
\]
Because each of these exact sequences is comprised functors valued in \emph{finite dimensional} vector spaces, they each canonically split in $\Vect$.   
This is to say there are canonical identifications in $\Vect^{\sf inj}$:
\[
\sT^{\sf fib}_q \oplus q^\ast \sT^{\sf fib}_q ~\simeq~ \sT^{\sf fib}_{pq}\qquad \text{ and }\qquad \epsilon_X^{{\sf dim}_{(-)}(q^{-1}q(-))}\oplus q^\ast \epsilon_Y^{{\sf dim}_{(-)}(p^{-1}(p(-))}~\simeq~ \epsilon_X^{{\sf dim}_{(-)}({pq}^{-1}(pq(-))}~.
\]
This verifies the asserted commutativity for each $K$-point of $\bun^2$.  
The asserted commutativity follows because the aforementioned short exact sequences manifestly pullback along morphisms $K\to K'$.
\end{proof}

The next result is the culmination of the formal manipulations above.  
It articulates a sense in which the functor $\Bun^{\un{\Bun}} \xra{\circ} \Bun$ respects direct sums of tangential structures.  
\begin{prop}\label{B.sums}
Let $\cB$, $\cB'$, $\cB''$ be tangential structures.
A functor $\cB \times \cB'\to \cB''$ over the direct sum functor $\Vect^{\sf inj} \times \Vect^{\sf inj}\xra{\oplus}\Vect^{\sf inj}$ determines a functor
\[
\Bun^{\cB}\underset{\Bun}\times \Bun^{\un{\Bun^{\cB'}}}    
\longrightarrow
\Bun^{\cB''} 
\]
over the functor $\Bun^{\un{\Bun}}  \xra{\circ} \Bun$.
\end{prop}

\begin{proof}
Using Observation~\ref{ob.2}, the sought functor is the result of base change of the top horizontal functor in Corollary~\ref{B.sum} along the rightward pair of horizontal functors in Lemma~\ref{splittings}.

\end{proof}

Recall Definition~\ref{def.epsilon} of the functor $\epsilon^\bullet\colon \ZZ_{\geq 0}\to \Vect^{\sf inj}$.
Via this functor $\epsilon^\bullet$, the functor $\ZZ_{\geq 0} \times \ZZ_{\geq 0} \xra{+}\ZZ_{\geq 0}$ evidently lies over the direct sum functor $\Vect^{\sf inj}\times \Vect^{\sf inj}\xra{\oplus} \Vect^{\sf inj}$.
Also, the identification $\RR^n\oplus \RR^{k}\simeq \RR^{n+k}$ in $\Vect^{\sf inj}$ results in a functor $\Vect^{\sf inj}_{/\RR^n} \times \Vect^{\sf inj}_{/\RR^k}\xra{\oplus} \Vect^{\sf inj}_{/\RR^{n+k}}$ over direct sum.  
Proposition~\ref{B.sums} applied to these functors gives the next result.  
\begin{cor}\label{vfr-iterate}
There are functors 
\[
\Bun^{\vfr}\underset{\Bun}\times \Bun^{\un{\Bun^{\vfr}}}
\xra{~{}~\circ~{}~}
\Bun^{\vfr} 
\qquad
\text{ and }
\qquad
\Mfld^{\vfr}_n\underset{\Bun}\times \Bun^{\un{\Mfld^{\vfr}_k}} 
\xra{~{}~\circ~{}~}
\Mfld^{\vfr}_{n+k}  ~,
\]
as well as 
\[
\Mfld^{\sfr}_n\underset{\Bun}\times \Bun^{\un{\Mfld^{\sfr}_k}} 
\xra{~{}~\circ~{}~}
\Mfld^{\sfr}_{n+k} ~,
\]
over $\Bun^{\un{\Bun}}\xra{\circ} \Bun$.

\end{cor}

\subsection{Suspension}
We construct a suspension of framed stratified spaces.  
This formation will play an important role in our definition of disk-stratifications, for instance, by constructing cells in a way amenable to inductive on dimension.  
The main result in this section is Lemma~\ref{autos-suspend}, which states that every automorphism of a suspended vari-framed stratified space $L$ is the suspension of an automorphism of $L$.
This result reflects the rigidity of vari-framings, which ultimately is responsible for the contact between higher categories and vari-framed differential topology.

For the next definition we make reference to the vari-framed stratified space $\DD^1$ of Example~\ref{pre-hemi}.  
The underlying space of $\DD^1$ is $[-1,1]$, and as such it is equipped with maps $\{-1\} \to \DD^1 \la \{1\}$, each of which is a constructible proper embedding.  
\begin{definition}\label{def.suspension}
The \emph{suspension} of a compact stratified space $X$ is the iterated pushout in $\strat$:
\[
\sS(X)~:=~\ast \underset{X\times \{-1\}}\coprod (X\times \DD^1) \underset{X\times \{1\}} \coprod \ast~.
\]
The \emph{fiberwise suspension} of a proper constructible bundle $X\to K$ is the iterated pushout in $\strat$:
\[
\sS^{\sf fib}(X)~:=~K\underset{X\times \{-1\}}\coprod (X\times \DD^1) \underset{X\times \{1\}} \coprod K~.
\]
\end{definition}

\begin{observation}\label{susp-facts}
For each proper constructible bundle $X\to K$, there is a natural factorization of the projection
\[
\sS^{\sf fib}(X)\longrightarrow K\times \DD^1 \xra{~\sf pr~}K
\]
through proper constructible bundles.   
Furthermore, for each conically smooth map $K'\to K$, the diagram of stratified spaces
\[
\xymatrix{
\sS^{\sf fib}(X_{|K'})\ar[rr]  \ar[d]
&&
K'  \times \DD^1 \ar[rr]  \ar[d]
&&
K'  \ar[d]
\\
\sS^{\sf fib}(X)\ar[rr]  
&&
K\times \DD^1 \ar[rr]
&&
K 
}
\]
is comprised of pullback squares.

\end{observation}

The next result articulates a sense in which suspension respects vari-framings as well as solid framings.  
\begin{lemma}\label{fr-sus}
The assignment
\begin{equation}\label{cpt-to-wreath}
(X\underset{\sf p.cbl}{\longrightarrow} K)\mapsto \Bigl(\sS^{\sf fib}(X)\longrightarrow K\times \DD^1 \longrightarrow K\Bigr)
\end{equation}
defines a functor
\[
\sS\colon \cBun  \longrightarrow
\cMfld^{\vfr}_1 \underset{\Bun}\times  \Bun^{\un{\cBun}} ~.
\]
This functor admits lifts
\[
\sS^{\fr}\colon \cBun^{\vfr}  \longrightarrow
\cMfld^{\vfr}_1 \underset{\Bun}\times  \Bun^{\un{\cBun}^{\vfr}}
\qquad
\text{ and }~{}~
\qquad
\sS^{\fr}\colon \cBun^{\sfr_n}  \longrightarrow
\cMfld^{\vfr}_1 \underset{\Bun}\times  \Bun^{\un{\cBun}^{\sfr_{n+1}}}
\]
for each dimension $n$.

\end{lemma}

\begin{proof}
Observation~\ref{susp-facts} gives that the assignment~(\ref{cpt-to-wreath}) defines a functor ${\sf cBun} \to \bun^2$ over $\strat$, which takes values in composable \emph{proper} constructible bundles.  
Applying the topologizing diagram, there results a functor $\cBun \to \Bun^{\un{\cBun}}$ between $\infty$-categories, regarded here as striation sheaves through the main result of~\cite{striat}.  
The composite functor $\cBun \to \Bun^{\un{\cBun}} \xra\pr \Bun$ is constant at $\DD^1$.  
The preferred vari-framing of $\DD^1$ determines a lift
\[
\cBun \longrightarrow \cMfld_1^{\vfr} \underset{\Bun}\times \Bun^{\un{\cBun}}
\]
as desired.  

Now, consider a proper constructible bundle $X\xra{\pi} K$.
We must explain how each fiberwise vari-framing of $X\to K$ canonically determines a fiberwise vari-framing of $\sS^{\sf fib}(X)\to K\times \DD^1$.  It will be clear that the fiberwise vari-framing pulls back among stratified maps of the $K$ argument.  
Restriction along the equator of the fiberwise suspension determines the map between spaces of lifts over $\Exit$
\begin{equation}\label{equat}
\vfr\bigl(\sS^{\sf fib}(X)\to K\times \DD^1\bigr) \longrightarrow \vfr(X\to K)~.
\end{equation}
We will explain that this map is an equivalence.
Using the results of~\S3 of \cite{striat}, the double pushout defining the fiberwise suspension is preserved by the exit-path functor:
\[
\exit(K)\underset{\exit(X)\times\{0\}}\amalg \exit(X)\times [1]\underset{\exit(X)\times\{1\}}\amalg \exit(K)\xra{~\simeq~}\exit\bigl(\sS^{\sf fib}(X)\bigr)
\]
over $\exit(K)\times [1]\xra{\simeq} \exit(K\times \DD^1)$.  
Because the fibers of $\sS^{\sf fib}(X)\to K\times \DD^1$ over $K\times \partial \DD^1$ are terminal, both $\sT$ and $\epsilon^{\sf dim}$ restrict along $\exit(K)\times \partial [1]\hookrightarrow \exit\bigl(\sS^{\sf fib}(X)\bigr)$ as the constant functor to $\Vect^{\sf inj}$ valued at the zero vector space.
Because the zero vector space is initial in $\Vect^{\sf inj}$, the restriction map~(\ref{equat}) is an equivalence of spaces.  
\end{proof}

In ordinary differential topology, an orientation on $M$ determines an orientation on $M\times \RR$ in a standard manner.  
This is similarly the case for spin structures, as well as framings.  
In the abstract, these are the \emph{data} of maps of spaces $B_n \to B_{n+1}$ over the standard map $\BO(n) \xra{-\times \RR} \BO(n+1)$ for the cases $B_n = {\sf BSO}(n)$ and $B_n={\sf BSpin}(n)$ and $B_n = {\sf EO}(n)$.  
The next definition imitates such data for the general case of $\cB$-structures on stratified spaces, so as to vary in constructible families.  
\begin{definition}\label{def.suspending}
A \emph{suspending tangential structure} is a tangential structure $\cB$ together with a \emph{framed suspension} functor
\[
\sS^{\cB} \colon \cBun^{\cB} \longrightarrow \Bun^{\cB}
\]
over the composite functor $\cBun \xra{\sS} \{\DD^1\}\underset{\Bun}\times \Bun^{\un{\cBun}} \xra{\circ} \Bun$.

\end{definition}

\begin{example}
The functor of Lemma~\ref{fr-sus}, given in terms of fiberwise suspensions, composed with the functor of Corollary~\ref{vfr-iterate}, which articulates compatibilities with vari-framings, gives a framed suspension functor for the tangential structure $\vfr$.

\end{example}

\begin{lemma}\label{autos-suspend}
\footnote{This Lemma is false without the dimension bound: see the erratum included in \S\ref{sec.erratum}.}
For each compact vari-framed stratified space $L$ with dimension $\dim(L)\leq 1$, the map between spaces of automorphisms 
\[
\sS^{\fr}\colon \Aut_{\Bun^{\vfr}}(L) \longrightarrow \Aut_{\Bun^{\vfr}}\bigl(\sS^{\fr}(L)\bigr)
\]
is an equivalence.

\end{lemma}

\begin{proof}
We seek to establish the equivalence between the maximal connected $\infty$-subgroupoids of $\Bun^{\vfr}$ that contain the respective objects $L$ and $\sS^{\fr}(L)$.
The maximal $\infty$-subgroupoid of $\Bun^{\vfr}$ corresponds, via the main result of~\cite{striat}, to the striation sheaf that classifies fiber bundles among stratified spaces which are equipped with a fiberwise vari-framing.  
Therefore, to prove the equivalence of the lemma it is enough to prove, for each smooth manifold $T$, that each $T$-point of the codomain lifts to an $T$-point of the domain.  
So consider an $T$-point of the striation sheaf ${\sf BAut}_{\Bun^{\vfr}}\bigl(\sS^{\fr}(L)\bigr)$; it classifies a proper fiber bundle $E\xra{p} T$, equipped with a fiberwise vari-framing $\exit(E) \xra{g} \vfr$, with each structured fiber $\bigl(E_t,g_{|\exit(E_t)}\bigr)$ is equivalent in $\Bun^{\vfr}$ to $\sS^{\fr}(L)$.
To construct the desired lift of this $T$-point to ${\sf BAut}_{\Bun^{\vfr}}(L)$ is the problem of constructing a fiberwise vari-framed proper fiber bundle $E_0\to T$ and an equivalence $\sS^{\sf fib,\fr}(E_0) \simeq E$ of fiberwise vari-framed fiber bundles over $T$ from the fiberwise framed suspension.

The fiberwise $0$-dimension strata of $E$ is a 2-sheeted covering over $T$.
The fiberwise vari-framing of $E\to T$ in particular implies this 2-sheeted cover is trivial.
We denote it as $T_-\sqcup T_+ \subset E$, with the subscripts marking if first coordinate agrees or disagrees with a collaring-coordinate about each cofactor.  
Taking the unzip along this closed constructible subspace gives the composable pair of proper constructible bundles
\[
\unzip_{T_-\sqcup T_+}(E) \xra{~q~} E \xra{~p~} T~.
\]
Because $T$ is a smooth manifold the functor $\sT_T \colon \exit(T) \to \Vect^{\sf inj}$ factors through $\Vect^\sim$, the maximal $\infty$-subgroupoid.  
Using that the fiberwise constructible tangent bundle is a morphism $\sT^{\sf fib} \colon \Exit \to \Vect^{\sf inj}$ of $\Vect^\sim$-modules, 
there results the short exact sequence of $\Vect$-valued functors from $\exit\bigl(\unzip_{T_-\sqcup T_+}(E)\bigr)$
\[
0 \longrightarrow \sT_q  \longrightarrow \sT_{pq} \longrightarrow q^\ast \sT_p  \longrightarrow 0~.
\]
(Here, and through this proof, we use the notation established at the beginning of the proof of Lemma~\ref{splittings}.)
Restricted to $\exit\bigl(\Link_{T_-\sqcup T_+}(E)\bigr)$, the cokernel term vanishes;
restricted to
\[
\exit\bigl(\unzip_{T_-\sqcup T_+}(E) \smallsetminus \Link_{T_-\sqcup T_+}(E)\bigr)\xra{\simeq} \exit(E\smallsetminus T_-\sqcup T_+)~,
\]
the kernel term vanishes and the cokernel term does not.
The first coordinate of the fiberwise vari-framing $\exit(E) \to \vfr$ therefore determines a non-vanishing parallel vector field on $\unzip_{T_-\sqcup T_+}(E) \smallsetminus \Link_{T_-\sqcup T_+}(E)$ in the sense of~\S8 of~\cite{aft1}.
We will now extend this vector field to all of $\unzip_{T_-\sqcup T_+}(E)$.

Let $\ov{\alpha}\colon \Link_{T_-\sqcup T_+}(E)\times[0,1) \hookrightarrow \unzip_{T_-\sqcup T_+}(E)$ be a choice of collaring, the existence of which is guaranteed by the results in~\S8 of~\cite{aft1}.
Denote the restriction $\alpha\colon \Link_{T_-\sqcup T_+}(E)\times \{\frac{1}{2}\} \to \unzip_{T_-\sqcup T_+}(E)\smallsetminus \Link_{T_-\sqcup T_+}(E) \cong E\smallsetminus (T_-\sqcup T_+)$.
The fiberwise vari-framing of $E\to T$ determines an identification of short exact sequences of $\Vect$-valued functors from $\exit\bigl(\Link_{T_-\sqcup T_+}(E)\bigr)$
\[
\xymatrix{
0  \ar[r]  \ar[d]
&
\sT^{\sf fib}_{\Link_{T_-\sqcup T_+}(E)}  \ar[r]  \ar[d]^-{\simeq}
&
\alpha^\ast \sT_{E\smallsetminus T_-\sqcup T_+}  \ar[r]  \ar[d]^-{\simeq}
&
\epsilon^1_{\Link_{T_-\sqcup T_+}(E)}  \ar[r]  \ar[d]^-{\simeq}
&
0  \ar[d]
\\
0  \ar[r]  \ar[u]
&
\alpha^\ast \epsilon^{\sf fib.dim-1}_{E\smallsetminus T_-\sqcup T_+}   \ar[r]  \ar[u]
&
\alpha^\ast \epsilon^{\sf fib.dim}_{E\smallsetminus T_-\sqcup T_+}  \ar[r]  \ar[u]
&
\alpha^\ast \epsilon^1_{E\smallsetminus T_-\sqcup T_+}  \ar[r]  \ar[u]
&
0 \ar[u]
}
\]
The above diagram grants that this vector field extends to a non-vanishing vector field on $\unzip_{T_-\sqcup T_+}(E)$ in a neighborhood of $\Link_{T_-\sqcup T_+}(E)$ which agrees with the one on $\unzip_{T_-\sqcup T_+}(E)\smallsetminus \Link_{T_-\sqcup T_+}(E)$ constructed in the paragraph above.

Flowing by the specified vector field on $\unzip_{T_-\sqcup T_+}(E)$ gives a partially defined continuous map
\[
\gamma\colon \unzip_{T_-\sqcup T_+}(E)\times \RR~\dashrightarrow~ \unzip_{T_-\sqcup T_+}(E)\times\RR
\]
over $T\times \RR$.
It restricts to a partially defined continuous map
\begin{equation}\label{e1}
\gamma_|\colon \Link_{T_-}(E)\times [0,\infty)\longrightarrow \unzip_{T_-\sqcup T_+}(E)
\end{equation}
over $T$.
This map further restricts to a conically smooth isomorphism over $T\times \DD^1$,
\begin{equation}\label{e111}
\gamma_|\colon \Link_{T_-}(E)\times \DD^1\longrightarrow \unzip_{T_-\sqcup T_+}(E)
\end{equation}
if and only if, for each $x\in \Link_{T_-}(E)$, the maximal domain of definition of the partially defined map $\gamma(x,-)\colon [0,\infty) \to \unzip_{T_-\sqcup T_+}(E)$ is $[0,a]$ for some $a\in (0,\infty)$.  
This condition is ensured if $\dim(L)\leq 1$.  
Upon applying $\exit(-)$, this isomorphism~(\ref{e111}) lies over $\vfr$.
In particular, we recognize an isomorphism of stratified spaces $\sS^{\sf fib, \fr}\bigl(\Link_{T_-}(E)\bigr) \cong E$ over $T$ which, upon applying $\exit(-)$, lies over $\vfr$.
The result follows from from the identification of fiberwise vari-framed constructible bundles $\Link_{T_-}(E) \simeq L$ over $T$.
\end{proof}

The following definition is supported by Lemma~\ref{fr-sus}, which articulates that suspensions naturally carry compact vari-framed stratified spaces to vari-framed stratified spaces. 
\begin{definition}[Hemispherical disks]\label{hemi-disks}
The \emph{hemispherical $n$-disk} is the vari-framed $n$-manifold defined inductively as the framed suspension
\[
\DD^n ~:=~\sS^{\fr}(\DD^{n-1})~\in~\cBun^{\vfr}
\]
with $\DD^0=\ast$.  
We set the convention $\DD^{-1}:= \emptyset$.  

\end{definition}

\begin{cor}\label{no-autos}
\footnote{
This result is false without the dimension bound; see the erratum included in~\S\ref{sec.erratum}.  
}
For each dimension $n<3$, the space of automorphisms
\[
\Aut_{\Bun^{\vfr}}(\DD^n)~\simeq~\ast
\]
is contractible.

\end{cor}

\begin{proof}
By definition, $\DD^n = \sS^{\sf fr}(\DD^{n-1})$.
The result follows from Lemma~\ref{autos-suspend}, by induction on $n<3$; the base case of $n=0$ is clear. 
\end{proof}

\subsection{Closed covers as limit diagrams}

Here we show that, for each tangential structure $\cB$, purely closed covers are limit diagrams in $\Bun^\cB$.

Consider a stratified space $X$.
Recall from~\S8.3 of~\cite{aft1} that $X$ is \emph{finitary} if it admits a finite open cover by basic singularity types; or equivalently, if $X$ is the interior of a compact stratified space with boundary.  
In particular, if $X$ is compact, then it is finitary.
Denote by ${\sf Cls}(X)$ the poset, ordered by inclusion, of \emph{proper constructible subspaces}.
Specifically, an element of ${\sf Cls}(X)$ is a closed subset $P\subset X$ which is a union of strata of $X$.  
The \emph{reversed mapping cylinder construction} of~\S6.6 of~\cite{striat} defines an equivalence between $\infty$-categories
\[
{\sf Cylr}\colon {\sf Cls}(X)^{\op}~\simeq~ (\Bun^{\cls})^{X/}
~,\qquad
(P\subset X)\mapsto \Bigl({\sf Cylr}(X\la P)  \to \Delta^1\Bigr)~,
\]
to the $\infty$-category of closed morphisms from $X$.  
Because $\Bun^\cB\to \Bun$ is closed-coCartesian, this morphism canonically lifts as a functor
\[
{\sf Cylr}\colon {\sf Cls}(X)^{\op}~\simeq~ (\Bun^{\cB,\cls})^{X/}~.
\]
Composition defines a functor
\[
{\sf Cls}(X)^{\op}~\simeq~ (\Bun^{\cB,\cls})^{X/}
\hookrightarrow (\Bun^{\cB})^{X/}
\]
to the $\infty$-undercategory.  
In particular, each full subposet $\cU \subset {\sf Cls}(X)$ defines a functor
\begin{equation}\label{20}
(\cU^{\op})^{\tl} \longrightarrow \Bun^\cB
\end{equation}
whose value on the cone-point is $X$.  
\begin{prop}\label{cls.lims}
Let $\cB$ be a tangential structure.  
Let $X$ be a finitary stratified space, equipped with a $\cB$-framing.  
Let $\cU \to {\sf Cls}(X)$ be a fully faithful right fibration whose codomain is the poset of proper constructible subspaces of $X$.  
The functor
\[
(\cU^{\op})^{\tl} 
\xra{~(\ref{20})~}
\Bun^\cB
\]
is a limit diagram if and only if the union $\underset{P\in \cU} \bigcup P$ is $X$.  

\end{prop}

\begin{proof}
By Observation~\ref{tau-closed-covers}, it suffices to prove the result in the case that the tangential structure $\cB \xra{\simeq} \Vect^{\sf inj}$. 

Consider the union $U := \underset{P\in \cU} \bigcup P\subset X$.
Then $U$ is a union of strata of $X$.
By definition of a stratified space, $X$ is a countable open union of basic singularity types.  
It follows that $X$ is locally of finite depth.
Therefore $U$ is locally closed.
It also follows that $X$ is locally compact.
Therefore the union $U \subset X$ is closed.  
We conclude that $U\subset {\sf Cls}(X)$ is a proper constructible subspace of $X$.  
There results a closed morphism $X\to U$ in $\Bun$, as well as a closed morphism $U \to P$ in $\Bun$ for each object $P\in \cU$.

So if $X$ is the limit of the functor $\cU^{\op} \xra{P\mapsto P} \Bun$, then there is a morphism $U\to X$ for which the composition $X\to U \to X$ is the identity morphism of $X$.  
That is, $X$ is a retract of $U$ in $\Bun$.
However, since $X\to U$ is a closed morphism, it is an epimorphism,
We conclude that if $X$ is equivalent to $\limit (\cU^{\op} \to \Bun)$, then the union $U = \underset{P\in \cU} \bigcup P\subset X$ is all of $X$.

We now prove the converse: if the union $U = X$ is all of $X$, then $X$ is the limit of 
\begin{equation}\label{21}
\cU^{\op} 
\longrightarrow 
\Bun~.
\end{equation}  
So let $(\cU^{\op})^{\tl}\longrightarrow \Bun$ be an extension of~(\ref{21}).  
Denote the value of this functor on the cone-point as $Z\in \Bun$.  
We first construct a morphism $Z\to X$ over $\cU^{\op}$, then we show that this morphism over $\cU^{\op}$ is unique.

The closed-active factorization system~(Theorem~6.5.6 of~\cite{striat}) assures for each $P\in \cU$, the given morphism $Z\to P$ factors uniquely $Z\xra{\sf cls}Z_P \xra{\sf act} P$ as a closed morphism followed by an active morphism.  
By the uniqueness of each such factorization, the assignment $P\mapsto Z_P$ defines a functor $Z_\bullet \colon \cU^{\op} \to (\Bun^{\sf cls})^{Z/}$. 
Furthermore, the morphism $Z_P \to P$ for each $P\in \cU$ defines a natural transformation between functors $\cU^{\op} \to \Bun^{Z/}$ from $Z_\bullet$ to the original functor $P\mapsto P$.  
Consider the union $Z_U:=\underset{P\in \cU}\bigcup Z_P \subset Z$.
As argued in the first paragraph of this proof, this union is a proper constructible subspace of $Z$.  
Therefore, there exists a closed morphism $Z\to Z_U$ in $\Bun$ over the functor $Z_\bullet\colon \cU^{\op} \to \Bun$.
From its construction, any morphism from $Z$ to $X$ over the functor $\cU^{\op} \to \Bun$ factors uniquely through this closed morphism $Z \to Z_U$.
We are therefore reduced to showing that there is a unique morphism from $Z_U$ to $X$ over the functor $\cU^{\op} \to \Bun$.

The above natural transformation from $Z_\bullet$ to the given functor $P\mapsto P$ defines a functor
\[
\cU^{\op} \longrightarrow
\Ar(\Bun)^{\sf cls}
\]
to the $\infty$-subcategory of the $\infty$-category of arrows in $\Bun$, consisting of all objects and those morphisms, which are natural transformations by closed morphisms.  
Via the reversed mapping cylinder construction, such a functor is equivalent to a functor between categories,
\[
E_\bullet \colon \cU \longrightarrow {\sf Cbl}(\Delta^1)^{\sf cls}~,
\]
to the category whose objects are constructible bundles over $\Delta^1$ and whose morphisms are proper constructible bundles over $\Delta^1$ between such.  
Consider the colimit of the composite functor 
\[
E_U~:=~\colim\Bigl(\cU \xra{E_\bullet}{\sf Cbl}(\Delta^1)^{\sf cls} \hookrightarrow {\sf Cbl}(\Delta^1) \Bigr)
\]
to the category of constructible bundles over $\Delta^1$ and maps over $\Delta^1$ between such.  
This colimit exists: the colimit certainly exists as a stratified topological space; because each morphism in this diagram is closed, this stratified topological space inherits a unique conically smooth structure compatible with each term in the diagram.  
Realized in this way, there are canonical identifications of the base changes
\[
Z_U~\cong~ (E_U)_{|\Delta^{\{0\}}}
\qquad \text{ and }\qquad
X=U~\cong~(E_U)_{|\Delta^{\{1\}}}
\]
over $\cU$.  
Therefore, this $E_U$ defines a morphism in $\Bun$ from $Z_U$ to $X$ in $\Bun$ over the functor $\cU^{\op} \to \Bun$.

We now prove that the morphism from $Z_U$ to $X$ in $\Bun$ over the functor $\cU^{\op}\to \Bun$ is unique. 
Let $S$ be a sphere (of some dimension).  
It suffices to show that each map from $S$ to the space of morphisms $Z_U$ to $X$ over $\cU^{\op}$ is homotopic to the constant such map at the morphism constructed above.  
Such a map from $S$ is given by a functor $\cU^{\tr} \to {\sf Cbl}(\Delta^1\times S)^{\sf cls}$ extending the functor $E_\bullet\times S$, and whose value on the cone-point is a constructible bundle 
\[
\w{E}_U 
\longrightarrow 
\Delta^1  \times S
\]
with identifications
\[
(\w{E}_{U})_{|\Delta^{\{0\}}\times S }~\cong~ Z_U\times S
~{}~{}~
\text{ over }
\Delta^{\{0\}}\times S 
\qquad \text{ and } \qquad 
(\w{E}_{U})_{|\Delta^{\{1\}}\times S }~\cong~ X \times S
~{}~{}~
\text{ over }
\Delta^{\{1\}}\times S  .
\]
From the definition of $E_U$ above as a colimit, these stipulated identifications of $\w{E}$ assemble as an identification
\[
E_U\times S~\cong~ \w{E}_U
~{}~{}~
\text{ over }
\Delta^1 \times S~,
\]
compatibly with the stipulated identifications.
In other words, the given functor $\cU^{\tr} \to {\sf Cbl}(\Delta^1\times S)^{\sf cls}$ agrees with the product of the colimit diagram defining $E_U$ with $S$.  
We conclude that the given map from $S$ to the space of morphisms from $Z_U$ to $X$ in $\Bun$ over the functor $\cU^{\op}\to \Bun$ is homotopic to the constant map at $E_U$.  
\end{proof}

\subsection{Disks}\label{sec.disks}
We introduce compact vari-framed \emph{disk}-stratified spaces.
We do so through a universal property, as the smallest collection containing $\emptyset$ and $\ast$ that is closed under the formation of framed suspension as well as \emph{closed covers}, which we now define.  
The disk-stratified condition captures our intuition for a `finely stratified space', and they will ultimately play the role of a basis for a descent among stratified spaces concerning gluing along strata.

A 2-term open cover of a smooth manifold determines a pushout diagram
\[
\xymatrix{
M  
&
V  \ar[l]
\\
U   \ar[u]
&
U\cap V  \ar[u]  \ar[l]
}
\]
among smooth maps.
In the case of stratified spaces we concern ourselves with analogous pushouts in which each arrow is not an open embedding but the inclusion of a closed union of strata.  
This can be phrased as a pullback diagram in $\Bun$, which we call a \emph{purely closed cover}, using the monomorphism $(\Strat^{\sf p.cbl.inj})^{\op} \underset{\rm Def~\ref{def.classes}}\hookrightarrow \Bun$.  
The $\infty$-category $\Bun$ allows for another class of pullback diagram, which we call \emph{refinement-closed covers}, which embody how a refinement of a stratum of a stratified space determines a refinement of that stratified space.  

\begin{definition}[Closed covers]\label{def.closed-cover}
For each tangential structure $\cB$, 
a limit diagram $[1]\times[1]\to \Bun^{\cB}$, written 
\[
\xymatrix{
X   \ar[r]   \ar[d] 
&
X''   \ar[d]
\\
X'   \ar[r]
&
X_0,
}
\]
is
\begin{itemize}
\item a \emph{purely closed cover} if the diagram is comprised of closed morphisms;
\item a \emph{refinement-closed cover} if the horizontal arrows are refinement morphisms while the fiberwise arrows are closed morphisms.
\end{itemize}
The limit diagram is a {\it closed cover} if it is either a purely closed cover or a refinement-closed cover.
\end{definition}

\begin{remark}\label{explicate-pure-cov}
The opposite of a purely closed cover $([1]\times[1])^{\op}\to (\Bun^{\cls})^{\op}\simeq \Strat^{{\pcbl}, {\sf inj}}$ is a pushout diagram in $\Strat$ by proper constructible embeddings; this is Proposition~\ref{cls.lims}, above.
Through the results of~\S3 of~\cite{striat}, we conclude that the composite functor $([1]\times[1])^{\op} \to \Strat^{{\pcbl}, {\sf inj}} \xra{\exit}\Cat_\infty$ is a pushout by fully faithful functors. 

\end{remark}

\begin{definition}\label{def.disks}
Let $\cB$ be a suspending tangential structure.
The $\infty$-category of \emph{compact disk-stratified $\cB$-structured spaces} is the smallest full $\infty$-subcategory 
\[
\cDisk^\cB~\subset~\cBun^{\cB}
\]
with the following properties.
\begin{enumerate}
\item For each $i=-1,0$, and for each $\cB$-structure $g$ on $\DD^i$, the object $(\DD^i,g)$ belongs to $\cDisk^\cB$.  

\item The framed suspension $\sS^{\fr}(X)$ belongs to $\cDisk^\cB$ whenever $X$ does.

\item An object $X\in \cBun^{\cB}$ belongs to $\cDisk^\cB$ whenever there is a closed cover in $\Bun^{\cB}$
\[
\xymatrix{
X   \ar[r]   \ar[d] 
&
X''   \ar[d]
\\
X'   \ar[r]
&
X_0
}
\]
in which $X'$ and $X_0$ and $X''$ belong to $\cDisk^\cB$.

\end{enumerate}
For each dimension $n$, the $\infty$-category of \emph{compact disk-stratified $\cB$-structured $n$-manifolds} is the intersection
\[
\cDisk^{\cB}_n~:=~ \cDisk^{\cB}\underset{\Bun^{\cB}}\times \Bun^{\cB}_{\leq n}~.
\]

\end{definition}

\begin{remark}
Definition~\ref{def.disks} can be approached iteratively.
For instance, (1) grants that $\DD^0=\ast$ belongs to $\cDisk^{\vfr}$.
Thereafter, (2) inductively grants that $\DD^{n+1}:=\sS^{\sf fr}(\DD^{n})$ belongs to $\cDisk^{\vfr}$ for each $n\geq 0$.
Point~(3) thereafter gives, in the case $X_0=\DD^{-1}=\emptyset$ of a purely closed cover, that finite disjoint unions of copies of $\DD^n$ belong to $\cDisk_n^{\vfr}$.  
Also, for each $0\leq i \leq n$ the two standard closed morphisms $i_\pm \colon \DD^n\to \DD^i$ (visit Remark~\ref{explicate-pure-cov}) give that the pullback $\DD^n\underset{\DD^i}\times \DD^n$, which is the stratified space obtained by gluing together two $n$-disks along oppositely arranged $i$-hemispheres, belongs to $\cDisk^{\vfr}$.  
Iterating such purely closed covers, we see that pasting diagrams belong to $\cDisk^{\vfr}$, in particular.
More complex stratified spaces are obtained through refinement-closed covers.  
Through these formations one sees that simplices, equipped with a particular vari-framing, belong to $\cDisk^{\vfr}$.  
This is illustrated as Figure~\ref{fig:covers}.  
Thereafter, we see that triangulations of closed manifolds belong to $\cDisk^{\vfr}$, provided the triangulation is equipped with a vari-framing.

\end{remark}

\begin{figure}
\begin{subfigure}{.4\textwidth}
\begin{tikzpicture}
\tikzstyle{every node}=[font=\tiny]

\begin{scope}
\draw[fill=lightgray] (0,0) [out=75, in=105] to (2,0) [out=-105, in=-75] to (0,0);
\draw (0,0) to (2,0);
\draw[fill] (0,0) circle [radius=0.07];
\draw[fill] (2,0) circle [radius=0.07];
\end{scope}

\draw[style={->}] (2.5,0) to (3.5,0);

\begin{scope}[shift={(4,0)}]
\draw[fill=lightgray] (0,0) [out=75, in=105] to (2,0);
\draw (0,0) to (2,0);
\draw[fill] (0,0) circle [radius=0.07];
\draw[fill] (2,0) circle [radius=0.07];
\end{scope}

\draw[style={->}] (1,-0.8) to (1,-1.4);
\draw[style={->}] (5,-0.8) to (5,-1.4);

\begin{scope}[shift={(0,-2)}]
\draw[fill=lightgray] (2,0) [out=-105, in=-75] to (0,0);
\draw (0,0) to (2,0);
\draw[fill] (0,0) circle [radius=0.07];
\draw[fill] (2,0) circle [radius=0.07];
\end{scope}

\draw[style={->}] (2.5,-2) to (3.5,-2);

\begin{scope}[shift={(4,-2)}]
\draw (0,0) to (2,0);
\draw[fill] (0,0) circle [radius=0.07];
\draw[fill] (2,0) circle [radius=0.07];
\end{scope}

\end{tikzpicture}
\caption{}
\end{subfigure}\hfill%
\begin{subfigure}{.4\textwidth}
\begin{tikzpicture}
\tikzstyle{every node}=[font=\tiny]

\begin{scope}
\draw[fill=lightgray] (0,0) [out=75, in=105] to (2,0) [out=-105, in=-75] to (0,0);
\draw (0,0) to (2,0);
\draw (2,0) to (3,0);
\draw[fill] (0,0) circle [radius=0.07];
\draw[fill] (2,0) circle [radius=0.07];
\draw[fill] (3,0) circle [radius=0.07];
\end{scope}

\draw[style={->}] (3.5,0) to (4.5,0);

\begin{scope}[shift={(5,0)}]
\draw (0,0) to (1,0);
\draw[fill] (0,0) circle [radius=0.07];
\draw[fill] (1,0) circle [radius=0.07];
\end{scope}

\draw[style={->}] (1,-0.7) to (1,-1.3);
\draw[style={->}] (5.5,-0.7) to (5.5,-1.3);

\begin{scope}[shift={(0,-2)}]
\draw[fill=lightgray] (0,0) [out=75, in=105] to (2,0) [out=-105, in=-75] to (0,0);
\draw (0,0) to (2,0);
\draw[fill] (0,0) circle [radius=0.07];
\draw[fill] (2,0) circle [radius=0.07];
\end{scope}

\draw[style={->}] (3.5,-2) to (4.5,-2);

\begin{scope}[shift={(5,-2)}]
\draw[fill] (0,0) circle [radius=0.07];

\end{scope}

\end{tikzpicture}
\caption{}
\end{subfigure}%

\vspace{0.7cm}

\begin{subfigure}{.4\textwidth}

\begin{tikzpicture}

\begin{scope}
\draw[fill=lightgray] (1.62,0.43) [out=160, in=75] to (0,0) [out=-75, in=160] to (0.38,-0.44) [out=340, in=-105] to (2,0) [out=105, in=-20] to (1.62,0.43);
\draw (0.38,-0.44) to (1.62, 0.43);

\draw[fill] (0,0) circle [radius=0.07];
\draw[fill] (0.38, -0.44) circle [radius=0.07];
\draw[fill] (1.62,0.43) circle [radius=0.07];
\draw[fill] (2,0) circle [radius=0.07];
\end{scope}

\draw[style={->}] (2.5,0) to (3.5,0);

\begin{scope}[shift={(4,0)}]

\draw[fill=lightgray] (1.62,0.43) [out=160, in=75] to (0,0) [out=-75, in=160] to (0.38,-0.44);
\draw (0.38,-0.44) to (1.62, 0.43);

\draw[fill] (0,0) circle [radius=0.07];
\draw[fill] (0.38, -0.44) circle [radius=0.07];
\draw[fill] (1.62,0.43) circle [radius=0.07];

\end{scope}

\draw[style={->}] (1,-0.7) to (1,-1.3);
\draw[style={->}] (5,-0.7) to (5,-1.3);

\begin{scope}[shift={(0,-2)}]

\draw[fill=lightgray] (0.38,-0.44) [out=340, in=-105] to (2,0) [out=105, in=-20] to (1.62,0.43);
\draw (0.38,-0.44) to (1.62, 0.43);
\draw[fill] (2,0) circle [radius=0.07];
\draw[fill] (0.38, -0.44) circle [radius=0.07];
\draw[fill] (1.62,0.43) circle [radius=0.07];

\end{scope}

\draw[style={->}] (2.5,-2) to (3.5,-2);

\begin{scope}[shift={(4,-2)}]
\draw (0.38,-0.44) to (1.62, 0.43);
\draw[fill] (0.38, -0.44) circle [radius=0.07];
\draw[fill] (1.62,0.43) circle [radius=0.07];
\end{scope}

\end{tikzpicture}

\caption{}
\end{subfigure}\hfill%
\begin{subfigure}{.4\textwidth}
\begin{tikzpicture}
\tikzstyle{every node}=[font=\tiny]

\begin{scope}
\draw[fill=lightgray] (0,0) [out=75, in=105] to node (TOP) {} (2,0) [out=-105, in=-75] to (0,0);
\draw[fill] (TOP) circle [radius=0.07];
\draw[fill] (0,0) circle [radius=0.07];
\draw[fill] (2,0) circle [radius=0.07];
\end{scope}

\draw[style={->}] (2.5,0) to (3.5,0);

\begin{scope}[shift={(4,0)}]
\draw[fill=lightgray] (0,0) [out=75, in=105] to (2,0) [out=-105, in=-75] to (0,0);
\draw[fill] (0,0) circle [radius=0.07];
\draw[fill] (2,0) circle [radius=0.07];
\end{scope}

\draw[style={->}] (1,-0.7) to (1,-1.3);
\draw[style={->}] (5,-0.7) to (5,-1.3);

\begin{scope}[shift={(0,-2.2)}]
\draw (0,0) [out=75, in=105] to node (TOP) {} (2,0);
\draw[fill] (TOP) circle [radius=0.07];
\draw[fill] (0,0) circle [radius=0.07];
\draw[fill] (2,0) circle [radius=0.07];
\end{scope}

\draw[style={->}] (2.5,-2) to (3.5,-2);

\begin{scope}[shift={(4,-2.2)}]
\draw (0,0) [out=75, in=105] to (2,0);
\draw[fill] (0,0) circle [radius=0.07];
\draw[fill] (2,0) circle [radius=0.07];
\end{scope}

\end{tikzpicture}
\caption{}
\end{subfigure}
\caption{(A)--(C) depict purely closed covers; (D) depicts a refinement-closed cover.
(Note:
indication of the tangential structure $\cB$ has been suppressed in these pictures.
So, strictly speaking, each of (A)--(D) depicts a diagram in $\Bun$.)}
\label{fig:covers}
\end{figure}
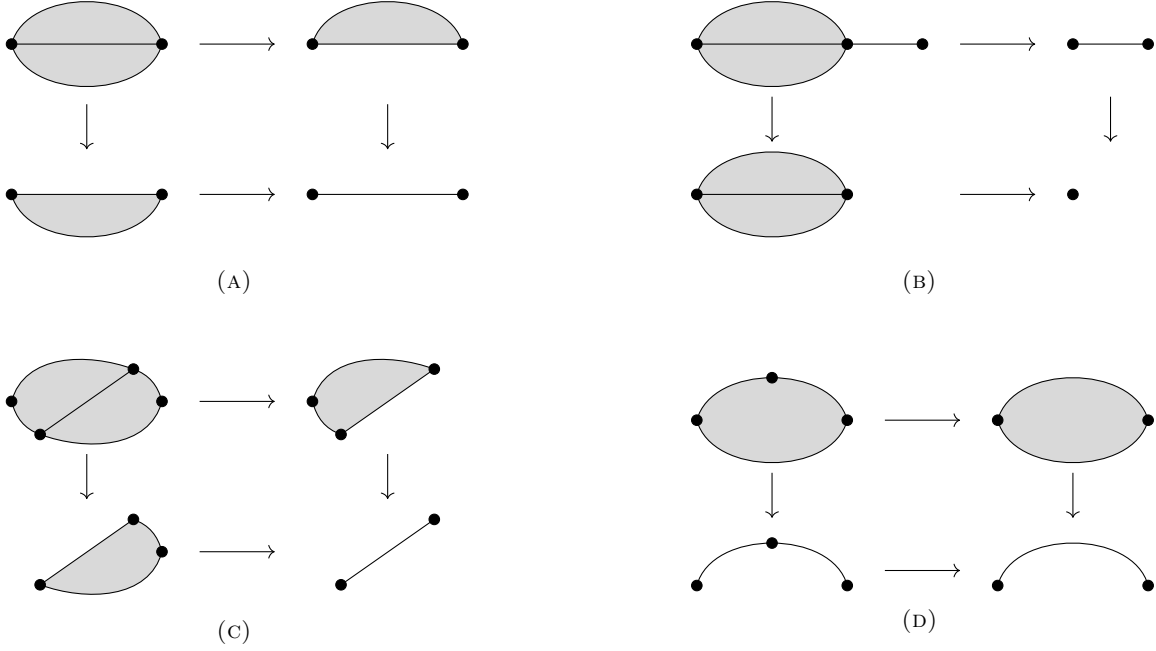

The $0$-dimensional case is explicit, as the next result demonstrates.  
\begin{lemma}\label{disk-zero} 
There is an equivalence of $\infty$-categories
\[
\cDisk^{\vfr}_{\leq 0}~\simeq~\Fin^{\op}
\]
with the opposite of finite sets.

\end{lemma}

\begin{proof}
This proof is similar to that of the main result of~\cite{striat}: we show that both $\infty$-categories, when viewed as sheaves on stratified spaces, classify the same structure.  

By inspection, there is a unique equivalence of the restricted fiberwise constructible vector bundles $\sT^{\sf fib}_{|\Bun_{\leq 0}}\simeq \epsilon^{\sf fib.dim}_{|\Bun_{\leq 0}}$.
Therefore, the functor  $\vfr_{\leq 0}\to \Vect^{\sf inj}_{\leq 0}$ is an equivalence.
It follows that the projection $\cDisk^{\vfr}_{\leq 0} \xra{\simeq} \cBun_{\leq 0}$ is an equivalence.  
The $\infty$-category $\cDisk^{\vfr}_{\leq 0}$ classifies proper constructible bundles with finite fibers.  
Each such constructible bundle $E\to K$ has a unique path lifting property in $\strat$,
\[
\xymatrix{
\Delta^{\{1\}}  \ar[rr]    \ar[d]
&&
E  \ar[d]
\\
\Delta^1  \ar[rr]  \ar@{-->}[urr]^-{\exists !}
&&
K,
}
\]
which implies the functor $\exit(E)\to \exit(K)$ is a right fibration whose fibers are $0$-types.  
The latter is classified by a functor $\exit(K)\to \Fin^{\op}$.
By inspection, this assignment restricts along conically smooth maps $K'\to K$, and so we have the desired functor $\cBun_{\leq 0} \to \Fin^{\op}$.  

Inspecting the case that $K=\ast$ reveals that this functor is essentially surjective.
Inspecting the case that $K=\Delta^1$ reveals that this functor is bijective on mapping components.  
The above unique lifting property implies the morphism spaces of $\cDisk_{\leq 0}^{\vfr}$ are $0$-types.
We conclude that this functor $\cDisk_{\leq 0}^{\vfr} \to \Fin^{\op}$ is an equivalence of $\infty$-categories.  
\end{proof}

The $1$-dimensional case is manageable as the next result indicates.
This result is useful for comparison with more combinatorial entities, such as $\bdelta$.  
\begin{lemma}\label{disk-1}
The $\infty$-category $\cDisk_1^{\vfr}$ is an ordinary category.

\end{lemma}

\begin{proof}
We must show that, for each pair of objects $D$ and $D'$ in $\cDisk_1^{\vfr}$, the space of morphisms $\cDisk_1^{\vfr}(D,D')$ has contractible components.
Recall the group $\infty$-category $\un{\sO}$ of Construction~\ref{def.O}.
By inspection, there is an identification of group $\infty$-categories $\un{\sO}_{|\leq 1}\simeq (\ZZ/2\ZZ)^{\tl}$.
In particular, the projection map between spaces of morphisms
\[
\cDisk_1^{\vfr}(D,D') \longrightarrow \cDisk_{\leq 1}(D,D')
\]
has fibers which are $0$-types; here we are using the same notation for objects in $\cDisk^{\vfr}$ and their unstructured projections to $\cDisk$.
So it suffices to prove that $\cDisk_{\leq 1}(D,D')$ is a $0$-type.

Let $s_0\in S$ be a pointed connected smooth manifold, and consider an $S$-point of the striation sheaf $\cDisk_1(D,D')$.
This classifies a proper constructible bundle $E\to S\times \Delta^1$ whose fibers are bounded above in dimension by $1$, equipped with identifications $E_{|S\times\Delta^{\{0\}}} \simeq D\times S$ and $E_{|S\times\Delta^{\{1\}}} \simeq D'\times S$ over $S$.
We wish to construct an equivalence $E\simeq E_{|\{s_0\}\times \Delta^1}\times S$ over $S\times \Delta^1$.

Consider the link system over $S$:
\[
D\times S\cong E_{|S\times\Delta^{\{0\}}}\xla{~\pi~}\Link_{E_{|S\times\Delta^{\{0\}}}}(E)\xra{~\gamma~} E_{|S\times\Delta^{\{1\}}}\cong D'\times S~;
\]
the map $\pi$ is proper and constructible while the map $\gamma$ is a refinement.  
In~\S6 of~\cite{striat} we use flows to construct an isomorphism over $S\times \Delta^1$:
\[
\Bigl(D\times S \underset{\Link_{E_{|S\times\Delta^{\{0\}}}}(E)\times\Delta^{\{0\}}} \coprod \Link_{E_{|S\times\Delta^{\{0\}}}}(E)\times \Delta^1 \underset{\Link_{E_{|S\times\Delta^{\{0\}}}}(E)\times \Delta^{\{1\}}}\coprod D'\times S\Bigr)
\xra{~\cong~}E~.
\]
In other words, the given $S$-point of $\cDisk_1(D,D')$ admits a factorization as an $S$-point of $\cDisk^{\sf c.cr}_1(D,D'')$ composed with an $S$-point of $\cDisk_1^{\sf ref}(D'',D')$.  
Therefore, we may consider these two cases separately.

Suppose the given $S$-point of $\cDisk_1(D,D')$ factors through $\cDisk_1^{\sf ref}(D,D') \simeq \Strat^{\sf ref}(D,D')$.  
By definition, the dimension of $D'$ is bounded above by $1$, and each connected stratum of $D$ is either a singleton or an open interval.  
It follows from the connectivity of $S$ that this $S$-point of $\Strat^{\sf ref}(D,D')$ is constant, at $\gamma_{|\{s_0\}}$.  

Suppose the given $S$-point of $\cDisk_1(D,D')$ factors through $\cDisk_1^{\vfr, \sf c.cr}(D,D')\simeq  \Strat^{\sf p.cbl}(D',D)$. 
Denote by $D_0\subset D$ the $0$-dimensional strata; denote by $\{\ov{D}'_\alpha\}$ the collection of the closures of each $1$-dimensional stratum of $D'$, this collection is indexed by a finite set.  
As established in~\S7 of~\cite{aft1}, there is the blow-up square in $\strat$
\[
\xymatrix{
\Link_{D'_0}(D')  \ar[r]  \ar[d]
&
\underset{\alpha}\coprod \ov{D}_\alpha \ar[d]
\\
D_0  \ar[r]
&
D.
}
\]
This diagram is a pushout by proper constructible maps.
By definition, the dimension of $D'$ is bounded above by $1$, and the closure of each connected stratum of $\ov{D}_\alpha$ is either a singleton or a closed interval.  It follows that $\Link_{D'_0}(D')$ is a finite set.  
We are therefore reduced to the case that $D'$ is either a finite set or a closed interval.  

If $D'$ is a finite set, then any constructible bundle $D'\to D$ factors through $D_0\subset D$, the $0$-dimensional strata of $D$.
This reduces us to the situation that both $D$ and $D'$ are finite sets.  
The result in this case follows directly from Lemma~\ref{disk-zero}.  

If $D'$ is a closed interval, then any constructible bundle $D'\to D$ factors surjectively through a constructible closed subspace $D''\subset D$ which is either a singleton or a closed interval.
This reduces us to the situation that $D$ is $\ast$ or $\DD^1$.  
The result in the first case is trivially true.
The result for the second case is true because any surjective constructible bundle $\DD^1 \to \DD^1$ is an isomorphism, the space of which is contractible.  
\end{proof}

The next result facilitates inductive arguments on dimension; we use it within a number of upcoming proofs.  
\begin{lemma}\label{truncate}
For each dimension $n$ there is a localization
\[
(-)_{\leq n}\colon \Bun^{\vfr} \rightleftarrows \Mfld^{\vfr}_{n}
\]
which restricts to a localization
\[
(-)_{\leq n}\colon \cDisk^{\vfr} \rightleftarrows \cDisk^{\vfr}_{n}~.
\]
The units of these adjunctions are by closed morphisms.  

\end{lemma}

\begin{proof}
As established in~\cite{aft1} (see~\S2), there is an $n$-skeleton functor $(-)_{\leq n}:\strat^{\emb}\ra \strat^{\emb}$ on stratified spaces and open embeddings thereamong, given by removing strata of dimension greater than $n$. 
Taking this $n$-skeleton fiberwise defines a functor $(-)_{\leq n}:\bun \ra \bun$ over $\strat$. 
Namely, assign to a constructible bundle $M\xra{\pi} K$ the $n$-skeleton $(M|K)_{\leq n}$ relative $K$, which we define as the stratified subspace of $M$ consisting of those $x\in M$ for which there is a bound of local dimensions of the fiber ${\sf dim}_x(\pi^{-1}\pi(x))\leq n$.
The restriction $(M|K)_{\leq n}\xra{\pi_|} K$ is again a constructible bundle. 

Applying the topologizing diagram (\S2 from~\cite{striat}) to this functor $(-)_{\leq n}:\bun\ra \bun$ gives a functor of $\oo$-categories $(-)_{\leq n}: \Bun\ra \Bun$. 
By construction, this functor factors through the $\oo$-subcategory $\Bun_{\leq n}$. 
We now construct a unit natural transformation ${\sf id} \ra (-)_{\leq n}$.  
We do this by applying the topologizing diagram to an endofunctor $\bun \to \bun$ over $\strat \xra{-\times \Delta^1} \strat$ with the property that over $\Delta^{\{0\}}$ this endofunctor is the identity functor and over $\Delta^{\{1\}}$ it is the functor $(-)_{\leq n}$.
This endofunctor is, for each stratified space $K$, the assignment of groupoids of constructible bundles
\[
(M\xra{\pi}K)~\mapsto~\Bigl(M\underset{(M|K)_{\leq n} \times \Delta^{\{0\}}} \amalg (M|K)_{\leq n} \times \Delta^{1}\xra{~\pi_|\times {\sf id}_{\Delta^1}~} K\times \Delta^1\Bigr)~;
\]
this assignment evidently pulls back along maps among the $K$-argument.
This assignment has the requisite restrictions over $\Delta^{\{0\}}$ and $\Delta^{\{1\}}$.  
By construction, this unit is by closed morphisms. 

We extend this to the general $\cB$-structured case.
We do this by dint of the applying requirement which ensures that the restriction of $\Bun^{\cB} \ra \Bun$ to closed morphisms is a coCartesian fibration, applied to fact that the unit of the above localization is implemented by closed morphisms. 
Picture this geometrically as follows. For each $K$-point of $\Bun^{\cB}$ represented by a constructible bundle $M\ra K$, we can again form the stratified space $M\times\{0\} \cup (M\times\Delta^1|K\times\Delta^1)_{\leq n}$. 
By the coCartesian property, a $\cB$-structure on $M\times \{0\}$ canonically extends to a fiberwise $\cB$-structure on the entire space over $K\times \Delta^1$. 
Formally, we construct the corresponding correspondence of $\oo$-category $\widetilde{\Bun^{\cB}}\ra \widetilde\Bun \ra [1]$, and we show the composite functor is both a Cartesian and a coCartesian fibration. First, it is again manifestly Cartesian. 
To check the coCartesian property amounts to the existence of coCartesian morphisms with any fixed source $M\in\Bun^{\cB} \simeq \widetilde{\Bun^{\cB}}_{|\{0\}}$ lying over the single non-identity morphism in $[1]$. 
We first choose the lift in $\widetilde{\Bun}$, using that $\widetilde\Bun \ra [1]$ is coCartesian. 
This lift is a closed morphism, by construction, and $\Bun^{\cB} \ra \Bun$ is a coCartesian fibration over closed morphisms; consequently, a second lift can be chosen, and we obtain that the composite functor $\widetilde{\Bun^{\cB}}\ra [1]$ is coCartesian and Cartesian.

Lastly, we set  $\cB=\vfr$ and observe the first localization.
Because $(M|K)_{\leq n} \to K$ is proper whenever the constructible bundle $M\to K$ is proper, the first localization restricts as a localization $\cBun^{\vfr} \rightleftarrows \cMfd_n^{\vfr}$.  
Finally, by construction $M_{\leq n}$ is disk-stratified whenever the vari-framed stratified space $M$ is disk-stratified.
This completes the result.  
\end{proof}

\subsection{Wreath}\label{sec.wreath}
We compare the \emph{wreath} construction, which ultimately defines $\btheta_n$, with iterated constructible bundles.

\subsubsection{\bf Finite correspondences}
We will consider the $\infty$-category 
\begin{equation}\label{e500}
\Corr(\Fin)
\end{equation}
of correspondences of finite sets. An object of $\Corr(\Fin)$ is a finite set; the groupoid of morphisms from $I$ to $J$ consists of diagrams $I\la K \ra J$ among finite sets, with composition given by base change:
\[
(I_0 \la K_{01} \to I_1) \circ (I_1 \la K_{12} I_2)~:=~(I_0 \la K_{01}\underset{I_1}\times K_{12} \to I_2)~.
\]
Likewise, an object of the $\infty$-category $\ov{\Corr}(\Fin)$ is a finite set $i\in I$ together with an element in it; the groupoid of morphisms in $\ov{\Corr}(\Fin)$ from $(i\in I)$ to $(j\in J)$ consists of spans $(i\in I) \la (k\in K) \ra (j\in J)$ among pointed finite sets, and composition is given by base change.  
There is a forgetful functor 
\begin{equation}\label{e502}
\ov{\Corr}(\Fin)
\longrightarrow 
\Corr(\Fin)
~,\qquad
(i\in I)\mapsto I ~.
\end{equation}

More precisely, $\Corr(\Fin)$ represents the presheaf
\[
\bdelta^{\op}
\longrightarrow
\Spaces
~,\qquad
[p]\mapsto \Cat^{\sf lim}\bigl(\TwAr([p]) , \Fin \bigr)~,
\]
whose value on $[p]$ is the space of limit-preserving functors to the category of finite sets from the twisted arrow category of $[p]$.
The $\infty$-category $\ov{\Corr}(\Fin)$ over $\Corr(\Fin)$ represents the presheaf on $(\bdelta_{/\Corr(\Fin)})^{\op}$ whose value on $\TwAr([p]) \to \Fin$ is the space of lifts along the forgetful functor $\Fin^{\ast/}\to \Fin$~.

\begin{remark}\label{t600}
Lemma~\ref{corr.conduche} of the appendix states that the natural forgetful functor $\ov{\Corr}(\Fin)\to \Corr(\Fin)$ is an exponentiable fibration, in the sense of Definition~\ref{def.efib}.
\end{remark}

\begin{remark}\label{t601}
The following assertions are immediate consequences of the main result in~\cite{fibrations}.
\begin{itemize}
\item	
The $\infty$-category $\Corr(\Fin)$ represents the presheaf on the $\infty$-category $\Cat_{\infty}$ of (small) $\infty$-categories whose value on $\cK$ is the $\infty$-groupoid of \emph{finite exponentiable fibrations} and isomorphisms among them.
Here, a functor $\cE\to \cK$ is a \emph{finite exponentiable fibration} if
\begin{itemize}
\item
it is an exponentiable fibration (Definition~\ref{def.efib});

\item
for each morphism $c_1 \to \cK$, the $\infty$-category $\Fun_{/\cK}(c_1 , \cE)$ of sections over this morphism is a finite $0$-type.  
\end{itemize}

\item
The functor $\ov{\Corr}(\Fin)\to \Corr(\Fin)$ is the universal finite exponentiable fibration.  

\end{itemize}

\end{remark}

\begin{example}\label{t605}
Consider the ordinary category $\Fin_\ast := \Fin^{\ast/}$ of based finite sets, an object we will typically denote as $I_+ = I\amalg \{+\}$.  
Consider the full subcategory of the undercategory,
\[
\Fin_{\ast \star}~\subset~ \Fin_\ast^{\star_+/}
\]
consisting those based maps $(\star_+ \xra{f} I_+)$ from the 2-element based set for which $f(\star)\neq +$.  
Lemma~\ref{fin-conduche} of the appendix gives that the evident projection functor $\Fin_{\ast \star} \to \Fin_\ast$ is a finite exponentiable fibration, as in Remark~\ref{t601}.
By Remark~\ref{t601}, this finite exponentiable fibration is classified by a functor 
\begin{equation}\label{e501}
\Fin_\ast 
\longrightarrow
\Corr(\Fin)~.
\end{equation}
Explicitly, this functor evaluates on an object $I_+\in \Fin_\ast$ as the object $I\in \Corr(\Fin)$, and on a morphism $f\colon I_+ \to J_+$ as the span $(I \la f^{-1}(J) \to J)$.
\end{example}

\subsubsection{\bf Wreath construction}
We recall the wreath construction.

Using Observation~\ref{t600}, after presentability considerations, Lemma~\ref{fin-conduche} of the appendix gives that base change along $\ov{\Corr}(\Fin)\to \Corr(\Fin)$ is a left adjoint.  
\begin{definition}[Wreath]\label{def.wreath}
For each $\infty$-category $\cD\to \Corr(\Fin)$ over based finite sets, 
the \emph{wreath} functor 
\[
\cD\wr - \colon \Cat_\infty \longrightarrow {\Cat_\infty}_{/\cD}
\]
is the right adjoint to the composite functor
\[
{\Cat_\infty}_{/\cD} \longrightarrow {\Cat_\infty}_{/\Corr(\Fin)} \xra{~{}~-\underset{\Corr(\Fin)}\times \ov{\Corr}(\Fin)~{}~} {\Cat_\infty}_{/\ov{\Corr}(\Fin)} \xra{~\rm forget~} \Cat_\infty~.
\]

\end{definition}

\begin{remark}\label{t602}
After Remark~\ref{t601}, the wreath construction admits the following description.
Let $\cE\to \cD$ be a finite exponentiable fibration; let $\cC$ be an $\infty$-category.
There is a canonical identification between $\infty$-categories over $\cD$,
\[
\cD\wr \cC~\simeq~\Fun^{\sf rel}_{\cD}(\cE,\cC)
~,
\]
with the \emph{relative functor $\infty$-category over $\cD$}.  That is, for each $\infty$-category $\cJ\to \cD$ over $\cD$, there is a canonical identification between the space of sections
\begin{equation}\label{e505}
\Cat_{/\cD}(\cJ , \cD \wr \cC)
~\simeq~
\Cat(\cE_{|\cJ} , \cC)
\end{equation}
and the space of functors from the base change of $\cE\to \cD$ along $\cJ \to \cD$.
\end{remark}

\begin{example}\label{simp-circle}
We regard the opposite of the simplex category $\bdelta^{\op}$ as an $\infty$-category over 
$\Corr(\Fin)$ 
by way of the simplicial circle
\[
\Delta[1]/\partial \Delta[1] \colon \bdelta^{\op} \longrightarrow \Fin_\ast
\overset{~(\ref{e501})~}\hookrightarrow
\Corr(\Fin)~.
\]

\end{example}

\begin{observation}\label{wreath-ff}
For each $\infty$-category $\cD$ over $\Corr(\Fin)$, a fully faithful functor $\cC\hookrightarrow \cC'$ between $\infty$-categories determines a fully faithful functor $\cD\wr \cC \to \cD\wr \cC'$. 

\end{observation}

\begin{lemma}\label{disks-to-fin}
For each dimension $n$, and each suspending tangential structure $\cB$, there is a functor
\[
\cDisk_n^\cB \longrightarrow \Corr(\Fin)
~,\qquad
(D,\varphi)
\mapsto D_n~,
\]
whose value on an object is the $0$-type which is its $n$-dimensional strata.  

\end{lemma}

\begin{proof}
It suffices to prove the result for the case where
$\cB \xra{\simeq} \Vect^{\sf inj}	$ 
is an equivalence.  
Consider the full $\infty$-subcategory 
$\Exit^{=n}\subset\Exit$
consisting of those pointed stratified spaces $(x\in X)\in \Exit$ for which $x\in X_n$ belongs to an $n$-dimensional stratum of $X$.  
Denote the base change of the resulting projection $\Exit^{=n} \to \Bun$ along a functor $\cJ \to \Bun$ as $\Exit^{=n}_{|\cJ} \to \cJ$. 
Through Remark~\ref{t601}, the result follows from the projection $\Exit^{=n}_{|\cDisk_n} \to \cDisk_n$ being a finite exponentiable fibration, which we show now.  

The fact that $\Exit^{=n}_{|\cDisk_n} \to \cDisk_n$ is an exponentiable fibration follows upon noting the following.
\begin{itemize}
\item
The fully faithful inclusion $\Exit^{=n}_{|\cDisk_n} \hookrightarrow \Exit_{|\cDisk_n}$ is a left fibration.  

\item
After the previous point, Corollary~\ref{cor.conduche} of the appendix gives that $\Exit^{=n}_{|\cDisk_n} \hookrightarrow \Exit_{|\cDisk_n}$ is an exponentiable fibration.

\item
Lemma~\ref{exit-exp} of the appendix states that the functor $\Exit \to \Bun$ is an exponentiable fibration.  
Thereafter, Lemma~\ref{conduche}(3) implies that the base change $\Exit_{|\cDisk_n} \to \cDisk_n$ is an exponentiable fibration.  

\item
Lemma~\ref{conduche}(3) implies that the composition of two composable exponentiable fibrations is again an exponentiable fibration.  
Applying this to the sequence of exponentiable fibrations $\Exit^{=n}_{|\cDisk_n} \to \Exit_{|\cDisk_n} \to \cDisk_n$ gives that the functor $\Exit^{=n}_{|\cDisk_n} \to \cDisk_n$ is an exponentiable fibration, as desired.  

\end{itemize}

It remains to show that the exponentiable fibration $\Exit^{=n}_{|\cDisk_n} \to \cDisk_n$ is \emph{finite}, in the sense of Remark~\ref{t601}.
From the definition of $\cDisk_n$, each morphism in $\cDisk_n$ is, in particular, a proper constructible bundle $X\to \Delta^1$ for which each stratum of $X$ is contractible and of dimension at most $n$.
It follows that, for each such morphism $X\to \Delta^1$, the space is a finite $0$-type which consists of those conically smooth sections $\Delta^1 \xra{\sigma} X$ whose value $\sigma(\Delta^{\{0\}}) \in X_n$ belongs to an $n$-dimensional stratum.

\end{proof}

\begin{lemma}\label{wreath-formal}
Let $n$ be a dimension, and let $\cB$ be a suspending tangential structure.
For each $\infty$-category $\cC$ there is a functor
\[
\cDisk_n^{\cB} \wr \cC \longleftarrow \cDisk_n^{\cB}\underset{\Bun}\times \Bun^{\un{\cC}}
\]
over $\cDisk_n^{\cB}$.
Furthermore, if $\cC$ has an initial object then this functor admits a fully faithful left adjoint over $\cDisk_n^{\cB}$.

\end{lemma}

\begin{proof}
It suffices to prove the result for the case 
$\cB\xra{\simeq}\Vect^{\sf inj}$ is an equivalence.  
Recall from the proof of Lemma~\ref{disks-to-fin}, the full $\infty$-subcategory
\begin{equation}\label{fin-exit}
\Exit^{=n}_{|\cDisk_n^{\cB}}
\hookrightarrow
\Exit_{|\cDisk_n}
\end{equation}
over $\cDisk_n$.
In particular, for each $\infty$-category $\cC$, and each $\infty$-category $\cJ$ over $\cDisk_n$, there is a restriction functor
\begin{equation}\label{Fun-fin-exit}
\Fun(\Exit_{|\cJ},\cC) \longrightarrow \Fun(\Exit^{=n}_{|\cJ}, \cC)~.
\end{equation}
The arrow in the statement of the lemma follows upon 
unwinding the definitions of the domain and codomain of that arrow (see~(\ref{e505}), for instance).

Now, because~(\ref{fin-exit}) is fully faithful, then a left adjoint to~(\ref{Fun-fin-exit}) over $\cDisk_n$ is fully faithful, should it exist. 
Indeed, for each functor $\cJ \to \cDisk_n$, any left adjoint to~(\ref{Fun-fin-exit}) is computed on sections over $\cJ$ via left Kan extension: 
\[
\bigl(\Exit^{=n}_{|\cJ}\xra{\cF}\cC\bigr) \mapsto \Bigl(\Exit_{|\cJ} \ni (x\in D)\mapsto \underset{(x'\in D')\in (\Exit^{=n}_{|\cJ})_{/(x\in D)}} \colim \cF(x'\in D'_n)\Bigr)~.
\]
The full $\infty$-subcategory $\Exit^{=n}_{|\cJ} \subset \Exit_{|\cJ}$ has the property that, for each object $(x\in D)\in \Exit_{|\cJ}$, 
the over $\infty$-category indexing this colimit
\[
(\Exit^{=n}_{|\cJ})_{/(x\in D)}
\]
is either empty or has a final object.
Indeed, each stratum of an object of $\cDisk_n$ is contractible and has dimension bounded by $n$.
We conclude that such a left adjoint exists if $\cC$ has an initial object.

\end{proof}

\begin{cor}\label{wreath}
For each pair of dimensions $i\geq 0$ 
and $j$, and for each suspending tangential structure $\cB\to \Vect^{\sf inj}$ for which the restricted projection $\cB_{\leq 0} \xra{\simeq} \Vect^{\sf inj}_{\leq 0}\simeq \ast$ is an equivalence, there is a fully faithful functor
\[
\cDisk_i^\cB \wr \cDisk^\cB_j  ~\hookrightarrow~\cDisk_i^\cB \underset{\Bun}\times \Bun^{\cDisk^\cB_j}
\]
over $\cDisk_i^{\cB}$.
In particular, there are fully faithful functors
\[
\cDisk_1^{\vfr}\wr \cDisk^{\vfr}_{n-1}  ~\hookrightarrow~\cDisk_1^{\vfr} \underset{\Bun}\times \Bun^{\cDisk^{\vfr}_{n-1}}
\]
and
\[
\cDisk_1^{\sfr} \wr \cDisk^{\sfr}_{n-1}  ~\hookrightarrow~\cDisk_1^{\sfr} \underset{\Bun}\times \Bun^{\cDisk^{\sfr}_{n-1}}~,
\]
each over $\cDisk_1^{\vfr}$.

\end{cor}

\begin{proof}
After Lemma~\ref{wreath-formal} it need only be checked that $\cDisk^\cB_j$ has an initial object.  
Because of the condition on $\cB_{\leq 0}$, this is the case exactly because $\cDisk_j$ has an initial object, which is $\ast$.  
\end{proof}

\subsection{Iterated linear orders}
We recall the definition of Joyal's category $\btheta_n$ as well as some notions within it.  
We first recall the following notions within the simplex category $\bdelta$.

\begin{definition}\label{def.segal-covs}
\begin{itemize}
\item[~]
\item {\bf Inerts:} The \emph{inert} subcategory $\bdelta_{\sf inrt} \subset \bdelta$ consists of the same objects and those morphisms $[p]\xra{\rho} [q]$ for which $\rho(i-1) = \rho(i)-1$ for each $0<i\leq p$.  

\item {\bf Actives:} The \emph{active} subcategory $\bdelta_{\sf act}\subset \bdelta$ consists of the same objects and those morphisms $[p]\xra{\rho}[q]$ that preserve extrema: $\rho(0)=0$ and $\rho(p)=q$.  

\item {\bf Segal covering diagrams:} A colimit diagram $[1]\times[1]\to \bdelta$, written
\[
\xymatrix{
[p_0]  \ar[r]  \ar[d]
&
[p']  \ar[d]
\\
[p'']  \ar[r]
&
[p],
}
\]
is a \emph{Segal covering} diagram if each arrow is inert.   

\item {\bf Univalence diagram:} The \emph{univalence} diagram is the colimit diagram in $\bdelta$
\[
\xymatrix{
&
\{1<3\}  \ar[r]  \ar[d]
&
\ast  \ar[dd]
\\
\{0<2\}  \ar[r]  \ar[d]
&
\{0<1<2<3\}  \ar[dr]
&
\\
\ast \ar[rr]
&&
\ast.
}
\]
\end{itemize}

\end{definition}

\begin{remark}
Consider the subcategory $\Fin_{\sf inj}\subset \Fin$ consisting of the same objects and those morphisms which are \emph{injective}.  
There is a monomorphism $\Fin_{\sf inj}^{\op} \to \Fin_\ast$ given by 1-point compactifications and collapse-maps thereamong.
Recall from Example~\ref{simp-circle} the functor $\bdelta^{\op}\to \Fin_\ast$.  
With respect to these functors, there is an identification $\bdelta_{\sf inrt}^{\op} \simeq \bdelta^{\op}\underset{\Fin_\ast}\times \Fin_{\sf inj}^{\op}$.

\end{remark}

\begin{remark}\label{act-inrt}
It is standard that $(\bdelta_{\sf act},\bdelta_{\sf inrt})$ is a factorization system on $\bdelta$.  

\end{remark}

The next result is definitional.  
\begin{prop}\label{diagrams-to-colims}
The standard functor 
\[
\bdelta \xra{~[\bullet]~} \Cat_\infty
\]
carries both Segal covering diagrams and the univalence diagram to a colimit diagram.  

\end{prop}

\begin{proof}
Let $[1]\times [1]\to \bdelta$, written 
\[
\xymatrix{
[p_0]  \ar[r]  \ar[d]
&
[p']  \ar[d]
\\
[p'']  \ar[r]
&
[p],
}
\]
be a Segal covering diagram.
Because colimits commute with one another, to argue that this is a colimit diagram among $\infty$-categories it suffices to argue the case that $[p'']=\{0<1\}$ and $[p]=\{1\}$ and $[p'] = \{1<\dots,p\}$.  
That is, for each $\infty$-category $\cC$, the canonical map of spaces of functors
\[
\Map\bigl([p],\cC\bigr) \longrightarrow \Map\bigl(\{0<1\},\cC\bigr) \underset{\Map(\{1\},\cC)}\times \Map\bigl(\{1<\dots<p\},\cC\bigr)
\]
must be an equivalence.  
This is manifestly the case.

Consider the functor $\cE^{\tr} \to \bdelta$ representing the univalence diagram.
Consider the $\infty$-category $E:=\colim(\cE\to \bdelta \to \Cat_\infty)$.  
This $\infty$-category corepresents the data of a morphism together with a left and a right inverse:
\[
\Map_{\Cat_\infty}(E,\cC)~\simeq~\Bigl\{d\xra{f^R}c \xra{f} d \xra{f^L} c\text{ and } f\circ f^R\simeq {\sf id}_d  \text{ and } {\sf id}_c\simeq   f^L\circ f, \text{ all in }\cC \Bigr\}~.
\]
Manifestly, left and right inverses are unique whenever they exist.
Thus, we identify the above space simply as $\Map\bigl([1],\cC^\sim\bigr)$, the space of morphisms in the maximal $\infty$-subgroupoid of $\cC$.  
There is the further identification $\cC^\sim \xra{\simeq} \Map\bigl([1],\cC^\sim\bigr)$ induced by the unique functor $[1]\to \ast$.  
In other words, the unique functor $E\to \ast$ is an equivalence of $\infty$-categories.  
In conclusion, the composite functor $\cE^{\tr} \to \bdelta \to \Cat_\infty$ is a colimit diagram.  
\end{proof}

We give a definition of Joyal's category $\btheta_n$~(\cite{joyaltheta}). This follows Definition~3.9 in~\cite{berger}, adapted through the wreath Construction~2.4.4.1 of~\cite{HA}.  
\begin{definition}[$\btheta_n$]\label{def.theta}
For $n\geq 0$, the $\infty$-category over $\bdelta^{\op}$, 
\[
\btheta_n^{\op} \longrightarrow \bdelta^{\op}~,
\]
is defined inductively as
$
\btheta_{n}^{\op}:=\bdelta^{\op} \wr \btheta_{n-1}^{\op}
$
for $n>0$, while $\btheta_0^{\op}:=\ast \xra{\{[0]\}} \bdelta^{\op}$.

\end{definition}

\begin{remark}\label{r3}
Fix $n>0$.
An object in $\btheta_n$ is the data of an object $[p]\in \bdelta$ together with, for each $0<i\leq p$, an object $T_i\in \btheta_{n-1}$.  
Following the notation of~\cite{berger}, we denote such data as
\[
[p]\bigl( T_1,\dots,T_p)~\in \btheta_n~.
\]
With this notation, here are some objects in $\btheta_2$:
\[
[0]
~,\qquad
[1]\bigl( [3] \bigr)
~,\qquad
[4]\bigl( [1] , [0] , [5] , [3] \bigr)
~,\qquad
[3]\bigl([6] , [0] , [2]  \bigr)
~\in \btheta_2~.
\]

\end{remark}

\begin{remark}
We have presented $\btheta_n^{\op}$ as an $\infty$-category.
However, for each pair of objects $T,T'\in \btheta_n^{\op}$, the space of morphisms $\btheta_n^{\op}(T',T)$ is a $0$-type; in other words, $\btheta_n^{\op}$ is an ordinary category.  
Indeed, this follows quickly by induction on $n$, using that each of $\bdelta^{\op}$ and $\Fin_\ast$ and $\Fin_{\ast\star}$ are ordinary categories.  

\end{remark}

We record some notions within the category $\btheta_n:= (\btheta_n^{\op})^{\op}$.  
\begin{definition}\label{def.n-Segal-cov}
Let $n\geq 0$.  
\begin{itemize}

\item {\bf Inerts:}  The \emph{inert} $\infty$-subcategory of $\btheta_n^{\op}$ is defined inductively as $\btheta_{n,\sf inrt}^{\op}:=\bdelta_{\sf inrt}^{\op} \wr \btheta_{n-1,\sf inrt}^{\op}$ for $n>0$ while $\btheta_{0,\sf inrt} = \btheta_0 = \ast$.  

\item {\bf Actives:}  The \emph{active} $\infty$-subcategory of $\btheta_n^{\op}$ is defined inductively as $\btheta_{n,\sf act}^{\op}:=\bdelta_{\sf act}^{\op} \wr \btheta_{n-1,\sf act}^{\op}$ for $n>0$ while $\btheta_{0,\sf act} = \btheta_0 = \ast$.  

\item {\bf Cells:}
For each $0\leq i \leq n$, the \emph{$i$-cell} $c_i\in \btheta_n^{\op}$ is the initial object if $i=0$ and if $i>0$ it is the object $\ast \xra{\{c_i\}} \btheta_n^{\op}$ representing the pair of functors $\ast \xra{\{[1]\}} \bdelta^{\op}$ and $\ast \simeq \{[1]\}\underset{\Fin_\ast}\times \Fin_{\ast \star} \xra{\{c_{i-1}\}} \btheta_{n-1}^{\op}$.   

\item {\bf Segal covering diagrams:} A colimit diagram $[1]\times[1]\to \btheta_n$, written
\[
\xymatrix{
T_0  \ar[r]  \ar[d]
&
T'  \ar[d]
\\
T''  \ar[r]
&
T,
}
\]
is a \emph{Segal covering} diagram if each arrow is inert.   

\item {\bf Univalence diagrams:} For $n>0$, a colimit diagram $\cE^{\tr} \to \btheta_n$ is a \emph{univalence} diagram if it has either of the following two properties:
\begin{itemize}

\item The projection $\cE^{\tr} \to \bdelta$ factors through the functor $\ast \xra{\{c_1\}} \bdelta$, and the functor $(\cE^{\tr})^{\op} \simeq (\cE^{\tr})^{\op} \underset{\Fin_\ast}\times \Fin_{\ast \star} \to \btheta_{n-1}^{\op}$ is the opposite of a univalence diagram.

\item The projection ${\cE^{\tr}} \to \bdelta$ is the univalence diagram, and the functor $(\cE^{\tr})^{\op}\underset{\Fin_\ast}\times \Fin_{\ast \star} \to \btheta_{n-1}^{\op}$ factors through $\ast \xra{\{c_0\}} \btheta_{n-1}^{\op}$.

\end{itemize}

\end{itemize}

\end{definition}

\begin{remark}\label{n-act-inrt}
The active-inert factorization system on $\bdelta$ (see Remark~\ref{act-inrt}) determines an active-inert factorization system on $\btheta_n$ for each $n\geq 0$. 

\end{remark}

\begin{observation}\label{theta-stand-bump}
We note that, for each $0\leq k \leq n$, there are monomorphisms among categories
\[
\iota_{k}\colon \btheta_k^{\op} ~\hookrightarrow~ \btheta_n^{\op} ~\hookleftarrow~ \btheta_{n-k}^{\op} \colon c_k \wr -~,
\]
the first of which is fully faithful; these functors preserve Segal covering diagrams and univalence diagrams.  
For $k=0$ the left functor is the inclusion of the initial object while the right functor is the identity.  
For $0<k\leq n$ these functors are given through induction as
\[
\btheta_k^{\op}:=\bdelta^{\op} \wr \btheta_{k-1}^{\op}
~\overset{{\sf id}_{\bdelta^{\op}} \wr \iota_{k}}\hookrightarrow~
\bdelta^{\op} \wr \btheta_{n-1}^{\op} =: \btheta_n^{\op}:= \bdelta^{\op}\wr \btheta_{n-1}^{\op}
~\overset{\{[1]\} \wr (c_{k-1}\wr-)}\hookleftarrow~
\{\star_+\}\wr \btheta_{n-k}^{\op}
\simeq 
\btheta_{n-k}^{\op}~.  
\]
The injectivity assertions follow quickly by induction, for which it is useful that all $\infty$-categories here are ordinary categories.  
That these functors preserve Segal covering diagrams follows because, by induction, they preserve inert morphisms, and, by induction, they preserve colimits.  
That these functors preserve univalence diagrams is direct from definitions.  

\end{observation}

\subsection{Cellular realization}
We define a fully faithful functor $\btheta_n^{\op} \hookrightarrow \cDisk_n^{\vfr}$, which we call \emph{cellular realization}.  
The existence of this functor, and that it is fully faithful, is the culmination of choices behind the definition of $\Bun$ and its vari-framed version. 
This fully faithful functor founds the contact between higher categories and vari-framed differential topology, as we articulate as factorization homology of the coming section.

The next result establishes the cellular realization functor for dimension $1$.
Recall that a morphism in $\cDisk_1^{\vfr}$ is \emph{closed/active} if it is carried to a closed/active morphism by the forgetful functor $\cDisk_1^{\vfr} \to \cBun$.  
\begin{lemma}\label{delta-to-disk}
There is a functor
\[
\lag-\rag\colon \bdelta^{\op} \longrightarrow  \cDisk^{\vfr}_1
\]
with the following properties.
\begin{enumerate}
\item  The functor is fully faithful.

\item The functor carries $[0]$ to $\DD^0$ and $[1]$ to $\DD^1$.

\item The functor carries the opposite of inert morphisms to closed morphisms.

\item The functor carries the opposite of active morphisms to active morphisms.

\item The functor carries Segal covers to purely closed covers.

\item Let $[p]\xra{\rho}[q]$ be a morphism in $\bdelta$.
For each factorization $\lag \rho\rag\colon \lag [q]\rag \xra{c} D \to \lag [p]\rag$ in $\cDisk^{\vfr}_1$ in which $c$ is a closed morphism, the object $D$ belongs to the essential image of $\bdelta^{\op}$.

\end{enumerate}

\end{lemma}

\begin{proof}
Consider the full $\infty$-subcategory $\cD\subset \cDisk^{\vfr}_1$ consisting of those $D$ for which there is a refinement morphism $D\to \DD^1$ or $D\to \DD^0$ in $\cDisk^{\vfr}_1$.  
By inspection, and likewise, this full $\infty$-subcategory has property~(6): 
if a morphism $D'\to D''$ in $\cD$ factors through a closed morphism $D'\to D$ then $D$ belongs to $\cD$.  

We will construct an equivalence of $\infty$-categories
\[
\cD\xra{~\simeq~}\bdelta^{\op}
\]
as striation sheaves, according to the main result of~\cite{striat}. 
Let $K$ be a stratified space.  
Because the morphism spaces in $\cDisk_n^{\vfr}$ are $0$-types (Lemma~\ref{disk-1}), the space of $K$-points of each of these striation sheaves is a $0$-type.

Consider a functor $\exit(K)\xra{(X\to K,\phi)} \cD$.
Denote the fiberwise $0$-dimensional strata as $X_0\to K$, and likewise $X_{>0}\to K$ for its complement.  
Choose a conically smooth embedding 
\[
e\colon X\hookrightarrow \RR\times K
\]
over $K$ for which the pullback fiberwise vari-framing $e_{>0}^\ast \partial_t\colon \exit(X_{>0}) \to \vfr$ is in the same component of the restriction $\phi_{>0}\colon \exit(X_{>0}) \to \vfr$. 
This embedding $e$ determines a fiberwise linear order $\leq$ on the constructible bundle $X_0\to K$, by which we mean a constructible closed subspace of the fiber product $X_0 \underset{K}\times X_0$ that intersects each fiber $X_{|\{k\}}\times X_{|\{k\}}$ as a linear order.  
This subspace $\leq$ is constructible and closed, and does not depend on the choice of embedding $e$.  
Taking values in $\cD$ implies $X_0 \to K$ is a surjective proper constructible bundle and has the following unique path lifting property in $\strat$
\[
\xymatrix{
\Delta^{\{1\}}  \ar[rr]    \ar[d]
&&
X_0  \ar[d]
\\
\Delta^1  \ar[rr]  \ar@{-->}[urr]^-{\exists !}
&&
K.
}
\]
Furthermore, the constructible fiberwise linear order enhances the right fibration $\exit(X_0) \to \exit(K)$ to a Cartesian fibration whose fibers are non-empty finite linearly ordered sets.
Such a Cartesian fibration is classified by a functor $\exit(K) \xra{(X_0\to K,\leq)} \bdelta^{\op}$.

We have thus produced a well-defined assignment of $K$-points from those of $\cD$ to those of $\bdelta^{\op}$.  
Tracing through the construction of this assignment, it restricts along conically smooth maps $K'\to K$, thereby producing the desired functor 
\[
\cD \longrightarrow \bdelta^{\op}~.  
\]

We now wish to show this functor is an equivalence of $\infty$-categories.
It is clearly essentially surjective, and surjective on mapping components.  
To see that this functor is injective on mapping components follows upon observing that, in the situation of the preceding argument, the space of embeddings $e\colon X\hookrightarrow \RR\times K$ over $K$ is connected.  
This proves property~(1).

Properties~(2),(3),(4), and (5) follow by direct inspection of the functor $\cD\to \bdelta^{\op}$ just constructed.
\end{proof}

\begin{remark}
The value $\lag [0]\rag$ is a point, $\ast$, regarded as a vari-framed manifold.
For $p>0$, the value $\lag [p]\rag$ is a var-framed refinement of the interval $\DD^1=[-1,1]$ whose $0$-dimensional strata bijective with the set $\{0,\dots , p\}$.  
See Figure~\ref{fig.66}.

\end{remark}

\begin{figure}[ht]
\begin{tikzpicture}

\node at (0,0) {$[0] \ \mapsto$};
\tikzstyle{every node}=[font=\tiny]
\draw[fill] (0.75,0) circle [radius=0.07] node[below]{0};
\tikzstyle{every node}=[font=\normalsize]

\node at (0,-0.75) {$[1] \ \mapsto$};
\tikzstyle{every node}=[font=\tiny]
\draw[fill] (0.75,-0.75) circle [radius=0.07] node[below]{0};
\draw[fill] (1.75,-0.75) circle [radius=0.07] node[below]{1};
\draw[midarrow=0.5] (0.75,-0.75) to (1.75,-0.75); 
\tikzstyle{every node}=[font=\normalsize]

\node at (0,-1.5) {$[2] \ \mapsto$};
\tikzstyle{every node}=[font=\tiny]
\draw[fill] (0.75,-1.5) circle [radius=0.07] node[below]{0};
\draw[fill] (1.75,-1.5) circle [radius=0.07] node[below]{1};
\draw[fill] (2.75,-1.5) circle [radius=0.07] node[below]{2};
\draw[midarrow=0.5] (0.75,-1.5) to (1.75,-1.5);
\draw[midarrow=0.5] (1.75,-1.5) to (2.75,-1.5); 
\tikzstyle{every node}=[font=\normalsize]

\node at (0.5,-2) {$\vdots$};

\end{tikzpicture}
\caption{Some values of the cellular realization functor $\lag - \rag \colon \bdelta^{\op} \to \cDisk^{\vfr}_1$.}
\label{fig.66}
\end{figure}
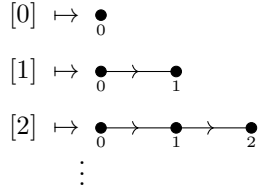

The definition of cellular realization is supported by the inductive Definition~\ref{def.theta} of $\btheta_n$ combined with Corollary~\ref{wreath}, which compares the wreath construction to iterated constructible bundles, and the vari-framed lift of $\Bun^{\un{\Bun}} \xra{\circ} \Bun$ provided by Corollary~\ref{vfr-iterate}.  
\begin{definition}[Cellular realization]\label{def.cellular}
For each dimension $n$, the \emph{cellular realization} functor
\[
\lag-\rag\colon \btheta_n^{\op} \longrightarrow \cDisk^{\vfr}_n
\]
is the composition
\[
\lag-\rag\colon \btheta_n^{\op} ~:=~\bdelta^{\op}\wr \btheta_{n-1}^{\op} \underset{\rm induction}\longrightarrow \bdelta^{\op} \wr \cDisk^{\vfr}_{n-1} \underset{\rm Cor~\ref{wreath}}\longrightarrow \bdelta^{\op} \underset{\Bun}\times \Bun^{\cDisk^{\vfr}_{n-1}} \underset{\rm Cor~\ref{vfr-iterate}}{\xra{~\circ~}} \cDisk^{\vfr}_n
\]
provided $n>1$.
For $n=0$ it is the functor $\ast \xra{\{\DD^0\}} \cDisk^{\vfr}_0$; for $n=1$ it is the functor of Lemma~\ref{delta-to-disk}.   

\end{definition}

\begin{remark}\label{r4}
Intuitively, the value $\lag T\rag$ is literally the pasting diagram associated to $T\in \btheta_n$ which can be regarded as a stratified subspace of $\RR^n$ as it inherits a vari-framing.  
To make this intuition precise and assemble this description as a functor in any point-set sense is not practical, if possible at all.  
The functor $\lag-\rag$ fully embraces the setting of $\infty$-categories.  
See Figure~\ref{fig.77}.

\end{remark}

%
%
%
%
%
%
%
%
%

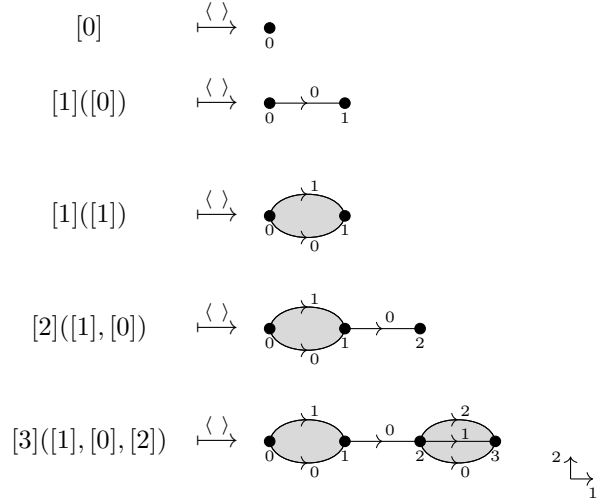
\begin{figure}[ht]
\begin{tikzpicture}


\node at (-0.5,5.5) {$[0]$};
\node at (1.2,5.65) {$\overset{\langle \ \rangle}{\longmapsto}$};
\tikzstyle{every node}=[font=\tiny]
\draw[fill] (1.9,5.5) circle [radius=0.07] node[below]{0};


\tikzstyle{every node}=[font=\normalsize]
\node at (-0.5,4.5) {$[1]([0])$};
\node at (1.2,4.65) {$\overset{\langle \ \rangle}{\longmapsto}$};
\tikzstyle{every node}=[font=\tiny]

\draw[fill] (1.9,4.5) circle [radius=0.07] node[below]{0};
\draw[fill] (2.9,4.5) circle [radius=0.07] node[below]{1};
\draw[->-] (1.9,4.5) to (2.9,4.5);
\node at (2.5, 4.65) {0};

\tikzstyle{every node}=[font=\normalsize]

\node at (-0.5,3) {$[1]([1])$};
\node at (1.2,3.15) {$\overset{\langle \ \rangle}{\longmapsto}$};

\tikzstyle{every node}=[font=\tiny]

\draw[fill=lightgray] (1.9,3) to[out=80,in=100] (2.9,3) to[out=-100,in=-80] (1.9,3);

\draw[fill] (1.9,3) circle [radius=0.07] node[below]{0};
\draw[fill] (2.9,3) circle [radius=0.07] node[below]{1};

\draw[midarrow=0.5] (1.9,3) to[out=80,in=100] (2.9,3);
\draw[midarrow=0.5] (1.9,3) to[out=-80,in=-100] (2.9,3);

\node at (2.5,3.4) {1};
\node at (2.5,2.6) {0};

\tikzstyle{every node}=[font=\normalsize]

\node at (-0.5,1.5) {$[2]([1],[0])$};
\node at (1.2,1.65) {$\overset{\langle \ \rangle}{\longmapsto}$};

\tikzstyle{every node}=[font=\tiny]

\draw[fill=lightgray] (1.9,1.5) to[out=80,in=100] (2.9,1.5) to[out=-100,in=-80] (1.9,1.5);
\draw[->-] (2.9,1.5) to (3.9,1.5);

\draw[fill] (1.9,1.5) circle [radius=0.07] node[below]{0};
\draw[fill] (2.9,1.5) circle [radius=0.07] node[below]{1};
\draw[fill] (3.9,1.5) circle [radius=0.07] node[below]{2};

\draw[midarrow=0.5] (1.9,1.5) to[out=80,in=100] (2.9,1.5);
\draw[midarrow=0.5] (1.9,1.5) to[out=-80,in=-100] (2.9,1.5);

\node at (2.5,1.9) {1};
\node at (2.5,1.1) {0};
\node at (3.5, 1.65) {0};

\tikzstyle{every node}=[font=\normalsize]

\node at (-0.5,0) {$[3]([1],[0],[2])$};
\node at (1.2,0.15) {$\overset{\langle \ \rangle}{\longmapsto}$};

\tikzstyle{every node}=[font=\tiny]

\draw[fill=lightgray] (1.9,0) to[out=80,in=100] (2.9,0) to[out=-100,in=-80] (1.9,0);
\draw[fill=lightgray] (3.9,0) to[out=80,in=100] (4.9,0) to[out=-100,in=-80] (3.9,0);

\draw[fill] (1.9,0) circle [radius=0.07] node[below]{0};
\draw[fill] (2.9,0) circle [radius=0.07] node[below]{1};
\draw[fill] (3.9,0) circle [radius=0.07] node[below]{2};
\draw[fill] (4.9,0) circle [radius=0.07] node[below]{3};

\draw[midarrow=0.5] (1.9,0) to[out=80,in=100] (2.9,0);
\draw[midarrow=0.5] (1.9,0) to[out=-80,in=-100] (2.9,0);

\draw[->-] (2.9,0) to (3.9,0);

\draw[midarrow=0.5] (3.9,0) to (4.9,0);
\draw[midarrow=0.5] (3.9,0) to[out=80,in=100] (4.9,0);
\draw[midarrow=0.5] (3.9,0) to[out=-80,in=-100] (4.9,0);

\node at (2.5,0.4) {1};
\node at (2.5,-0.4) {0};
\node at (3.5, 0.15) {0};
\node at (4.5,0.4) {2};
\node at (4.5, 0.1) {1};
\node at (4.5, -0.4) {0};


\draw[->] (5.9,-0.5) to (6.2,-0.5) node[below]{1};
\draw[->] (5.9,-0.5) to (5.9,-0.2) node[left]{2};

\end{tikzpicture}
\caption{Typical values of the cellular realization functor $\lag - \rag \colon \btheta_2^{\op} \to \cDisk^{\vfr}_2$.}
\label{fig.77}
\end{figure}

\begin{observation}\label{wreath-ff-bun-bun}
For each $0\leq k \leq n$ the fully faithful functor 
$
\iota_k\colon \btheta_k^{\op} \hookrightarrow \btheta_n^{\op}
$
lies over $\cDisk^{\vfr}_k \hookrightarrow \cDisk^{\vfr}_n$.  
This follows by induction on $i$ from Observation~\ref{wreath-ff}, with base case the fully faithful functor $\ast \hookrightarrow \bdelta^{\op}$.  

\end{observation}

Recall the $\infty$-category $\cBun^n$ of~\S\ref{sec.iterated-bdls} that classifies $n$-fold proper constructible bundles among stratified spaces.  
There is a $\cB$-structured version of this $\infty$-category that we highlight as the following notation.
\begin{notation}\label{iterated-tau's}
For $\cB$ a tangential structure, we inductively denote
\[
\cBun^{\cB, n}~:=~ \cBun^{\cB} \underset{\Bun}\times \Bun^{\un{\cBun}^{\cB,n-1}}
\]
where $\Bun^{\cB,0}:= \ast$.  
\end{notation}

With Notation~\ref{iterated-tau's} we notice from the Definition~\ref{def.cellular} that the cellular realization functor factors
\[
\btheta_n^{\op} \longrightarrow \cDisk^{\vfr_{\leq 1},n} \xra{~\circ~} \cDisk_n^{\vfr}~.
\]
\begin{lemma}\label{theta-ff-in-wreath}
The functor $\btheta_n^{\op} \to \cDisk^{\vfr_{\leq 1},n}$ is fully faithful, and it carries Segal covers to purely closed covers.

\end{lemma}

\begin{proof}
Lemma~\ref{delta-to-disk} gives that $\bdelta^{\op} \to \cDisk^{\vfr}_1$ is fully faithful.  
Thereafter, Observation~\ref{wreath-ff} gives that $\btheta_n^{\op} \to (\cDisk^{\vfr})^{\wr n}$ is fully faithful.  
Lastly, Corollary~\ref{wreath} gives that $\btheta_n^{\op} \to (\cDisk^{\vfr})^{\wr n} \to \cDisk^{\vfr_{\leq 1},n}$ is fully faithful.
The statement about covers follows immediately by induction, with the case $n=1$ given by Lemma~\ref{delta-to-disk}.  
\end{proof}

\begin{lemma}\label{theta-preserves} 
For each dimension $n$, the functor
\[
\btheta_n^{\op} \xra{~\lag-\rag~} \cDisk^{\vfr}_n
\]
caries Segal covers to purely closed covers.

\end{lemma}

\begin{proof}
Lemma~\ref{delta-to-disk} gives the case of $n=1$.  
This implies the assertion for the functor $\btheta_n^{\op} \to \cDisk^{\vfr_{\leq 1},n}$.
Using Lemma~\ref{theta-ff-in-wreath} just above, it remains to verify that the functor $\cDisk^{\vfr_{\leq 1},n}\xra{\circ} \cDisk^{\vfr}_n$ carries purely closed covers to purely closed covers.  
Observation~\ref{tau-closed-covers} gives that purely closed covers in $\Bun^{\cB}$ are detected by the projection $\Bun^{\cB} \to \Bun$.
So it suffices to verify that the functor $\Bun^n \xra{\circ} \Bun$ carries purely closed covers to purely closed covers. 
By induction, we may reduce to the case $n=2$.
This case follows from the following observation, which follows because the sheaf $\Bun$ on $\strat$ is a striation sheaf, and in particular it is \emph{cone-local} (see~\S4 of~\cite{striat} for a discussion of this term):
\begin{itemize}
\item[~]
For each pushout diagram in $\strat$
\[
\xymatrix{
Y_0  \ar[r]  \ar[d]
&
Y'' \ar[d]
\\
Y'  \ar[r]
&
Y
}
\]
in which each morphism is a proper constructible embedding, and for each constructible bundle $X\to Y$, the square of pullbacks
\[
\xymatrix{
X_{|Y_0}  \ar[r]  \ar[d]
&
X_{|Y''} \ar[d]
\\
X_{|Y'}  \ar[r]
&
X
}
\]
too is a pushout. 

\end{itemize}
\end{proof}

\begin{theorem}\label{theta-ff}
\footnote{
This result is false without the dimension bound; see the erratum included in~\S\ref{sec.erratum}.
}
For each dimension $n<3$, the cellular realization functor
\[
\lag-\rag\colon \btheta_n^{\op} \longrightarrow \cDisk^{\vfr}_n
\]
is fully faithful.

\end{theorem}

\begin{proof}
Let $T,T'\in \btheta_n$ be objects.
We must show that the map of spaces 
\begin{equation}\label{theta-map}
\btheta_n(T',T) \xra{\lag-\rag} \cDisk^{\vfr}_n\bigl(\lag T\rag,\lag T'\rag\bigr)
\end{equation}
is an equivalence.  We will prove ~(\ref{theta-map}) is an equivalence by induction on the dimension of the underlying stratified space of $\lag T'\rag$. 
Suppose the dimension of $\lag T'\rag$ is zero.
Necessarily, $T'=c_0$ is the $0$-cell and $\lag T'\rag = \DD^0$ is the hemispherical $0$-disk, which is just $\ast$.  
As such, both $\btheta_n(c_0,T)$ and $\cDisk^{\vfr}_n\bigl(\lag T\rag,\DD^0\bigr)$ are compatibly identified as the space $\lag T\rag_0$ which is the $0$-dimensional stratum of the stratified space $\lag T\rag$.  
This proves the base case of our induction. 

Now suppose~(\ref{theta-map}) is an equivalence whenever the dimension of the underlying stratified space of $\lag T'\rag$ is less than $k'$.  
We proceed by induction on the dimension of the underlying stratified space of $\lag T\rag$.
Suppose the dimension of $\lag T\rag$ is zero.
Necessarily, $T=c_0$ is the $0$-cell and $\lag T\rag=\DD^0$ is the hemispherical $0$-disk, which is just $\ast$.  
As such, both $\btheta_n(T',c_0)$ and $\cDisk^{\vfr}_n\bigl(\DD^0, \lag T'\rag\bigr)$ are terminal.  
This proves the base case of our nested induction. 
So suppose~(\ref{theta-map}) is an equivalence whenever the dimension of the underlying stratified space of $\lag T\rag$ is less than $k$.

By construction, each object of $\btheta_n^{\op}$ can be witnessed as a finite iteration of Segal covers among the cells $c_k$ $(0\leq k \leq n)$.
Because Segal covers are in particular limit diagrams in $\btheta_n^{\op}$, there is a finite limit diagram $\cU^{\tl} \to \btheta_n^{\op}$ whose value on the cone-point is $T'$ and whose value on each $U\in \cU$ is a $k$-cell for some $0\leq k \leq n$.  
Lemma~\ref{theta-preserves} gives that the composite functor $\cU^{\tl} \to \btheta_n^{\op} \to \cDisk^{\vfr}_n$ too is a finite limit diagram whose value on the cone-point is $\lag T'\rag$ and whose value on each $U\in \cU$ is a hemispherical $k$-disk for some $0\leq k \leq n$.  
This explains the vertical equivalences
\[
\xymatrix{
\btheta_n(T',T)   \ar[d]_-{\simeq}  \ar[rr]^-{\lag-\rag} 
&&
\cDisk^{\vfr}_n\bigl(\lag T\rag,\lag T'\rag\bigr)  \ar[d]^-{\simeq}
\\
\underset{U\in \cU}{\sf lim} \btheta_n(U,T)    \ar[rr]^-{\lag-\rag} 
&&
\underset{U\in \cU} {\sf lim} \cDisk^{\vfr}_n\bigl(\lag T\rag,\lag U\rag\bigr).
}
\]
Therefore, the map~(\ref{theta-map}) is an equivalence if and only if it is for $T'= c_{j'}$ for each $0\leq j' \leq n$.  '

Lemma~\ref{theta-preserves} gives that each object $T\in \btheta_n^{\op}$ can be witnessed as a finite colimit diagram $\cV^{\tr} \to \btheta_n^{\op}$ with the value on the cone-point $T$ and with the value on each $V\in \cV$ a $k$-cell for some $0\leq k \leq n$.  
Lemma~\ref{cls-mor-sections} grants that the composite functor $\cV^{\tr} \to \btheta_n^{\op} \to \cDisk^{\vfr}_n$ is again a finite colimit diagram whose value on the cone-point is $\lag T\rag$ and whose value on each $V\in \cV$ is a hemispherical $k$-disk for some $0\leq k \leq n$.  
This explains the vertical equivalences
\[
\xymatrix{
\btheta_n(T',T)   \ar[d]_-{\simeq}  \ar[rr]^-{\lag-\rag} 
&&
\cDisk^{\vfr}_n\bigl(\lag T\rag,\lag T'\rag\bigr)  \ar[d]^-{\simeq}
\\
\underset{V\in \cV}{\sf lim} \btheta_n(T',V)    \ar[rr]^-{\lag-\rag} 
&&
\underset{V\in \cV} {\sf lim} \cDisk^{\vfr}_n\bigl(\lag V\rag,T'\bigr).
}
\]
Therefore, the map~(\ref{theta-map}) is an equivalence if and only if it is for $T= c_j$ for each $0\leq j \leq n$.

We have reduced the problem of showing~(\ref{theta-map}) is an equivalence to the case $T=c_k$ and $T'=c_{k'}$, with the assumption that the map~(\ref{theta-map}) is an equivalence whenever the dimensions of $\lag T\rag$ and $\lag T'\rag$ are smaller than $k$ and $k'$, respectively.  
Lemma~\ref{fact-sys} states a closed-active factorization system on $\cDisk^{\vfr}_n$.
Lemma~\ref{theta-ff-in-wreath} implies such a factorization system for $\btheta_n^{\op}$.   
This is to say that the composition maps from the coends
\begin{equation}\label{coends-t}
\circ\colon \btheta_n^{\sf act}(-,T) \underset{\btheta_n^\sim}\bigotimes \btheta_n^{\sf cls}(T',-) \xra{~\simeq~} \btheta_n(T',T)
\end{equation}
and
\begin{equation}\label{coends-d}
\circ\colon \cDisk_n^{\vfr, \sf cls}\bigl(\lag T\rag,-\bigr) \underset{\cDisk_n^{\vfr,\sim}}\bigotimes \cDisk_n^{\vfr, \sf act}\bigl(-,\lag T'\rag\bigr) \xra{~\simeq~} \cDisk^{\vfr}_n\bigl(\lag T\rag,\lag T'\rag\bigr)
\end{equation}
are equivalences.
Now, consider a composable pair of morphisms $\DD^k\to D_0\to \DD^{k'}$ in $\cDisk^{\vfr}_n$ in which the first morphism is closed and the second is active.  
The target of a closed morphism from $\DD^k$ is either $\DD^l$ or $\partial \DD^l$ for some $k \geq l\leq k'$.  

We now rule out the case that this target is $\partial \DD^l$ for some $l\leq k$.
We show by contradiction that there are no active morphisms from $\partial \DD^l$ to $\DD^{k'}$.
So assume there exists an active morphism $\partial \DD^l \xra{a} \DD^{k'}$.  
Necessarily, $l-1\leq k'$.
Consider the composition $\DD^{l-1} \xra{i} \partial \DD^l \xra{a} \DD^{k'}$ with the standard creation morphism of Example~\ref{stand-cr} -- this composite is again an active morphism.
By induction, this composite morphism $\DD^{l-1} \xra{ai}\DD^{k'}$ is in the image of $\btheta_n^{\op}$.
There is a unique such morphism in $\btheta_n^{\op}$, which is $\DD^{l-1}\xra{i} \DD^{k'}$, and it does not factor through $\partial \DD^{l-1}$.  
This is a contradiction.

The conclusion of the previous paragraph is that $D_0 \simeq \DD^l$.
With this, the equivalences~(\ref{coends-t}) and~(\ref{coends-d}) reduce us, by induction, to showing that each of the two maps
\[
\btheta_n^{\sf cls,\op}(c_{k},c_k) \xra{~\lag-\rag~} \cDisk_n^{\vfr,\sf cls}\bigl(\DD^k,\DD^{k}\bigr)
\qquad\text{ and }\qquad
\btheta_n^{\sf act,\op}(c_{k'},c_{k'}) \xra{~\lag-\rag~} \cDisk_n^{\vfr,\sf act}\bigl(\DD^{k'},\DD^{k'}\bigr)
\]
is an equivalence of spaces.  
However, every closed endomorphism $\DD^k \to \DD^k$ is an isomorphism in $\cDisk^{\vfr}_n$.
We are thus reduced to showing the righthand equivalence which is just for spaces of active morphisms between cells of the same dimension.  
We will now explain the solid diagram among mapping spaces
\[
\xymatrix{
\btheta_n^{\sf act,\op}(c_k,c_k)\ar[rr]  \ar[d]_-{\lag-\rag}
&&
\underset{{c_l\xra{\cls}c_k}_{,l<k}} {\sf lim} \btheta_n(c_l, c_k) \ar[d]^-{\lag-\rag}_-{\simeq}
\\
\cDisk_n^{\vfr, \sf act}(\DD^k,\DD^k) \ar[r]  \ar@{-->}[d]^-{(b)}   \ar@{-->}[rd]^-{(a)}
&
\cDisk^{\vfr}_n(\DD^k, \partial \DD^k)  \ar[r]^-{\simeq}_-{(1)}  \ar[dr]_-{\simeq}^-{(2)}
&
\underset{{\DD^k\xra{\cls}\DD^l}_{,l<k}} {\sf lim} \cDisk^{\vfr}_n(\DD^k, \DD^l)
\\
\Aut_{\cDisk_n^{\vfr}}(\partial \DD^k)  \ar[r]
&
\cDisk_n^{\vfr, \sf act}(\partial \DD^k, \partial \DD^k)    \ar[r]
&
\cDisk^{\vfr}_n(\partial \DD^k, \partial \DD^k).
}
\]
The middle horizontal equivalence~(1) is from the standard hemispherical closed cover of the hemispherical $(k-1)$-sphere $\partial \DD^k$ by hemispherical disks, each of dimension smaller than $k$.
The upper left vertical arrow is induced by the cellular realization functor; it is an equivalence by induction.
The lower diagonal equivalence~(2) is from the closed-active factorization system on $\cDisk^{\vfr}_n$.
The bottom horizontal maps are inclusions of components.
We will now argue the existence of the dashed arrow (a)

Consider a constructible bundle $E\to \Delta^1$ classified by an active endomorphism in $\Bun$ of $\DD^k$.
Consider the fiberwise boundary $\partial E \subset E$.
The projection $\partial E\to \Delta^1$ is a constructible bundle classifying an endomorphism in $\Bun$ of $\partial \DD^k$.  
For the morphism in $\Bun$ classifying $E\to \Delta^1$ to be active, it is equivalent to the condition that the continuous map $E\to \Delta^1$ is a fiber bundle of underlying topological spaces.  
It follows that the constructible bundle $\partial E\to \Delta^1$ too is a fiber bundle of underlying topological spaces.
Thereafter, it is classified by an \emph{active} endomorphism of $\partial \DD^k$.  
With this consideration, we conclude the factorization which is~(a).

Continuing with the situation of the previous paragraph, if $\partial E\to \Delta^1$ is a fiber bundle of stratified spaces (not just underlying topological spaces), then necessarily $E\to \Delta^1$ too is a conically smooth fiber bundle of stratified spaces.
We conclude that the inclusion of components
\[
\Aut_{\cDisk_n^{\vfr}}(\partial \DD^k)
\xra{~\simeq~}
\cDisk_n^{\vfr,\sf act}(\DD^k,\DD^k)
\]
is an equivalence of spaces.
This provides the factorization which is~(b).  
Because $\btheta_n^{\sf aut}(c_k,c_k) \simeq \ast$ is terminal, we have reduced the problem of showing~(\ref{theta-map}) is an equivalence to that of showing the space of automorphisms
\[
\Aut_{\cDisk_n^{\vfr}}(\partial \DD^k)~\simeq~\ast
\]
is terminal.
This is the statement of Corollary~\ref{no-autos}.
\end{proof}

The following two lemmas were used in the above proof of Theorem~\ref{theta-ff}, the fully faithfulness of cellular realization.

\begin{lemma}\label{cls-mor-sections}
For each inert morphism $T\to T'$ in $\btheta_n^{\op}$ there is a natural section $T'\to T$.  
Furthermore, for each solid pullback diagram in $\btheta_n^{\op}$ among inert morphisms
\[
\xymatrix{
T        \ar[d]      \ar[r]
&
T''\ar[d]\ar@{-->}@/_.7pc/[l]
\\
\ar@{-->}@/^.7pc/[u]T'  \ar[r]&T_0\ar@{-->}@/^.7pc/[l]\ar@{-->}@/_.7pc/[u]
}
\]
the natural sections determine a pushout diagram.  

\end{lemma}

\begin{proof}
We first consider the $n=1$ case where $\btheta_1^{\op} =\bdelta^{\op}$.
Let $\rho^{\op}$ be an inert morphism in $\bdelta^{\op}$ from $[q]$ to $[p]$; it is the data of a map of linearly ordered sets $[p]\xra{\rho}[q]$.  
We define the section $\sigma^{\op}\colon [p]\to [q]$ to the morphism $\rho^{\op}$ in $\bdelta^{\op}$ as the data of the map of linearly ordered sets $[q]\xra{\sigma}[p]$ given as follows.
Declare $\sigma(i)= \rho^{-1}(i)$ whenever $i$ lies in the image of $\rho$; declare $\sigma(i) = 0$ whenever $i<\rho(j)$ for all $j\in [p]$; declare $\sigma(i)=q$ whenever $i>\rho(j)$ for all $j\in [p]$.  
Using that $\rho^{\op}$ is inert, these assignments are well-defined and respect linear orders.
The construction of this section is functorial in the following sense:
For $T_0\to T'\to T$ a composition of inert morphisms in $\bdelta^{\op}$, the section of the composite $T_0 \hookrightarrow T$ is the composite of the sections.

It remains to show that if $T = T' \underset{T_0}\times T''$ is a pullback in $\bdelta^{\op}$ among inert morphisms then $T$ is the pushout of the diagram formed by the sections. 
Because the square is comprised of inert morphisms, the pullback in question indeed exists.
For the same reason, as a linearly ordered set, the underlying set of this pullback is the pushout of the underlying sets of the constituents of the pullback.  
The desired pushout follows by inspecting the constructions of the sections to the inert morphisms.

If $n=0$, the result is trivially true.
We proceed by induction on $n>0$.  
Assume the result for $\btheta_{n-1}^{\op}$. 
Consider an inert morphism $([p],(T_i)_{0<i\leq p}) \to ([p'],(T_{i'})_{0<i'\leq p'})$ in $\btheta_n^{\op}$.  
By definition, this is the data of an inert morphism $[p]\xra{\rho^{\op}} [p']$ in $\bdelta^{\op}$ together with, for each $0<i'\in p'$, an inert morphism $T_{\rho(i')} \to T'_{i'}$ in $\btheta_{n-1}^{\op}$.  
By induction, there is a natural section $[p']\to [p]$ as well as a natural section $T'_{i'} \ra T_{\rho(i')}$ for each $0<i'\leq p'$.
Together, these define a section
$([p'],(T_{i'})_{0<i'\leq p'}) \to ([p],(T_i)_{0<i\leq p})$ in $\btheta_n^{\op}$.
Because it is so for $\bdelta^{\op}$ and $\btheta_{n-1}^{\op}$ by induction, the construction of this section associated to the given inert morphism is functorial: the composition of the sections is the section of the composite.  
It remains to check that this process converts an inert colimit diagram in $\btheta_n$ to a limit diagram. 
For this, we use that the wreath construction is a right adjoint; specifically, we can recognize a diagram in $\bdelta^{\op}\wr \cC$ as a limit if the corresponding diagram in $\bdelta^{\op}$ is a limit and the corresponding diagram in $\cC$ is a limit.
\end{proof}

\begin{lemma}
The cellular realization $\lag-\rag: \btheta_n^{\op} \ra \cDisk_n^{\vfr}$ preserves the dashed colimits of Lemma~\ref{cls-mor-sections}.
\end{lemma}
\begin{proof}
Each section $\lag T'\rag \ra \lag T\rag$ is a creation, which is to say that it lies in the image of the limit preserving monomorphism $\Strat^{\cbl,\op}\ra \Bun$. 
As a diagram in $\Strat^{\cbl}$, it is a pullback.  
\end{proof}

\section{Factorization homology}
We now give a definition of factorization homology.
We do this in two conceptual steps.
The first step can be interpreted as extending sheaves from a basis for a topology.
The second step can be interpreted as integration.
We leave to later works a thorough examination of the properties of factorization homology.

\subsection{Higher categories}

We recall a slight modification of Rezk's definition of $(\infty,n)$-categories.  
\begin{definition}[After~\cite{rezk-n}]\label{def.n-cat}
The $\infty$-category $\Cat_{(\infty,n)}$ of $(\infty,n)$-categories is equipped with a functor $\btheta_n \to \Cat_{(\infty,n)}$ and is initial among all such for which
\begin{itemize}
\item $\Cat_{(\infty,n)}$ is presentable;

\item {\bf Segal:} The functor $\btheta_n \to \Cat_{(\infty,n)}$ carries Segal covering diagrams to colimit diagrams;

\item {\bf Univalent:} The functor $\btheta_n \to \Cat_{(\infty,n)}$ carries univalence diagrams to colimit diagrams.  

\end{itemize}

\end{definition}

\noindent
Necessarily, the given functor $\btheta_n\to \Cat_{(\infty,n)}$ is fully faithful.  
As such, the restricted Yoneda functor gives a presentation
\[
\Cat_{(\infty,n)}   \hookrightarrow   \Psh(\btheta_n)
\]
as the full $\oo$-subcategory consisting of those functors $\cC\colon \btheta_n^{\op} \to \Spaces$ that carry the opposites of both Segal covering diagrams and univalence diagrams to limit diagrams.
We will sometimes refer to such presheaves as \emph{univalent Segal $\btheta_n$-spaces}.

\begin{example}
For each $n>0$ consider the ordinary category ${n\sf Cat}$ of ordinary categories enriched over the ordinary Cartesian category ${(n-1)\sf Cat}$, where ${ 0\sf Cat}:={\sf Set}$.  
In general, for $\sV$ an ordinary category that admits finite products, there is a functor $\bdelta^{\op} \wr \sV^{\op} \to \Cat(\sV)^{\op}$ to $\sV$-enriched categories; this is constructed in~\cite{berger}.  
By induction on $n$, there results a functor $\btheta_n \hookrightarrow {n \sf Cat}$; in~\cite{berger} this functor is shown to be fully faithful.  
The left Kan extension of the defining functor $\btheta_n\to \Cat_{(\infty,n)}$ along this functor defines a functor between $\infty$-categories
\[
{n\sf Cat} \longrightarrow \Cat_{(\infty,n)}~.
\]
In particular, each strict $n$-category determines an $(\infty,n)$-category.  

\end{example}

\begin{remark}
The work of Barwick and Schommer-Pries~\cite{clark-chris} establishes an axiomatic approach to $(\infty,n)$-categories, in the background of quasi-categories. 
They show that each candidate quasi-category of $(\infty,n)$-categories is equivalent to Rezk's Definition~\ref{def.n-cat} above.  
 \end{remark}

\begin{example}\label{assoc}
For $n=1$, the standard functor $\bdelta^{\op} \to {\sf Assoc}$ over $\Fin_\ast$ to the associative $\infty$-operad determines a fully faithful functor 
\[
\fB \colon \Alg_{\sf Assoc}(\Spaces^\times) ~\hookrightarrow~\Cat_{(\infty,1)}^{\ast/}~\subset~\Psh(\bdelta)^{\ast/}
\]
from the $\infty$-category of associative algebras in the Cartesian symmetric monoidal $\infty$-category of spaces to pointed $(\infty,1)$-categories.
In this way, associative monoids in spaces give examples of $(\infty,1)$-categories. (The fact that this functor $\fB$ factors as indicated is because $\bdelta^{\op} \to {\sf Assoc}$ is an \emph{approximation}, as developed in~\S2.3.3 of~\cite{HA}.)
\end{example}

\begin{example}\label{n-assoc}
For $n>0$, the previous Example~\ref{assoc} inductively defines a functor 
\[
\btheta_n^{\op}\simeq \bdelta^{\op}\wr \btheta_{n-1}^{\op} \longrightarrow \cE_1\wr \cE_{n-1} \xra{~\times~} \cE_n
\]
over $\Fin_\ast$ to the $\infty$-operad of little $n$-disks.  
This functor determines a functor
\[
\fB^n\colon \Alg_{\cE_n}(\Spaces^\times) \longrightarrow \Cat_{(\infty,n)}^{\ast/}~\subset~\Psh(\btheta_n)^{\ast/}
\]
from $\cE_n$-algebras in the Cartesian symmetric monoidal $\infty$-category of spaces to pointed $(\infty,n)$-categories.  
In this way, $\cE_n$-algebras in spaces give examples of $(\infty,n)$-categories.  
In particular, $\cE_\infty$-algebras in spaces, such as commutative monoids and infinite-loop spaces, give examples of $(\infty,n)$-categories.  (The fact that this functor $\fB^n$ factors as indicated is because $\bdelta^{\op} \to {\sf Assoc}$ is an \emph{approximation}, because the wreath construction respects inert-coCartesian morphisms (see~\S2.4.4 of~\cite{HA}),
and because the map from the wreath to $\cE_n$ respects inert-coCartesian morphisms (see~\S5.1.2 of~\cite{HA}).)
\end{example}

\begin{remark}\label{commutative}
Example~\ref{n-assoc} in particular gives that an $\cE_\infty$-space determines an $(\infty,n)$-category for each $n\geq 0$.
This can be achieved more directly through the foundational work of Segal in~\cite{segal-gamma}.
Namely, in that work Segal shows how an infinite-loop space $C$ determines a functor $C\colon \Fin_\ast \to \Spaces$ that satisfies a \emph{reduced} condition, meaning $C(\ast)\simeq \ast$ is terminal, and what has been termed the \emph{Segal} condition,
meaning $C$ carries each pullback diagram in $\Fin_\ast$ among \emph{inert} morphisms to pullback squares among spaces.  (In this situation of based finite sets, a based map $I_+\to J_+$ is \emph{inert} if the restriction $I_{|J} \to J$ is an isomorphism.)

\end{remark}

\begin{example}\label{space-cat}
For each $0\leq k \leq n$, right Kan extension along the functor $\iota_k \colon \btheta_k \hookrightarrow \btheta_n$ of Observation~\ref{theta-stand-bump} defines a functor between $\infty$-categories
\[
(\iota_k)_\ast\colon \Cat_{(\infty,k)} ~\hookrightarrow~ \Cat_{(\infty,n)}
\]
which is fully faithful.  
Thus, $(\infty,k)$-categories are examples of $(\infty,n)$-categories.
In particular, there is a fully faithful functor from spaces, regarded as $\infty$-groupoids:
\[
\Spaces \simeq \Cat_{(\infty,0)} ~\hookrightarrow~\Cat_{(\infty,n)}~. 
\]

\end{example}

\subsection{Labeling systems from higher categories}
Toward the construction of factorization homology from higher categories, we explain here how an $(\infty,n)$-category determines a labeling system on a sufficiently finely stratified vari-framed $n$-manifold.

The next result makes use of the cellular realization functor $\lag-\rag\colon \btheta_n^{\op} \hookrightarrow \cDisk_n^{\vfr}$ of Definition~\ref{def.cellular}, which for $n<3$ Theorem~\ref{theta-ff} verifies is fully faithful.
We postpone the proof of this result to the end of this section.  
\begin{lemma}\label{fC}
The restricted Yoneda functor $(\cDisk_n^{\vfr})^{\op} \to \Psh(\btheta_n)$ takes values in $(\infty,n)$-categories.  

\end{lemma}

\begin{notation}\label{def.fC}
After Lemma~\ref{fC}, we denote the factorized restricted Yoneda functor as
\[
\fC\colon (\cDisk^{\vfr}_n)^{\op} \longrightarrow \Cat_{(\infty,n)}~.
\]

\end{notation}

\begin{remark}\label{fC-explain}
The construction of $\fC$, as a functor among $\infty$-categories, embodies many of the choices and constructions in this article.
To given some intuition for the value $\fC(M)$ on a compact vari-framed disk-stratified $n$-manifold $M$,
there is a non-identity $i$-morphism of this $(\infty,n)$-category for each connected $i$-dimensional stratum of $M$.  
Figure~\ref{fig.88} depicts the $2$-category which is the value of $\fC$ on the apparent compact vari-framed disk-stratified $2$-manifold.  
Namely, the set of objects in this $2$-category is the set of 0-strata;
the underlying 1-category, with the specified objects, is freely generated by the set of 1-strata, with source/target as indicated by the arrows on each edge;
the 2-category, with specified underlying 1-category, is freely generated by the set of 2-strata, with source/target as indicated by the double-arrows in each region.
Note that the source/target of these 2-morphisms are not, simply, generating 1-morphisms, but are composites there among.  
\end{remark}

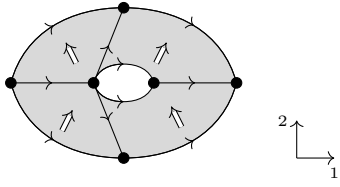
\begin{figure}
\begin{tikzpicture}
\tikzstyle{every node}=[font=\tiny]

\draw[fill=lightgray] (0,1) [out=75, in=180] to node [opacity=0] (TOP) {} (1.5,2) [out=0, in=105] to (3,1) [out=-105, in=0] to node [opacity=0] (BOTTOM) {} (1.5,0) [out=-180, in=-75] to (0,1);

\draw[fill=white] (1.1,1) [out=75, in=180] to (1.5,1.25) [out=0,in=105] to (1.9,1) [out=-105,in=0] to (1.5, 0.75) [out=-180,in=-75] to (1.1,1);

\draw[fill] (0,1) circle [radius=0.07];
\draw[fill] (3,1) circle [radius=0.07];

\draw[fill] (1.1,1) circle [radius=0.07];
\draw[fill] (1.9,1) circle [radius=0.07];

\draw[fill] (1.5,2) circle [radius=0.07];
\draw[fill] (1.5,0) circle [radius=0.07];

\draw[midarrow=0.5] (0,1) [out=75, in=180] to (1.5,2);
\draw[midarrow=0.5] (0,1) [out=-75,in=-180] to (1.5,0);
\draw[midarrow=0.5] (1.5,2) [out=0, in=105] to (3,1);
\draw[midarrow=0.5] (1.5,0) [out=0, in=-105] to (3,1);

\draw[->] (1.1,1) [out=75, in=180] to (1.5,1.25);
\draw[->] (1.1,1) [out=-75,in=-180] to (1.5,0.75);


\draw[->-] (1.1,1) to (1.5,2);
\draw[->-] (1.1,1) to (1.5,0);
\draw[->-] (0,1) to (1.1,1);
\draw[->-] (1.9,1) to (3,1);

\draw[-{Implies},double distance=1.5pt,shorten >=2pt,shorten <=2pt] (0.9, 1.2) to (0.7, 1.6);
\draw[-{Implies},double distance=1.5pt,shorten >=2pt,shorten <=2pt] (1.9, 1.2) to (2.1, 1.6);

\draw[-{Implies},double distance=1.5pt,shorten >=2pt,shorten <=2pt] (0.65, 0.3) to (0.85, 0.7);

\draw[-{Implies},double distance=1.5pt,shorten >=2pt,shorten <=2pt] (2.3, 0.35) to (2.1, 0.75);

\draw[->] (3.8,0) to (4.3,0) node[below]{1};
\draw[->] (3.8,0) to (3.8,0.5) node[left]{2};

\end{tikzpicture}
\caption{The $2$-category which is the value of $\fC$ on an object of $\cDisk^{\vfr}_2$.}
\label{fig.88}
\end{figure}

Consider the sequence of functors
\begin{equation}\label{RKE}
\Cat_{(\infty,n)} \longrightarrow \Psh(\btheta_n) \longrightarrow \Fun(\cDisk^{\vfr}_n,\Spaces)
\end{equation}
in which the second is given by right Kan extension.
Explicitly, for $\cC$ an $(\infty,n)$-category, this right Kan extension evaluates on a compact vari-framed disk-stratified $n$-manifold $M$ as the space
\[
\Cat_{(\infty,n)}\bigl(\fC(M),\cC\bigr)
\]
of functors between $(\infty,n)$-categories.

\begin{remark}
We think of the space $\Cat_{(\infty,n)}\bigl(\fC(M),\cC\bigr)$ appearing in the above expression as the space of $\cC$-labeling systems on the vari-framed disk-stratified $n$-manifold $M$.  
In other words, the $(\infty,n)$-category corepresents $M$-labeling systems.

\end{remark}

\begin{cor}\label{RKE-ff}
\footnote{
This result is incorrect without the dimension bound; see the erratum in~\S\ref{sec.erratum}.
}
For $n<3$, the composite functor~(\ref{RKE})
\[
\Cat_{(\infty,n)} ~\hookrightarrow~ \Fun(\cDisk^{\vfr}_n,\Spaces)~,\qquad \cC\mapsto \Bigl(M \mapsto \Cat_{(\infty,n)}\bigl(\fC(M),\cC\bigr)\Bigr)
\]
is fully faithful.

\end{cor}

\begin{proof}
The functor is given as a composition of two functors.  
The first is fully faithful, from the defining universal property of the $\infty$-category $\Cat_{(\infty,n)}$, as discussed above.
That the second functor is fully faithful is an immediate because the cellular realization $\lag-\rag\colon \btheta_n^{\op} \to \cDisk_n^{\vfr}$ is fully faithful for $n<3$ (Theorem~\ref{theta-ff}).
\end{proof}

\begin{proof}[Proof of Lemma~\ref{fC}]
We will consider the restricted Yoneda functor $\fC\colon (\cMfd_n^{\vfr})^{\op} \to \Psh(\btheta_n^{\op})$ which enlarges the restricted Yoneda functor of Notation~\ref{def.fC}. 
Lemma~\ref{theta-preserves} directly states that Segal covering diagrams are carried by the cellular realization functor to purely closed covers.  
Because purely closed covers are in particular limit diagrams in $\cDisk_n^{\vfr}$, each presheaf $\fC(M)$ on $\btheta_n$ carries the opposites of Segal covering diagrams to limit diagrams of spaces.

It remains to verify that each value $\fC(M)$ carries the opposites of univalence diagrams to limit diagrams of spaces.  
For this, we use the criterion of Lemma~\ref{univ-criterion}, which reduces us to showing that the only $k$-idempotents of the Segal $\btheta_n$-space $\fC(M)$ are identity $k$-morphisms.

Consider the unique (up to equivalence) vari-framed $1$-manifold $D$ whose underlying stratified space is the pushout $\ast \underset{\partial \DD^1}\amalg \DD^1$ in $\strat$.  
Up to equivalence, there is a unique refinement morphism $D\to S^1$.
As such, for each $p\geq 0$, there is a fiber bundle $\w{D}_p\to D$ among stratified spaces which refines the $p$-sheeted covering map $S^1 \xra{z\mapsto z^p}S^1$, as well as a refinement $\w{D}_p \to D$, of which there are $p$.  
There results a composite morphism $F_p\colon D \xra{\sf cr} \w{D}_p \xra{\sf ref} D$ in $\cMfd_1^{\vfr}$.  
Thereafter, for each $0<k\leq n$, there is a morphism $F_p^k\colon D^k \to D^k$ in $\cMfd_n^{\vfr}$ among $(k-1)$-fold framed suspensions, $D^k:=\sS^{\fr,\circ (k-1)}(D)$.  

From the definition of the stratified space $D$ as a pushout in $\strat$, there is a canonical map of stratified spaces $\DD^1 \xra{u} D$.  
This map is a constructible bundle, and therefore determines the morphism $D\to \DD^1$ in $\cBun$ which classifies the reversed mapping cylinder of the map of stratified spaces: ${\sf Cylr}(u) \to \Delta^1$.  
This cylinder is equipped with a fiberwise vari-framing: $\exit({\sf Cylr}(u)) \to \vfr$ over $\Exit$.  
We arrive at a creation morphism $D \to \DD^1$ in $\cMfd_1^{\vfr}$.  
Taking iterated framed suspensions determines a creation morphism
\[
D^k:=\sS^{\fr,\circ(k-1)}(D) \longrightarrow \sS^{\fr,\circ(n-1)}(\DD^1) \cong \DD^k
\]
in $\cMfd_n^{\vfr}$ for each $0<k\leq n$.  
Applying $\fC$ determines a map $c_k \to \fC(D^k)$ of presheaves on $\btheta_n$.
By inspection, this map factors through an equivalence from a quotient of a $k$-cell
\[
c_k \simeq c_{k-1} \wr c_1 \longrightarrow c_{k-1} \wr (c_1/\partial c_1)~\simeq~ \fC(D^k)
\]
of presheaves on $\btheta_n^{\op}$.
As such the Segal space $\fC(D^k)$ is univalent, and therefore presents an $(\infty,n)$-category;
this $(\infty,n)$-category $\fC(D^k)$ corepresents a $k$-endomorphism of a $(k-1)$-morphism of an $(\infty,n)$-category.  
As such, there is a preferred functor $\fC(D^k) \xra{e^k} {\sf Idem}^k$ between $(\infty,n)$-categories, together with, for each $p\geq 0$, an identification of functors $e^k\circ \fC(F_p^k) \simeq e^k$.

Consequently, to prove that each functor ${\sf Idem}^k \to \fC(M)$ factors through $c_{k-1} \to \fC(M)$ it is enough to prove that each morphism $M \xra{\epsilon^k} D^k$ in $\cMfd_n^{\vfr}$ factors through the standard creation map $\DD^{k-1}\to D^k$ whenever there is an identification of morphisms $\epsilon^k\circ F_p^k \simeq \epsilon^k$ in $\cMfld_n^{\vfr}$.  
This statement is implied by its version without vari-framings.  
Namely, we have reduced to proving the following assertion.
\begin{itemize}
\item[~]
For each commutative diagram in $\cBun$
\[
\xymatrix{
M  \ar[rr]^-{\epsilon^k}  \ar[dr]_-{\epsilon^k}
&&
D^k    \ar[dl]^-{F_2^k}
\\
&
D^k  
&
}
\]
there is a factorization $\epsilon^k\colon M \xra{\cls} M_{<k} \to D^k$ through the unit of the adjunction of Lemma~\ref{truncate} which is given by forgetting strata of dimension at least $k$.

\end{itemize}
Using the closed-active factorization system on $\cBun$ (Lemma~\ref{closed-active}), we can reduce to the case that the morphism $\epsilon^k$ is active.  In this case the assertion becomes the statement that the topological dimension of $M$ is less than $k$.

Such a commutative diagram in $\cBun$ is a constructible bundle $X\to \Delta^2$ together with identifications of its restrictions along $\Delta^S\subset \Delta^2$ for various non-empty linearly ordered subsets $\emptyset\neq S\subset \{0<1<2\}$.  
Consider the link system among stratified spaces 
\[
M \underset{\sf p.cbl.surj}{\xla{~\pi~}}\Link_{M}(X)\underset{\sf ref}{\xra{~\gamma~}} X_{|\Delta^{\{1<2\}}}~;
\]
here, the leftward map is proper and constructible and surjective while the rightward map is a refinement.  
The naturality of links grants a surjective proper constructible bundle $\Link_{M}(X) \to \Link_{\Delta^{\{0\}}}(\Delta^2) \cong \Delta^1$.  
We obtain a surjective proper constructible bundle $\Link_{M}(X) \to M\times \Delta^1$.  
Because the topological dimension of the stratified space $D^k$ is bounded above by $k$, then the topological dimension of the $\Link_{M}(X)$ is bounded above by $(k+1)$, and therefore the topological dimension of $M$ is bounded above by $k$. 
 
Let $M_k \subset M$ be the open subspace consisting of the $k$-dimensional strata; we must explain why $M_k$ is empty.  
For dimension reasons, the surjective proper constructible bundle $\Link_{M}(X)_{|M_k\times \Delta^1} \to M_k\times \Delta^1$ has finite fibers which are empty if and only if $M_k$ is empty.  
By assumption, there is an identification of stratified spaces $\Link_{M}(X)_{|M_k\times \Delta^{\{0\}}} \cong\Link_{M}(X)_{|M_k\times \Delta^{\{1\}}}$ over $M_k$.  
On the other hand, from the construction of the morphism $F_2^k$ in $\cBun$, the cardinality of the fiber $\Link_{M}(X)_{|\{x\}\times \Delta^{\{0\}}}$ is twice that of the fiber $\Link_{M}(X)_{|\{x\}\times \Delta^{\{1\}}}$ for each $x\in M_k$.
We conclude that $M_k=\emptyset$, as desired.  
\end{proof}

\subsection{Factorization homology}\label{sec.fact.def}
Recall the functor~(\ref{RKE}), which assigns to each $(\infty,n)$-category a copresheaf on compact disk-stratified vari-framed $n$-manifolds.
By definition, such stratified disks lie fully among all such stratified manifolds.
We can now define factorization homology as the left Kan extension from this full $\infty$-subcategory to copresheaves of manifolds.

\begin{definition}\label{def.fact.homology}
Factorization homology is the composite
\[
\xymatrix{
\displaystyle\int:\Cat_{(\oo,n)}\ar@{->}[rr]&&\Fun\bigl(\cDisk_n^{\vfr},\spaces\bigr)\ar@{^{(}->}[rr]&&\Fun\bigl(\cMfd_n^{\vfr},\spaces\bigr)
}
\]
of right Kan extension along $\btheta_n^{\op} \hookrightarrow \cDisk^{\vfr}_n$ followed by left Kan extension along $\cDisk_n^{\vfr}\hookrightarrow \cMfd_n^{\vfr}$.
\end{definition}

Equivalently, given an $(\oo,n)$-category $\cC$, factorization homology is defined by the following two Kan extensions:
\[\xymatrix{
\btheta_n^{\op}\ar@{->}[d]\ar[rr]^\cC&&\spaces\\
\cDisk_n^{\vfr}\ar@{-->}[urr]_\cC\ar@{_{(}->}[d]\\
\cMfd_n^{\vfr}\ar@/_.5pc/@{-->}[uurr]_{\displaystyle\int\cC}\\}\]
This left Kan extension is equivalent to the classifying space of the unstraightening construction $\cC_M$ of the composite
\[
\cDisk^{\vfr}_{n/M}\longrightarrow \cDisk_n^{\vfr} \overset{\cC}\longrightarrow \spaces~.
\]
This unstraightening construction can be thought of the $\oo$-category of $\cC$-labeled disk-stratifications over $M$. 
By that token the factorization homology over a compact vari-framed $n$-manifold $M$,
\[
\int_M\cC
~\simeq~ 
\sB\bigl(\cC_M\bigr)
~\simeq~
\colim\Bigl(  \cDisk^{\vfr}_{n/M}\longrightarrow \cDisk_n^{\vfr} \xra{~\Map(\fC(-),\cC)~} \Spaces \Bigr)~,
\]
is the classifying space of $\cC$-labeled disk-stratifications over $M$, as well as the colimit over the $\infty$-overcategory $\Disk^{\vfr}_{n/M}$.

\begin{remark}
For $M$ a compact framed smooth $n$-manifold, and for $\cC$ an $(\infty,n)$-category, heuristically the factorization homology
$\int_M \cC$ is the integral, or average, of sufficiently fine $\cC$-labeled vari-framed refinements of $M$.  
This description as a left Kan extension makes this heuristic precise as well as functorial in $M$ up to coherent homotopy.

\end{remark}

Examination of this factorization homology, and its development in the enriched case, will take place in other works.  Here are some easy values of this factorization homology.  
\begin{example} 
For each $0\leq i \leq n$, and for each $(\infty,n)$-category $\cC$, the value $\int_{\DD^i}\cC$ is the space of $i$-morphisms of $\cC$.  In particular, $\int_{\DD^0}\cC \simeq \cC^\sim$ is the space of objects.  

\end{example}

\begin{remark}
Evidently, the domain of factorization homology can be extended to arbitrary presheaves on $\btheta_n$ via the same prescription as Definition~\ref{def.fact.homology}:
\[
\displaystyle\int
\colon
{\sf PShv}(\btheta_n)
\xra{~\sf RKan~}
\Fun\bigl(\cDisk_n^{\vfr},\spaces\bigr)
\xra{~\sf LKan~}
\Fun\bigl(\cMfd_n^{\vfr},\spaces\bigr)~.
\]

\end{remark}

\section{Appendix: some $\infty$-category theory}
We go over some notions within $\infty$-category theory.

\subsection{Monomorphisms}\label{sec.monos}
In this section we characterize monomorphisms among $\infty$-categories.
We first recall the following standard definition.

\begin{definition}[Mono/Epi]\label{def.epi}
A map $X\to Y$ between spaces is a \emph{monomorphism} if, when regarded as a functor between $\infty$-groupoids, it is fully faithful.  
A morphism $f\colon [1]\to \cX$ in an $\infty$-category is a \emph{monomorphism} if, for each object $\ast \xra{x} \cX$, the composite functor
\[
[1]\xra{~f~}\cX \longrightarrow \Psh(\cX) \xra{~x^\ast~} \Psh(\ast) \simeq \Spaces
\]
is a monomorphism between spaces.  
A morphism $f\colon [1]\to \cX$ in an $\infty$-category is an \emph{epimorphism} if $[1] \simeq  [1]^{\op} \xra{f^{\op}}\cX^{\op}$ is a monomorphism.  

\end{definition}

\begin{remark}
Alternatively, a map between spaces $f\colon X\to Y$ is a monomorphism if and only if it is an inclusion of path components.
This is to say that, for any choice of base point $x\in X$, the homomorphism between homotopy groups $\pi_q(X;x) \to \pi_q(Y;f(x))$ is an isomorphism for each $q>0$.  
\end{remark}

\begin{example}
Let $X$ be a space for which its suspension
\[
\sS(X):=\ast \underset{X} \amalg \ast~\simeq~\ast
\]
is contractible.
Because epimorphisms are preserved by co-base change, we conclude that the unique map $X\to \ast$ is an epimorphism.

\end{example}

Because it is the case for monomorphisms among spaces, both monomorphisms and epimorphisms satisfy a certain two-out-of-three property:
\begin{observation}
Let $[2]\to \cX$ be a functor between $\infty$-categories.
For each $0\leq i<j\leq 2$, denote the restriction $f_{ij}\colon \{i<j\} \to \cX$.  
\begin{itemize}
\item If $f_{01}$ and $f_{12}$ are monomorphisms, then so is $f_{02}$.
\item If $f_{01}$ and $f_{12}$ are epimorphisms, then so is $f_{02}$.
\item If $f_{12}$ and $f_{02}$ are monomorphisms, then so is $f_{01}$.  
\item If $f_{01}$ and $f_{02}$ are epimorphisms, then so is $f_{12}$.
\end{itemize}
  
\end{observation}

Monomorphisms, respectively epimorphisms, are closed under the formation of limits, respectively colimits, in the following sense.  
\begin{observation}\label{limit-closed}
For $\cX$ an $\infty$-category, the monomorphisms and the epimorphisms form full $\infty$-subcategories $\Ar^{\sf mono}(\cX)\subset \Ar(\cX)\supset \Ar^{\sf epi}(\cX)$ of the $\infty$-category of arrows, $\Ar(\cX):=\Fun([1],\cX)$.  
Furthermore, each diagram among $\infty$-categories
\[
\xymatrix{
\cJ  \ar[r]  \ar[d]
&
\Ar^{\sf mono}(\cX)  \ar[d]
\\
\cJ^{\tl}  \ar[r]^-{{\sf lim}}  \ar@{-->}[ur]
&
\Ar(\cX),
}
\]
in which the bottom horizontal arrow is a limit diagram, factors as a limit diagram.
Likewise, each diagram among $\infty$-categories
\[
\xymatrix{
\cJ  \ar[r]  \ar[d]
&
\Ar^{\sf epi}(\cX)  \ar[d]
\\
\cJ^{\tr}  \ar[r]^-{\colim}  \ar@{-->}[ur]
&
\Ar(\cX),
}
\]
in which the bottom horizontal arrow is a colimit diagram, factors as a colimit diagram.

\end{observation}

\begin{lemma}
A functor $F\colon \cC \to \cD$ between $\infty$-categories is a monomorphism if and only if the map between maximal $\infty$-subgroupoids $\cC^\sim \to \cD^\sim$ is a monomorphism, and, for each pair of objects $\partial [1] \xra{x\sqcup y} \cC$, the map between spaces of morphisms 
\[
\cC(x,y) \simeq \Map^{x\sqcup y/}([1],\cC) \xra{~F\circ -~} \Map^{F\circ(x\sqcup y)/}([1],\cD) \simeq \cD(F(x),F(y))
\]
is a monomorphism. 

\end{lemma}
\begin{proof}
Let $F\colon \cC\to \cD$ be a functor between $\infty$-categories.
Restriction along the functor $\partial [1]  \to [1]$ determines the downward maps in the diagram of spaces of functors
\[
\xymatrix{
\Map_{\Cat_{\infty}}([1],\cC) \ar[rr]^-{F\circ-}  \ar[d]
&&
\Map_{\Cat_{\infty}}([1],\cD)  \ar[d]
\\
\Map_{\Cat_{\infty}}(\partial [1],\cC)  \ar[rr]^-{F\circ -}
&&
\Map_{\Cat_{\infty}}(\partial [1],\cD).
}
\]
Suppose $F$ is a monomorphism.
Then the horizontal maps in this diagram are monomorphisms.
It follows that, for each pair of objects $x\sqcup y \colon \partial[1]\to \cC$, the map of fibers $\cC(x,y)\to \cD(F(x),F(y))$ is a monomorphism.  
Also, $F$ being a monomorphism implies the map between spaces $\cC^\sim \simeq \Map_{\Cat_\infty}(\ast,\cC) \xra{F\circ-} \Map_{\Cat_\infty}(\ast,\cD)\simeq \cD^\sim$ is a monomorphism.

Now suppose, that the map of spaces $\cC^\sim \to \cD^\sim$ is a monomorphism and, for each pair of objects $x\sqcup y\colon \partial [1]\to \cC$, that the map of spaces $\cC(x,y)\to \cD\bigl(F(x),F(y)\bigr)$ is a monomorphism.
Let $\cK$ be an $\infty$-category.
We must show the map between spaces of functors
\[
\Map_{\Cat_\infty}(\cK, \cC) \xra{~F\circ-~}\Map_{\Cat_\infty}(\cK, \cD)
\]
is a monomorphism.
First note that the fully faithful functor $\bdelta_{\leq 1} \hookrightarrow \Cat_\infty$ generates $\Cat_\infty$ under colimits.  
That is, for each full $\infty$-subcategory $\bdelta_{\leq 1}\subset \cS\subset \Cat_\infty$ that is closed under the formation of colimits, then the inclusion $\cS\to \Cat_\infty$ is an equivalence.  
Using this, choose a colimit diagram $\cK_\bullet\colon \cJ^{\tr} \to \Cat_\infty$ that carries the cone-point to $\cK$ and each $j\in \cJ$ to $[0]$ or $[1]$.  
Using that monomorphisms are closed under the formation of limits (Observation~\ref{limit-closed}), we conclude that the map of spaces of functors
\[
\Map_{\Cat_\infty}(\cK, \cC)\xra{\simeq}\underset{j\in \cJ}{\sf lim} \Map_{\Cat_\infty}(\cK_j, \cC) \longrightarrow \underset{j\in \cJ}{\sf lim} \Map_{\Cat_\infty}(\cK_j, \cD) \xla{\simeq} \Map_{\Cat_\infty}(\cK, \cD)
\]
is a monomorphism provided it is in the cases that $\cK = \ast$ and $\cK = [1]$.
The case of $\cK = \ast$ is the assumption that $\cC^\sim \to \cD^\sim$ is a monomorphism.
With the above square diagram, the case of $\cK=[1]$ follows from the case of $\cK=\ast$ together with the additional assumption about mapping spaces.  
\end{proof}

\begin{example}
Fully faithful functors among $\infty$-categories are monomorphisms.
\end{example}

\begin{example}
For $\cE \xra{\pi} \cB$ a functor between $\infty$-categories, the collection of $\pi$-Cartesian morphisms determines a monomorphism $\cE^{{\sf Cart}_{/\pi}}\hookrightarrow \cE$: a functor $\cK \to \cE$ factors through $\cE^{{\sf Cart}_{/\pi}}$ whenever each composition $[1]\to \cK \to \cE$ is a $\pi$-Cartesian morphism.  
In the case that $\cB \simeq \ast$ is terminal, we see that the functor $\cE^\sim\hookrightarrow \cE$ from the maximal $\infty$-subgroupoid is a monomorphism.  

\end{example}

\begin{example}
For each monomorphism $\cC\xra{F} \cD$ among $\infty$-categories, and for each $\infty$-category $\cK$, the functor between functor $\infty$-categories
\[
\Fun(\cK,\cC)\xra{~F\circ-~} \Fun(\cK,\cD)
\]
too is a monomorphism.

\end{example}

\subsection{Cospans}\label{sec.cspan}
We record some facts about $\infty$-categories of cospans.
See~\cite{barwick} for a more thorough development.  

We denote the functor from finite sets to posets,
\[
\cP_{\neq \emptyset}(-)\colon \Fin\longrightarrow \poset~,
\]
whose value on a set $S$ is the poset of non-empty subsets of $S$, ordered by inclusion.  
We will use the same notation for its precomposition with the forgetful functor $\bdelta \to \Fin$ given by forgetting linear orders.

Let $\cC$ be an $\infty$-category, and let $\cC^\sim\subset \cL,\cR \subset \cC$ be a pair of $\infty$-subcategories each of which contains the maximal $\infty$-subgroupoid of $\cC$.
The simplicial space $\cSpan(\cC)^{\cL\text{-}\cR}$ is the subfunctor of the composite functor
\[
\bdelta^{\op} \xra{\cP_{\neq \emptyset}(-)} \poset^{\op} \xra{\Map(-,\cC)} \Spaces
\]
consisting of those values $\cP_{\neq \emptyset}(\{0,\dots,p\}) \xra{F} \cC$ that satisfy the following conditions.
\begin{itemize}
\item The functor $F$ carries colimit diagrams to colimit diagrams.

\item The functor $F$ carries inclusions $S\subset T$ to morphisms in $\cL$ whenever ${\sf Min}(S)={\sf Min}(T)$.

\item The functor $F$ carries inclusions $S\subset T$ to morphisms in $\cR$ whenever ${\sf Max}(S)={\sf Max}(T)$.

\end{itemize}
Explicitly, the value $\cSpan(\cC)^{\cL\text{-}\cR}[0]\simeq \cC^\sim$ is the maximal $\infty$-subgroupoid of $\cC$;
the value on $[1]$ is the pullback among spaces
\[
\xymatrix{
\cSpan(\cC)^{\cL\text{-}\cR}[1]  \ar[r]  \ar[d]
&
\Ar(\cL)^\sim   \ar[d]^-{\ev_t}
\\
\Ar(\cR)^\sim  \ar[r]^-{\ev_t}
&
\cC^\sim
}
\]
involving the spaces of arrows in $\cL$ and in $\cR$ and evaluations at their targets.
So a point in $\cSpan(\cC)^{\cL\text{-}\cR}$ is an object of $\cC$, and a $1$-simplex in $\cSpan(\cC)^{\cL\text{-}\cR}$ is a \emph{cospan} in $\cC$ by $\cL$ and $\cR$, by which we mean a diagram in $\cC$
\[
c_-  \xra{~l~} c_0 \xla{~r~} c_+
\]
in which $l$ is a morphism in $\cL$ and $r$ is a morphism in $\cR$.

\begin{criterion}\label{cspan-rep}
Let $\cL \to \cC\la \cR$ be essentially surjective monomorphisms among $\infty$-categories.
Suppose each diagram $c_+ \xla{r} c_0 \xra{l} c_-$ in $\cC$ admits a pushout whenever the morphism $l$ belongs to $\cL$ and $r$ belongs to $\cR$.  
Then the simplicial space $\cSpan(\cC)^{\cL\text{-}\cR}$ presents an $\infty$-category.  

\end{criterion}

We consider an $\infty$-subcategory 
\[
{\sf Pair}~\subset~ \Fun\bigl(\cP_{\neq \emptyset}(\{\pm\}),\Cat_\infty\bigr)~.
\]
The objects are those $\cL \to \cC\la \cR$ for which both functors are essentially surjective monomorphisms and which satisfy the condition of Criterion~\ref{cspan-rep}.
The morphisms are those which preserve the pushouts of Criterion~\ref{cspan-rep}.  
Manifest from its construction, these simplicial spaces of cospans organize as a functor
\begin{equation}\label{cspan}
\cSpan\colon {\sf Pair} \longrightarrow \Cat_\infty~,\qquad (\cC^\sim\subset\cL,\cR\subset \cC)\mapsto \cSpan(\cC)^{\cL\text{-}\cR}~.
\end{equation}  
By inspection, this functor~(\ref{cspan}) preserves finite products.  
The previous criterion thus gives the next observation.
\begin{observation}\label{cspan-ot-rep}
The functor~(\ref{cspan}) lifts as a functor to symmetric monoidal $\infty$-categories
\[
\cSpan\colon \CAlg({\sf Pair}^\times)\longrightarrow \Cat_\infty^{\ot}
\]
from commutative algebras in the Cartesian $\infty$-operad associated to the $\infty$-category ${\sf Pair}$.

\end{observation}

\begin{observation}\label{from-gpd}
For each symmetric monoidal $\infty$-category $\cC$ there is a canonical identification of symmetric monoidal $\infty$-groupoids 
\[
\cC^\sim~ \simeq ~\cSpan(\cC)^{\cC^\sim\text{-}\cC^\sim}~.
\]
In particular, for each symmetric monoidal pair $(\cL\subset \cC\supset \cR)\in \CAlg({\sf Pair}^{\times})$, there is a canonical symmetric monoidal functor
\[
\cC^\sim\longrightarrow \cSpan(\cC)^{\cL\text{-}\cR}
\]
from the maximal symmetric monoidal $\infty$-subgroupoid.

\end{observation}

\subsection{Exponentiability in $\Cat_\infty$}\label{sec.expo}
In this section we verify a couple relevant examples of \emph{exponentiable fibrations} $\cE \to \cB$. See \cite{fibrations} for a more extensive treatment of this subject.

\subsubsection{\bf Basic notions}
Each functor $\cE \xra{\pi} \cB$ between $\infty$-categories determines a functor
\[
\pi_!\colon {\Cat_\infty}_{/\cE} \xra{-\circ \pi~} {\Cat_\infty}_{/\cB}~,\qquad (\cK\to \cE)\mapsto (\cK\to \cE\xra{\pi} \cB)
\]
given by composing with $\pi$.  
This functor preserves colimits.
This functor has a right adjoint
\[
\pi_!\colon {\Cat_\infty}_{/\cE}  ~\rightleftarrows~{\Cat_\infty}_{/\cB} \colon \pi^\ast~,
\]
which we refer to as \emph{base change}, which evaluates as $\pi^\ast \colon (\cK\to \cB)\mapsto (\cE_{|\cK}\to \cE)$ where $\cE_{|\cK}:= \cK\underset{\cB}\times \cE$ is the fiber product.

\begin{definition}\label{def.efib}
A functor $\cE\xra{\pi} \cB$ between $\infty$-categories is an \emph{exponentiable fibration} if the base change functor $\pi^\ast$ is a left adjoint.
In this case, the right adjoint
\[
\pi^\ast \colon {\Cat_\infty}_{/\cB}~ \rightleftarrows~{\Cat_\infty}_{/\cE}\colon \pi_\ast
\]
is the \emph{exponential} functor.  

\end{definition}
\noindent
By adjunction, the value of this exponential functor on $\tau \to \cE$ has the following universal property:
\begin{itemize}
\item[~]
For each functor $\cK \to \cB$, there is a natural identification of the space of functors over $\cB$
\[
\Map_{/\cB}\bigl(\cK,\pi_\ast \tau\bigr)~\simeq~ \Map_{/\cE}\bigl(\cE_{|\cK},\tau\bigr)
\]
with the space of functors over $\cE$ from the pullback.  
\end{itemize}

\begin{observation}\label{exp-op}
A functor $\cE\to \cB$ is an exponentiable fibration if and only if its opposite $\cE^{\op}\to \cB^{\op}$ is.  

\end{observation}

\begin{remark}\label{non-exp}
Not every functor is an exponentiable fibration.  
For instance, base change along the functor $\{0<2\}\to [2]$ carries the pushout diagram among $\infty$-categories over $[2]$
\[
\xymatrix{
\{1\}  \ar[r]  \ar[d]
&
\{1<2\}  \ar[d]
&&
\emptyset  \ar[r]  \ar[d]
&
\{2\}  \ar[d]
\\
\{0<1\}  \ar[r]
&
[2]
&\text{ to the diagram }&
\{0\}  \ar[r]
&
\{0<2\}
}
\]
over $\{0<2\}$ which is \emph{not} a pushout.

\end{remark}

The following result is an $\infty$-categorical version of a result of Giraud's (\cite{giraud}), which is also the main result in~\cite{conduche} of Conduch\'e.
The result articulates a sense in which Remark~\ref{non-exp} demonstrates the only obstruction to exponentiability.  
For convenient latter application, we state one of the assertions in the result in terms of \emph{suspension} of an $\infty$-category:
\begin{itemize}
\item[~] 
For $\cJ$ an $\infty$-category, its \emph{suspension} is the pushout in the diagram among $\infty$-categories
\[
\xymatrix{
\cJ \times \{0<2\}  \ar[r]  \ar[d]
&
\cJ\times[2]  \ar[d]
\\
\{0<2\}  \ar[r]
&
\cJ^{\tl\tr}.
}
\]
\end{itemize}
This construction is evidently functorial in $\cJ$.
Notice the evident fully faithful functors $\cJ^{\tl}  \hookrightarrow \cJ^{\tl\tr}  \hookleftarrow \cJ^{\tr}$ from the cones.
\begin{lemma}\label{conduche}
The following conditions on a functor $\cE \xra{\pi} \cB$ between $\infty$-categories are equivalent.
\begin{enumerate}

\item The functor $\pi$ is an exponentiable fibration. 

\item The base change functor $\pi^\ast \colon {\Cat_\infty}_{/\cB} \to {\Cat_\infty}_{/\cE}$ preserves colimits.  

\item For each functor $[2]\to \cB$, the diagram among pullbacks
\[
\xymatrix{
\cE_{|\{1\}}  \ar[r]  \ar[d]
&
\cE_{|\{1<2\}}   \ar[d]
\\
\cE_{|\{0<1\}}   \ar[r]
&
\cE_{|[2]} 
}
\]
is a pushout among $\infty$-categories.

\item For each functor $[2]\to \cB$, and for each lift $\{0\} \amalg\{2\} \xra{\{e_0\}\amalg \{e_2\}} \cE$ along $\pi$, the canonical functor from the coend
\[
\cE_{|\{0<1\}}(e_0,-)\underset{\cE_{|\{1\}}}\bigotimes \cE_{|\{1<2\}}(-,e_2)
\xra{~\circ~}  
\cE_{|[2]}(e_0,e_2)
\]
is an equivalence of spaces.

\item
For each functor $[2]\to \cB$, the canonical map of spaces
\[
\underset{[p]\in \bdelta^{\op}}\colim~\Map_{/\ast^{\tl\tr}}([p]^{\tl \tr}, \cE_{|\ast^{\tl\tr}})\xra{~\circ~} \Map_{/\{0<2\}}(\{0<2\},\cE_{|\{0<2\}})
\]
is an equivalence.
Here we have identified $[2]\simeq \ast^{\tl\tr}$ as the suspension of the terminal $\infty$-category, and we regard each suspension $[p]^{\tl\tr}$ as an $\infty$-category over $\ast^{\tl\tr}$ by declaring the fiber over the left/right-cone-point to be the left/right-cone-point.

\item For each functor $[2]\to \cB$, and for each lift $\{0<2\}\xra{(e_0\xra{h} e_2)}\cE$ along $\pi$, the $\infty$-category of factorizations of $h$ through $\cE_{|\{1\}}$ over $[2]\to \cB$
\[
\sB ({\cE_{|\{1\}}}^{e_0/})_{/(e_0\xra{h}e_2)}~\simeq~\ast~\simeq~\sB ({\cE_{|\{1\}}}_{/e_2})^{(e_0\xra{h}e_2)/}
\]
has contractible classifying space.  
Here, the two $\infty$-categories in the above expression agree and are the fiber over $h$ of the functor $\ev_{\{0<2\}}\colon \Fun_{/\cB}([2],\cE) \to \Fun_{/\cB}(\{0<2\},\cE)$.

\end{enumerate}

\end{lemma}

\begin{proof}
By construction, the $\infty$-category $\Cat_\infty$ is presentable, and thereafter each over $\infty$-category ${\Cat_\infty}_{/\cC}$ is presentable.
The equivalence of~(1) and~(2) follows by way of the adjoint functor theorem (Cor. 5.5.2.9 of~\cite{HTT}), using that base-change is defined in terms of finite limits.  
The equivalence of~(4) and~(6) visibly follows from Quillen's Theorem A.  
The equivalence of~(4) and~(5) follows upon observing the map of fiber sequences among spaces
\[
\xymatrix{
\cE_{|\{0<1\}}(e_0,-)\underset{\cE_{|\{1\}}}\bigotimes \cE_{|\{1<2\}}(-,e_2)  \ar[r]  \ar[d]^-{\circ}
&
\underset{[p]\in \bdelta^{\op}}\colim~\Map_{/\ast^{\tl\tr}}([p]^{\tl \tr}, \cE_{|\ast^{\tl\tr}})  \ar[r]^-{\ev_{0,2}}  \ar[d]^-{\circ}
&
\cE_{|\{0\}} \times \cE_{|\{2\}}  \ar[d]^-{=}
\\
\cE_{|[2]}(e_0,e_2)  \ar[r]
&
\Map_{/\{0<2\}}(\{0<2\},\cE_{|\{0<2\}})  \ar[r]^-{\ev_{0,2}} 
&
\cE_{|\{0\}} \times \cE_{|\{2\}},
}
\]
where the top sequence is indeed a fibration sequence because pullbacks are universal in the $\infty$-category of spaces.  
By construction, there is the pushout expression $\{0<1\}\underset{\{1\}}\amalg \{1<2\} \xra{\simeq}[2]$ in $\Cat_\infty$; this shows~(2) implies~(3).

\medskip

We now prove the equivalence between~(3) and~(5).
Consider an $\infty$-category $\cZ$ under the diagram $\cE_{|\{0<1\}} \la \cE_{|\{1\}} \to \cE_{|\{1<2\}}$.  
We must show that there is a unique functor $\cE_{|[2]} \to \cZ$ under this diagram.  
To construct this functor, and show it is unique, it is enough to do so between the complete Segal spaces these $\infty$-categories present:
\[
\Map([\bullet],\cE_{|[2]})~ \overset{\exists !}\dashrightarrow~ \Map([\bullet],\cZ)
\]
under $\Map([\bullet],\cE_{|\{0<1\}}) \la \Map([\bullet],\cE_{|\{1\}}) \to \Map([\bullet],\cE_{|\{1<2\}})$.

So consider a functor $[p]\xra{f} [2]$ between finite non-empty linearly ordered sets.  
Denote the linearly ordered subsets $A_i:=f^{-1}(i)\subset [p]$.  
We have the diagram among $\infty$-categories
\begin{equation}\label{two-squares}
\xymatrix{
A_1  \ar[r]  \ar[d]
&
A_1\star A_2  \ar[d]
&&
\{1\}  \ar[r]  \ar[d]
&
\{1<2\}  \ar[d]
\\
A_0\star A_1  \ar[r]
&
[p]
&
\text{ over the diagram }
&
\{0<1\} \ar[r]
&
[2].
}
\end{equation}
We obtain the solid diagram among spaces of functors
\begin{equation}\label{flat}
\Small
\xymatrix{
\Map_{/\{0<1\}}(A_0\star A_1,\cE_{|\{0<1\}})   \ar[ddd]
&
\Map_{/\{1\}}(A_1,\cE_{|\{1\}})  \ar[r]  \ar[l]  \ar@/_3pc/[ddd]
&
\Map_{/\{1<2\}}(A_1\star A_2, \cE_{|\{1<2\}})  \ar[ddd]
\\
&
\Map_{/[2]}([p],\cE) \ar@{-->}[d]^-{\exists !}  \ar[ul]  \ar[ur]
&
\\
&
\Map([p],\cZ)   \ar[dl]  \ar[dr]
&
\\
\Map_{/\{0<1\}}(A_0\star A_1,\cZ) 
&
\Map_{/\{1\}}(A_1,\cZ)  \ar[r]  \ar[l]
&
\Map_{/\{1<2\}}(A_1\star A_2, \cZ)  
}
\end{equation}
and we wish to show there is a unique filler, as indicated.  
\\
{\bf Case that $f$ is consecutive:}
In this case the left square in~(\ref{two-squares}) is a pushout.
It follows that the upper and the lower flattened squares in~(\ref{flat}) are pullbacks.
And so there is indeed a unique filler making the diagram~(\ref{flat}) commute.  
\\
{\bf Case that $f$ is not consecutive:}
In this case $A_1=\emptyset$ and $A_0\neq \emptyset \neq A_2$.  
Necessarily, there are linearly ordered sets $B_0$ and $B_2$ for which $B_0^{\tr} \simeq A_0$ and $B_2^{\tl} \simeq A_2$.  
We recognize $B_0^{\tr} \underset{\{0\}}\amalg \{0<2\}\underset{\{2\}} \amalg B_2^{\tl} \xra{\simeq} [p]$ as an iterated pushout.  
So the canonical maps among spaces to the iterated pullbacks
\[
\Small
\xymatrix{
\Map_{/[2]}([p],\cE_{|[2]})  \ar[rr]^-{\simeq}
&&
\Map(B_0^{\tr},\cE_{|\{0\}})  \underset{\cE_{|\{0\}}^{\sim}}\times \Map_{/\{0<2\}}(\{0<2\},\cE_{|\{0<2\}}) \underset{\cE_{|\{2\}}^{\sim}} \times \Map(B_2^{\tl},\cE_{|\{2\}})  
}
\]
and
\[
\Small
\xymatrix{
\Map([p],\cZ)  \ar[rr]^-{\simeq}
&&
\Map(B_0^{\tr},\cZ)  \underset{\cZ^{\sim}}\times \Map(\{0<2\},\cZ) \underset{\cZ^{\sim}} \times \Map(B_2^{\tl},\cZ)   
}
\]
are equivalences.
This reduces us to the case that $[p] \to [2]$ is the functor $\{0<2\} \to [2]$.  
We have the solid diagram among spaces
\[
\xymatrix{
\Map_{/\{0<2\}}(\{0<2\},\cE_{|\{0<2\}})   \ar@{-->}[rr]^-{\exists !}
&&
\Map(\{0<2\},\cZ)  
\\
|\Map_{/[2]}([\bullet]^{\tl \tr}, \cE_{|\ast^{\tl\tr}})|  \ar[rr]  \ar[u]^-{\circ}
&&
|\Map([\bullet]^{\tl \tr}, \cZ)|  \ar[u]^-{\simeq}_-{\circ}.
}
\]
The right vertical map is an equivalence by the Yoneda lemma for $\infty$-categories. 
(Alternatively, the codomain is the classifying space of the $\infty$-category which is the unstraightening of the indicated functor from $\bdelta^{\op}$ to spaces, and the domain maps to this $\infty$-category finally.)
Assumption~(5) precisely gives that the left vertical map is an equivalence.  
The unique filler follows.  

It is immediate to check that the unique fillers just constructed are functorial among finite non-empty linearly ordered sets over $[2]$.

\medskip

It remains to show~(4) implies~(1). To do this we make use of the presentation $\Cat_\infty\hookrightarrow \Psh(\bdelta)$ as complete Segal spaces.  
Because limits and colimits are computed value-wise in $\Psh(\bdelta)$, and because colimits in the $\infty$-category $\Spaces$ are universal, then colimits in $\Psh(\bdelta)$ are universal as well.
Therefore, the base change functor
\[
\pi^\ast \colon \Psh(\bdelta)_{/\cB} \longrightarrow \Psh(\bdelta)_{/\cE} \colon \w{\pi}_\ast
\]
has a right adjoint, as notated.
Because the presentation $\Cat_\infty \hookrightarrow \Psh(\bdelta)$ preserves limits, then the functor $\cE\xra{\pi}\cB$ is an exponentiable fibration provided this right adjoint $\w{\pi}_\ast$ carries complete Segal spaces over $\cE$ to complete Segal spaces over $\cB$.  

So let $\tau \to \cE$ be a complete Segal space over $\cE$.  
To show the simplicial space $\w{\pi}_\ast \tau$ satisfies the Segal condition we must verify that, for each functor $[p]\to \cE$ with $p>0$, the canonical map of spaces of simplicial maps over $\cE$
\[
\Map_{/\cE}([p],\w{\pi}_\ast \tau) \longrightarrow \Map_{/\cE}(\{0<1\},\w{\pi}_\ast \tau)  \underset{\Map_{/\cE}(\{1\},\w{\pi}_\ast \tau)} \times  \Map_{/\cE}(\{1<\dots<p\},\w{\pi}_\ast \tau)
\]
is an equivalence.
Using the defining adjunction for $\w{\pi}_\ast$, this map is an equivalence if and only if the canonical map of spaces of functors
\[
\Map_{/\cB}(\pi^\ast [p],\tau) \longrightarrow \Map_{/\cB}(\pi^\ast \{0<1\},\tau)  \underset{\Map_{/\cB}(\pi^\ast \{1\},\tau)} \times  \Map_{/\cB}(\pi^\ast \{1<\dots<p\}, \tau)
\]
is an equivalence.
This is the case provided the canonical functor among pullback $\infty$-categories from the pushout $\infty$-category
\[
\cE_{|\{0<1\}} \underset{\cE_{|\{1\}}} \amalg \cE_{|\{1<\dots<p\}} \longrightarrow \cE_{|[p]}
\]
is an equivalence between $\infty$-categories over $\cB$.  
(Here we used the shift in notation $\pi^\ast \cK := \cE_{|\cK}$ for each functor $\cK\to \cB$.)
This functor is clearly essentially surjective, so it remains to show this functor is fully faithful.
Let $e_i$ and $e_j$ be objects of $\cE$, each which lies over the object of $[p]$ indicated by the subscript.
We must show that the map between spaces of morphisms
\[
\bigl(\cE_{|\{0<1\}} \underset{\cE_{|\{1\}}} \amalg \cE_{|\{1<\dots<p\}} \bigr)(e_i,e_j) \longrightarrow \cE_{|[p]}(e_i,e_j)
\]
is an equivalence.  
This is directly the case whenever $1<i\leq j\leq p$ or $0\leq i \leq j \leq 1$.  
We are reduced to the case $i=0<j$.  
This map is identified with the map from the coend
\[
\cE_{|\{0<1\}}(e_0,-)\underset{\cE_{|\{1\}}}\bigotimes \cE_{|\{1<j\}}(-,e_j)
\xra{~\circ~}  
\cE_{|\{0<1<j\}}(e_0,e_j)~.
\]
Condition~(4) exactly gives that this map is an equivalence, as desired.

It remains to verify this Segal space $\w{\pi}_\ast \tau$ satisfies the univalence condition.
So consider a univalence diagram $\cU^{\tr} \to \cB$.  
We must show that the canonical map
\[
\Map_{/\cB}(\ast,\w{\pi}_\ast \tau) \longrightarrow \Map_{/\cB}(\cU,\w{\pi}_\ast \tau)
\]
is an equivalence of spaces of maps between simplicial spaces over $\cB$.
Using the defining adjunction for $\w{\pi}_\ast$, this map is an equivalence if and only if the map of spaces
\[
\Map_{/\cE}(\cE_{|\ast},\tau)\longrightarrow \Map_{/\cB}(\cE_{|\cU},\tau)
\]
is an equivalence.
Because the presentation of $\cB$ as a simplicial space is complete, there is a canonical equivalence $\cE_{|\cU} \simeq \cE_{|\ast}\times \cU$ over $\cU$.  
That the above map is an equivalence follows because the presentation of $\tau$ as a simplicial space is complete.  
\end{proof}

\subsubsection{\bf (co)Cartesian fibrations}

The construction of the pushforward $\pi_\ast$ implies the following.  
\begin{observation}\label{t5}
Let 
\[
\xymatrix{
\cS'  \ar[rr]   \ar[d]
&&
\cE'  \ar[d]   \ar[rr]^-{\pi'}
&&
\cB' \ar[d]
\\
\cS   \ar[rr]
&&
\cE  \ar[rr]^-{\pi}
&&
\cB
}
\]
be a diagram among $\infty$-categories in which each square is a pullback.
If $\pi$ is an exponentiable fibration, then the functor $\pi'$ is also an exponentiable fibration and the canonical functor between $\infty$-categories over $\cB'$
\[
\pi'_\ast \cS'
\xra{~\simeq~}
(\pi_\ast \cS)_{|\cB'}
\]
is an equivalence.  

\end{observation}

We recall the Definitions~2.4.1.1 and~2.4.2.1 of~\cite{HTT}.
\begin{definition}\label{d1}
Let $\cE \xra{\pi}\cB$ be a functor between $\infty$-categories.  
Let $\cB_0\subset \cB$ be an $\infty$-subcategory.
\begin{itemize}

\item
\begin{itemize}
\item
A morphism $f\colon c_1 \xra{\lag e_s \xra{f} e_t\rag} \cE$ is \emph{$\pi$-coCartesian} if the canonical diagram among $\infty$-categories
\[
\xymatrix{
\cE^{e_t/}  \ar[rr]  \ar[d]
&&
\cE^{e_s/}  \ar[d]
\\
\cB^{\pi e_t/}
\ar[rr]
&&
\cB^{\pi e_s/}
}
\]
is a pullback.  

\item
The functor $\pi \colon \cE\to \cB$ is a \emph{$\cB_0$-coCartesian fibration} if each diagram among $\infty$-categories
\[
\xymatrix{
c_0  \ar[rr]  \ar[d]_-{\lag s \rag}
&&
\cE  \ar[d]^-{\pi}
\\
c_1  \ar[r]  \ar@{-->}[urr]
&
\cB_0 \ar[r]
&
\cB
}
\]
admits a $\pi$-coCartesian filler.

\end{itemize}

\item
\begin{itemize}
\item
A morphism $f\colon c_1 \xra{\lag e_s \xra{f} e_t\rag} \cE$ is \emph{$\pi$-Cartesian} if the canonical diagram among $\infty$-categories
\[
\xymatrix{
\cE_{/e_s} \ar[rr]  \ar[d]
&&
\cE_{/e_t}  \ar[d]
\\
\cB_{/\pi e_s}
\ar[rr]
&&
\cB_{/\pi e_t}
}
\]
is a pullback.  

\item
The functor $\pi \colon \cE\to \cB$ is a \emph{$\cB_0$-Cartesian fibration} if each diagram among $\infty$-categories
\[
\xymatrix{
c_0  \ar[rr]  \ar[d]_-{\lag t \rag}
&&
\cE  \ar[d]^-{\pi}
\\
c_1  \ar[r]  \ar@{-->}[urr]
&
\cB_0 \ar[r]
&
\cB
}
\]
admits a $\pi$-Cartesian filler.

\end{itemize}

\end{itemize}

\end{definition}

\begin{example}
If the inclusion $\cB_0 \hookrightarrow \cB$ is an equivalence, then a functor $\cE\to \cB$ is a $\cB_0$-(co)Cartesian fibration if and only if it is a (co)Cartesian fibration.
If the inclusion $\cB_0 = \cB^\sim \hookrightarrow \cB$ is the maximal $\infty$-subgroupoid, every functor $\cE\to \cB$ is a $\cB_0$-(co)Cartesian fibration.  

\end{example}

In~\S2.4.1 of~\cite{HTT} it is shown that, for $\cE\xra{\pi} \cB$ a (co)Cartesian fibration, a composition of two $\pi$-(co)Cartesian morphisms in $\cE$ is again a $\pi$-(co)Cartesian morphism in $\cE$.  
With this, we make the following.
\begin{observation}\label{t3}
Let $\cE\xra{\pi}\cB$ be a functor between $\infty$-categories, and let $\cB_0\subset \cB$ be an $\infty$-subcategory.
If $\pi$ is a $\cB_0$-(co)Cartesian fibration, there is a unique $\infty$-subcategory $\cE_0\subset \cE$ with the following universal property.
\begin{itemize}
\item[]
A functor $\cJ \to \cE$ factors through $\cE_0$ if and only if each morphism in $\cJ$ is carried to a $\pi$-(co)Cartesian morphism in $\cE$ over a morphism in $\cB_0$.

\end{itemize}

\end{observation}

\begin{terminology}

With the notation of Observation~\ref{t3}, $\cE_0\subset \cE$ is the \emph{$\infty$-subcategory of $\pi$-$\cB_0$-(co)Cartesian morphisms} in $\cE$.

\end{terminology}

\begin{cor}\label{cor.conduche}
If a functor $\cE \xra{\pi} \cB$ is either a coCartesian fibration or a Cartesian fibration then it is an exponentiable fibration.
In particular, both left fibrations and right fibrations are exponentiable fibrations.  

\end{cor}

\begin{proof}
Using Observation~\ref{exp-op}, each case implies the other.
So we will concern ourselves only with the coCartesian case.
We will invoke criterion~(6) of Lemma~\ref{conduche}. 
So fix a functor $[2]\to \cB$.
For each $0\leq i<j\leq 2$ we will denote $f^{ij}\colon \{i<j\} \to \cB$ for the resulting morphisms of $\cB$; we will denote $f^{ij}_!\colon \cE_{|\{i\}} \to \cE_{|\{j\}}$ for the coCartesian functor between fiber $\infty$-categories; and we will denote $u^{ij}$ for a $\pi$-coCartesian lift of $f^{ij}$.  Because $f^{12}\circ f^{01} \simeq f^{02}$ then $f^{12}_!\circ f^{01}_!\simeq f^{12}_!$ and $u^{12}\circ u^{01} \simeq u^{02}$, whenever the latter composition has meaning.  

Consider a lift $\{0<2\}\xra{(e_0\xra{h}e_2)} \cE$.  
There is a unique factorization $h\colon e_0 \xra{u^{01}} f^{01}_!(e_0) \xra{u^{12}} f^{02}_!(e_0) \xra{\ov{h}} e_2$ in which $\ov{h}$ is a morphism in the fiber $\infty$-category $\cE_{|\{2\}}$.  
We have the diagram in $\cE$
\[
\xymatrix{
e_0  \ar[rr]^-{h}  \ar[dr]_-{u^{01}}
&&
e_2 
\\
&
f^{01}_!(e_0)  \ar[ur]_-{\ov{h}\circ u^{12}}
&
,
}
\]
which is a lift of $[2]\to \cB$ along $\pi$.
As such, it defines an object of the $\infty$-category $({\cE_{|\{1\}}}^{e_0/})_{/(e_0\xra{h}e_2)}$.
By the universal properties of $\pi$-coCartesian morphisms, this object is initial.  
Therefore, the classifying space of $({\cE_{|\{1\}}}^{e_0/})_{/(e_0\xra{h}e_2)}$ is contractible.  
\end{proof}

\begin{lemma}\label{fib.lemma}
Let $\cB_0\subset \cB$ be an $\infty$-subcategory, and let $\cE\xra{\pi} \cB$ be a $\cB_0$-Cartesian fibration.  
Let $\cE_0\subset \cE$ be the $\infty$-subcategory of $\cB_0$-Cartesian morphisms.  
For each functor $\cS\xra{\tau} \cE$ that is an $\cE_0$-coCartesian fibration, the functor
\[
\pi_\ast \cS \longrightarrow  \cB
\] 
is a $\cB_0$-coCartesian fibration.

\end{lemma}

\begin{proof}
We first characterize $\cB_0$-coCartesian morphisms with respect to the projection $\pi_\ast \cS\to \cB$.  
Let $\{0<1\} \to \cB_0 \subset \cB$ select a morphism.  
By definition of the $\infty$-category $\pi_\ast \cS \to \cB$ over $\cB$, a lift of this morphism in $\cB_0$ to a morphism in $\pi_\ast \cS$ is a functor
\[
F_{01}\colon \cE_{|\{0<1\}} \longrightarrow \cS
~,\qquad
\text{ over }\cE~.
\]
Such a morphism in $\pi_\ast \cS$ over the selected morphism in $\cB_0$ is coCartesian if and only if the following condition is satisfied.
\begin{itemize}
\item[]
Let $\cK$ be an $\infty$-category.
Let $\{0<1\}\star \cK \to \cB$ be a functor extending the given functor $\{0<1\}\to \cB_0 \subset \cB$.  
Then each diagram extending $F_{01}$,
\begin{equation}\label{e100}
\xymatrix{
\cE_{|\{0\}}  \ar[rr]  \ar[d]
&&
\cE_{|\{0<1\}}  \ar[rr]^-{F_{01}}    \ar[d]
&&
\cS  \ar[d]^-{\tau}
\\
\cE_{|\{0\}\star \cK}  \ar[rr]   \ar[urrrr]^{F_{0\cK}~{}~{}~{}~{}~{}~{}~{}~{}~{}}
&&
\cE_{|\{0<1\}\star \cK}  \ar[rr]   \ar@{-->}[urr]_-{\exists !}
&&
\cE    ,
}
\end{equation}
admits a unique filler.  

\end{itemize}
Using that $\cE \to \cB$ is a $\cB_0$-Cartesian fibration, there is a canonical functor $\cE_{|\{0\}\star \cK}\xla{f} \cE_{|\{1\}\star \cK}$ under $\cE_{|\cK}$, equipped with an equivalence between $\infty$-categories over $\{0<1\}\star \cK$:
\begin{eqnarray}
\nonumber
{\sf Cylr}\bigl( \cE_{|\{0\}\star \cK}\xla{f} \cE_{|\{1\}\star \cK} \bigr) 
\underset{ \cE_{|\cK}\times \{0<1\} } \coprod  \cE_{|\cK}
&
~:=~
&
\Bigl( \cE_{|\{0\}\star \cK} \underset{ \cE_{|\{1\}\star \cK}\times \{0\}} \coprod \cE_{|\{1\}\star \cK}\times \{0<1\}  \Bigr)  
\underset{ \cE_{|\cK}\times \{0<1\} } \coprod  \cE_{|\cK}
\\
\label{e102}
&
~\simeq~
&
\cE_{|\{0<1\}\star \cK}~.
\end{eqnarray}
Through this identification, the diagram~(\ref{e100}) is equivalent to the diagram
\begin{equation}\label{e101}
\xymatrix{
\cE_{|\{1\}}  \ar[rrrr]  \ar[d]_-{f}
&&
&&
 \cE_{|\{1\}}\times \{0<1\}  \ar[rrrr]^-{F_{01}}     \ar[dd]
&&
&&
\cS  \ar[d]^-{\tau}
\\
\cE_{|\{0\}}  \ar[rr]
&&
\cE_{|\{0\}\star \cK}  \ar[urrrrrr]
&&
&&
&&
\cE
\\
&&
\cE_{|\{1\}\star \cK} \times\{0\}  \ar[u]  \ar[rr]
&&
\cE_{|\{1\}\star \cK}\times \{0<1\}   \ar[urrrr]   \ar@{-->}[uurrrr]^-{\exists !} 
&&
&&
\\
&&
&&
\cE_{|\cK}\times \{0<1\}  \ar[u]     \ar[rr]
&&
\cE_{|\cK}  \ar[uurr]  \ar@{-->}[uuurr]^-{\exists !} 
&&
.
}
\end{equation}
By construction of this diagram~(\ref{e101}) from~(\ref{e100}), for each $e\in \cE_{|\{1\}\star \cK}$ the composite functor
\[
\{0<1\} = \{e\}\times\{0<1\} \longrightarrow \cE_{|\{1\}\star \cK} \times\{0<1\}\longrightarrow \cE
\]
factors through $\cE_0$; i.e., this composite functor selects a $\pi$-$\cB_0$-Cartesian morphism.  
Consider the subdiagram of~(\ref{e101}):
\begin{equation}\label{e103}
\xymatrix{
\cE_{|\{1\}\star \cK} \times\{0\}    \ar[rr]  \ar[d]
&&
\cS  \ar[d]^-{\tau}
\\
\cE_{|\{1\}\star \cK}\times \{0<1\}   \ar[rr]   \ar@{-->}[urr]^-{F_{1\cK}}
&&
\cE  .
}
\end{equation}
By assumption, the functor $\cS\to \cE$ is an $\cE_0$-coCartesian fibration.
Therefore, the diagram~(\ref{e103}) admits an initial filler, which has the following property: for each $e\in \cE_{|\{1\}\star \cK}$ the composite functor
\[
\{0<1\} = \{e\}\times\{0<1\} \longrightarrow \cE_{|\{1\}\star \cK} \times \{0<1\}
\xra{~F_{1\cK}~}
\cS
\]
selects a $\tau$-$\cE_0$-coCartesian morphism.
Consider the subdiagram of~(\ref{e101}):
\begin{equation}\label{e104}
\xymatrix{
\cE_{|\cK}\times \{0<1\}   \ar[rr]  \ar[d]
&&
\cE_{|\{1\}\star \cK}\times \{0<1\}  \ar[rr]^-{F_{1\cK}} 
&&
\cS  \ar[d]^-{\tau}
\\
\cE_{|\cK}  \ar[rrrr]   \ar@{-->}[urrrr]
&&&&
\cE  .
}
\end{equation}
The right vertical functor in this diagram~(\ref{e104}) is a localization.
Furthermore, a $\tau$-coCartesian morphisms in $\cS$ is an equivalence if it is carried by $\tau$ to an equivalence in $\cE$.  
The existence and uniqueness of a filler in~(\ref{e104}) follows.  
It remains to verify that the functor $F_{1\cK}$ is unique among those for which the subdiagram of~(\ref{e101}),
\begin{equation}\label{e105}
\xymatrix{
\cE_{|\{1\}}\times \{0<1\}   \ar[rr]^-{F_{01}}  \ar[d]
&&
\cS  \ar[d]^-{\tau}
\\
\cE_{|\{1\}\star \cK}\times \{0<1\}    \ar[rr]     \ar@{-->}[urr]^-{F_{1\cK}} 
&&
\cE  ,
}
\end{equation}
commutes.
Using that coCartesian morphisms in $\cS$ compose, this diagram commutes if and only if the following condition is satisfied.
\begin{itemize}
\item[] {\bf (C):}
For each $e\in \cE_{|\{1\}}$, the composite functor
\[
\{0<1\} = \{e\}\times\{0<1\} \longrightarrow \cE_{|\{1\}} \times \{0<1\}
\xra{~F_{10}~}
\cS
\]
selects a $\tau$-coCartesian morphism.  
\end{itemize}
Now assume $F_{01}$ satisfies the condition~{\bf (C)}.
As established above, for each $e\in \cE_{|\cK}$, the composite functor
\[
\{0<1\} = \{e\}\times\{0<1\} \longrightarrow \cE_{|\{1\}\star\cK} \times \{0<1\}
\xra{~F_{1\cK}~}
\cS
\]
selects an equivalence in $\cS$.  
It follows that $F_{1\cK}$ is, in fact, the unique such filler in~(\ref{e105}).
We have concluded that the given morphism $F_{01}$ in $\pi_\ast\cS$ is a $\pi$-coCartesian morphism provided it satisfies the condition~{\bf (C)}.

With the above condition for checking the coCartesian property for certain morphisms in $\pi_\ast \cS\to \cB$, we are at last prepared to show that the functor $\pi_\ast \cS\to \cB$ is a $\cB_0$-coCartesian fibration.
Now let $\{0<1\}\to \cB_0 \subset \cB$ select a morphism.  
Using that $\pi$ is a $\cB_0$-Cartesian fibration, there is a canonical functor $\cE_{|\{0\}} \xla{f}\cE_{|\{1\}}$ equipped with an identification
\begin{equation}\label{e106}
\cE_{|\{0\}} \underset{\cE_{|\{1\}}\times\{0\}}\coprod \cE_{|\{1\}}\times\{0<1\}
~\simeq~
\cE_{|\{0<1\}}
\end{equation}
over $\{0<1\}$. 
Through this identification, for each $e\in \cE_{|\{1\}}$, the composite functor
\[
\{0<1\} = \{e\}\times\{0<1\} \longrightarrow \cE_{|\{1\}}\times\{0<1\} \longrightarrow \cE_{|\{0<1\}} \longrightarrow \cE
\]
selects a $\pi$-$\cB_0$-Cartesian morphism.  
From the definition of the $\infty$-category $\pi_\ast \cS\to \cB$ over $\cB$, it is enough to show that each solid diagram
\[
\xymatrix{
\cE_{|\{1\}}  \ar[rr]^-{f}  \ar[d]
&&
\cE_{|\{0\}}  \ar[rr]^-{F_0} 
&&
\cS  \ar[d]
\\
\cE_{|\{1\}}\times\{0<1\}  \ar[rrrr]  \ar@{-->}[urrrr]_-{F_{01}}
&&&&
\cE
}
\]
admits a filler that satisfies condition~{\bf (C)}. 
Such a filler exists because $\tau$ is assumed to be an $\cE_0$-coCartesian fibration.  
\end{proof}

\subsubsection{\bf Doubly based finite sets}
Recall from 
Example~\ref{t605} 
the $\infty$-category $\Fin_{\ast\star}$ over $\Fin_\ast$.  
The defining functor $\Fin_{\ast \star} \hookrightarrow \Fin_{\ast}^{\star_+/}$ is a monomorphism.
That is, for each functor $\cK \to \Fin_\ast$, the map
\begin{equation}\label{sub-class}
\Map_{/\Fin_\ast}(\cK,\Fin_{\ast \star}) ~\hookrightarrow~ \Map_{/\Fin_\ast}(\cK,\Fin_{\ast}^{\star_+/})
\end{equation}
is a monomorphism of spaces.
The given functor $\cK \to \Fin_\ast$ classifies a left fibration $\cE \to \cK$, equipped with a section $\sigma_+ \colon \cK \to \cE$, with the property that, for each $k\in \cK$, the fiber $\cE_k$ is a $0$-type. 
As such, the monomorphism of spaces~(\ref{sub-class}) is naturally identified as the inclusion of path components
\[
\{\cK \xra{\sigma} \cE\mid \text{ for each } k\in \cK,~\sigma_+(k) \nsim \sigma(k)\in \cE \}~\hookrightarrow~ \Map_{/\cK}(\cK,\cE)
\]
consisting of those sections which are object-wise never equivalent to the given section $\sigma_+$.  
In the case that $\cK$ is a suspension, this monomorphism~(\ref{sub-class}) is yet more explicit, which we highlight as the following simple result of this discussion.

\begin{observation}\label{doubly-classifies}
Let $\cJ^{\tl \tr} \to \Fin_\ast$ be a functor from a suspension.
Denote the value of this functor on the left/right-cone-point as the based finite set $L_+/R_+$, and denote the value of this functor on the unique morphism between cone-points as the based map $L_+\xra{f_{0}}R_+$. 
Evaluation at the left-cone-point defines an identification of the monomorphism of spaces of sections $\Map_{/\Fin_\ast}(\cJ^{\tl\tr},\Fin_{\ast \star})\hookrightarrow \Map_{/\Fin_\ast}(\cJ^{\tl\tr},\Fin_{\ast}^{\star_+/})$ with the inclusion of sets
\[
f_{0}^{-1}(R)~\hookrightarrow~ L_+~.
\]
\end{observation}

\begin{lemma}\label{fin-conduche}
The projection $\Fin_{\ast\star} \to \Fin_\ast$ is an exponentiable fibration.

\end{lemma}

\begin{proof}
We verify (4) of Lemma~\ref{conduche} applied to the functor $\Fin_{\ast\star} \to \Fin_\ast$.  
Let $[2]\xra{I_+\xra{f}J_+\xra{g}K_+} \Fin_\ast$ be a functor.  
Choose lifts $(I_+,i)$ and $(K_+,k)$ to $\Fin_{\ast \star}$.  
Observation~\ref{doubly-classifies} identifies the space of morphisms 
${\Fin_{\ast \star}}_{|[2]}\bigl((I_+,i),(K_+,k)\bigr)$ as terminal if $g(f(i))=k$ and as empty otherwise.  
In the latter case, the coend too is empty, which verifies the criterion.
Now assume $g(f(i))=k$.  
We must show that the coend too is terminal.  
This coend is indexed by a $0$-type, so it is identified simply as the space
\[
\underset{j\in J}\coprod {\Fin_{\ast \star}}_{|\{0<1\}}\bigl((I_+,i),(J_+,j)\bigr)\times {\Fin_{\ast \star}}_{|\{1<2\}}\bigl((J_+,j),(K_+,k)\bigr)~.
\]
Using Observation~\ref{doubly-classifies} we see that the cofactor indexed by $f(i)\in J$ is terminal while the other cofactors are empty. 
\end{proof}

\subsubsection{\bf Finite correspondences}
Recall from the top of~\S\ref{sec.wreath} the $\infty$-category $\ov{\Corr}(\Fin)$ over $\Corr(\Fin)$.  

\begin{lemma}\label{corr.conduche}
The projection $\ov{\Corr}(\Fin) \to \Corr(\Fin)$ is an exponentiable fibration.
Also, for each morphism $c_1\to \Corr(\Fin)$, the $\infty$-category $\Fun_{/\Corr(\Fin)}\bigl(c_1,\ov{\Corr}(\Fin)\bigr)$ of sections is a finite $0$-type.  

\end{lemma}

\begin{proof}
The second statement is immediate from the definition of $\ov{\Corr}(\Fin)$ as an $\infty$-category over $\Corr(\Fin)$. 

For the first statement, we verify (4) of Lemma~\ref{conduche} applied to the functor $\ov{\Corr}(\Fin) \to \Corr(\Fin)$.  
Let 
\[
[2]\xra{I_0 \la K_{01} \ra I_1 \la K_{12} \ra I_2 } \Corr(\Fin)
\]
be a functor.  
Choose lifts $(i_0\in I_0)$ and $(i_2\in I_2)$ in $\ov{\Corr}(\Fin)$ of the values on $0,2\in [2]$.
By definition of the $\infty$-category $\ov{\Corr}(\Fin)$ over $\Corr(\Fin)$, there is a canonical identification of the space of morphisms
\[
\ov{\Corr}(\Fin)_{|[2]}\bigl( (i_0\in I_0) , (i_2\in I_2)\bigr)
~\simeq~
(K_{01})_{|i_0}\underset{I_1}\times (K_{12})_{|i_2}~.
\]
Recognize the righthand side of this identification as precisely the relevant coend in~(4) of Lemma~\ref{conduche}; in this situation, the coend is indexed by the finite $0$-type $I_1$.

\end{proof}

\subsubsection{\bf Absolute exit-paths}
We prove that the functor $\Exit \to \Bun$ between $\infty$-categories, given by forgetting section data, is exponentiable.  
This result is essential for our method of defining structured versions of $\Bun$.

\begin{lemma}\label{exit-exp}
The functor $\Exit \to \Bun$ is an exponentiable fibration.

\end{lemma}

\begin{proof}
We check the criterion~(5) of Lemma~\ref{conduche}.
Fix a functor $[2]\to \Bun$.  
This functor classifies a constructible bundle $X\to \Delta^2$.  
We will denote the restriction $X_S:=X_{|\Delta^{S}}$ for each non-empty linearly ordered subset $S\subset\{0<1<2\}$.  

Through~\S3 of~\cite{striat} we establish a natural identification of the $\infty$-category of lifts
\[
\Fun_{/\{0<2\}}(\{0<2\},\Exit_{|\{0<2\}})~\simeq~\exit\bigl(\Link_{X_0}(X_{02})\big)
\]
as the exit-path $\infty$-category of the link.
Likewise, there is an identification of the $\infty$-category of lifts 
\[
\Fun_{/[2]}([2],\Exit_{|[2]})~\simeq~\exit\Bigl(\Link_{\Link_{X_0}(X_{01})}\bigl(\Link_{X_0}(X)\Bigr)
\]
as the exit-path $\infty$-category of the iterated link.  
In~\S6 of~\cite{striat} we construct a refinement morphism among stratified spaces
\[
\gamma\colon  
\Link_{\Link_{X_0}(X_{01})}\bigl(\Link_{X_0}(X)\bigr)  
~\longrightarrow~
\Link_{X_0}(X_{02})
\]
from the iterated link to the link -- in brief, this morphism is obtained using flows of collaring vector fields of links inside blow-ups.
This morphism is compatible with the above identifications in that there is a commutative diagram among $\infty$-categories
\[
\xymatrix{
\Fun_{/[2]}([2],\Exit_{|[2]})  \ar[d]^-{\simeq}     \ar[rr]^-{\circ}_-{\ev_{\{0<2\}}}
&&
\Fun_{/\{0<2\}}(\{0<2\},\Exit_{|\{0<2\}})  \ar[d]^-{\simeq}
\\
\exit\bigl(  \Link_{\Link_{X_0}(X_{01})}\bigl(\Link_{X_0}(X)\bigr)    \bigr)    \ar[rr]^-{\exit(\gamma)}
&&
\exit\bigl(\Link_{X_0}(X_{02})\big).
}
\]
Restricting the righthand terms of this diagram to maximal $\infty$-subgroupoids gives the commutative diagram among $\infty$-categories
\[
\xymatrix{
\Fun_{/[2]}([2],\Exit_{|[2]})_{|\Map_{/\{0<2\}}(\{0<2\},\Exit_{|\{0<2\}}) }\ar[d]^-{\simeq}     \ar[rr]^-{\circ}_-{\ev_{\{0<2\}}}
&&
\Map_{/\{0<2\}}(\{0<2\},\Exit_{|\{0<2\}})  \ar[d]^-{\simeq}
\\
\underset{p\in P} \coprod    \Bigl( \Link_{\Link_{X_0}(X_{01})}\bigl(\Link_{X_0}(X)\bigr)   \Bigr)_{|\Link_{X_0}(X_{02})_p}    \ar[rr]^-{\exit(\gamma)_|}
&&
\underset{p\in P} \coprod \Link_{X_0}(X_{02})_p
}
\]
where the coproducts are indexed by the strata of the link.

To verify criterion~(5) of Lemma~\ref{conduche} we must explain why the top horizontal functor induces an equivalence on classifying spaces.
By the commutativity of the previous recent diagram, we must explain why the bottom horizontal map induces an equivalence on classifying spaces.    
In~\S4 of~\cite{aft1} it is proved that, for $\w{X} \to X$ a refinement between stratified spaces, the associated functor between exit-path $\infty$-categories $\exit(\w{X}) \to \exit(X)$ is a localization.  
In particular, for each stratum $X_p\subset X$, the map $\exit(\w{X}_{|X_p}) \to X_p$ induces an equivalence on classifying spaces. 
\end{proof}

\begin{prop}\label{t4}
The functor $\Exit \xra{\pi}\Bun$ has the following properties.
\begin{itemize}

\item
It is a $\Bun^{\sf cls}$-Cartesian fibration.

\item
It is a $\Bun^{\sf p.cbl}$-Cartesian fibration.

\item
It is a $\Bun^{\sf ref}$-coCartesian fibration.

\item
It is a $\Bun^{\sf opn}$-coCartesian fibration.

\end{itemize}

\end{prop}

\begin{proof}
We first consider the assertion as it concerns the $\infty$-subcategory $\Bun^{\sf cls}\subset \Bun$.

We begin by characterizing some closed-Cartesian morphisms with respect to the functor $\pi\colon \Exit \to \Bun$. 
Let $\{0<1\}\to \Bun^{\sf cls}\subset \Bun$ select a closed morphism in $\Bun$.  
By definition of the $\infty$-category $\Bun$, this morphism is the data of a constructible bundle $X_{01}\to \Delta^1$ over the standardly stratified 1-simplex.  
A lift of this morphism $\{0<1\}\xra{\gamma} \Exit$ is a section, in the category of stratified spaces, of this constructible bundle
\[
\xymatrix{
&&
X_{01} \ar[d]
\\
\Delta^1   \ar[rr]^-=  \ar@{-->}[urr]^-{\gamma}
&&
\Delta^1   .
}
\]
Such a section selects points $x_0 \in X_0:=X_{|\Delta^{\{0\}}}$ and $x_1\in X_1:=X_{|\Delta^{\{1\}}}$ in these fibers.  
By definition, such a morphism $\gamma$ of $\Exit$ is $\pi$-Cartesian if and only if 
the canonical diagram among $\infty$-categories
\[
\xymatrix{
\Exit_{/(x_0\in X_0)} \ar[rr]  \ar[d]
&&
\Exit_{/(x_1\in X_1)}  \ar[d]
\\
\Bun_{/X_0}
\ar[rr]
&&
\Bun_{/X_1}
}
\]
is a pullback. 
From the definition of these $\infty$-overcategories in this diagram, the condition that this square is a pullback is equivalent to the following condition.
\begin{itemize}
\item[]
Let $\cK$ be an $\infty$-category.
Let $\cK\star [1]\to \Bun$ be a functor extending the given functor $[1] =\{0<1\}\to \Bun$.  
Each solid diagram among $\infty$-categories
\[
\xymatrix{
\{1\}  \ar[rr]  \ar[d]
&&
[1]    \ar[d]    \ar[rr]^-{\gamma}
&&
\Exit \ar[d]^-{\pi}
\\
\cK\star \{1\}\ar[rr]  \ar[urrrr]
&&
\cK\star [1]  \ar[rr]  \ar@{-->}[urr]_-{\exists !}
&&
\Bun
}
\]
admits a unique filler.  
\end{itemize}
From the definition of the functor between $\infty$-categories $\Exit\to \Bun$, it is enough to verify the above condition just in the cases that $\cK = \exit(K)$ is the exit-path $\infty$-category of a compact stratified space $K$.  
Using that the functor $\Exit \to \Bun$ is an exponentiable fibration (Lemma~\ref{exit-exp}), the canonical diagram among $\infty$-categories
\[
\xymatrix{
&
\Exit_{|\cK}\times[1]  \ar[dr]^-{s}  \ar[dl]_-{\sf pr}
&&
\cK\times \Exit_{|[1]}  \ar[dr]^-{\sf pr}  \ar[dl]_-{t}
&
\\
\Exit_{|\cK}  \ar[drr]
&&
\Exit_{|\cK\times c_1\times [1]}  \ar[d]
&&
\Exit_{|[1]}  \ar[dll]
\\
&&
\Exit_{|\cK\star [1]}
&&
}
\]
is a colimit diagram over $\cK \star [1]$. 
Putting these two facts together, the above condition can be rephrased as the following condition.
\begin{itemize}
\item[{\bf (C):}]
Let $K$ be a compact stratified space.  
Let $X\to K\times \Delta^{\{s<t\}} \times \Delta^{\{0<1\}}$ be a constructible bundle, equipped with an identification $X \cong K\times \Delta^{\{s<t\}}\times X_{01}$ over $K\times \Delta^{\{s<t\}}\times \Delta^{\{0<1\}}$.  
Each solid diagram among stratified spaces
\[
\xymatrix{
&
K  \ar[r]^-1  \ar[d]^-{t}
&
K\times \Delta^{\{0<1\}}  \ar[dd]^-{t}   \ar[r]^-{\sf pr}
&
\Delta^{\{0<1\}}  \ar[r]^-{\gamma}
&
X  \ar[dd]
\\
K \ar[r]^-{s}  \ar[d]_-{1}
&
K\times \Delta^{\{s<t\}}  \ar[dr]  \ar[urrr]
&
&
&
\\
K\times \Delta^{\{0<1\}}  \ar[rr]^-{s}  \ar[uurrrr]
&
&
K\times \Delta^{\{s<t\}} \times \Delta^{\{0<1\}} \ar[rr]^-=  \ar@{-->}[uurr]_-{\exists !}
&
&
K\times \Delta^{\{s<t\}} \times \Delta^{\{0<1\}}
}
\]
admits a unique filler.  
\end{itemize}

Now, introduced in~\S6.6 of~\cite{striat} is the \emph{reversed cylinder} construction for stratified spaces, where it is shown to define an equivalence between $\infty$-categories,
\[
{\sf Cylr}\colon( \Strat^{\sf p.cbl.inj})^{\op}
\xra{~\simeq~}
\Bun^{\cls}~,
\]
from the opposite of the $\infty$-category associated to the Kan-enriched category of stratified spaces and proper constructible embeddings among them.  
Using this reverse cylinder construction, the assumption that the given morphism $[1]=\{0<1\}\to \Bun$ selects a closed morphism offers the existence of a unique proper constructible embedding between base changes of $X$,
\[
X_{|K\times\Delta^{\{s<t\}}\times\Delta^{\{0\}} }
\xla{~f~}
X_{|K\times\Delta^{\{s<t\}}\times\Delta^{\{1\}} }   ~,
\]
together with an isomorphism between stratified spaces
\[
{\sf Cylr}\Bigl(
X_{|K\times\Delta^{\{s<t\}}\times\Delta^{\{0\}} }
\xla{~f~}
X_{|K\times\Delta^{\{s<t\}}\times\Delta^{\{1\}} }
\Bigr)
\xra{~\cong~}
X
\]
over $K\times\Delta^{\{s<t\}}\times\Delta^{\{0<1\}}$.
This isomorphism results in an identification between $\infty$-categories
\[
{\sf Cylr}\Bigl(  
\exit(X_{|K\times\Delta^{\{s<t\}}\times\Delta^{\{0\}} })
\xla{~f~}
\exit( X_{|K\times\Delta^{\{s<t\}}\times\Delta^{\{1\}} })  
\Bigr)
~\simeq~
\exit(X)
\]
over $\exit(K)\times c_1\times [1]$.
In particular, the composite projection 
\begin{equation}\label{e107}
\exit(X) \longrightarrow \exit(K)\times c_1\times [1]
\longrightarrow [1]
\end{equation}
is a Cartesian fibration.
From this we see that the above condition~{\bf (C)} is true provided $\gamma$ is a Cartesian morphism of the composite functor~(\ref{e107}).
Inspecting the reverse cylinder construction, this is to say that the map from the standardly stratified 1-simplex
\[
\Delta^{\{0<1\}}
\xra{~\gamma~}
X 
\cong
{\sf Cylr}\Bigl(
X_{|K\times\Delta^{\{s<t\}}\times\Delta^{\{0\}} }
\xla{~f~}
X_{|K\times\Delta^{\{s<t\}}\times\Delta^{\{1\}} }
\Bigr)
\xra{~\rm retract~}
X_{|K\times\Delta^{\{s<t\}}\times\Delta^{\{0\}} }
\]
selects an equivalence in the exit-path $\infty$-category of $X_{|K\times\Delta^{\{s<t\}}\times\Delta^{\{0\}} }$.  
This concludes our characterization of some $\pi$-Cartesian morphisms in $\Exit$.

Using this characterization, it is evident that each solid diagram among $\infty$-categories
\[
\xymatrix{
\{1\}  \ar[rr]  \ar[d]
&&
\Exit  \ar[d]^-{\pi}
\\
\{0<1\}  \ar[r]  \ar@{-->}[urr]
&
\Bun^{\sf cls}\ar[r]
&
\Bun
}
\]
admits a $\pi$-Cartesian filler.

Finally, the assertions concerning the $\infty$-subcategory $\Bun^{\sf p.cbl}\subset \Bun$ follows by the same argument as that above as it concerns $\Bun^{\sf cls}$.
The key point being that the \emph{reversed cylinder} construction, which played an essential role above, exists for proper constructible maps, more generally.
Likewise, the other two $\infty$-subcategories of $\Bun$ follow a similar argument, but premised on the \emph{open cylinder} construction, ${\sf Cylo}$, of~\S6.6 of~\cite{striat}.  
\end{proof}

\subsection{Higher idempotents}
In this section we define the notion of a $k$-idempotent in an $(\infty,n)$-category, and more generally in a Segal $\btheta_n$-space.  
We show that, if a Segal $\btheta_n$-space has only identity $k$-idempotents for each $0<k\leq n$, then it presents an $(\infty,n)$-category.

Let ${\sf Idem}$ denote the unique 2-element monoid which is not a group.  
This monoid is commutative, and so defines a functor ${\sf Idem}\colon \Fin_\ast \to {\sf Set} \hookrightarrow \Spaces$ which satisfies a Segal condition in the sense of~\cite{segal-gamma}.
Precomposing with the simplicial circle defines a functor
\[
\fB{\sf Idem} \colon \bdelta^{\op} \xra{c_1/\partial c_1} \Fin_\ast \xra{\sf Idem} \Spaces~.
\]
For each $0<k\leq n$ consider the presheaf ${\sf Idem}^k$ on $\btheta_n$ which is the left Kan extension:
\[
\xymatrix{
\bdelta^{\op} \ar[rr]^-{\fB{\sf Idem}}    \ar[d]_-{c_{k-1}\wr-}
&&
\Spaces
\\
\btheta_k^{\op}  \ar[rr]^-{\iota_k}
&&
\btheta_n^{\op}  \ar@{-->}[u]_-{{\sf Idem}^k}.
}
\]
Because ${\sf Idem}\colon \Fin_\ast \to \Spaces$ is a Segal $\Fin_\ast$-space, then this $\btheta_n$-space ${\sf Idem}^k$ too is Segal. 
Because the maximal subgroup of the monoid ${\sf Idem}$ is trivial, then this Segal $\btheta_n$-space ${\sf Idem}^k$ is univalent.  
Therefore, ${\sf Idem}^k$ presents an $(\infty,n)$-category.
As such, ${\sf Idem}^k$ corepresents a $k$-idempotent of a Segal $\btheta_n$-space.

\begin{observation}\label{idem-epi}
For each $0<k\leq n$, the projection ${\sf Idem}^k \to c_{k-1}$ between Segal $\btheta_n$-spaces is an epimorphism.
This follows by induction, using that the map of monoids ${\sf Idem} \to \ast$ is an epimorphism.  

\end{observation}

We say a $k$-idempotent of a Segal $\btheta_n$-space, ${\sf Idem}^k \to \cC$, is a \emph{$k$-identity} if there is a factorization ${\sf Idem}^k \to c_{k-1} \to \cC$.  
After Observation~\ref{idem-epi}, it is a \emph{condition} on a $k$-idempotent to be an identity.

\begin{lemma}\label{univ-criterion}
Let $\cC$ be a Segal $\btheta_n$-space.
Suppose, for each $0<k\leq n$, that each $k$-idempotent of $\cC$ is an identity, by which we mean each map of presheaves ${\sf Idem}^k \to \cC$ factors through ${\sf Idem}^k\to c_{k-1}$.
Then $\cC$ presents an $(\infty,n)$-category.

\end{lemma}

\begin{proof}
We will make ongoing use of Observation~\ref{theta-stand-bump}, where a number of monomorphisms among the $\btheta_i$ are established.  
Consider a univalence diagram $\cE^{\tr} \to \btheta_n$. 
From its definition, there is a maximal $0<k\leq n$ for which there is a factorization
\[
\cE^{\tr} \longrightarrow \bdelta \xra{~\iota~} \btheta_{n-k+1} \xra{c_{k-1}\wr -} \btheta_n
\]
through a univalence diagram of $\bdelta$.  
In this way, we reduce to the case $n=k=1$, where the $\cE\to \btheta_n$ is the diagram in $\bdelta$
\[
\ast \la \{0<2\}  \to \{0<1<2<3\} \la \{1<3\} \to \ast~.
\]
By definition, the colimit $\w{E}$ of this diagram in the $\infty$-category $\Psh^{\sf Segal}(\bdelta)$, of Segal $\bdelta$-spaces, corepresents an equivalence.
This diagram receives an evident map from the diagram in $\bdelta$
\[
\ast \la \{0<2\} \to \{0<1<2\} ~.
\]
The colimit $\w{R}$ of this latter diagram in Segal $\bdelta$-spaces corepresents a pair of retractions.
By direct inspection, this Segal $\bdelta$-space $\w{R}$ is complete, and therefore presents an $\infty$-category.  
We obtain a composite map of Segal $\bdelta$-spaces
\[
{\sf Idem}^1  \longrightarrow \w{R} \longrightarrow \w{E}~.
\]
Therefore, there is a natural map
$
\Map(\w{E},\cC) \longrightarrow \Map({\sf Idem}^1,\cC)
$
from the space of equivalences in $\cC$ to the space of pairs of idempotents of $\cC$.  
This map fits into a commutative triangle of maps among spaces
\[
\xymatrix{
\Map(\w{E},\cC) \ar[rr]
&&
\Map({\sf Idem}^1,\cC)
\\
&
\Map(c_0,\cC)\simeq \cC[0]  \ar[ur]  \ar[ul]
&
.
}
\]
It is standard that the leftward diagonal map is a monomorphism.  
The assumption on $\cC$ exactly grants that the rightward diagonal map is an equivalence.  
It follows that this diagram is comprised of equivalences among spaces.
This proves that $\cC$ carries the univalent diagram to a limit diagram, as desired.
\end{proof}

\section{Appendix: erratum to previous version}\label{sec.erratum}
\noindent
The statement of Lemma~\ref{autos-suspend} of this article is not true in dimensions at least 3.
Because the statement of Lemma~\ref{autos-suspend} is not true in the cases that $\dim(\sS(L)) \geq 3$, the following results therefore require an additional dimension hypothesis in addition to how they are stated in a previous version of this article:
Theorem~\ref{main.theorem}, Corollary~\ref{no-autos}, Theorem~\ref{theta-ff}, Corollary~\ref{RKE-ff}.
\smallskip
The main construction of the article, factorization homology,
\[
\int\colon 
\Cat_{(\infty,n)}
\longrightarrow
\Fun\bigl( \cMfd_n^{\vfr} , \Spaces \bigr)
~,
\]
is unaffected.  
In particular, the construction of $\Spaces$-valued invariants of closed framed $n$-manifolds from the data of an $(\infty,n)$-category is unaffected.
The assertion of fully faithfulness of factorization homology is unaffected for $n<3$.  

\smallskip

In this appendix, we quote Lemma~\ref{autos-suspend} and its proof. The incorrect clause is highlighted and in bold, and we explain the incorrectness of this step.
Afterward, we construct a counterexample to the statement of Lemma~\ref{autos-suspend} in the case that the dimension $n=\dim(\sS(L))=3$.
\\

\subsection{Lemma~\ref{autos-suspend} and its proof}
We recall the statement and proof of Lemma~\ref{autos-suspend}.
The error in this proof occurs in the clause in bold highlight
This clause is correct only for $\dim(\sS(L)) < 3$; consequently, Lemma~\ref{autos-suspend}, and results that logically depend on it, require an additional dimension hypothesis.

\begin{lemma}[Lemma \ref{autos-suspend} of a previous version]
For each compact vari-framed stratified space $L$, the map between spaces of automorphisms 
\[
\sS^{\fr}\colon \Aut_{\Bun^{\vfr}}(L) \xra{~\simeq~}\Aut_{\Bun^{\vfr}}\bigl(\sS^{\fr}(L)\bigr)
\]
is an equivalence.  

\end{lemma}

\begin{proof}
We seek to establish the equivalence between the maximal connected $\infty$-subgroupoids of $\Bun^{\vfr}$ that contain the respective objects $L$ and $\sS^{\fr}(L)$.
The maximal $\infty$-subgroupoid of $\Bun^{\vfr}$ corresponds, via the main result of~[AFR], to the striation sheaf that classifies fiber bundles among stratified spaces which are equipped with a fiberwise vari-framing.  
Therefore, to prove the equivalence of the lemma it is enough to prove, for each smooth manifold $T$, that each $T$-point of the codomain lifts to an $T$-point of the domain.  
So consider an $T$-point of the striation sheaf ${\sf BAut}_{\Bun^{\vfr}}\bigl(\sS^{\fr}(L)\bigr)$; it classifies a proper fiber bundle $E\xra{p} T$, equipped with a fiberwise vari-framing $\exit(E) \xra{g} \vfr$, with each structured fiber $\bigl(E_t,g_{|\exit(E_t)}\bigr)$ is equivalent in $\Bun^{\vfr}$ to $\sS^{\fr}(L)$.
To construct the desired lift of this $T$-point to ${\sf BAut}_{\Bun^{\vfr}}(L)$ is the problem of constructing a fiberwise vari-framed proper fiber bundle $E_0\to T$ and an equivalence $\sS^{\sf fib,\fr}(E_0) \simeq E$ of fiberwise vari-framed fiber bundles over $T$ from the fiberwise framed suspension.

The fiberwise $0$-dimension strata of $E$ is a 2-sheeted covering over $T$.
The fiberwise vari-framing of $E\to T$ in particular implies this 2-sheeted cover is trivial.
We denote it as $T_-\sqcup T_+ \subset E$, with the subscripts marking if first coordinate agrees or disagrees with a collaring-coordinate about each cofactor.  
Taking the unzip along this closed constructible subspace gives the composable pair of proper constructible bundles
\[
\unzip_{T_-\sqcup T_+}(E) \xra{~q~} E \xra{~p~} T~.
\]
Because $T$ is a smooth manifold the functor $\sT_T \colon \exit(T) \to \Vect^{\sf inj}$ factors through $\Vect^\sim$, the maximal $\infty$-subgroupoid.  
Using that the fiberwise constructible tangent bundle is a morphism $\sT^{\sf fib} \colon \Exit \to \Vect^{\sf inj}$ of $\Vect^\sim$-modules, 
there results the short exact sequence of $\Vect$-valued functors from $\exit\bigl(\unzip_{T_-\sqcup T_+}(E)\bigr)$
\[
0 \longrightarrow \sT_q  \longrightarrow \sT_{pq} \longrightarrow q^\ast \sT_p  \longrightarrow 0~.
\]
(Here, and through this proof, we use the notation established at the beginning of the proof of Lemma~3.12.)
Restricted to $\exit\bigl(\Link_{T_-\sqcup T_+}(E)\bigr)$, the cokernel term vanishes;
restricted to
\[
\exit\bigl(\unzip_{T_-\sqcup T_+}(E) \smallsetminus \Link_{T_-\sqcup T_+}(E)\bigr)\xra{\simeq} \exit(E\smallsetminus T_-\sqcup T_+)~,
\]
the kernel term vanishes and the cokernel term does not.
The first coordinate of the fiberwise vari-framing $\exit(E) \to \vfr$ therefore determines a non-vanishing parallel vector field on $\unzip_{T_-\sqcup T_+}(E) \smallsetminus \Link_{T_-\sqcup T_+}(E)$ in the sense of~\S8 of~[AFT].
We will now extend this vector field to all of $\unzip_{T_-\sqcup T_+}(E)$.

Let $\ov{\alpha}\colon \Link_{T_-\sqcup T_+}(E)\times[0,1) \hookrightarrow \unzip_{T_-\sqcup T_+}(E)$ be a choice of collaring, the existence of which is guaranteed by the results in~\S8 of~[AFT].
Denote the restriction $\alpha\colon \Link_{T_-\sqcup T_+}(E)\times \{\frac{1}{2}\} \to \unzip_{T_-\sqcup T_+}(E)\smallsetminus \Link_{T_-\sqcup T_+}(E) \cong E\smallsetminus (T_-\sqcup T_+)$.
The fiberwise vari-framing of $E\to T$ determines an identification of short exact sequences of $\Vect$-valued functors from $\exit\bigl(\Link_{T_-\sqcup T_+}(E)\bigr)$
\[
\xymatrix{
0  \ar[r]  \ar[d]
&
\sT^{\sf fib}_{\Link_{T_-\sqcup T_+}(E)}  \ar[r]  \ar[d]^-{\simeq}
&
\alpha^\ast \sT_{E\smallsetminus T_-\sqcup T_+}  \ar[r]  \ar[d]^-{\simeq}
&
\epsilon^1_{\Link_{T_-\sqcup T_+}(E)}  \ar[r]  \ar[d]^-{\simeq}
&
0  \ar[d]
\\
0  \ar[r]  \ar[u]
&
\alpha^\ast \epsilon^{\sf fib.dim-1}_{E\smallsetminus T_-\sqcup T_+}   \ar[r]  \ar[u]
&
\alpha^\ast \epsilon^{\sf fib.dim}_{E\smallsetminus T_-\sqcup T_+}  \ar[r]  \ar[u]
&
\alpha^\ast \epsilon^1_{E\smallsetminus T_-\sqcup T_+}  \ar[r]  \ar[u]
&
0 \ar[u]
}
\]
The above diagram grants that this vector field extends to a non-vanishing vector field on $\unzip_{T_-\sqcup T_+}(E)$ in a neighborhood of $\Link_{T_-\sqcup T_+}(E)$ which agrees with the one on $\unzip_{T_-\sqcup T_+}(E)\smallsetminus \Link_{T_-\sqcup T_+}(E)$ constructed in the paragraph above.

Flowing by the vector field on $\unzip_{T_-\sqcup T_+}(E)$ just constructed gives a partially defined conically smooth map
\[
\gamma\colon \unzip_{T_-\sqcup T_+}(E)\times \DD^1~\dashrightarrow~ \unzip_{T_-\sqcup T_+}(E)\times\DD^1
\]
over $T\times \DD^1$ which, upon applying $\exit(-)$, lies over $\vfr$.
This map $\gamma$ extends the map $\alpha$ above, and 
\hl{\bf 
has the property that its restriction}
\[
\mathcolorbox{yellow}{
\gamma_|\colon \Link_{T_-}(E)\times\DD^1\longrightarrow \unzip_{T_-\sqcup T_+}(E)
}
\]  
\hl{\bf is defined and is an isomorphism over $T$} and, upon applying $\exit(-)$, it lies over $\vfr$.
In particular, we recognize an isomorphism of stratified spaces $\sS^{\sf fib, \fr}\bigl(\Link_{T_-}(E)\bigr) \cong E$ over $T$ which, upon applying $\exit(-)$, lies over $\vfr$.
The result follows from from the identification of fiberwise vari-framed constructible bundles $\Link_{T_-}(E) \simeq L$ over $T$.
\end{proof}

\subsection{The incorrect step in Lemma \ref{autos-suspend}}
Flowing by the specified vector field on $\unzip_{T_-\sqcup T_+}(E)$ gives a partially defined continuous map
\[
\gamma\colon \unzip_{T_-\sqcup T_+}(E)\times \RR~\dashrightarrow~ \unzip_{T_-\sqcup T_+}(E)\times\RR
\]
over $T\times \RR$.
It restricts to a partially defined continuous map
\begin{equation}\label{e1}
\gamma_|\colon \Link_{T_-}(E)\times [0,\infty)\longrightarrow \unzip_{T_-\sqcup T_+}(E)
\end{equation}
over $T$.
This map further restricts to a conically smooth isomorphism over $T\times \DD^1$,
\[
\gamma_|\colon \Link_{T_-}(E)\times \DD^1\longrightarrow \unzip_{T_-\sqcup T_+}(E)
\]
if and only if, for each $x\in \Link_{T_-}(E)$, the maximal domain of definition of the partially defined map $\gamma(x,-)\colon [0,\infty) \to \unzip_{T_-\sqcup T_+}(E)$ is $[0,a]$ for some $a\in (0,\infty)$.  
This condition is ensured if $\dim(L)\leq 1$.  
However, if $\dim(L)>1$, this condition might not be satisfied.
For instance, in the case that $T=\ast$ and $L=\DD^2$ so that $E=\sS^{\fr}(\DD^2)=\DD^3$, it is possible to choose a path of vari-framings on $\DD^3$ from the standard one to one for which, for some $x\in \Link_{T_-}(E)\cong \DD^2$, the resulting flow
\[
\gamma(x,-)\colon [0,\infty) \longrightarrow \unzip_{T_-\sqcup T_+}(E)\cong \DD^2\times \DD^1
\]
is defined for all non-negative time.  
\\

\subsection{Counterexample to the statement of Lemma~\ref{autos-suspend}}
We show that the statement of Lemma~\ref{autos-suspend} is not true in the case that $L=\DD^2$.  
For this, we first give a limit expression for the space of vari-framings on a stratified space. 
In summary, we will show
\[
\pi_1\bigl(\Aut^{\vfr}(\DD^3) \bigr)
~\cong~
\ZZ/2\ZZ
\qquad \text{ yet }\qquad
\pi_1\bigl(\Aut^{\vfr}(\DD^2) \bigr) = 0
~.
\]
This calculation also leads to the conclusion that the functor
\[
\lag - \rag \colon
\btheta_n^{\op}
\longrightarrow
\bcD_n^{\vfr}
\]
is not fully faithful for $n \geq 3$, since $\Aut_{\btheta_n}(c_n) \simeq \ast$ for all $n\geq 0$.
\\

\noindent
Let $X$ be a stratified space.
The space $\vfr(X)$ of vari-framings on $X$ is the space of lifts:
\[
\xymatrix{
&&
\ZZ_{\geq 0} \ar[d]^-{i\mapsto \RR^i}
\\
\exit(X) \ar[rr]^-{\sT X} \ar@{-->}[urr]
&&
\Vect^{\sf inj}
~.
}
\]
This space of lifts is a torsor for the limit of the diagram of loop spaces:
\[
\sd\bigl( \exit(X) \bigr)^{\op}
\longrightarrow
\sd(\ZZ_{\geq 0})^{\op}
\xra{~\{i_0<\dots<i_p\}\mapsto \sO(i_0)\times \prod_{0<r\leq p} \sO(i_r - i_{r-1})~}
\Spaces_\Omega
\]
indexed by the subdivision of the exit-path $\oo$-category of $X$.
If there exists vari-framing $\varphi\in \vfr(X)$, then $\varphi$ determines an equivalence with the limit
\[
\vfr(X)
~\simeq~
\limit\Bigl(
\sd\bigl( \exit(X) \bigr)^{\op}
\longrightarrow
\sd(\ZZ_{\geq 0})^{\op}
\xra{~\{i_0<\dots<i_p\}\mapsto \sO(i_0)\times \prod_{0<r\leq p} \sO(i_r - i_{r-1})~}
\Spaces
\Bigr)
~.
\]
Here are some identifications of this limit.
\begin{itemize}
\item
For $X=\DD^1$, this limit is
\[
\vfr(\DD^1)
~\simeq~\sO(1)
~.
\]
In particular, the space $\vfr(\DD^1)$ is a finite $0$-type.

\item
For $X=\DD^2$, this limit is equivalent to the pullback
\[
\vfr(\DD^2)
~\simeq~
\limit\Bigl(
\sO(1)\times\sO(1)
\xra{~\oplus~}
\sO(2)
\xla{~\oplus~}
\sO(1)\times\sO(1)
\Bigr)
~.
\]
So the space $\vfr(\DD^2)$ is a $0$-type, and its set of path- omponents is countably infinite.

\item
For $X=\DD^3$, the space $\vfr(\DD^3)$ is equivalent to the limit of the diagram:
\[
\xymatrix{
\sO(1)\times\sO(1)\times \sO(1) \ar[d]_-{{\sf id} \times \oplus} \ar[rr]^-{\oplus\times {\sf id}}
&&
\sO(2)\times \sO(1) \ar[d]_-{\oplus} 
&&
\sO(1)\times \sO(1)\times \sO(1) \ar[ll]_-{\oplus\times {\sf id}} \ar[d]_-{{\sf id} \times \oplus}
\\
\sO(1)\times\sO(2) \ar[rr]^-{\oplus} 
&&
\sO(3)
&&
\sO(1)\times \sO(2) \ar[ll]_-{\oplus}
\\
\sO(1)\times\sO(1)\times \sO(1) \ar[u]^-{{\sf id} \times \oplus} \ar[rr]^-{\oplus\times {\sf id}}
&&
\sO(2)\times \sO(1) \ar[u]^-{\oplus} 
&&
\sO(1)\times \sO(1)\times \sO(1) \ar[ll]_-{\oplus\times {\sf id}} \ar[u]^-{{\sf id} \times \oplus}
.
}
\]
So each path-component of the space $\vfr(\DD^3)$ is equivalent with the space $\Omega^3_0 \sO(3)$, the space of null-homotopic based maps from $S^3$ to $\sO(3)$.  
In particular, the homotopy group $\pi_\ast\bigl(\vfr(\DD^3)\bigr)$ is nontrivial for infinitely many values of $\ast$; for instance, $\pi_1\bigl(\vfr(\DD^3)\bigr) \cong \ZZ/2\ZZ$.  
\\

\end{itemize}

\noindent
Now, each automorphism of a stratified space $X=(X\to P)$ lies over an automorphism of the stratifying poset $P$: there is a canonical homomorphism
\begin{equation}\label{e2}
\Aut(X)
\longrightarrow
\Aut(P)
~.
\end{equation}
Consider the hemispherically stratified $n$-disk: $X=\DD^n=\bigl(\DD^n \to P_{\DD^n})$ where the stratifying poset $P_{\DD^n}$ is the right-cone of the poset structure on the set $\{\pm\}\times [n-1]$ that declares $(\sigma,i)<(\tau,j)$ to mean $i<j$.  
There is a canonical isomorphism between groups:
\begin{equation}\label{e5}
\sO(1)^{\times n}
~\cong~
\Aut(P_{\DD^n})
~.
\end{equation}
The inclusion, $\Aut_0(\DD^n) \xra{\simeq} \Aut(\DD^n)$, of those automorphisms that preserve the origin $0\in \DD^n$, is an equivalence.
Consider the resulting restriction homomorphism
\[
\Aut(\DD^n) ~\simeq~
\Aut_0(\DD^n)
\longrightarrow
\Aut(\DD^n\smallsetminus\{0\})
~.
\]
Consider the commutative diagram:
\[
\xymatrix{
\Diff_{\sf c}(\RR^n) \ar[rr] \ar[d]
&&
\Aut(\DD^n) \ar[rr]^-{(\ref{e2})}_-{\rm (a)} \ar[d]
&&
\Aut(P_{\DD^n}) 
\overset{(\ref{e5})}\cong
\sO(1)^{\times n}
\ar[d]^-{\cong}
\\
\ast \ar[rr]
&&
\Aut(\DD^n \smallsetminus\{0\}) \ar[rr]^-{(\ref{e2})}_{\rm (b)}
&&
\Aut(P_{\DD^n\smallsetminus\{0\}})
.
}
\]
By inspection, the right vertical arrow is an isomorphism, as indicated.  
Through an application of the isotopy-extension theorem, the left square is a pullback.  
For $n\leq 3$, the topological group $\Diff_{\sf c}(\RR^n)$ is contractible. The $n=3$ case, follows from the equivalence $\Diff(D^3)\simeq \Diff(S^2) \simeq \sO(3)$, which is due to Smale \cite{smale}.
It follows that the homomorphism~${\rm (a)}$ is an equivalence for this range of low dimensions:
\begin{equation}\label{e3}
\Aut(\DD^n)
\xra{~\simeq~}
\Aut(P_n) \cong \sO(1)^{\times n}
\qquad
\text{( $n\leq 3$ )}
~.
\end{equation}
\\

\noindent
By definition of the $\infty$-category $\Bun^{\vfr}$, automorphisms in $\Bun^{\vfr}$ of the object $(X,\varphi)\in \Bun^{\vfr}$ is the pullback in spaces
\[
\xymatrix{
\Aut^{\vfr}(X,\phi)
\ar[rr] \ar[d]
&&
\Aut(X) \ar[d]^-{\rm Orbit_\varphi}
\\
\ast \ar[rr]^{\lag \varphi \rag}
&&
\vfr(X)
.
}
\]
\begin{itemize}
\item
In the case that $X=\DD^1$, this pullback reveals that $\Aut^{\vfr}(\DD^1)\simeq \ast$ is contractible.  

\item
In the case $X=\DD^2$, there is a pullback diagram among spaces:
\[
\xymatrix{
\Aut^{\vfr}(\DD^2)
\ar[rr] \ar[d]
&&
\sO(1)^{\times 2} \ar[d]^-{\rm diagonal}
\\
\ast \ar[rr]^-{\lag \uno \rag}
&&
\bigl( \sO(1)^{\times 2} \underset{\sO(2)}\times \sO(1)^{\times 2} \bigr)
.
}
\]
We conclude that $\Aut^{\vfr}(\DD^2)\simeq \ast$ is contractible.

\item
In the case $X=\DD^3$, there is a pullback diagram among spaces:
\[
\xymatrix{
\Aut^{\vfr}(\DD^3)
\ar[rr] \ar[d]
&&
\sO(1)^{\times 3} \ar[d]^-{\rm diagonal}
\\
\ast \ar[rr]^-{\lag \varphi \rag}
&&
\vfr(\DD^3)
.
}
\]
After the description of $\vfr(\DD^3)$ above, and the identifications of $\vfr(\DD^2)$ and $\Aut(\DD^2)$ above, we conclude that $\Aut^{\vfr}(\DD^3)$ is not contractible: for each positive value of $\ast$, there is an isomorphism
\[
\pi_{\ast+4}(S^3)
~\cong~
\pi_\ast\bigl(\Aut_{\Bun^{\vfr}}(\DD^3) \bigr)
~.
\]
Therefore, the homomorphism
\[
\sS^{\fr}
\colon \Aut^{\vfr}(\DD^2)
\longrightarrow
\Aut^{\vfr}(\DD^3)
\qquad
\qquad
\text{( using that $\DD^3 = \sS^{\fr}(\DD^2)$ )}
\]
is \emph{not} an equivalence.

\end{itemize}

\end{document}